\documentclass[11pt,letterpaper]{amsart}

\usepackage{pstricks}

\usepackage[T1]{fontenc}
\usepackage[latin9]{inputenc}
\usepackage{geometry}
\usepackage{wrapfig}
\usepackage{multicol}
\usepackage{enumerate}   
\usepackage{graphicx}
\usepackage{soul}
\usepackage{xcolor}
\usepackage{amssymb}
\usepackage{relsize}

\newtheorem{theorem}{Theorem}[section]
\newtheorem{lemma}[theorem]{Lemma}
\newtheorem{proposition}[theorem]{Proposition}

{ \theoremstyle{definition}
\newtheorem{definition}[theorem]{Definition}}
{ \theoremstyle{remark}
\newtheorem{remark}[theorem]{Remark}}
\usepackage{placeins}
\usepackage{cite}
\usepackage{caption}
\usepackage{enumerate}
\usepackage{afterpage}
\usepackage{enumitem}
\usepackage{bmpsize}
\usepackage{hyperref}
\usepackage{tabu}
\usepackage{enumitem}   
\numberwithin{equation}{section}
\usepackage{stmaryrd}
\usepackage{tikz}
\usetikzlibrary{matrix,graphs,arrows,positioning,calc,decorations.markings,shapes.symbols}
\definecolor{dullmagenta}{rgb}{0.4,0,0.4}   % #660066
\definecolor{darkblue}{rgb}{0,0,0.4}

\def\i{ \mathsf{i}}

\def\sfU { \mathsf{U}}
\def\sfV{\mathsf{V}}

\def\P {{\mathbb{P}^{\theta, K}_{\infty}}}
\def\PN {{\mathbb{P}^{\theta, K}_{N}}}

\def\I {{I}}

\def\G {\mathsf{G}}
\def\bX {{\bf X}}
\def\U {U_{K}^{n, k}}

\def\XXO {{\mathfrak{X}^{\theta,K}_{N}}}
\def\XXI {{\mathfrak{X}^{\theta,K}_{\infty}}}

\def\Re{\mathsf{Re}}
\def\Im{\mathsf{Im}}
\def\Df{\mathsf{D}}

\psset{unit=1pt}
\psset{arrowsize=4pt 1}
\psset{linewidth=1pt}
\newrgbcolor{mygreen}{.2 .7 .2}
\newrgbcolor{myred}{.7 .3 .2}
\newgray{mygray}{.90}
\countdef\x=23
\countdef\y=24
\countdef\z=25
\countdef\t=26

\def\tbox(#1,#2)#3{
\x=#1 \y=#2 
\multiply\x by 12 
\multiply\y by 12 
\z=\x \t=\y
\advance\z by 12 
\advance\t by 12 
\psline(\x,\y)(\x,\t)(\z,\t)(\z,\y)(\x,\y)
\advance\x by 6
\advance\y by 6 
\rput(\x,\y){{\bf #3}}}

\def\gbox(#1,#2)#3{
\x=#1 \y=#2 
\multiply\x by 12 
\multiply\y by 12 
\z=\x \t=\y
\advance\z by 12 
\advance\t by 12 
\psline[linecolor=gray, linewidth=0.5pt](\x,\y)(\x,\t)(\z,\t)(\z,\y)(\x,\y)
\advance\x by 6
\advance\y by 6 
\rput(\x,\y){\gray{#3}}}

\def\ebox(#1,#2)#3{
\x=#1 \y=#2 
\multiply\x by 12 
\multiply\y by 12 
\advance\x by 6
\advance\y by 6 
\rput(\x,\y){#3}}

\def\ggbox(#1,#2){
\x=#1 \y=#2 
\multiply\x by 12 
\multiply\y by 12 
\z=\x \t=\y
\advance\z by 12 
\advance\t by 12 
\psframe[fillstyle=solid, fillcolor=mygray, linewidth=0pt](\x,\y)(\z,\t)
\psline[linecolor=gray, linewidth=0.5pt](\x,\y)(\x,\t)(\z,\t)(\z,\y)(\x,\y)}

\def\tline(#1,#2)(#3,#4){
\x=#1 \y=#2 \z=#3 \t=#4
\multiply\x by 12 
\multiply\y by 12 
\multiply\z by 12 
\multiply\t by 12 
\psline(\x,\y)(\z,\t)}

\title[$\beta$-Krawtchouk corners processes and multi-level loop equations]{Global asymptotics for $\beta$-Krawtchouk corners processes via multi-level loop equations}
\date{\today}
\author{Evgeni Dimitrov and Alisa Knizel}

\begin{document}

\maketitle 

\begin{abstract}
We introduce a two-parameter family of probability distributions, indexed by $\beta/2 = \theta > 0$ and $K \in \mathbb{Z}_{\geq 0}$, that are called $\beta$-Krawtchouk corners processes. These measures are related to Jack symmetric functions, and can be thought of as integrable discretizations of $\beta$-corners processes from random matrix theory, or alternatively as non-determinantal measures on lozenge tilings of infinite domains. We show that as $K$ tends to infinity the height function of these models concentrates around an explicit limit shape, and prove that its fluctuations are asymptotically described by a pull-back of the Gaussian free field, which agrees with the one for Wigner matrices. The main tools we use to establish our results are certain multi-level loop equations introduced in our earlier work \cite{DK21}.
\end{abstract}

\tableofcontents

%-------------------------------------------------------------------------------------------------------------------------------------------------------------------------------------------------
% Section 1
%
%-------------------------------------------------------------------------------------------------------------------------------------------------------------------------------------------------
\section{Introduction and main result}\label{Section1}

%-------------------------------------------------------------------------------------------------------------------------------------------------------------------------------------------------
% Section 1.0
%
%-------------------------------------------------------------------------------------------------------------------------------------------------------------------------------------------------
\subsection{Preface}\label{Section1.0} In the paper we consider a two-parameter family of probability measures $\P$, which are indexed by $\beta/2 = \theta > 0$ and $K \in \mathbb{Z}_{\geq 0}$. We call these measures {\em $\beta$-Krawtchouk corners processes}, and they form a special subclass of {\em discrete $\beta$-corners processes}, which we introduced in \cite{DK21} as integrable discretizations of the $\beta$-corners processes from random matrix theory. The goal of this paper is to describe the global behavior of these models as $K \rightarrow \infty$. The tools we use to study $\P$ are certain {\em multi-level loop equations} that were developed in \cite{DK21}, which generalize the two-level loop equations from \cite{DK2020} and the single-level loop equations from \cite{BGG}, and which originate from the works of Nekrasov and his collaborators \cite{NPS, NS, N}. Loop equations (also known as Schwinger-Dyson equations) have proved to be a very efficient tool in the study of global fluctuations of log-gases and random matrices, see \cite{BoGu, BoGu2, JL, KS,S} and the references therein. It is worth mentioning that loop equations were introduced and have been widely used in the physics literature, cf. \cite{AM90, Ey1, EyCh, MI83} and the references therein.

Part of our motivation for studying $\P$ is to showcase our multi-level loop equations and explain how they can be used to obtain meaningful asymptotic results. That being said, we believe that the $\beta$-Krawtchouk corners processes are interesting models in their own right. They sit on the interface of $\beta$-corners processes (and more generally random matrix theory), {\em Macdonald processes} (and more generally integrable probability), and random lozenge tilings, and at least when $\theta = 1$ bear strong connections to orthogonal polynomials. In the present paper we have opted to introduce and motivate the models from the lozenge tiling perspective initially when $\theta = 1$ -- this is done in Section \ref{Section1.1}. In Section \ref{Section1.2} we introduce the class of models for general $\theta >0$ and explain some of its symmetric function origin, and connections to representation theory. Section \ref{Section1.3} contains the main results we establish about the $\beta$-Krawtchouk corners processes as well as a general discussion of how we use the multi-level loop equations. In Section \ref{Section1.4} we give an outline of the paper and end with some acknowledgments.

%-------------------------------------------------------------------------------------------------------------------------------------------------------------------------------------------------
% Section 1.1
%
%-------------------------------------------------------------------------------------------------------------------------------------------------------------------------------------------------
\subsection{Lozenge tilings}\label{Section1.1} Fix a polygonal domain in the triangular grid. By gluing together two triangles along a common side one can form three types of rhombuses, or {\em lozenges}, and one can consider a uniformly sampled covering of the domain by such lozenges. The latter model is called a {\em uniform lozenge tiling} of the domain and it satisfies the following {\em tiling Gibbs property}: if we fix a sub-domain as well as the tiling outside of this sub-domain, then the conditional distribution of the tiling of the sub-domain is uniformly distributed over all possible tilings. See Figure \ref{S1_1}.

\begin{figure}[h]
\centering
\scalebox{0.25}{\includegraphics{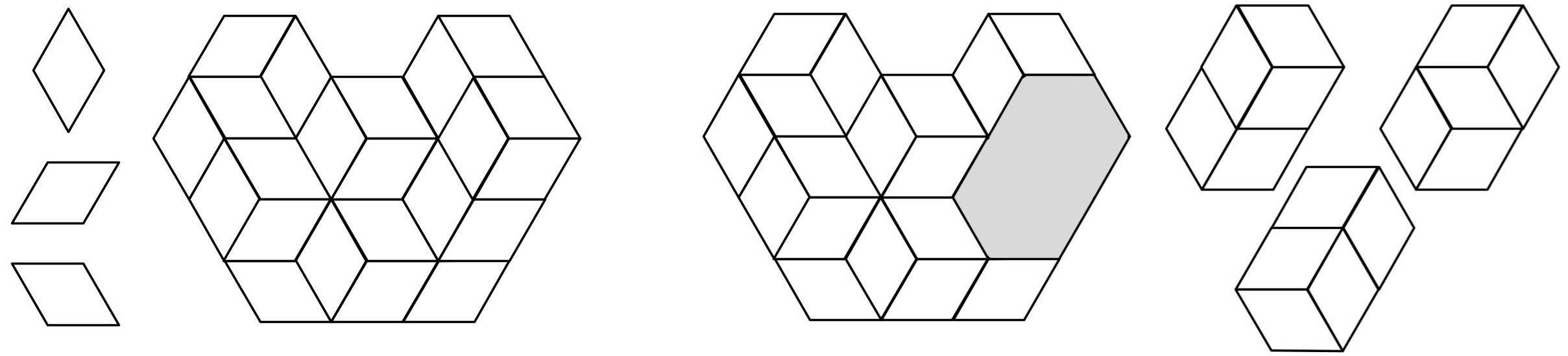}}
\caption{The left part depicts the three lozenges and a tiling of a domain. The gray region is a sub-domain and the conditional distribution of its tiling is uniform over all three tilings shown on the right. }
\label{S1_1}
\end{figure}
The core question one is interested in understanding is what happens to a random lozenge tiling as we increase the size of the domain in some regular way. Part of the interest in lozenge tilings comes from their exact solvability. Indeed, for such models (and more generally for planar {\em dimer models}) many quantities (such as partition and correlation functions) can be expressed as determinants of an inverse Kasteleyn matrix. One can then analyze these determinants in order to either predict, or in some cases rigorously prove, universal physical phenomena in the large scale limit. We refer the interested reader to the surveys of Kenyon \cite{Kenyon00, Kenyon09} for background on dimer models and to the book by Gorin \cite{GorBook} for a comprehensive introduction to random lozenge tilings.

A general feature of lozenge tiling models is that as the domain increases, the tiling develops {\em frozen facets} or {regions} where asymptotically only one type of lozenge remains present, as well as {\em liquid regions} where all three types of lozenges are asymptotically present. The frozen and liquid regions are separated by what are called {\em arctic boundaries} or {\em curves}, which for various models are algebraic curves that can be computed explicitly, see \cite{ADPZ, KOS, KO07}. The presence of liquid and frozen regions gives rise to rich asymptotic behavior for random lozenge tiling models (and more generally dimer models), and depending on the scaling one performs various universal scaling limits arise. Below we summarize some of asymptotic results available for dimer models and random lozenge tilings. We mention that the literature on dimer models, or even specifically lozenge tilings, is vast and so we do not attempt to give a full account. What we have tried to do is to include some of the earliest works establishing a particular kind of an asymptotic result, as well as a few more recent and/or more comprehensive results. The interested reader is referred to the introductions of the more recent papers we include for detailed explanations of the results, and a more comprehensive historical background on the works available for each scaling regime. For a friendly introduction to and overview of many of these results we refer to \cite{GorBook}.
\vspace{3mm}

{\raggedleft {\em Global laws of large numbers.}} A natural way to encode a tiling model is through a {\em height function}, which is an integer-valued random function. One way to define such a function, is by insisting that its value at any point is equal to the number of vertical lozenges $\lozenge$ to the right of that point. We mention that different papers define the height function differently, but they are all related through some deterministic transformations. One of the first results established for dimer models, is a {\em limit shape phenomenon}, which states that the height function (after suitable normalization) concentrates around a global limit, called the {\em limit shape}. This was first established for domino tilings for the Aztec diamond (a certain dimer model that is a cousin of lozenge tilings) in \cite{CEP}. Later, this result was substantially generalized by \cite{CKP}, which expressed the limit shape as the unique maximizer of a certain explicit concave surface tension functional. For lozenge tilings, this was rewritten by \cite{KO07} as the solution to the complex Burgers equation. A general feature of the limit shape of tiling models is that it is curved in the liquid region and flat in frozen regions.
\vspace{3mm}

{\raggedleft {\em Global central limit theorems.}} Once we know that the height function concentrates around a limit shape, it is natural to ask what the fluctuations around this limit shape look like. As was conjectured in \cite{KO07}, see also \cite[Lectures 11-12]{GorBook}, it is believed that the fluctuations are described by a suitable pull-back of the {\em Gaussian free field} (GFF) on $\mathbb{H}$ with Dirichlet boundary condition. The nature of the pull-back is dependent on the domain that is being tiled and for certain models is known explicitly, although its general form remains unknown. Some of the earlier works that establish convergence to the GFF for uniform lozenge tilings include \cite{Kenyon08} and \cite{BorFer}. More recent results include \cite{P1, BufGor18}.
\vspace{3mm}

{\raggedleft {\em Other asymptotics.}} In the bulk of the liquid region lozenge tilings converge to certain {\em ergodic Gibbs translation-invariant} (EGTI) measures on tilings of the whole plane, which were classified in \cite{Sheff05}, see\cite{KOS,Aggarwal19}. Near the arctic boundary, which separates the liquid from the frozen regions, the asymptotic behavior of random lozenge tilings is described by {\em Kardar-Parisi-Zhang} (KPZ) statistics. Some of the early works that establish edge asymptotics for uniform tiling models include domino tilings of the Aztec diamond \cite{Joh05} and lozenge tilings of the hexagon \cite{BKMM}, while for a recent comprehensive treatment of lozenge tilings of general polygonal domains we refer to the two-part paper \cite{Huang24, AggHuang21}. Near a {\em turning point} a random lozenge tiling (under appropriate scaling) converges to the {\em GUE corners process} or {\em minors process}, which is the joint law of the eigenvalues of a random matrix from the Gaussian universality ensemble (GUE) and its principal top-left corners, see \cite{Bar,Ne}. Some of the early works that establish convergence of lozenge tilings near turning points include \cite{JohNor, Nor, OR06}, while for a recent comprehensive treatment of lozenge tilings of general polygonal domains we refer to \cite{AggGorin}. Near a {\em cusp} of the arctic boundary one observes the Pearcey process, see \cite{OR07} and \cite{HYZ}. Near {\em cuspidal turning points} instead of the Pearcey process one observes the {\em cusp-Airy process}, see \cite{OR06, DJM16}. \\

\begin{figure}[h]
\centering
\scalebox{0.43}{\includegraphics{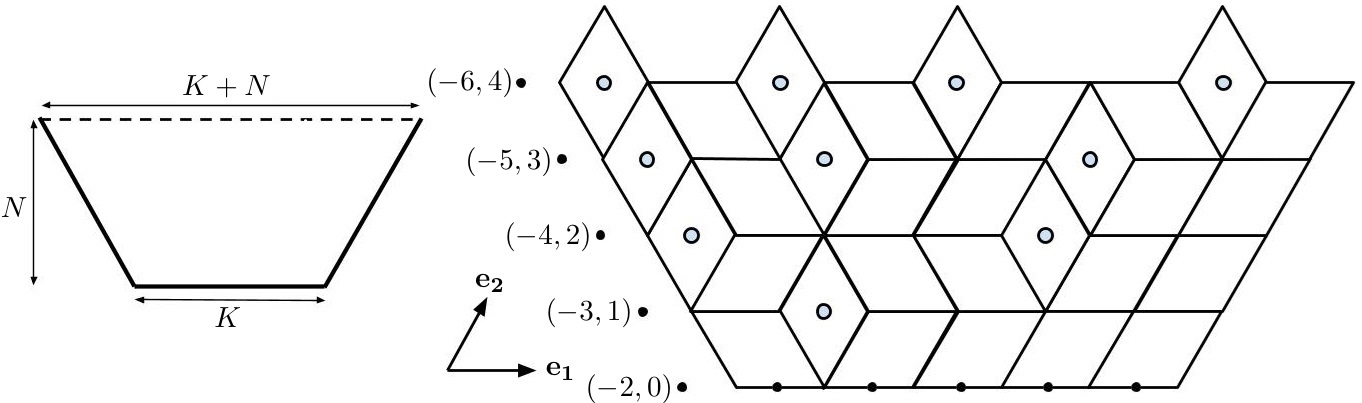}}
\caption{The left part depicts a trapezoidal domain. The right part depicts a random tiling of a trapezoid with base $K =5$ and height $N = 4$. For the right side, the particle locations on the fourth level are $\ell_1^4 = 2$, $\ell_2^4 = -1$, $\ell_3^4 = -3$ and $\ell^4_4 = -5$.}
\label{S1_2}
\end{figure}
In the present paper we deal with lozenge tilings of {\em trapezoidal domains}, also sometimes called {\em half-hexagons}. These are domains that depend on two parameters $K \in \mathbb{Z}_{\geq 0}$ (which is length of the short base of the trapezoid) and $N \in \mathbb{N}$ (which is the height of the trapezoid), see the left part of Figure \ref{S1_2}. For such domains the top boundary is special and we allow vertical lozenges $\lozenge$ to stick out from it, see the right part of Figure \ref{S1_2}. From the combinatorics of the model, one has that any tiling has precisely $k$ vertical lozenges on the $k$-th level of the model, and so there are precisely $N$ lozenges that stick out from the top. We place a particle in the middle of each vertical lozenge, and denote the locations of the particles on level $k$ by $(\ell_1^k, \dots, \ell^k_k)$, where $\ell_1^k > \ell_2^k > \cdots > \ell_k^k$. Another consequence of the combinatorics of the model is that if one looks at the lozenges of two consecutive levels, then they need to {\em interlace} in the sense that for $i = 1, \dots, N-1$ we have 
\begin{equation}\label{S1Interlace}
\ell^{i+1} \succeq \ell^i, \mbox{ which means } \lambda_1^{i+1} \geq \lambda_1^{i} \geq \lambda_2^{i+1} \geq \lambda_2^{i} \geq \cdots \geq \lambda^{i}_{i} \geq \lambda_{i+1}^{i+1},
\end{equation}
where we have written $\lambda^j_i = \ell_i^j + i$ for $j = 1, \dots, N$ and $i = 1, \dots, j$. 

For random tilings of trapezoids we allow the top level $(\ell_1^N, \dots, \ell^N_N)$ to have some distribution, but conditionally on the top level the remaining tiling needs to be uniformly distributed. In particular, the model still satisfies the tiling Gibbs property away from the top boundary. One readily observes that once the location of the vertical lozenges $\{\ell_i^j\}_{1 \leq i \leq j \leq N}$ are known the rest of the tiling can be uniquely reconstructed. Given this, one can view random lozenge tilings as measures on interlacing sequences of integer vectors, sometimes called {\em Gelfand-Tsetlin patterns}. Phrased in terms of $\{\ell_i^j\}_{1 \leq i \leq j \leq N}$ the tiling Gibbs property is seen to be equivalent to the statement that conditional on $(\ell^n_1, \dots, \ell^n_n)$ the distribution of $(\ell^1, \dots, \ell^{n-1})$ is uniform, subject to the interlacing conditions in (\ref{S1Interlace}), and is independent of $(\ell^{n+1}, \dots, \ell^{N})$. 

When $(\ell_1^N, \dots, \ell^N_N)$ are deterministic and converge to some limiting profile, random tilings of trapezoids have been previously studied using determinantal point processes in \cite{DJM16, DM15, DM18,DM20}, and independently in \cite{P1, P2}. They have also been studied using {\em Schur generating functions} in \cite{BufGor18, GorPan15}. In the present paper we consider the case then $\ell^N = (\ell_1^N, \dots, \ell^N_N)$ is {\em random} and has distribution
\begin{equation}\label{S1Krawt}
\mathbb{P}_{N}^K(\ell^N) \propto \prod_{1 \leq i < j \leq N} (\ell^N_i - \ell^N_j)^2 \cdot \prod_{i = 1}^N \binom{K+ N -1}{\ell_i^N + N},
\end{equation}
where $K-1 \geq \ell_1^N > \cdots > \ell_N^N \geq -N$. As before, $(\ell^1, \dots, \ell^{N-1})$ are uniformly distributed conditional on $\ell^N$ subject to interlacing. Using the Weyl dimension formula, the distribution of the whole vector $(\ell^1, \dots, \ell^N)$ becomes
\begin{equation}\label{S1Krawt2}
\mathbb{P}_{N}^K(\ell^1, \dots, \ell^N) \propto {\bf 1} \{\ell^1 \preceq \cdots \preceq \ell^N\} \cdot \prod_{1 \leq i < j \leq N} (\ell^N_i - \ell^N_j) \cdot \prod_{i = 1}^N \binom{K+ N -1}{\ell_i^N + N}.
\end{equation}

Let us briefly explain some of the properties of the measures in (\ref{S1Krawt2}) and why we are interested in them. Firstly, if one starts from (\ref{S1Krawt2}) and projects to $(\ell^1, \dots, \ell^n)$, where $1 \leq n \leq N$, then the pushforward measure is precisely of the form (\ref{S1Krawt2}) with $N$ replaced with $n$. We prove this statement more generally later in Lemma \ref{Consistent}. In particular, the measures in (\ref{S1Krawt2}) form a consistent family of measures (in $N$ for fixed $K$), and so there is a measure $\mathbb{P}_{\infty}^K$ on infinite sequences $(\ell^1, \ell^2, \dots)$ whose projection to any $(\ell^1, \dots, \ell^{N})$ is of the form (\ref{S1Krawt2}). The measures $\mathbb{P}_{\infty}^K$ form a one-parameter family of measures (indexed by $K$) that one can view as measures on lozenge tilings of the {\em vertically infinite} trapezoidal domain with base $K$. We mention that this measure is no longer ``uniform'', since there are infinitely many possible tilings of such a domain. Nevertheless, the conditional distribution of the tiling of any {\em finite} sub-domain is still uniform given the tiling outside of it, i.e. we still have the tiling Gibbs property from the beginning of this section. 

If one projects $\mathbb{P}_{\infty}^K$ to $\ell^N$, then its distribution (given by (\ref{S1Krawt})) upon shifting all entries by $N$ becomes the {\em Krawtchouk ensemble} from \cite[Section 5.1]{jo-en} with $p = q$. Consequently, all one-dimensional slices of $\mathbb{P}_{\infty}^K$ are given by Krawtchouk ensembles with different parameters that depend on the level, and $\mathbb{P}_{\infty}^K$ provides a coupling of all of these ensembles which satisfies the tiling Gibbs property, or equivalently the conditional uniformity of levels subject to the interlacing condition in (\ref{S1Interlace}). For this reason we refer to the measures $\mathbb{P}_{\infty}^K$ as {\em Krawtchouk corners processes}. We mention that the name of Krawtchouk ensembles comes from their connection to {\em Krawtchouk orthogonal polynomials}, see \cite{KLS}, and they occur in: the random domino tilings of the Aztec diamond and certain simplified directed first-passage percolation models \cite{jo-en}; stochastic systems of non-intersecting paths \cite{BBDT,KOR}; distributions on partitions arising from the representation theory of the symmetric group \cite{BO07}.

The term ``corners process'' comes from random matrix theory and we chose this name, since one can view $\mathbb{P}_{\infty}^K$ as an {\em integrable discretization} of  the GUE corners process that was mentioned earlier. Indeed, if we set $\ell_i^j = K/2 + (1/2) \cdot \sqrt{K} \cdot x_{j-i+1}^j$ in (\ref{S1Krawt2}) and take the $K \rightarrow \infty$ limit using the usual Stirling approximation argument from the classical DeMoivre-Laplace theorem we arrive at the density on $(x^1, \dots, x^N) \in \mathbb{R}^{N(N+1)/2}$, given by
\begin{equation}\label{S1GUECorners}
f(x^1, \dots, x^N) \propto {\bf 1} \{x^1 \preceq \cdots \preceq x^N\} \cdot \prod_{1 \leq i < j \leq N} (x^N_j - x^N_i) \cdot \prod_{i = 1}^Ne^{- (x^N_i)^2/2},
\end{equation}
which is precisely the GUE corners process, see \cite[Theorem 20.1]{GorBook}. We mention that for vectors $x^k\in \mathbb{R}^k$ and $x^{k+1} \in \mathbb{R}^{k+1}$ we have written $x^k \preceq x^{k+1}$ to mean $x_{k+1}^{k+1} \geq x^k_k \geq \cdots \geq x^k_1 \geq x^{k+1}_1$ (note that the indices are reversed compared to (\ref{S1Interlace}) so that $x_k^k$ is the largest and $x^k_1$ is the smallest  -- this is the usual convention from random matrix theory). We also mention that the last argument can be turned into a precise weak-convergence statement by adapting \cite[Proposition 4.3]{DK21}.

The usual way random tilings that satisfy the tiling Gibbs property arise in the literature is by taking a uniform tiling of a finite domain, and it is generally not easy to construct infinite random tilings with this property. Examples of such measures are the EGTI measures we mentioned earlier. Those measures are qualitatively different from the Krawtchouk corners processes $\mathbb{P}_{\infty}^K$ as they are translation-invariant and there is no boundary that the tiling is interacting with. Our measures are closer to the ones considered in \cite{BorFer}, as they describe random tilings of a sector in the plane that still has a boundary creating frozen facets, and have levels described by orthogonal polynomial ensembles (in \cite{BorFer} the authors get {\em Chariler ensembles} connected to {\em Charlier orthogonal polynomials}). The origin of the measures $\mathbb{P}_{N}^K$ in (\ref{S1Krawt2}) and the fact that they are consistent, comes from a connection to symmetric function theory, which we will elaborate once we introduce our general model in the next section. Here, we only mention that $\mathbb{P}_{N}^K$ arises as a special case of the {\em Schur processes} from  \cite{OR03}. This is also the case for the lozenge tilings in \cite{BorFer}, although in that paper the authors consider a {\em time evolution} of the tiling (i.e. there is an extra dimension).

%-------------------------------------------------------------------------------------------------------------------------------------------------------------------------------------------------
% Section 1.2
%
%-------------------------------------------------------------------------------------------------------------------------------------------------------------------------------------------------
\subsection{General model}\label{Section1.2}  In this section we introduce the general model that we study, which is a one-parameter generalization of the Krawtchouk corners process $\mathbb{P}_{\infty}^K$ from Section \ref{Section1.1}.

For $N \in \mathbb{N}$, $K \in \mathbb{Z}_{\geq 0}$ and $\beta/2 = \theta  > 0$  we define
\begin{equation}\label{S1GenState}
\begin{split}
&\Lambda^K_N = \{  (\lambda_1, \dots, \lambda_N) \in \mathbb{Z}^N : \lambda_1\geq  \lambda_2 \geq \cdots \geq \lambda_N,\mbox{ and }0  \leq \lambda_i  \leq K \}, \\
&\mathbb{W}^{\theta,K}_{k} = \{ (\ell_1, \dots, \ell_k):  \ell_i = \lambda_i - i \cdot \theta, \mbox{ with } (\lambda_1, \dots, \lambda_k) \in \Lambda^K_k\}, \\
&\XXO= \{ (\ell^1, \ell^{2}, \dots, \ell^{N}) \in \mathbb{W}^{\theta, K}_{1} \times \mathbb{W}^{\theta,K}_{2} \times \cdots \times \mathbb{W}^{\theta,K}_{N} : \ell^N \succeq \ell^{N-1} \succeq \cdots \succeq \ell^{1} \},
\end{split}
\end{equation}
where $\ell \succeq m$ means that if $ \ell_i = \lambda_i - i \cdot\theta$ and $m_i = \mu_i - i \cdot \theta$ then $\lambda_1 \geq \mu_1 \geq \lambda_2 \geq \mu_2 \geq \cdots \geq \mu_{N-1} \geq \lambda_N$. The set $\XXO$ is the state space of our point configuration $(\ell^1, \ell^{2}, \dots, \ell^{N})$. 

We consider measures on $\XXO$ of the form
\begin{equation}\label{S1PDef}
\PN(\ell^1, \dots, \ell^{N}) = \frac{1}{Z(\theta, K,N)} \cdot H_N^t(\ell^N) \cdot \prod_{j = 1}^{N-1}  \I(\ell^{j+1}, \ell^j), \mbox{ where }
\end{equation}
\begin{equation}\label{S1PDef2}
\begin{split}
\I(\ell^{j+1}, \ell^{j}) = & \prod_{1 \leq p < q \leq j+1}\frac{\Gamma(\ell^{j+1}_p - \ell^{j+1}_q + 1 - \theta)}{\Gamma(\ell^{j+1}_p - \ell^{j+1}_q) } \cdot \prod_{1 \leq p < q \leq j} \frac{\Gamma(\ell^{j}_p - \ell^{j}_q + 1)}{\Gamma(\ell^{j}_p - \ell^{j}_q + \theta)} \\
&\times\prod_{1 \leq p < q \leq j+1} \frac{\Gamma(\ell^{j}_p -
  \ell^{j+1}_q)}{ \Gamma(\ell^{j}_p - \ell^{j+1}_q + 1 - \theta)}  \cdot
\prod_{1 \leq p \leq q \leq j} \frac{\Gamma(\ell^{j+1}_p -
  \ell^{j}_q + \theta)}{\Gamma(\ell^{j+1}_p - \ell^{j}_q + 1)} ,
\end{split}
\end{equation}
\begin{equation}\label{S1PDef3}
\begin{split}
H_N^t(\ell^N) = \prod_{1 \leq p < q \leq N} &\frac{\Gamma(\ell^N_p - \ell^N_q + 1)}{\Gamma(\ell^N_p - \ell^N_q + 1 - \theta)} \cdot \prod_{p = 1}^N w^{\theta,K}_N(\ell^N_p).
\end{split}
\end{equation}
The constant $Z(\theta, K,N)$ in (\ref{S1PDef}) is a normalization constant and the weight $w_N^{\theta, K}$ in (\ref{S1PDef3}) is 
\begin{equation}\label{S1PDef4}
w^{\theta,K}_N(x) = \frac{1}{\Gamma(x +N\theta+1) \Gamma(K -x + 1- \theta) }.
\end{equation}
Notice that when $\theta = 1$, the measure $\PN$ becomes equivalent to (\ref{S1Krawt2}). Indeed, when $\theta = 1$ we have that $\I(\ell^{j+1}, \ell^{j}) = 1$ and $H^t_N(\ell^N)$ agrees with the product in (\ref{S1Krawt2}) upto a deterministic factor, which can be absorbed into the normalization constant.

We mention that unlike the $\theta = 1$ case, the particles $\ell_i^j$ no longer lie on $\mathbb{Z}$ but rather on the {\em shifted lattices} $\mathbb{Z} - i \theta$. We have chosen this convention in part to make the formulas in (\ref{S1PDef2}) and (\ref{S1PDef3}) more compact; however, one can easily go back to integer coordinates using the identities $ \ell^j_i = \lambda^j_i - i \cdot\theta$, which we will frequently use without mention.\\

The measures $\PN$ are special cases of discrete $\beta$-corners processes -- a class of models recently introduced in \cite{DK21} as integrable discretizations of $\beta$-corners processes from random matrix theory, such as the {\em Hermite}, {\em Laguerre} and {\em Jacobi} $\beta$-corners processes, see \cite{BGJ}. The word ``integrable'' reflects the symmetric function origin of the measures $\PN$, which we now explain briefly and discuss at length in Section \ref{Section2}. 

Starting from the measures $\mathbb{P}_N^K$ in (\ref{S1Krawt2}), one can replace the terms ${\bf 1}\{ \ell^i \preceq \ell^{i+1}\}$ with $s_{\lambda^{i+1}/ \lambda^i}(1)$, which is a {\em skew Schur polynomial} with a single variable equal to one. The terms in (\ref{S1PDef2}) are then obtained by replacing $s_{\lambda^{i+1}/ \lambda^i}(1)$ with $J_{\lambda^{i+1}/ \lambda^i}(1; \theta)$ which is a {\em skew Jack polynomial} with parameter $\theta$. We mention that this change affects the nature of the Gibbs property, which is no longer uniform subject to interlacing, but quite more involved and closer in form to the Gibbs property satisfied by $\beta$-corners processes. In \cite{GS} the authors referred to this new discrete Gibbs property as a {\em Jack-Gibbs property}. The lift from $\theta = 1$ to general $\theta > 0$, and correspondingly from Schur to Jack symmetric functions, is by now understood to be the ``correct'' way to obtain a generalization to arbitrary $\beta = \theta/2$. For example, in \cite{GS} the authors used Jack symmetric functions to construct general $\beta$-analogues of Dyson Brownian motion. 

For more background on models based on Jack symmetric functions and the connection to $\beta$-log gases (or $\beta$-ensembles) and $\beta$-corners processes we refer to \cite{CDM23, DK2020, DK21,GH22,huang}, and the references therein. Here, we mention that the interest in $\beta$-corners processes stems from their connection to random matrix theory, Macdonald processes \cite{BorCor}, and probability distributions on the irreducible representations of $U(N)$, see \cite[Section 3.2]{BufGor18} for the $\theta = 1$ case. In fact, the measures in (\ref{S1PDef}) are special cases of the ascending Macdonald processes with a pure $\beta$ specialization of length $K$, see Section \ref{Section2} for the details. The connection to Macdonald processes ensures that the measures $\PN$ are consistent, which gives rise to the following core definition of the paper.
\begin{definition}\label{BKCC} Fix $\theta > 0$ and $K \in \mathbb{Z}_{\geq 0}$. As will be shown in Lemma \ref{Consistent}, there exists a unique probability measure $\P$ on 
\begin{equation}\label{InfState}
\XXI = \{(\ell^1, \ell^2, \ell^3, \dots): \ell^{i+1} \succeq \ell^i \mbox{ for all $i \geq 1$}\},
\end{equation}
such that the pushforward of $\P$ onto $(\ell^1, \dots, \ell^N)$ is given by $\PN$ in (\ref{S1PDef}). We refer to the measure $\P$ as the {\em $\beta$-Krawtchouk corners process} with parameters $\beta = 2\theta$ and $K$.
\end{definition}
\begin{remark} Going back to the representation theory, one way to think about $\P$ is as a measure on the path-space of the {\em Gelfand-Tsetlin graph} or equivalently, as a {\em coherent system} of measures on the latter with Jack edge weights. When  $\theta = 1$ these systems are described in \cite{BO12} and for general $\theta > 0$ they implicitly appear in \cite{OO98}, and the measure $\P$ corresponds to setting $\beta_1^+ = \cdots = \beta_K^+ = 1$ and all other parameters to zero.
\end{remark}
\begin{remark}\label{DBE}  In Section \ref{Section2.3} we show that for each $n \in \mathbb{N}$ we have
\begin{equation}\label{S12E8}
\P(\ell^{n})  \propto \prod_{1 \leq i < j \leq n} \frac{\Gamma(\ell_i^{n} - \ell_j^n + 1 ) \Gamma(\ell_i^n - \ell_j^n  + \theta) }{\Gamma(\ell_i^{n} - \ell_j^n + 1 - \theta) \Gamma(\ell_i^n - \ell_j^n )} \cdot \prod_{i = 1}^n w_n^{\theta, K}(\ell_i^n),
\end{equation}
where $w_n^{\theta, K}$ is as in (\ref{S1PDef4}). Measures of the form (\ref{S12E8}) are called {\em discrete $\beta$-ensembles} and they were introduced and studied in \cite{BGG} for general weight functions. We refer to the measures in (\ref{S12E8}) as {\em $\beta$-Krawtchouk ensembles} and note that when $\theta = \beta/2 = 1$ they become (\ref{S1Krawt}). The measure $\P$ thus provides a natural coupling of infinitely many such ensembles, indexed by $n \in \mathbb{N}$.
\end{remark}

Starting from the next section we will almost exclusively talk about $\P$ as measures on $\XXI$ as in (\ref{InfState}), as opposed to lozenge tilings. Here, we mention that by placing vertical lozenges in locations $\{ (\lambda_i^j - i, i) : 1 \leq i \leq j < \infty \}$ in the coordinate system in Figure \ref{S1_2} one can alternatively view $\P$ as a {\em non free fermionic} measure on infinite lozenge tilings of the vertically infinite trapezoid of base $K$. The term ``non free fermionic'' reflects the fact that for general $\theta > 0$ (unlike $\theta = 1$) the measure $\P$ is not known to be determinantal. We finally mention that in principle one can lift our model higher in the hierarchy of symmetric functions and replace Jack with Macdonald symmetric functions in the definition of $\P$. In a different setting, lozenge tilings with semilocal interactions governed by Macdonald symmetric functions have been previously studied in \cite{Ahn20}.

%-------------------------------------------------------------------------------------------------------------------------------------------------------------------------------------------------
% Section 1.3
%
%-------------------------------------------------------------------------------------------------------------------------------------------------------------------------------------------------
\subsection{Main results and methods}\label{Section1.3} We now turn to our main results. They are formulated in terms of the {\em height function} $\mathcal{H}_K$, which we define presently. For $n \in \mathbb{N}$ and $i \in \llbracket 1, n \rrbracket$ we let
\begin{equation}\label{S1NewPart}
\tilde{\ell}_i^n = \ell_i^n - K/2 + \theta(n+1)/2,
\end{equation}
where two integers $p, q$ we write $\llbracket p, q \rrbracket = \{p, p+1, \dots, q\}$ with the convention that $\llbracket p,q \rrbracket = \emptyset$ if $q < p$. In words, $\tilde{\ell}^n_i$ are shifted versions of $\ell^n_i$ that are more symmetric. Specifically, note that from (\ref{S1GenState}) we have $\ell_i^n \in [-n \theta, K - \theta]$, while in view of (\ref{S1NewPart}) we have $\tilde{\ell}_i^n \in [- (n-1)\theta/2 - K/2, (n-1)\theta/2 + K/2]$, so that the range of $\tilde{\ell}_i^n$ is symmetric with respect to the origin.

With the above notation we define the height function $\mathcal{H}_K(x,s)$ for $(x,s) \in \mathbb{R} \times (0,\infty)$ as follows
\begin{equation}\label{S1DHF}
\mathcal{H}_K(x,s) = \sum_{i = 1}^n {\bf 1}\{\tilde{\ell}_i^n/K \geq x\}, \mbox{ where } n = \lceil s \cdot \theta^{-1} \cdot K \rceil.
\end{equation}
In words, $\mathcal{H}_K(x,s)$ counts the number of (scaled) particles to the right of location $x$. Note that since $\tilde{\ell}_i^n \in  [- (n-1)\theta/2 - K/2, (n-1)\theta/2 + K/2]$ for $i \in \llbracket 1, n\rrbracket$, we have that $\mathcal{H}_K(x,s) = n$ if $x \leq - (n-1)\theta/2 - K/2$ and $\mathcal{H}_K(x,s) = 0$ if $x > (n-1)\theta/2 + K/2$. 

Our first result shows that the scaled random height functions $K^{-1} \cdot \mathcal{H}_K(x,s)$ converge in probability to a {\em deterministic} height function $h(x,s)$ on $\mathbb{R} \times (0, \infty)$ that we proceed to describe. For $(x,s) \in \mathbb{R} \times (0, \infty)$ we define the function
\begin{equation}\label{S1Density}
\nu(x,s) = \begin{cases}  \mathlarger{ \frac{1}{\theta \pi} \cdot \arccos \left( \frac{1 - s}{\sqrt{(s+1)^2/4 - x^2} }  \right)} &\mbox{ if }  \mathlarger{ x \in \left(- \sqrt{s}, \sqrt{s} \right)}, \\ 
 \theta^{-1} \cdot {\bf 1} \{ s > 1 \} & \mbox{ if } x \in  \left(- \frac{s+1}{2} , \frac{s+1}{2} \right)  \setminus (-\sqrt{s}, \sqrt{s})              \\
 0 &   \mbox{ if }x \not \in \left(- \frac{s+1}{2} , \frac{s+1}{2} \right) .\end{cases}
\end{equation}
Using (\ref{S1Density}) we now define the deterministic height function
\begin{equation}\label{S1LimHeight}
h(x,s) = \int_x^{\infty} \nu(y,s) dy.
\end{equation}
With the above notation in place we can state our first result, which one can view as a global law of large numbers for the random height functions $\mathcal{H}_K(x,s)$.
\begin{theorem}\label{ThmLLN} Fix $\theta > 0$, assume that $(\ell^1, \ell^2, \dots)$ is distributed according to $\mathbb{P}^{\theta, K}_{\infty}$ as in Definition \ref{BKCC} and let $\mathcal{H}_K$ be as in (\ref{S1DHF}). Then, for any compact set $\mathcal{V} \subset \mathbb{R} \times (0,\infty)$ we have
\begin{equation}\label{S1LLNEq}
\lim_{K \rightarrow \infty} \sup_{(x,s) \in \mathcal{V}} \left| K^{-1} \cdot  \mathcal{H}_K(x,s) - h(x,s) \right| = 0,
\end{equation}
where $h(x,s)$ is as in (\ref{S1LimHeight}) and the convergence is in probability.
\end{theorem}

The second major result of the paper describes the asymptotic global fluctuations of the height function $\mathcal{H}_K(x,s)$ and identifies them with a suitable pull-back of the Gaussian free field. We introduce the domain in $\mathbb{R}^2$
\begin{equation}\label{S1Domain}
\mathcal{D} = \{(x,s) \in \mathbb{R}^2: s > 0, \sqrt{s} > x > - \sqrt{s} \},
\end{equation}
which is the liquid region in our model, i.e. the set in $\mathbb{R} \times (0, \infty)$ where $\nu(x,s)$ is not equal to zero or one, or equivalently where $h(x,s)$ is not straight. We also define the map $\Omega: \mathcal{D} \rightarrow \mathbb{H}$ via 
\begin{equation}\label{S1Map}
\Omega(x,s) = x + \i \cdot \sqrt{s - x^2},
\end{equation}
which in words maps horizontal lines in $\mathcal{D}$ to half-circles in $\mathbb{H}$, see Figure \ref{S1_3}.
\begin{figure}[h]
\centering
\scalebox{0.46}{\includegraphics{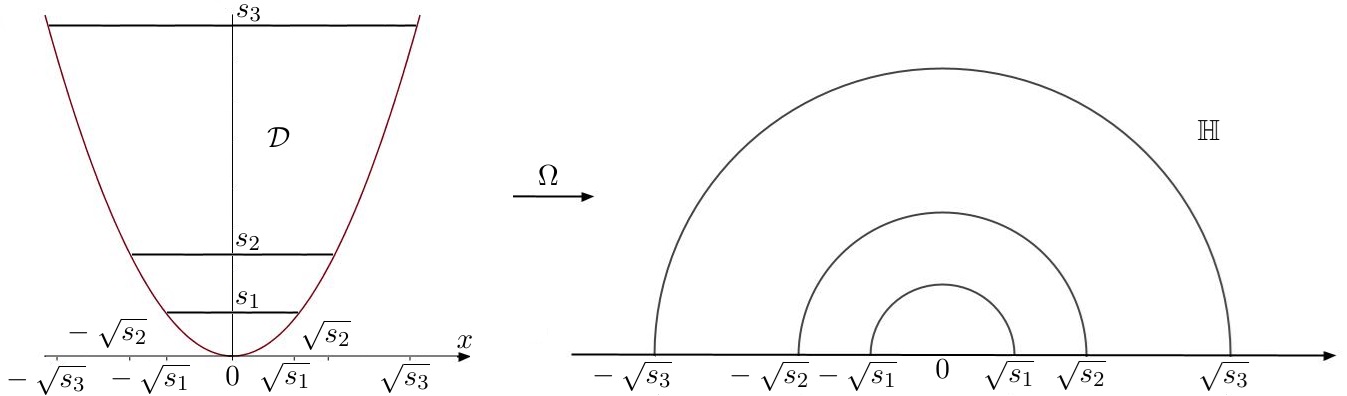}}
\caption{The liquid region $\mathcal{D}$ and the map $\Omega$ that sends it to $\mathbb{H}$. }
\label{S1_3}
\end{figure}
One readily verifies that $\Omega(x,s)$ defines a smooth bijection between $\mathcal{D}$ and $\mathbb{H}$ with inverse
\begin{equation}\label{S1MapInv}
\Omega^{-1}(z) = \left( \Re[z], |z|^2 \right).
\end{equation}

With the above notation in place we can state our second result, which one can view as a global central limit theorem for the random height functions $\mathcal{H}_K(x,s)$.
\begin{theorem}\label{ThmMain} Fix $\theta > 0$, assume that $(\ell^1, \ell^2, \dots)$ is distributed according to $\mathbb{P}^{\theta, K}_{\infty}$ as in Definition \ref{BKCC} and let $\mathcal{H}_K$ be as in (\ref{S1DHF}). Fix $m \in \mathbb{N}$, $s_1, \dots, s_m \in (0,\infty)$ and $m$ entire functions $f_1, \dots, f_m$, and consider the random variables
\begin{equation}\label{S1RVS}
X_i^K = \sqrt{\theta \pi} \cdot \int_{\mathbb{R}} \left(\mathcal{H}_K(x,s_i)  - \mathbb{E} \left[\mathcal{H}_K(x,s_i) \right] \right) \cdot f_i(x) dx \mbox{ for } i \in \llbracket 1, m \rrbracket.
\end{equation}
Then, as $K \rightarrow \infty$ the random vectors $X^K = (X^K_1, \dots, X^K_m)$ converge in the sense of moments to a Gaussian vector $(X_1, \dots, X_m)$ with mean zero and covariance
\begin{equation}\label{S1CovGFF}
\mathrm{Cov}(X_i, X_j) = \int_{-\sqrt{s_i} }^{\sqrt{s_i}} \int_{-\sqrt{s_j} }^{\sqrt{s_j}} f_i(x) f_j(y) \left( - \frac{1}{2\pi} \log \left| \frac{\Omega(x,s_i) - \Omega(y,s_j)  }{\Omega(x,s_i) - \overline{\Omega}(y,s_j) } \right|\right) dy dx,
\end{equation}
where $\Omega(x,s)$ is as in (\ref{S1Map}).
\end{theorem}
\begin{remark}\label{S1RemGFF} Let us explain the connection between Theorem \ref{ThmMain} and the GFF, see also Section \ref{Section8.1}. If $\mathcal{F}$ denotes the GFF on $\mathbb{H}$ with Dirichlet boundary conditions, then $\mathcal{F} \circ \Omega$ is a generalized centered Gaussian field on $\mathcal{D}$ and for entire functions $f_1, \dots, f_m$ the covariance of the random variables 
$$\int_{-\sqrt{s_i}}^{\sqrt{s_i}} \mathcal{F} \left( \Omega(x,s_i) \right) f_i(x) dx, \hspace{3mm} i \in \llbracket 1, m \rrbracket,$$
is precisely the right side of (\ref{S1CovGFF}). I.e. the natural $L^2$ pairings of the centered height functions $\mathcal{H}_K(x,s)$ in (\ref{S1RVS}) converge to those of $\mathcal{F} \circ \Omega$ in the sense of moments as $K \rightarrow \infty$.
\end{remark}
\begin{remark}\label{S1RemTheta}
We mention that in Theorem \ref{ThmMain} the parameter $\theta$ enters only as a simple factor in (\ref{S1RVS}). The pull-back map $\Omega$ and other expressions in the statement are independent of $\theta$. This is in strong agreement with the results in \cite{BGJ} for the $\beta$-Jacobi corners process and in \cite{HuangJGF} for certain Jack processes.
\end{remark}
\begin{remark}\label{S1RMTGFF} The pull-back map $\Omega$ in (\ref{S1Map}) is exactly the same as the one discovered for spectra of Wigner matrices in \cite{BorWig} upto a factor of $2$. This factor can be traced to the fact that the $N$ eigenvalues in \cite{BorWig} concentrate in $[-2 \sqrt{N}, 2 \sqrt{N}]$ as opposed to $[-\sqrt{N}, \sqrt{N}]$. Depending on the reader's background, this might be very surprising or totally expected. The reason this might be surprising is that Theorem \ref{ThmLLN} shows that the limit shape for our model is completely different than the one for Wigner matrices. Namely, the limiting particle density on level $s$ in our model is $s^{-1} \cdot  \nu(x,s)$ with $\nu(x,s)$ as in (\ref{S1Density}), which is not the same as Wigner's semicircle law. So even though the limit shapes are {\em different} the fluctuations are the {\em same}. The reason this might be expected, is that (as explained in Section \ref{Section1.1}) the Krawtchouk corners process is an integrable discretization of the GUE corners process, and the GUE is a special case of Wigner matrices. Therefore, the asymptotic fluctuations of the two models should agree.
\end{remark}

We next sketch our approach to proving Theorems \ref{ThmLLN} and \ref{ThmMain}. We focus on Theorem \ref{ThmMain} as it is the one containing most of the novelty. We consider the random analytic functions 
$$\mathsf{G}_{K}(z,s) = \sum_{i = 1}^n \frac{1}{z - \ell_i^n/K}, \mbox{ where } n = \lceil s K \rceil,$$
which are nothing but the Stieltjes transforms of the point processes formed by our scaled particle locations. The main technical result of the paper, Theorem \ref{MainTechThm}, shows that as $K \rightarrow \infty$ the centered functions $\mathsf{G}_{K}(z,s) $ converge to a centered Gaussian field (in $z$ and $s$) with an explicit covariance. One then deduces Theorem \ref{ThmMain} from Theorem \ref{MainTechThm} by a conceptually straightforward, although computationally involved, argument that occupies the majority of Section \ref{Section8}.

To prove Theorem \ref{MainTechThm} we need to show that all third and higher order joint cumulants of $\mathsf{G}_{K}(z,s) $ vanish, while the second order cumulant (or covariance) converges to some function as $K \rightarrow \infty$. For a very short introduction to the definition and basic properties of joint cumulants we refer the reader to Appendix \ref{AppendixA}. As mentioned a few times earlier, the main new tool we use to study the joint cumulants of $\mathsf{G}_{K}(z,s) $ are the multi-level loop equations from \cite{DK21}. These equations are quite involved so we forego stating them until the main text. In order to apply these equations to our model, we need to be able to linearize them, which requires that we obtain moment bounds for $\mathsf{G}_{K}(z,s)$ and for differences of two $\mathsf{G}_{K}(z,s) $'s at consecutive levels. These moment bounds are established in Propositions \ref{S33P1} and \ref{S43P1} by combining results from \cite{DD21}, the single-level loop equations from \cite{BGG} and the two-level loop equations from \cite{DK2020}. We mention that both the single and two-level loop equations are special cases of the multi-level loop equations and the reason we can apply them in our setting is that the restriction of $\P$ to each level is a discrete $\beta$-ensemble, see Remark \ref{DBE}.

Upon linearization, the single-level loop equations from \cite{BGG} state that 
\begin{equation}\label{S1SingleLevel}
M(\mathsf{G}_{K}(z_1,s_1),\mathsf{G}_{K}(z_2,s_1), \dots, \mathsf{G}_{K}(z_m,s_1) ) = [\mbox{Something explicit}] + [\mbox{Something small}],
\end{equation}
where the left side is the joint cumulant. When $m \geq 3$ the explicit part is zero, and when $m = 2$ it is the limiting covariance. In particular, for the single-level analysis the problem is essentially over once we linearize the loop equations. The situation is more involved for the multi-level analysis and after linearizing the equations and simplifying expressions we instead arrive at the statement that any subsequential limit $f_{\infty}(z,s)$ of 
$$M(\mathsf{G}_{K}(z_1,s_1),\mathsf{G}_{K}(z,s) )$$
satisfies the integral identity
\begin{equation}\label{S1PDE}
 \left[e^{\theta G\left(z/s, s\right)} - 1 \right] f_{\infty}(z,s) + \theta \cdot \int_{s}^{s_1}  \partial_{z} f_{\infty}(z,u)du -   \left[e^{\theta G\left(z/s, s\right)} - 1 \right] f_{\infty}(z,s)  = 0, 
\end{equation}
provided that $s_1 \geq s > 0$. Here, $G(z,s)$ is an explicit function (it is the Stieltjes transform of certain $s$-dependent measures that arise as limiting particle densities). Upon differentiation, (\ref{S1PDE}) becomes a first order PDE with a boundary condition given by the single-level covariance that we mentioned can be computed in (\ref{S1SingleLevel}). This equation can be solved using the method of characteristics and this is how we show the covariances converge (the characteristic curves are actually straight lines). 

We mention that the above argument is quite reductive and we have omitted many challenges along the way. For example, the existence of subsequential limits for the cumulants is not obvious and we establish it by leveraging our moment bounds from  Propositions \ref{S33P1} and \ref{S43P1} in conjunction with the multi-level loop equations. Using these inputs we are able to show a type of Lipschitz (or equicontinuity) property satisfied by the joint cumulants that ensures that subsequential limits exist, and they have sufficient regularity. In addition, when we deal with $m$-th order cumulants we get PDEs involving $2m-2$ variables, which makes the PDE non-linear for $m \geq 3$. The way we deal with this issue is by implementing an intricate inductive argument that allows us to deal at each step of the induction with only two variables, see the proof of Proposition \ref{S6Prop}.\\

In the remainder of this section we compare our approach to previous ones used to study related models. As mentioned earlier, when $\theta \neq 1$ one cannot use techniques based on determinantal point processes, nor the Schur generating functions in \cite{BufGor18,BufGor19}. One possible approach is based on the {\em Jack generating functions} in \cite{HuangJGF}, which have been used to prove GFF asymptotics in certain Markov processes based on Jack symmetric functions. The existing work in that direction is in a dynamical setting of $N$ walkers that evolve without crossing as opposed to corners processes. 

Another approach is based on the {\em dynamical loop equations} introduced in \cite{GH22}. The methods developed in \cite{GH22} are applied to study lozenge tilings on what are called {\em quadratic lattices} as opposed to the shifted integer lattices in our case, and their top level is deterministic while ours is random, which does play a role in their analysis (especially in identifying the GFF structure). That being said, the dynamical loop equations are quite general and can be applied to different models as explained in \cite[Section 3]{GH22}, and our model specifically satisfies \cite[Corollary 3.5]{GH22}. In addition, in the context of studying random tilings of domains with holes, \cite{BufGor19} were able to deal with the case of a random top level, so there is a hope that the dynamical loop equations can be used to study models like ours. We mention that in \cite{GH22} the authors also end up solving a PDE using the method of characteristics to compute the covariance; however, it is hard to make a direct comparison as they have a PDE for the complex slope while our PDE is for the joint cumulants of the Stieltjes transforms. At a very high level, our loop equations provide access to the ``level'' integrals of the Stieltjes transforms, while the dynamical loop equations describe how one level changes if the previous is fixed, i.e. they provide access to the ``level'' derivatives. In this sense, the two approaches are complementary to each other, and we are hopeful they can be used in tandem to study a variety models in the future.

%-------------------------------------------------------------------------------------------------------------------------------------------------------------------------------------------------
% Section 1.4
%
%-------------------------------------------------------------------------------------------------------------------------------------------------------------------------------------------------
\subsection{Paper outline}\label{Section1.4} In Section \ref{Section2} we introduce Jack processes, and explain how the measures $\PN$ from (\ref{S1PDef}) can be interpreted as special cases of them. We utilize this connection to show some projection properties of $\PN$, which imply these measures are consistent and each level is described by a discrete $\beta$-ensemble. In Section \ref{Section3} we introduce several functions that arise from the law of large numbers limit of different levels, and obtain bounds on the moments of certain analogues of $\mathsf{G}_{K}(z,s)$ from the introduction -- this is Proposition \ref{S33P1}. In Section \ref{Section4} we prove moment bounds for the differences of two Stieltjes transforms on consecutive levels -- this is Proposition \ref{S43P1}. In Proposition \ref{S52P1} we prove a large class of multi-level joint cumulant relations. The key ingredients behind Proposition \ref{S52P1} are the multi-level loop equations, and certain product expansion formulas, which allow us to linearize them.

In Section \ref{Section6} we show that the third and higher joint cumulants of $\mathsf{G}_{K}(z,s)$ vanish in the large $K$ limit, see  Proposition \ref{S6Prop}, and in Section \ref{Section7} we show that the $\mathsf{G}_{K}(z,s)$ converge to a Gaussian field with an explicit covariance, see Theorem \ref{MainTechThm}. In Section \ref{Section8} we match the covariance in Theorem \ref{MainTechThm} with a suitable pull-back of the Gaussian free field and prove Theorems \ref{ThmLLN} and \ref{ThmMain}. In Appendix \ref{AppendixA} we summarize some basic properties of joint cumulants used throughout the paper, and in Appendix \ref{AppendixB} we explain how the covariance formula in Theorem \ref{MainTechThm} was guessed.

%-------------------------------------------------------------------------------------------------------------------------------------------------------------------------------------------------
% Section 1.5
%
%-------------------------------------------------------------------------------------------------------------------------------------------------------------------------------------------------
\subsection*{Acknowledgments}\label{Section1.5} The authors would like to thank Alexei Borodin, Vadim Gorin and Jiaoyang Huang for many useful discussions and comments on earlier drafts of the paper. ED was partially supported by the Minerva Foundation Fellowship and the NSF grant DMS:2054703. AK is partially supported by NSF grant DMS:2153958.

%-------------------------------------------------------------------------------------------------------------------------------------------------------------------------------------------------
%    Section 2
%
%-------------------------------------------------------------------------------------------------------------------------------------------------------------------------------------------------
\section{Connection to Jack symmetric functions}\label{Section2} In this section we explain how the measures $\PN$ from (\ref{S1PDef}), can be interpreted as an ascending Jack process and we utilize this connection to derive some properties for $\PN$. In Section \ref{Section2.1} we recall the Jack symmetric functions and list some of their basic properties. Section \ref{Section2.2} introduces Jack measures and the ascending Jack process. The connection between the ascending Jack process and $\PN$ is made in Section \ref{Section2.3}.

%-------------------------------------------------------------------------------------------------------------------------------------------------------------------------------------------------
%    Section 2.1
%
%-------------------------------------------------------------------------------------------------------------------------------------------------------------------------------------------------
\subsection{Jack symmetric functions}\label{Section2.1} The exposition in this section for the most part follows \cite[Section 6A]{DK2020}, which in turn is based on \cite[Chapter VI]{Mac}. 

A {\em partition} is a sequence $\lambda = (\lambda_1, \lambda_2, \dots)$ of non-negative integers such that $\lambda_1 \geq \lambda_2 \geq \cdots$ and all but finitely many terms are zero. We denote the set of all partitions by $\mathbb{Y}$ and by $\varnothing$ the partition $\lambda$ such that $\lambda_i = 0$ for all $i \in \mathbb{N}$. The {\em length} $\ell(\lambda)$ of a partition is the number of non-zero $\lambda_i$ and the {\em weight} of $\lambda$ is $|\lambda|= \lambda_1 + \lambda_2 + \cdots$. An alternative representation is given by $\lambda = 1^{m_1} 2^{m_2} \cdots$, where $m_j(\lambda) = | \{ i \in \mathbb{N}: \lambda_j = i\}|$ is called the {\em multiplicity} of $j$ in $\lambda$. There is a natural ordering on the set of partitions, called the {\em reverse lexicographic order}, given by 
$$\lambda > \mu \iff \mbox{ there exists $k \in \mathbb{N}$ such that $\lambda_i = \mu_i$ for $i < k$ and $\lambda_k > \mu_k$.}$$

A {\em Young diagram} is a graphical representation of a partition $\lambda$ with $\lambda_1$ left-justified boxes in the top row, $\lambda_2$ in the second row and so on. In general, we do not distinguish between a partition and the Young diagram representing it. The {\em conjugate} of a partition $\lambda$ is the partition $\lambda'$, whose Young diagram is the transpose of the diagram $\lambda$. Explicitly, we have $\lambda_i' = |\{ j \in \mathbb{N}: \lambda_j \geq i\}|$. 

Given two diagrams $\lambda$ and $\mu$ such that $\mu \subseteq \lambda$ (as collections of boxes), we call the difference $\kappa =\lambda - \mu$ a {\em skew Young diagram}. A skew Young diagram $\kappa$ is a {\em horizontal $m$-strip} if $\kappa$ contains $m$ boxes and no two lie in the same column. If $\lambda - \mu$ is a horizontal strip we write $\lambda \succeq \mu$ and note that $\lambda \succeq \mu \iff  \lambda_1 \geq \mu_1 \geq \lambda_2 \geq \mu_2 \geq \cdots$. For a box $\square = (i,j) $ of $\lambda$ (that is, a pair $(i,j)$ such that $\lambda_i \geq j$) we denote by $a(i,j)$ and $l(i,j)$ its {\em arm} and {\em leg lengths}:
$$a( \square) = a(i,j) = \lambda_i - j, \hspace{5mm} l(\square) = l(i,j) = \lambda_j' - i.$$
Further, we let $a'(\square)$ and $l'(\square)$ denote the {\em co-arm} and {\em co-leg lengths}:
$$a'(\square) = j - 1, \hspace{5mm} l'(\square) = i - 1.$$

Let $\Lambda_X$ be the algebra of symmetric functions in variables $X = (x_1, x_2, \dots)$. One way to view $\Lambda_X$ is as the algebra of polynomials in the Newton power sums 
$$p_k(X) = \sum_{i = 1}^{\infty} x_i^k \mbox{ for } k \geq 1.$$
For any partition $\lambda$ we define 
$$p_\lambda(X) = \prod_{i = 1}^{\ell(\lambda)} p_{\lambda_i}(X),$$
and note that $p_{\lambda}(X)$ for $\lambda \in \mathbb{Y}$ form a linear basis in $\Lambda_X$. 

In what follows we fix a parameter $\theta$. Unless the dependence on $\theta$ is important we will suppress it from the notation and we do the same for the variable set $X$. We consider the following scalar product $\langle \cdot, \cdot \rangle$ on $\lambda \otimes \mathbb{Q}(\theta)$:
\begin{equation}\label{S21E1}
\langle p_{\lambda}, p_{\mu} \rangle = {\bf 1} \{ \lambda = \mu \} \cdot \theta^{-\ell(\lambda)} \prod_{i = 1}^{\lambda_1} i^{m_i(\lambda)} m_i(\lambda)!.
\end {equation}
\begin{proposition}\label{S21P1} \cite[Chapter VI]{Mac} There are unique symmetric functions $J_{\lambda} \in \Lambda \otimes \mathbb{Q}(\theta)$ for all $\lambda \in \mathbb{Y}$ such that 
\begin{itemize}
\item $\langle J_{\lambda}, J_{\mu} \rangle = 0$ unless $\lambda = \mu$;
\item the leading (with respect to the reverse lexicographic order) monomial in $J_\lambda$ is $\prod_{i = 1}^{\ell(\lambda)} x_i^{\lambda_i}$.
\end{itemize} 
\end{proposition}
The functions $J_\lambda$ in Proposition \ref{S21P1} are called {\em Jack symmetric functions} and they form a homogeneous linear basis for $\Lambda$ that is different from the $p_\lambda$ above. The {\em dual Jack symmetric functions} are defined through 
$$\tilde{J}_\lambda = J_\lambda \cdot \prod_{ \square \in \lambda} \frac{a(\square) + \theta \cdot  l (\square) + \theta }{ a(\square) + \theta \cdot  l(\square) + 1}.$$
\begin{remark}\label{S21R1} We mention that $J_\lambda$ are degenerations of {\em Macdonald symmetric function} $P_\lambda$ if one sets $t = q^{\theta}$ and let $q \rightarrow 1-$. Taking the same limit sends the dual Macdonald symmetric function $Q_\lambda$ to $\tilde{J}_\lambda$. This convention is slightly different from \cite[Chapter VI.10]{Mac}, where one sets $q = t^{\alpha}$ and sends $t \rightarrow 1-$. Our convention agrees with \cite[Chapter VI.10]{Mac} upon setting $\theta = \alpha^{-1}$. 
\end{remark}

\begin{definition} Take two infinite sequences of variables $X = (x_1, x_2, \dots)$ and $Y = (y_1, y_2, \dots)$ and let $(X,Y)$ denote their concatenation. For two partitions $\lambda$ and $\mu$ we define the {\em skew Jack symmetric functions} $J_{\lambda/ \mu}(X)$ and $\tilde{J}_{\lambda/ \mu}(X)$ as the coefficients in the expansions
\begin{equation}\label{S21E2} 
J_{\lambda}(X,Y) = \sum_{ \mu \in \mathbb{Y}} J_{\mu}(Y) J_{\lambda / \mu}(X) \hspace{2mm} \mbox{ and } \hspace{2mm} \tilde{J}_{\lambda}(X,Y) = \sum_{ \mu \in \mathbb{Y}} \tilde{J}_{\mu}(Y) \tilde{J}_{\lambda/ \mu}(X).
\end{equation}
In particular, if $\mu = \varnothing$ is the zero partition, we have that $J_{\lambda/ \varnothing}(X) = J_{\lambda}(X)$ and $\tilde{J}_{\lambda/ \varnothing}(X) = \tilde{J}_{\lambda}(X)$.
\end{definition}

A {\em specialization} is an algebra homomorphism from $\Lambda$ to $\mathbb{C}$. We say that a specialization $\rho$ of $\Lambda$ is {\em Jack-positive} if it takes non-negative values on all Jack symmetric functions, i.e. $J_\lambda(\rho) \geq 0$ for all $\lambda \in \mathbb{Y}$. The collection of all Jack-positive specializations is characterized by the following statement.
\begin{proposition}\label{S21Prop1}\cite[Theorem A]{KOO} For any fixed $\theta > 0$, Jack positive specializations can be parametrized by triplets $(\alpha, \beta, \gamma)$, where $\alpha, \beta$ are sequences of real numbers satisfying 
$$\alpha_1 \geq \alpha_2 \geq \cdots \geq 0, \hspace{2mm} \beta_1 \geq \beta_2 \geq \cdots \geq 0, \hspace{2mm} \sum_{i = 1}^{\infty} (\alpha_i + \beta_i) < \infty,$$
and $\gamma \geq 0$. The specialization corresponding to a triplet $(\alpha, \beta, \gamma)$ is uniquely determined by its values on the Newton power sums $p_k$, $k \geq 1$:
$$p_1 \rightarrow p_1 (\alpha, \beta, \gamma) := \gamma + \sum_{i = 1}^{\infty} (\alpha_i + \beta_i)$$
$$p_k \rightarrow p_k (\alpha, \beta, \gamma) := \sum_{i = 1}^{\infty} \alpha_i^k + (-\theta)^{k-1} \sum_{i = 1}^{\infty} \beta_i^k, \hspace{2mm} k \geq 2.$$
\end{proposition}
\begin{remark} We mention that \cite{Matveev19} proves an analogous classification result for Macdonald symmetric functions.
\end{remark}
\begin{remark} By degenerating Macdonald to Jack symmetric functions, one has from \cite[Proposition 2.2]{BorCor} that $J_{\lambda/ \mu}(\rho) \geq 0$ for any $\lambda, \mu \in \mathbb{Y}$ and Jack-positive specialization $\rho$.
\end{remark}

Of special interest to us is the case when $\alpha_1 = \alpha_2 = \cdots = \alpha_N = 1$ and all other parameters are equal to zero. This is known as a {\em pure-$\alpha$} specialization and we denote it by $1^N$. When $N = 1$ we simply denote it by $1$. If we set $\beta_1 = \beta_2 = \cdots = \beta_N = 1$ and all other parameters are equal to zero, we obtain what is known as a {\em pure-$\beta$} specialization, denoted by $1_\beta^N$. From \cite[Chapter VI, (10.20)]{Mac} we have 
\begin{equation}\label{S21E3}
\begin{split}
J_{\lambda}(1^N) &= {\bf 1}\{\ell(\lambda) \leq N \} \prod_{\square \in \lambda} \frac{N\theta + a'(\square) - \theta l'(\square)}{a(\square) + \theta l(\square) + \theta }  \\
&= {\bf 1}\{\ell(\lambda) \leq N \} \cdot \prod_{i = 1}^N \prod_{j = 1}^{\lambda_i} \frac{N\theta+(j-1) - (i-1) \theta}{\lambda_i - j + \theta (\lambda_j' - i) + \theta}.
\end{split}
\end{equation}
In addition, from \cite[Chapter VI, (7.14')]{Mac} (see also \cite[(2.4)]{GS}) we have
\begin{equation}\label{S21E4}
\begin{split}
J_{\lambda/\mu}(1) = {\bf 1}\{ \mu \preceq \lambda \}  & \times \prod_{1 \leq i < j \leq k-1} \frac{(\mu_i - \mu_j + \theta (j-i) + \theta)_{\mu_j - \lambda_{j+1}}  }{ (\mu_i - \mu_j + \theta (j-i) + 1)_{\mu_j - \lambda_{j+1}} },\\
&  \times \frac{(\lambda_i - \mu_j + \theta(j-i) + 1)_{\mu_j - \lambda_{j+1}} }{(\lambda_i - \mu_j + \theta (j-i) + \theta)_{\mu_j - \lambda_{j+1}}},
\end{split}
\end{equation}
where $k$ is any integer satisfying $\ell(\lambda) \leq k$ and $(b)_n = b (b+1) \cdots (b+n-1)$ is the {\em Pochhammer symbol}. Finally, from \cite[Chapter VI, Sections 5 and 10]{Mac} 
\begin{equation}\label{S21E5}
\begin{split}
\tilde{J}_{\lambda}(1_\beta^N; \theta) = J_{\lambda'}(1^N; \theta^{-1}) = {\bf 1}\{ \lambda_1 \leq N \} \cdot \prod_{i = 1}^N \prod_{j = 1}^{\lambda_i} \frac{N + \theta(i-1) - (j-1)}{\lambda_i - j + \theta(\lambda_j' - i) + 1},
\end{split}
\end{equation}
where the second equality follows from (\ref{S21E3}). 
\begin{remark} We mention that (\ref{S21E5}) can be found in the displayed equation below \cite[(6.6)]{DD21}, but the latter has a typo -- $\lambda_i' - j$ in the denominator should be replaced with $\lambda_j' - i$. This mistake does not propagate to \cite[(6.7)]{DD21}, which is the formula for $\tilde{J}_{\lambda}(1_\beta^M; \theta)$ used in the rest of the paper.
\end{remark}

%-------------------------------------------------------------------------------------------------------------------------------------------------------------------------------------------------
%    Section 2.2
%
%-------------------------------------------------------------------------------------------------------------------------------------------------------------------------------------------------
\subsection{Ascending Jack processes}\label{Section2.2} Using the notion of Jack-positive specializations, one can define probability measures on $\mathbb{Y}$ as follows. 
\begin{definition}\label{S22Def1} Fix $\theta > 0$ and two Jack-positive specializations $\rho_1$, $\rho_2$ such that the (non-negative) series $\sum_{\lambda \in \mathbb{Y}} J_{\lambda}(\rho_1) \tilde{J}_{\lambda}(\rho_2)$ is finite. We define the {\em Jack probability measure} $\mathcal{J}_{\rho_1, \rho_2}(\lambda)$ on $\mathbb{Y}$ by
$$\mathcal{J}_{\rho_1, \rho_2}(\lambda) = \frac{J_{\lambda}(\rho_1) \tilde{J}_{\lambda}(\rho_2)}{H_{\theta}(\rho_1; \rho_2)},$$
where the normalization constant is given by $H_{\theta}(\rho_1; \rho_2) = \sum_{ \lambda \in \mathbb{Y}} J_{\lambda}(\rho_1) \tilde{J}_{\lambda}(\rho_2)$. 
\end{definition}
\begin{remark} The construction of probability measures through specializations was first considered in \cite{okounkov} in the context of Schur measures (these correspond to $\theta = 1$, in which case Jack symmetric functions degenerate to Schur symmetric functions). Since then, the construction has been extended to Macdonald symmetric functions, see \cite{BorCor}, of which Jack symmetric functions are a special case.
\end{remark}
\begin{remark}\label{S22R1} When $\rho_1 = 1^N$ and $\rho_2 = (\alpha, \beta, \gamma)$ as in Proposition \ref{S21Prop1} with $\alpha_i \in [0, 1)$ for $i \in \mathbb{N}$, we have the following formula for the normalization constant in Definition \ref{S22Def1}
\begin{equation}\label{S22E1}
H_{\theta}(\rho_1; \rho_2) = e^{N\theta \gamma} \prod_{i = 1}^{\infty} (1 + \theta \beta_i)^N \cdot \prod_{i = 1}^{\infty} \frac{1}{(1- \alpha_i)^{N\theta}},
\end{equation}
where the convergence of the product is ensured when 
$$1 > \alpha_1 \geq \alpha_2 \geq \cdots \geq 0, \hspace{2mm} \beta_1 \geq \beta_2 \geq \cdots \geq 0 , \hspace{2mm} \sum_{i = 1}^{\infty} (\alpha_i + \beta_i) < \infty, \hspace{2mm} \gamma \in [0, \infty).$$
We mention that the formula in (\ref{S22E1}) can be deduced by setting $t = q^{\theta}$ in \cite[(2.23) and (2.31)]{BorCor} and letting $q \rightarrow 1^-$. We also mention that the formula in (\ref{S22E1}) is slightly different from \cite[(2.23) and (2.31)]{BorCor}. This difference is due to the way we define $(\alpha, \beta, \gamma)$ specializations in Proposition \ref{S21Prop1} (made after \cite{KOO}), which is different from the implicit definition of the $(\alpha, \beta, \gamma)$ specialization in \cite[(2.23)]{BorCor} through certain generating series. The generating series \cite[(1.13)]{Matveev19} agrees with \cite[(2.23)]{BorCor}, and then if we set $t = q^{\theta}$ and $q \rightarrow 1^-$ in \cite[(1.7)]{Matveev19} we obtain precisely our specialization with $\beta$ and $\gamma$ variables multiplied by $\theta^{-1}$.
\end{remark}

We next turn to defining the main object of interest in this section.
\begin{definition}\label{S22Def2} Fix $N \in \mathbb{N}$, $K \in \mathbb{Z}_{\geq 0}$ and $\theta > 0$. We define the {\em ascending Jack process} to be the probability measure on sequences of partitions $(\lambda^1, \dots, \lambda^N)$ such that 
\begin{equation}\label{S22E3}
\mathcal{J}^{\theta, K}_N(\lambda^1, \dots, \lambda^N) = (1 + \theta)^{-NK} \cdot \prod_{ i = 1}^N J_{\lambda^i/ \lambda^{i-1}}(1) \cdot \tilde{J}_{\lambda^N}(1^K_\beta).
\end{equation}
In (\ref{S22E3}) we have that $\lambda^0 = \varnothing$ is the zero partition.
\end{definition}
\begin{remark}\label{S22R2} The ascending Jack process $\mathcal{J}^{\theta, K}_N$ is a special case of the ascending Macdonald process \cite[Definition 2.7]{BorCor} corresponding to $a_i = 1$ for $i = 1, \dots, N$ and $\rho = 1_{\beta}^K$. From \cite[(2.31)]{BorCor} we have that (\ref{S22E3}) is a probability measure -- i.e. the sum over all $\lambda^1, \dots, \lambda^N \in \mathbb{Y}$ is equal to $1$. As mentioned in Remark \ref{S22R1} our specialization $1^K_{\beta}$ corresponds to $\beta_1 = \cdots = \beta_K = \theta$ in \cite[Proposition 2.2]{BorCor}.
\end{remark}

In the remainder of the section we derive several properties of the measures $\mathcal{J}^{\theta, K}_N$ from Definition \ref{S22Def2}. We first note that from \cite[Chapter VI.7]{Mac}(see also \cite[(2.28)]{BorCor} with $\hat \nu = \varnothing$) we have for $1 \leq k \leq N$ and $\nu \in \mathbb{Y}$
\begin{equation}\label{S22E4}
\sum_{\kappa \in \mathbb{Y}} J_{\kappa / \nu} (1^k) \tilde{J}_{\kappa}(1^K_{\beta}) = (1 + \theta)^{(N-k)K}  \cdot \tilde{J}_{\nu}(1^K_{\beta}).
\end{equation}
Through repeated uses of (\ref{S21E2}) and (\ref{S22E4}) into (\ref{S22E3}) we have for any $1 \leq n \leq N$ and $1 \leq i_1 < i_2 < \cdots < i_n \leq N$ that the projection of $\mathcal{J}^{\theta, K}_N$ to the subsequence $(\lambda^{i_1}, \dots, \lambda^{i_n})$ has distribution
\begin{equation}\label{S22E5}
\mathcal{J}^{\theta, K}_N(\lambda^{i_1}, \dots, \lambda^{i_n}) = (1 + \theta)^{- i_n \cdot K}\cdot \prod_{j = 1}^n J_{\lambda^{i_j}/ \lambda^{i_{j-1}}}(1^{i_j - i_{j-1}}) \cdot \tilde{J}_{\lambda^{i_n}}(1^K_\beta),
\end{equation}
where we use the convention $i_0 = 0$. One particular consequence of (\ref{S22E5}) is that the projection to $(\lambda^{1}, \dots, \lambda^{n})$ with $1 \leq n \leq N$ has distribution $\mathcal{J}^{\theta, K}_n$. Thus, $\{ \mathcal{J}^{\theta, K}_n \}_{n \geq 1}$ form a consistent family of measures on $\mathbb{Y}^n$ and by Kolmogorov's extension theorem there exists a unique measure $\mathcal{J}^{\theta, K}_{\infty}$ on $\mathbb{Y}^{\infty}$, whose projection to the first $N$ coordinates is precisely $\mathcal{J}^{\theta, K}_N$. 

From (\ref{S22E5}) applied to $n = 1$ we have that under $\mathcal{J}^{\theta, K}_N$, the partition $\lambda^m$ with $1 \leq m \leq N$ is distributed according to a Jack measure as in Definition \ref{S22Def1} with specializations $\rho_1 = 1^m$ and $\rho_2 = 1_{\beta}^K$. In particular, from (\ref{S21E3}) and (\ref{S21E5}) we conclude that the probability of $(\lambda^1, \dots, \lambda^N)$ under $\mathcal{J}^{\theta, K}_N$ is zero unless $\lambda^m_{m+1} = 0$ and $0 \leq \lambda^m_1 \leq K$ for all $m = 1 , \dots, N$. Furthermore, (\ref{S21E4}) implies that $(\lambda^1, \dots, \lambda^N)$  also has zero probability unless $\lambda^i \preceq \lambda^{i+1}$ for $i = 0, \dots, N-1$.

%-------------------------------------------------------------------------------------------------------------------------------------------------------------------------------------------------
%    Section 2.3
%
%-------------------------------------------------------------------------------------------------------------------------------------------------------------------------------------------------
\subsection{From $\mathcal{J}^{\theta, K}_N$ to $\PN$}\label{Section2.3} In this section we explain how the measure $\PN$ from (\ref{S1PDef}), can be interpreted as an ascending Jack process $\mathcal{J}^{\theta, K}_N$ as in Definition \ref{S22Def2}. We consider the following change of variables
\begin{equation}\label{S23E1}
\ell_i^j = \lambda_i^j - i \cdot \theta \mbox{ for } 1 \leq i \leq j \leq N.
\end{equation}
Below we show that if $(\lambda^1, \dots, \lambda^N)$ is distributed according to $\mathcal{J}^{\theta, K}_N$, then $(\ell^1, \dots, \ell^N)$ has distribution $\PN$. Let us first discuss the state spaces of the two measures. As explained in the last two paragraphs of Section \ref{Section2.2}, we have that $\mathcal{J}^{\theta, K}_N$ is supported on sequences of interlacing partitions, such that $\lambda_{m+1}^m = 0$ for all $m = 1, \dots, N$ (i.e. the $m$-th partition has length at most $m$) and $\lambda_1^m \leq K$. This shows that by dropping all (deterministically zero) elements of index that is greater than $m$ we have $(\lambda^m_1, \dots, \lambda^m_m) \in \Lambda^K_m$ as in (\ref{S1GenState}) for $m = 1, \dots, N$. Using the latter observation and the interlacing property, we conclude that $(\ell^1, \dots, \ell^N)$ defined through (\ref{S23E1}) a.s. lie inside of $\XXO$ as in (\ref{S1GenState}), so that these random elements take value in the correct space.

We next express the (skew) Jack symmetric functions involved in the definition of $\mathcal{J}^{\theta, K}_N$ through the variables $\ell_i^j$ from (\ref{S23E1}). Firstly, from \cite[(6.5)]{DD21} we have
\begin{equation}\label{S23E2}
J_{\lambda^N}(1^N) = \prod_{i = 1}^N \frac{\Gamma(\theta)}{\Gamma(i \theta)} \times \prod_{1 \leq i < j \leq N} \frac{\Gamma(\ell^N_i - \ell^N_j + \theta)}{\Gamma(\ell^N_i - \ell^N_j)}.
\end{equation}
In addition, from \cite[(6.7)]{DD21}
\begin{equation}\label{S23E3}
\tilde{J}_{\lambda^N}(1_{\beta}^K)  = \prod_{1 \leq i < j \leq N} \frac{\Gamma(\ell^N_i - \ell^N_j + 1)}{\Gamma(\ell^N_i - \ell^N_j + 1 - \theta)} \prod_{i = 1}^N \frac{\Gamma(K + \theta (i-1) +1)}{\Gamma(\ell^N_i + 1 + N\theta) \Gamma(K  - \ell^N_i + 1 - \theta)}.
\end{equation}
We mention that in \cite[(6.7)]{DD21} one needs to shift $\ell_i$ by $-N\theta$ (this is due to $\ell_i = \lambda_i + (N-i)\theta$ in that paper as opposed to $\ell_i = \lambda_i - i\theta$ as in (\ref{S23E1})). Finally, from \cite[(6.11)]{DK2020} we have that 
\begin{equation}\label{S23E4}
J_{\lambda^{j+1}/ \lambda^{j}}(1) = \I(\ell^{j+1}, \ell^{j}) ,
\end{equation}
where the latter is as in (\ref{S1PDef2}). Combining (\ref{S23E2}), (\ref{S23E3}) and (\ref{S23E4}) with (\ref{S22E3}) we conclude that the induced measure on $(\ell^1, \dots, \ell^N)$ under the map in (\ref{S23E1}) is precisely (\ref{S1PDef}) and moreover we find the following formula for the normalization constant
\begin{equation}\label{S23E5}
Z(\theta, K, N) = (1 + \theta)^{NK} \cdot \prod_{i = 1}^N \frac{\Gamma(i\theta)}{\Gamma(\theta) \cdot \Gamma(K + \theta (i -1 ) + 1)}.
\end{equation}

Using the projection formula for $\mathcal{J}^{\theta, K}_N$ from (\ref{S22E5}) we obtain the following formula for the joint distribution of $(\ell^{k_1}, \dots, \ell^{k_2})$ for any $N \geq k_2 \geq k_1 \geq 1$ 
\begin{equation}\label{S23E6}
\PN(\ell^{k_1}, \dots, \ell^{k_2}) = \frac{1}{Z(\theta, K,k_2)} \cdot \prod_{i = k_1 + 1}^{k_2} \frac{\Gamma(i \theta)}{\Gamma(\theta)} \cdot H^t_{k_2}(\ell^{k_2}) \cdot H^b_{k_1}(\ell^{k_1}) \cdot \prod_{j = k_1}^{k_2-1}  \I(\ell^{j+1}, \ell^j), \mbox{ where }
\end{equation}
$I(\ell^{j+1}, \ell^j)$ and $H^t_n$ are as in (\ref{S1PDef2}) and (\ref{S1PDef3}), while 
\begin{equation}\label{S23E7}
H^b_{n}(\ell^{n}) = \prod_{1 \leq i < j \leq n} \frac{\Gamma(\ell_i^n - \ell_j^n + \theta)}{\Gamma(\ell_i^n - \ell_j^n)}. 
\end{equation}
One consequence of (\ref{S23E6}) is that when $k_1 = k_2 = n$ for some $1 \leq n \leq N$ we have that 
\begin{equation}\label{S23E8}
\PN(\ell^{n}) = \frac{1}{Z( \theta, K, n )}\prod_{1 \leq i < j \leq n} \frac{\Gamma(\ell_i^{n} - \ell_j^n + 1 ) \Gamma(\ell_i^n - \ell_j^n  + \theta) }{\Gamma(\ell_i^{n} - \ell_j^n + 1 - \theta) \Gamma(\ell_i^n - \ell_j^n )} \cdot \prod_{i = 1}^n w_n^{\theta, K}(\ell_i^n),
\end{equation}
where $w_n^{\theta, K}$ is as in (\ref{S1PDef4}).

Applying (\ref{S23E6}) to $k_1 = 1$ and $k_2 = n$ for some $1 \leq n \leq N$, we conclude that the projection of $\PN$ to $(\ell^1, \dots, \ell^n)$ is precisely $\mathbb{P}^{\theta,K}_n$. The latter shows that $\{\PN\}_{N \geq 1}$ forms a consistent family of measures. By Kolmogorov's extension theorem we conclude the following statement.
\begin{lemma}\label{Consistent}
Fix $\theta > 0$ and $K \in \mathbb{Z}_{\geq 0}$. There exists a unique measure $\P$ on $\XXI$, as in (\ref{InfState}), whose projection to the first $N \in \mathbb{N}$ coordinates is precisely $\PN$ as in (\ref{S1PDef}).
\end{lemma}
\begin{remark} Observe that the measure $\P$ in Lemma \ref{Consistent} is precisely the $\beta$-Krawtchouk corners process in Definition \ref{BKCC}.
\end{remark}

%-------------------------------------------------------------------------------------------------------------------------------------------------------------------------------------------------
% Section 3
%
%-------------------------------------------------------------------------------------------------------------------------------------------------------------------------------------------------
\section{Single-level analysis} \label{Section3} In (\ref{S23E8}) we derived the marginal distribution of $\ell^n$ under the measure $\P$ from Definition \ref{BKCC}. The form of this distribution is that of a {\em discrete $\beta$-ensemble} from \cite{BGG}. In this section we utilize various results available for discrete $\beta$-ensembles to derive certain moment bounds that are used in our asymptotic analysis -- see Proposition \ref{S33P1}. In addition, in Section \ref{Section3.1} we summarize several functions related to the $\beta$-Krawtchouk corners process that are used throughout the paper.

%-------------------------------------------------------------------------------------------------------------------------------------------------------------------------------------------------
% Section 3.1
%
%-------------------------------------------------------------------------------------------------------------------------------------------------------------------------------------------------
\subsection{Law of large numbers}\label{Section3.1} Given a random $(\ell^1, \ell^2, \dots ) $, distributed according to $\P$ as in Definition \ref{BKCC}, and $n \in \mathbb{N}$, we consider the random {\em empirical measure}
\begin{equation}\label{S31E1}
\mu_{n} = \frac{1}{n} \cdot \sum_{i = 1}^n \delta(\ell_i^n/n).
\end{equation}
In this section we show that $\mu_n$ converge to a certain deterministic measure $\mu(x, s)$ when $n = \lceil s K \rceil \rightarrow \infty$ -- this is Proposition \ref{LLN}. Afterwards we use $\mu(x,s)$  to construct several functions that will play an important role in many of our arguments later in the paper.\\

We record the following result, which one can think of as a law of large numbers for the measures $\mu_n$. We mention that this result follows from \cite[Theorem 5.3]{BGG}. In the proof below we are merely verifying that the assumptions of \cite[Theorem 5.3]{BGG} are satisfied in our setup, which has already been done in a slightly different notation in \cite[Section 6]{DD21}.
\begin{proposition}\label{LLN}
Fix $\theta, s \in (0, \infty)$ and set $n_K = \lceil s \cdot K \rceil$. Suppose that $(\ell^1, \ell^2, \dots)$ is distributed according to $\P$ as in Definition \ref{BKCC}, and let $\mu_{n_K}$ be as in (\ref{S31E1}). Then, as $K \rightarrow \infty$ the measures $\mu_{n_K}$ converge weakly in probability to a certain deterministic probability measure $\mu(x,s)dx$ that depends on $\theta$ and $s$. More precisely, for each compactly supported Lipschitz function $f(x)$ on $\mathbb{R}$ and each $\varepsilon > 0$ the random variables
$$K^{1/2 - \varepsilon} \left| \int_{\mathbb{R}} f(x) \mu_{n_K}(dx) - \int_{\mathbb{R}}f(x) \mu(x,s)dx \right|$$
converge to zero in the sense of moments. Furthermore, the measure $\mu(x,s)$ is supported in $[-\theta, s^{-1}]$ and has density bounded by $\theta^{-1}$. 
\end{proposition}
\begin{proof} From (\ref{S23E8}) we have that $\ell^n = (\ell^{n}_1, \dots, \ell^{n}_n)$ is a discrete $\beta$-ensemble, i.e. it satisfies \cite[Definition 3.1]{BGG} with $j = 1$, $a_1= -n\theta - 1$ and $b_1 = K+1$. The regularity Assumption 1 in \cite[Section 3.2]{BGG} has been verified in \cite[Section 6.2]{DD21} and Assumption 2 holds with $\hat{a}_1 = -\theta$, $\hat{b}_1 = s^{-1}$. The convergence part of the proposition now follows from \cite[Theorem 5.3]{BGG}.

An explicit formula for $\mu(x,s)$ has been computed in \cite[Proposition 6.5]{DD21}, where it is denoted by $\phi_\beta^{\theta, {\tt M} + \theta}$, and the relationship is given by
\begin{equation}\label{S31E2}
\mu(x,s) = \phi_\beta^{\theta, {\tt M} + \theta}(x - \theta) \mbox{ with ${\tt M} = s^{-1}$}.
\end{equation}
The shift in the argument by $-\theta$ on the right side of (\ref{S31E2}) is due to the fact that the variables $\ell_i$ in \cite[Proposition 6.5]{DD21} are shifted by $n\theta$ compared to ours. The last part of the proposition follows from the formula for $\phi_\beta^{\theta, s^{-1} + \theta}$, although it can also be deduced from \cite[Section 5]{BGG}.
\end{proof}

In the remainder of this section we introduce several auxiliary functions related to the measures $\mu(x,s)$ from Proposition \ref{LLN} and discuss some of their properties. Let $w^{\theta,K}_n$ be as in (\ref{S1PDef4}). Using the functional equation $\Gamma(z+1) = z \Gamma(z)$ we observe that 
\begin{equation}\label{S31E3}
\frac{w^{\theta, K}_n(z-1)}{w^{\theta, K}_n(z)} = \frac{\Phi_{n,K}^-(z)}{\Phi_{n,K}^+(z)}, \mbox{ where } \Phi_{n,K}^-(z) = z + n\theta  \mbox{ and } \Phi_{n,K}^+(z) = K + 1 - \theta - z.
\end{equation}
It follows from \cite[Theorem 4.1]{BGG} that
\begin{equation}\label{S31E4}
R_{n,K}(z) = \Phi^{-}_{n,K}(zn) \cdot \mathbb{E}\left[ \prod_{i = 1}^n \frac{zn - \ell_i^n - \theta}{zn - \ell_i^n} \right] + \Phi^+_{n,K}(zn) \cdot \mathbb{E}\left[ \prod_{i = 1}^{n}\frac{zn- \ell^{n}_i + \theta - 1}{zn - \ell^{n}_i - 1} \right]
\end{equation}
is a degree $1$ polynomial in $z$. We mention that in deriving (\ref{S31E4}) we used that $\Phi^{\pm}_{n,K}$ are degree one polynomials and also that $\Phi^{-}_{n,K}(-n\theta) = 0 = \Phi^+_{n,K}(K + 1 - \theta)$.

From the proof of \cite[Proposition 5.11]{BGG}, see also \cite[(A.7)]{DK2020}, we have that if $n = \lceil s K \rceil$ for some fixed $s > 0$, then 
\begin{equation}\label{S31E5}
\lim_{K \rightarrow \infty} \mathbb{E}\left[ \prod_{i = 1}^n \frac{zn - \ell_i^n - \theta}{zn - \ell_i^n}  \right] = e^{-\theta G(z,s)} \mbox{ and } \lim_{K \rightarrow \infty}\mathbb{E}\left[ \prod_{i = 1}^{n}\frac{zn- \ell^{n}_i + \theta - 1}{zn - \ell^{n}_i - 1} \right] = e^{\theta G(z,s)},
\end{equation}
where $G(z,s)$ is the Stieltjes transform of $\mu(x,s)$, i.e.
\begin{equation}\label{S31E6}
G(z,s) = \int_{\mathbb{R}} \frac{\mu(x,s)dx}{z-x},
\end{equation}
and the convergence in (\ref{S31E5}) is uniform over compact subsets of $\mathbb{C} \setminus [-\theta, s^{-1}]$.

Dividing both sides of (\ref{S31E4}) by $n = \lceil s K \rceil$ and taking $K \rightarrow \infty$ we get using (\ref{S31E3}) and (\ref{S31E5})
\begin{equation*}
\begin{split}
 \lim_{K \rightarrow \infty} n^{-1} R_{n,K}(z) =  (z + \theta)  \cdot e^{-\theta G(z,s)} + (s^{-1} - z) \cdot e^{\theta G(z,s)}.
\end{split}
\end{equation*}
Since $n^{-1}R_{n,K}(z)$  is a degree $1$ polynomial for each $K$, so is its pointwise limit as $K \rightarrow \infty$. In particular, we have that there is a degree $1$ polynomial $R(z, s)$ such that 
\begin{equation}\label{S31E7}
\begin{split}
&R(z,s) =   (z + \theta) \cdot e^{-\theta G(z,s)} + (1/s - z) \cdot e^{\theta G(z,s)}.
\end{split}
\end{equation}
We next observe from (\ref{S31E6}) that as $|z| \rightarrow \infty$ we have 
\begin{equation}\label{S31E8}
\exp \left( \pm \theta G(z,s) \right) = 1 \pm \frac{\theta }{z} + O(|z|^{-2}).
\end{equation}
Using (\ref{S31E7}) and (\ref{S31E8}) we see that $R(z,s) = 1/s - \theta + O(|z|^{-1})$ and so we conclude that 
\begin{equation}\label{S31E9}
R(z,s) = 1/s - \theta.
\end{equation}
\begin{remark}
Once we have a formula for $R(z,s)$ we can use (\ref{S31E7}) to compute the Stieltjes transform $G(z,s)$ and then use that to find a formula for $\mu(x,s)$. This was already done in \cite[Lemma 3.6]{DK2020}. Specifically, setting 
\begin{equation}\label{S31E10}
\Phi^-(z,s) = z + \theta \mbox{ and } \Phi^+(z,s) = 1/s - z ,
\end{equation} 
we have from \cite[Lemma 3.6]{DK2020} that $\mu(x,s)$ is continuous on $[-\theta, 1/s]$, and given by
\begin{equation}\label{S31E11}
 \mu(x,s) = \frac{1}{\theta \pi } \cdot \arccos \left( \frac{R(x,s)}{2 \sqrt{\Phi^-(x,s)  \Phi^+(x,s)}}\right) =  \frac{1}{\theta \pi } \cdot \arccos \left( \frac{1/s - \theta }{2 \sqrt{(x + \theta )( 1/s - x )}}\right).
\end{equation}
In (\ref{S31E11}) we denote by $\arccos(x)$ the function, which is $\pi$ on $(-\infty, -1]$, $0$ on $[1 ,\infty)$, and the usual arccosine function on $(-1,1)$. The formula for $\mu(x,s)$ in (\ref{S31E11}) after some simple trigonometric  identities is seen match the one from \cite[Proposition 6.5]{DD21}, which was discussed in the proof of Proposition \ref{LLN}.
\end{remark}

Our next goal is to find explicit formulas for $ \exp ( \theta G(z,s)) $ and $ \exp ( -\theta G(z,s))$. These formulas will be given by certain square roots, and we isolate our convention in the following definition.
\begin{definition}\label{Branch} Throughout the paper all of the logarithms and square roots are defined with respect to the principal branch. Given $a, b \in \mathbb{R}$ with $a<b$, $f(z) = \sqrt{(z - a)(z- b)} $ means
$$f(z) = \begin{cases} \sqrt{z - a} \sqrt{z-b} &\mbox{ when $z \in \mathbb{C} \setminus (-\infty, b]$ }, \\ -\sqrt{a -z }\sqrt{b - z} &\mbox{ when $z \in (-\infty, a)$ }. \end{cases}$$ 
Observe that in this way $f$ is holomorphic on $\mathbb{C} \setminus [a,b]$, cf. \cite[Theorem 2.5.5]{SS}.
\end{definition}

Setting $Y =\exp ( \theta G(z,s))$ we see from (\ref{S31E7}) and (\ref{S31E9}) that $Y$ solves
$$X \cdot (1/ s - \theta) = X^2 \cdot   (1/s - z)   +  (z + \theta).$$
The above quadratic equation has two roots given by
$$X_{\pm} = \frac{(1/s - \theta) \pm 2 \sqrt{(z - z_+(s))(z - z_-(s))}  }{2 (1/ s - z)},$$
where we mention that the above square root is as in Definition \ref{Branch} and
\begin{equation}\label{S31E12}
z_\pm(s) =(1/2) \cdot (1/s - \theta) \pm \sqrt{ \theta/ s} .
\end{equation}
Setting $F(z,s) = (z - z_+(s))(z - z_-(s))$, we see that 
$$F(z,s) = (1/4) \cdot (1/s - \theta)^2 + (z+ \theta) (z- 1/s),$$
and so 
$$F(-\theta,s) = F(1/s) = (1/4) \cdot (1/s - \theta)^2 \geq 0,$$
which implies that 
\begin{equation}\label{S31E13}
-\theta \leq z_-(s) \leq z_+(s) \leq 1/s.
\end{equation}
The latter shows that $X_{\pm}$ is well-defined for $z \in \mathbb{C} \setminus [-\theta, 1/s]$.

From (\ref{S31E8}) we have $\exp ( \theta G(z,s)) = 1 + \theta/ z + O(|z|^{-2})$ as $z \rightarrow \infty$, and so we conclude that $\exp ( \theta G(z,s))  = X_-$, i.e.
\begin{equation}\label{S31E14}
\begin{split}
e^{\theta G(z,s)} = \hspace{2mm} &\frac{(1/s - \theta) - 2 \sqrt{(z - z_+(s))(z - z_-(s))}  }{2 (1/ s - z) },
\end{split}
\end{equation}
provided that $s > 0$ and $z \in \mathbb{C} \setminus [-\theta, 1/s]$. 

Similarly, we see from (\ref{S31E7}) that $Y^{-1}$ solves
$$X \cdot (1/ s - \theta) =  (1/s - z)   +  X^2 \cdot (z + \theta).$$
The above quadratic equation has two roots given by
$$X_{\pm} =\frac{(1/s - \theta) \pm 2 \sqrt{(z - z_+(s))(z - z_-(s))}  }{2 (z + \theta) }.$$
As before, we have that $X_{\pm}$ are well-defined for $z \in \mathbb{C} \setminus [-\theta, 1/s]$.

From (\ref{S31E8}) we have $\exp ( -\theta G(z,s)) = 1 - \theta/ z + O(|z|^{-2})$ as $z \rightarrow \infty$, and so we conclude that $\exp (-\theta G(z,s))  = X_+$, i.e.
\begin{equation}\label{S31E15}
\begin{split}
e^{ - \theta G(z,s)} = \frac{(1/s - \theta) + 2 \sqrt{(z - z_+(s))(z - z_-(s))}  }{2 (z + \theta) },
\end{split}
\end{equation}
provided that $s > 0$ and $z \in \mathbb{C} \setminus [-\theta, 1/s]$. \\

Combining (\ref{S31E10}), (\ref{S31E14}) and (\ref{S31E15}), we get 
\begin{equation}\label{S31E16}
\begin{split}
&Q(z,s) :=   \Phi^-(z,s) \cdot e^{- \theta G(z,s)} - \Phi^+(z,s) \cdot e^{\theta G(z,s)}  =  2 \sqrt{(z - z_+(s))(z - z_-(s))},
\end{split}
\end{equation}
and if we also use (\ref{S31E7}) we get
\begin{equation}\label{S31E17}
\begin{split}
\left(   e^{\theta G(z,s)}   -  1   \right)\left(  \Phi^+(z,s) - e^{-\theta G(z,s)} \Phi^{-}(z,s) \right) = \hspace{2mm} &R(z,s) -  \Phi^+(z,s)  -  \Phi^-(z,s) =  - 2\theta,
\end{split}
\end{equation}
provided that $s > 0$ and $z \in \mathbb{C} \setminus [-\theta, 1/s]$. The last equation in particular implies that $e^{\theta G(z,s)}   -  1 \neq 0$ for $s> 0$ and $z \in \mathbb{C} \setminus [-\theta, 1/s]$.

Using (\ref{S31E14}) we see that $G(z,s)$ is smooth in the region $\{ (z, s) \in \mathbb{C} \times (0, \infty): z \in \mathbb{C} \setminus [-\theta, s^{-1}]\}$ and also for each $s > 0$ and $m \in \mathbb{Z}_{\geq 0}$ the function $\partial_s^m G(z,s)$ is analytic in $z$ on $\mathbb{C} \setminus [-\theta, s^{-1}]$.  By a direct computation using (\ref{S31E14}) we have
\begin{equation}\label{SpecEqn}
\begin{split}
 \partial_s G(z,s) = \frac{\partial_s  e^{\theta G(z,s)} }{ \theta e^{\theta G(z,s)}} = \frac{  (ze^{\theta G(z,s)}   -  \theta -  z) \partial_z G(z,s) }{ s \left(e^{\theta G(z,s)} - 1 \right)},
\end{split}
\end{equation}
for $\{ (z, s) \in \mathbb{C} \times (0, \infty): z \in \mathbb{C} \setminus [-\theta, s^{-1}]\}$.

%-------------------------------------------------------------------------------------------------------------------------------------------------------------------------------------------------
% Section 3.2
%
%-------------------------------------------------------------------------------------------------------------------------------------------------------------------------------------------------
\subsection{Weak moment bounds}\label{Section3.2} Given a random $(\ell^1, \ell^2, \dots) $, distributed according to $\P$ as in Definition \ref{BKCC}, $L > 0$ and $n \in \mathbb{N}$, we consider the random analytic functions
\begin{equation}\label{S32E1}
G^n_L(z) = \sum_{i = 1}^n \frac{1}{z - \ell_i^n/L} \hspace{2mm} \mbox{ and } \hspace{2mm} \hat{G}^n_{L}(z) = G^n_L(z) - L \cdot G(z L/n,n/K),
\end{equation}
where $G(z,s)$ is as in (\ref{S31E6}). When $L = n$, $G^n_{n}(z)$ is $n$ times the Stieltjes transform of the empirical measures from (\ref{S31E1}) and $\hat{G}^n_{n}$ is a shifted version of it.

The goal of this section is to establish the following result.
\begin{proposition}\label{S32P1} Fix $\theta, \varepsilon >0$ and $\delta \in (0,1)$. Suppose that $(\ell^1, \ell^2, \dots)$ is distributed according to $\P$ as in Definition \ref{BKCC}. Then, for each $k \in \mathbb{N}$ and compact set $\mathcal{V} \subset \mathbb{C} \setminus [-\theta, \delta^{-1}]$ we have 
\begin{equation}\label{S32E2}
\lim_{K \rightarrow \infty} \sup_{z \in \mathcal{V}} \max_{ \delta K \leq n \leq \delta^{-1} K} K^{-k/2 - \varepsilon} \mathbb{E}\left[\left|\hat{G}_{n}^n(z) \right|^k \right] = 0.
\end{equation}
\end{proposition}
\begin{remark}\label{S32R1} In plain words, Proposition \ref{S32P1} says that the moments $\mathbb{E}\left[|\hat{G}_{n}^n(z) |^k \right]$ can grow at most as fast as $O(K^{k/2 + \varepsilon})$ for any $\varepsilon > 0$. Later in Proposition \ref{S33P1} we will show that $\mathbb{E}\left[|\hat{G}_{n}^n(z) |^k \right] = O(1)$. This is why we refer to (\ref{S32E2}) as {\em weak moment bounds}.
\end{remark}
\begin{remark}\label{S32R2} If we fix $s > 0$, $z \in \mathbb{C} \setminus [-\theta, s^{-1}]$ and set $n_K = \lceil s K \rceil$, then 
$$\frac{1}{n_K} \cdot \hat{G}_{n_K}^{n_K}(z)  = \int_{\mathbb{R}} f(x; z,s) \mu_{n_K}(dx) - \int_{\mathbb{R}} f(x;z,s) \mu(x,n_K/K) dx,$$
with $f(x;z,s) = \frac{1}{z -x}$. One can then use Proposition \ref{LLN} to show that
$$ \lim_{K \rightarrow \infty} K^{-k/2 - \varepsilon} \mathbb{E}\left[\left|\hat{G}_{n_K}^{n_K}(z) \right|^k \right] = 0.$$
The power of Proposition \ref{S32P1} is that this moment convergence happens uniformly as $z$ varies over a compact set in $\mathbb{C}$ that avoids the intervals $[-\theta, K/n]$ (that contain the singularities of $\hat{G}_{n}^n$) and as $K/n$ varies in a compact subset of $(0,\infty)$ (and not converging to some specified $s > 0$).
\end{remark}

To prove Proposition \ref{S32P1} we require an estimate from \cite{DD21}, which we establish next after introducing a bit of necessary notation. Let 
$$\hat{f}(\xi) = \int_{\mathbb{R}} f(x) e^{- i \xi x} dx$$
denote the Fourier transform of $f$. For a compactly supported Lipschitz function $g$ on $\mathbb{R}$ we define
\begin{equation}\label{S32E3}
\|{g}\|_{1/2}:=\left(\int_{\mathbb{R}} |t||\widehat{g}(t)|^2dt\right)^{1/2}, \quad \|{g}\|_{\operatorname{Lip}}=\sup_{x\neq y}\left|\frac{g(x)-g(y)}{x-y}\right|.
\end{equation}
We mention that both $\|{g}\|_{1/2}$ and $\|{g}\|_{\operatorname{Lip}}$ are finite for a compactly supported Lipschitz function $g$, cf. \cite[Lemma 2.5]{DD21}. The following is the key estimate we require.

\begin{proposition}\label{S32P2} Fix $\theta > 0$, $\delta \in (0,1)$ and let $(\ell^1, \ell^2, \dots)$ be distributed according to $\P$ as in Definition \ref{BKCC}. In addition, let $\mu_{n}$ be as in (\ref{S31E1}) and $\mu(x,s)$ be as in Proposition \ref{LLN}. If $n \geq 2$, $K/n \in [\delta, \delta^{-1}]$, $\gamma > 0$, $p\geq 2$ and $g$ is a compactly supported Lipschitz function, we have
\begin{equation}\label{S32E4}
\begin{split}
& \P  \left(\left|  \int_{\mathbb{R}}  g(x)\mu_{n}(dx)- \int_{\mathbb{R}} g(x)\mu(x, n/K) dx \right|\ge \gamma \|{g}\|_{1/2}+\frac{\theta \|{g}\|_{\operatorname{Lip}}}{n^p} \right)  \\
& \le  \exp \left(-2\pi^2\gamma^2\theta n^2+O\left(n\log n\right) \right),
\end{split}
\end{equation}
where the constant in the big $O$ notation depends on $\theta, p$ and $\delta$ alone.
\end{proposition}
\begin{proof} Consider the random variables $\tilde{\ell}^n = (\ell^n_1 + n \theta, \dots, \ell^n_1 + n \theta)$. From (\ref{S1PDef4})  and (\ref{S23E8}) we have 
$$\P(\tilde{\ell}^{n}) = \frac{1}{\hat{Z}( \theta, K, n )}\prod_{1 \leq i < j \leq n} \frac{\Gamma(\tilde{\ell}_i^{n} - \tilde{\ell}_j^n + 1 ) \Gamma(\tilde{\ell}_i^n - \tilde{\ell}_j^n  + \theta) }{\Gamma(\tilde{\ell}_i^{n} - \tilde{\ell}_j^n + 1 - \theta) \Gamma(\tilde{\ell}_i^n - \tilde{\ell}_j^n )} \cdot \prod_{i = 1}^n e^{-n V_n(\tilde{\ell}_i/ n)},$$
where 
$$V_n(x) =  \frac{1}{n} \cdot \log \left( \frac{ \Gamma(\tilde{\ell}_i^n +1) \Gamma(K + n\theta - \tilde{\ell}_i^n + 1 - \theta) }{n^{K + n \theta + 2 - \theta} e^{-K - n \theta + \theta}} \right),$$
and $\hat{Z}( \theta, K, n )$ is a new normalization constant. The latter measures agree with \cite[(6.11)]{DD21} with $N = n$, ${\tt M} = K/n$ and $M_N = \lceil  {\tt M} n \rceil = K$.

We next proceed to check that the latter measures satisfy \cite[Assumptions 2.1 and 2.2]{DD21}. To make the notation consistent one needs to take $N, {\tt M}, M_N$ as in the previous line. Since $K/n \in [\delta, \delta^{-1}]$ by assumption we see that we can take $A_0 = \delta^{-1}$ and $a_0 = \delta$ in \cite[Assumption 2.1]{DD21}. We can also take $A_1 = 1$ and $q_n = 0$ by the definition of ${\tt M}$. The continuity of $V_n$ is assured by the continuity of the gamma function on $(0, \infty)$. If we set 
$$V(x) = x \log x + (K/n + \theta - x) \log (K/n + \theta - x),$$
we see that
$$V'(x) = \log x - \log (K/n + \theta - x).$$
The latter shows that $V$ satisfies \cite[(2.3)]{DD21} with $A_4 = 1$ and the second part of \cite[(2.2)]{DD21} with $I = [0, K/n + \theta]$ and 
$$A_3 = \sup_{y \in [\delta, \delta^{-1}]} \sup_{x \in (0, y + \theta)}  x \left| \log x \right| + (y + \theta - x) \cdot \left |\log (y + \theta - x) \right|,$$
which is a finite constant that depends on $\delta$ and $\theta$ alone. Finally, one can repeat the proof of \cite[(6.12)]{DD21} verbatim to show that there exists a constant $A_2 > 0$ such that if $x \in I_n = [0, K/n + \theta (n-1)/n] $, then 
$$|V_n(x) - V(x)| \leq A_2 \cdot n^{-1} \cdot \log(n+1).$$
In particular, the first part of \cite[(2.2)]{DD21} is satisfied with this choice of $A_2$ and $p_n = n^{-1} \log (n+1)$. 

Now that \cite[Assumptions 2.1 and 2.2]{DD21} are verified for the distribution of $\tilde{\ell}^n$, we obtain (\ref{S32E4}) as a direct consequence of \cite[Proposition 2.6]{DD21} applied to $\tilde{\ell}^n$ and the function $\tilde{g}(x) = g(x -\theta)$.
\end{proof}

In the remainder of this section we give the proof of Proposition \ref{S32P1}.
\begin{proof}[Proof of Proposition \ref{S32P1}] Fix $\varepsilon, \delta > 0$, $k \in \mathbb{N}$ and a compact set $\mathcal{V}$ as in the statement of the proposition. Let $\eta > 0$ be sufficiently small so that for $I_\eta := [-\eta - \theta, \delta^{-1} + \eta]$ we have $ I_\eta \cap \mathcal{V} = \emptyset$. Let $h(x)$ be a smooth function such that $0 \leq h(x) \leq 1$,  $h(x) = 1$ if $x \in [-\eta/2 - \theta, \delta^{-1} + \eta/2]$, $h(x) = 0$ if $x \leq -\eta - \theta$ or $x \geq \delta^{-1}  + \eta$ and $\sup_{ x \in I_\eta} |h'(x)| \leq \eta^{-1} \cdot 10$. Notice that for $n/K \in [\delta, \delta^{-1}]$ we have $\mu_{n}$  as in (\ref{S31E1}) and $\mu(x, n/K)$ as in Proposition \ref{LLN} are both supported in $[-\theta, \delta^{-1}]$ and so
\begin{equation*}
\left| \int_{\mathbb{R}}  g_v(x) \cdot h(x)  \mu_{n}(dx) -  \int_{\mathbb{R}} g_v(x) \cdot h(x)  \mu(x, n/K)dx \right|  = n^{-1} \left| \hat{G}_{n}^n(v) \right|,
\end{equation*}
where $g_v(x) = (v -x)^{-1}$. Let us denote 
\begin{equation*}
c_1(\mathcal{V}) := \sup_{v \in \mathcal{V}} \| g_v \cdot h \|_{1/2} \mbox{ and } c_2(\mathcal{V}) := \sup_{v \in  \mathcal{V}} \| g_v \cdot h \|_{\operatorname{Lip}},
\end{equation*} 
which from \cite[(A.16)]{DK2020} are finite positive constants that depend on $\mathcal{V}, \delta$ (and implicitly on our choices of $\eta, h$). 

We now apply Proposition \ref{S32P2} for the function $g_v \cdot h$ with $\gamma = r \cdot n^{\varepsilon/(2k) - 1/2}$, $r > 0$ and $p = 3$ to get for some constant $C$ that depends on $\delta$ and $\theta$ and all $v \in \mathcal{V}$ and $n, K$ with $n \geq 2$, $n/K \in [\delta, \delta^{-1}]$
$$\P \left(  \left| \hat{G}_{n}^n(v) \right| \geq c_1(\mathcal{V}) \cdot r n^{1/2+\varepsilon/(2k)} + c_2(\mathcal{V}) \cdot n^{-2} \right) \leq \exp\left( C n \log(n) - 2 \pi^2 r^2 \theta n^{1+\varepsilon/k }\right).$$
Combining the latter with \cite[Lemma 2.2.13]{Durrett}, we conclude that
\begin{equation*}
\mathbb{E} \left[ \left|\hat{G}_{n}^n(v)\right|^k \right] =O \left( n^{k/2 + \varepsilon/2} \right),
\end{equation*}
where the constant in the big $O$ notation depends on $\varepsilon, \delta, \mathcal{V}, \theta$ and $k$, and is uniform for $v \in \mathcal{V}$ and $n, K$ with $n \geq 2$, $n/K \in [\delta, \delta^{-1}]$. The latter moment estimate implies (\ref{S32E2}), since $K \geq \delta n$. 
\end{proof}

%-------------------------------------------------------------------------------------------------------------------------------------------------------------------------------------------------
% Section 3.3
%
%-------------------------------------------------------------------------------------------------------------------------------------------------------------------------------------------------
\subsection{Strong moment bounds}\label{Section3.3} The goal of this section is to establish the following result.
\begin{proposition}\label{S33P1} Fix $\theta >0$ and $\delta \in (0,1)$. Suppose that $(\ell^1, \ell^2,  \dots)$ is distributed according to $\P$ as in Definition \ref{BKCC}, and $\hat{G}^n_L$ are as in (\ref{S32E1}). Then, for any $k \in \mathbb{N}$ and any compact set $\mathcal{V} \subset \mathbb{C} \setminus [-\theta, \delta^{-1}]$ we have 
\begin{equation}\label{S33E1}
 \sup_{z \in \mathcal{V}} \max_{ \delta K \leq n \leq \delta^{-1} K} \mathbb{E}\left[\left|\hat{G}_{n}^n(z) \right|^k \right] = O(1),
\end{equation}
where the constant in the big $O$ notation depends on $\theta, \delta, \mathcal{V}$ and $k$ alone.
\end{proposition}
\begin{remark}\label{MomentBoundRem} Before we go to the proof let us explain how Proposition \ref{S33P1} will be applied later in the text. Fixing $\theta$ and $\delta$ as in the statement of the proposition, we consider $n, L, K$ such that $n/L, n/K \in [\delta, \delta^{-1}]$. For these parameters we have for $\hat{G}^n_L$ as in (\ref{S32E1}) that
 \begin{equation}\label{MBRem}
\begin{split}
&\hat{G}_{L}^n(z) = G_L^n(z) - L \cdot G(zL/n, n/K) =  \zeta(z),\\
& \partial_z  \hat{G}_{L}^n(z) = \partial_z G_L^n(z) - (L^2/n) \cdot \partial_z G(zL/n, n/K)=  \zeta(z),\\
& \partial^2_z  \hat{G}_{L}^n(z) =  \partial_z^2 G_L^n(z) - (L^3/n^2) \cdot \partial_z^2 G(zL/n, n/K) = \zeta(z).
\end{split}
\end{equation}
In (\ref{MBRem}) we have written $\zeta(z)$ for a generic random analytic function on $\mathbb{C} \setminus [-\theta \delta^{-1}, \delta^{-2}]$ such that 
\begin{equation}\label{REQ1}
\mathbb{E}\left[|\zeta(z)|^k \right] = O(1),
\end{equation}
where the constant in the big $O$ notation depends on $\theta, \delta, k$ and a compact subset $\mathcal{V}_1 \subset \mathbb{C} \setminus [-\theta \delta^{-1}, \delta^{-2}]$ and (\ref{REQ1}) holds uniformly as $z$ varies in $\mathcal{V}_1$.  

Let us explain how (\ref{MBRem}) follows from Proposition \ref{S33P1} briefly. The first line in (\ref{MBRem}) follows from the moment bounds of $\hat{G}_{n}^n(z)$ in Proposition \ref{S33P1}, the identity
\begin{equation}\label{REQ2}
\hat{G}_{L}^n(z) = (L/n) \cdot \hat{G}^n_n(zL/n),
\end{equation}
and the fact that if $\mathcal{V}_1$ is a compact subset of $\mathbb{C} \setminus [-\theta \delta^{-1}, \delta^{-2}]$, then 
$$\mathcal{V} := \{z \in \mathbb{C}: z = s w \mbox{ for some } w\in \mathcal{V}_1, s \in [\delta, \delta^{-1}]\}$$
is a compact subset of $\mathbb{C} \setminus[-\theta, \delta^{-1}]$. Once we have the first line in (\ref{MBRem}) we obtain the second and third from Cauchy's inequalities, see \cite[Corollary 4.3]{SS}.
\end{remark}

\begin{proof} The proof we present below is an adaptation of the proof of \cite[Proposition 3.11]{DK2020}, which in turn is based on the proof of \cite[Proposition 2.18]{BGG}. Essentially, the idea is to reduce the problem to a certain self-improvement estimate claim for the joint moments of $\hat{G}_{n}^n(v_1), \dots, \hat{G}_{n}^n(v_k)$. One uses Proposition \ref{S32P1} to establish the base case of this claim and by iterating it finitely many times one ultimately obtains the $O(1)$ bound in (\ref{S33E1}). The way one obtains the improvement at each step in the claim is by utilizing the single-level loop or Nekrasov's equations from \cite{BGG}. 

Even though the main ideas of the proof already appeared in \cite{DK2020} we decided to include it here for a few reasons. Firstly, there are various statements within the proof, which will be useful for us in later parts of the paper. Secondly, the notation in \cite{DK2020} is rather different from the one in this paper and translating between the two is somewhat hard. Finally, the present proposition is formulated more generally than \cite[Proposition 3.11]{DK2020}, where there is a single level as opposed to a large number of levels (indexed by $n \in[\delta K, \delta^{-1} K]$), and various estimates throughout \cite{DK2020} need to be reestablished in this more general form. \\

For any $\eta > 0$ we denote
\begin{equation}\label{Deta}
\mathcal{D}_{\eta} = \{z \in \mathbb{C}: |z - x| > \eta \mbox{ for all $x \in [-\theta, \delta^{-1}]$} \}.
\end{equation}
In addition, for compact sets $\mathcal{V}_1\subset \mathcal{D}_{\eta}$, and $\mathcal{V}_2 \subset \mathbb{C}$ we write $\xi_n(z;x,y)$ to mean a generic random function of $n \in [\delta K, \delta^{-1}K]$, $x,y \in \mathcal{V}_2$ and $z \in \mathcal{D}_{\eta}$ that is analytic on $\mathcal{D}_{\eta}$ for each fixed $x,y \in \mathcal{V}_2$ and such that almost surely
$$ \max_{\delta K \leq n \leq \delta^{-1} K}  \sup_{z \in \mathcal{V}_1} \sup_{x,y \in \mathcal{V}_2} \left|\xi_n(z; x,y) \right| = O(1),$$
where the constant in the big $O$ notation depends on $\mathcal{V}_1, \mathcal{V}_2, \theta, \delta$ and $\eta$, but not $K$ provided it is large enough. How large $K$ needs to be will be clear from the context. To ease the notation, we drop $x,y, n$ and simply write $\xi(z)$. The meaning of $\xi(z)$ may be different from line to line or even within the same equation. For the sake of clarity we split the proof into five steps.\\

{\bf \raggedleft Step 1.} The goal of this step is to obtain a certain product expansion formula that will be used in the next step to linearize the Nekrasov's equations. The key statement is in (\ref{S33E2}) and it is analogous to the asymptotic expansions of \cite[Section 4]{DK2020}.\\

Let us fix a compact set $\mathcal{V}_2 \subset \mathbb{C}$ and an $\eta > 0$. We further fix a compact set $\mathcal{V}_1 \subset \mathcal{D}_{\eta}$. Note that since $\ell_i^n/n \in [-\theta, \delta^{-1}]$ almost surely, we can find $B > 0$ (depending on $\eta, \delta, \mathcal{V}_2$) such that if $K \geq B$ and $n \geq \delta K$ we have $\ell_i^n/n \pm x/n \not \in \mathcal{D}_{\eta}$ for each $x \in \mathcal{V}_2$ and $i \in \llbracket 1, n \rrbracket$.

Below we show that if $K \geq B$, $n \in [\delta K, \delta^{-1}K]$, $x, y\in \mathcal{V}_2$ and $z \in \mathcal{D}_{\eta}$, then
\begin{equation}\label{S33E2}
\begin{split}
\prod_{i = 1}^n \frac{zn - \ell^n_i + x}{zn - \ell^n_i + y} = & e^{(x - y) G(z, n/K)} \cdot \left[1 + \frac{(x-y)\hat{G}_{n}^n(z)}{n} + \frac{(x^2 - y^2) \partial_z G(z,n/K)}{2n} \right] \\
&  + \frac{\xi(z) [ \hat{G}_{n}^n(z)]^2}{n^2} +  \frac{\xi(z) \partial_z \hat{G}_{n}^n(z)}{n^2} +\frac{\xi(z) \hat{G}_{n}^n(z)}{n^2} + \frac{\xi(z)}{n^2},
\end{split}
\end{equation}
where we recall that $G(z,s)$ is as in (\ref{S31E6}) and $G_{n}^n(z), \hat{G}_{n}^n(z)$ are as in (\ref{S32E1}).\\

We first have by Taylor expansion of the logarithm and (\ref{S32E1}) that
\begin{equation*}
\begin{split}
& \prod_{ i = 1}^n \frac{zn - \ell^n_i + x}{zn - \ell^n_i + y} = \exp \left[ \sum_{i = 1}^n \log \left( 1 + \frac{1}{n} \cdot \frac{x}{z - \ell^n_i/n } \right) - \sum_{i = 1}^n \log \left( 1 + \frac{1}{n} \cdot \frac{y}{z - \ell^n_i/n } \right) \right] \\
&=  \exp\left[ \frac{(x-y) G_{n}^n(z) }{n}+ \frac{(x^2 - y^2) \partial_z G_{n}^n(z)}{2n^2} + \frac{ \xi(z)}{n^2}\right].
\end{split}
\end{equation*}
Setting $F_1(u) = \frac{ e^u - 1 - u}{u^2}$ and noting that $n^{-1}G_{n}^n(z)  = \xi(z)$, we see that
\begin{equation*}
\begin{split}
&\prod_{ i = 1}^n \frac{zn - \ell^n_i + x}{zn - \ell^n_i + y} = e^{(x - y) G(z, n/K)} \cdot \Bigg{[}1 + \frac{(x-y) \hat{G}_{n}^n(z) }{n}+ \frac{(x^2 - y^2) \partial_z G_{n}^n(z)}{2n^2} \\
& + \frac{(x-y)(x^2 - y^2) \hat{G}_{n}^n(z) \partial_z G_n^n(z) }{2n^3}   + \frac{(x-y)^2 [\hat{G}_{n}^n(z)]^2}{n^2} F_1 \left(\frac{(x - y) \hat{G}_{n}^n(z)}{n}\right) \\
&+ \frac{ (x-y)^2 (x^2 - y^2)[\hat{G}_{n}^n(z)]^2\partial_z G_{n}^n(z)}{2n^4} F_1 \left(\frac{(x - y) \hat{G}_{n}^n(z)}{n} \right)  \Bigg{]} +\frac{ \xi(z)}{n^2}.
\end{split}
\end{equation*}
The latter implies (\ref{S33E2}).\\

{\bf \raggedleft Step 2.} In this step we give the output of applying the single-level loop equations from \cite{BGG} and linearize the resulting equations using (\ref{S33E2}) from Step 1. Equations (\ref{S33E9}) and (\ref{S33E11}) below are the main outputs of this step. We first introduce a bit of notation.\\

If $\xi_1, \dots, \xi_m$ are bounded complex-valued random variables we write $M(\xi_1, \dots, \xi_m)$ for their joint cumulant. We refer the reader to Appendix \ref{AppendixA} for the definition and some basic properties of joint cumulants that we use in the paper. Let us fix a compact set $\mathcal{V}_1 \subset \mathbb{C}\setminus [-\theta, \delta^{-1}]$. We can find an $\eta > 0$ and a positively oriented contour $\Gamma$ that encloses $[-\theta, \delta^{-1}]$ such that $\Gamma \subset \mathcal{D}_\eta$, $\mathcal{V}_1 \subset \mathcal{D}_{\eta}$ with $\mathcal{D}_{\eta}$ as in (\ref{Deta}) and with $\Gamma$ and $\mathcal{V}_1$ being at least distance $\eta/2$ away from each other. We then fix $m \in \mathbb{Z}_{\geq 0}$ and $m+1$ points $v_0, v_1, \dots, v_m \in \mathcal{V}_1$. For any set $A= \{a_1, \dots, a_r\}  \subseteq \llbracket 1, m \rrbracket $ and bounded complex-valued random variable $\xi$ we write
\begin{equation}\label{S3Cum}
M(\xi; A): = M\left(\xi, \hat{G}_{n}^n(v_{a_1}), \hat{G}_{n}^n(v_{a_2}), \dots, \hat{G}_{n}^n(v_{a_r}) \right). 
\end{equation}

We then have the following consequence of \cite[(B.5)]{DK2020}
\begin{equation}\label{S3Loop}
\begin{split}
0 = &\sum_{A \subseteq \llbracket 1, m \rrbracket } \frac{1}{2\pi \i}  \oint_{\Gamma}dz  \prod_{a \in A^c} \left( \frac{n^{-1}}{(v_a - z)(v_a - z + n^{-1} )} \right) \frac{\Phi^{-}_{n,K}(zn)}{2n(z-v_0)} M\left(\prod_{i = 1}^n \frac{zn - \ell_i^n - \theta}{zn - \ell_i^n}; A\right)  \\
&+ \frac{1}{2\pi \i}  \oint_{\Gamma} dz  \frac{\Phi^{+}_{n, K}(zn)}{2 n (z- v_0)} M\left(\prod_{i = 1}^n \frac{zn - \ell_i^n + \theta - 1}{zn - \ell_i^n-1}; \llbracket 1, m \rrbracket \right) ,
\end{split}  
\end{equation}
where $\Phi^{\pm}_{n,K}$ are as in (\ref{S31E3}) and $A^c = \llbracket 1, m \rrbracket \setminus A$. We mention that to identify (\ref{S3Loop}) with \cite[(B.5)]{DK2020} one needs to use the following correspondences:
$$n \leftrightarrow N, \hspace{3mm} \Phi^{\pm}_{n,K}(zn) \leftrightarrow N\cdot \Phi_N^{\pm}(zN + N \theta) , \hspace{3mm} \ell_i^{n} + n \theta \leftrightarrow \ell_i, \hspace{3mm} 2 \leftrightarrow H(z), \hspace{3mm} K + 1 + (n-1)\theta \leftrightarrow s_N,$$
$$ \Gamma \leftrightarrow \Gamma - \theta, \hspace{3mm}, v_0 \leftrightarrow v, \hspace{3mm} z_-(n/K) \leftrightarrow \alpha + \theta, \hspace{3mm} z_+(n/K) \leftrightarrow \beta + \theta, \hspace{3mm} \mathbb{C} \leftrightarrow \mathcal{M} $$
where we recall that $z_{\pm}(s)$ are as in (\ref{S31E12}), $\alpha, \beta, H(z)$ are as in \cite[Assumption 5, Section 3]{DK2020} and $\mathcal{M}, \Phi^{\pm}_N$ are as in \cite[Assumption 3, Section 3]{DK2020}. We further note that the first two lines in \cite[(B.5)]{DK2020} are equal to zero since from (\ref{S31E3})
$$\Phi_N^-(0) = n^{-1}\Phi_{n,k}^-(-n\theta) = 0 \mbox{ and } \Phi^+_N(s_N) = n^{-1} \Phi^{+}_{n,K}(K + 1 - \theta) = 0.$$

We now apply (\ref{S33E2}) to conclude that there exists $B$ (depending on $\theta, \delta, \eta$) such that for $K \geq B$ and $n \in [\delta K, \delta^{-1}K]$
\begin{equation}\label{S3Loop2}
\begin{split}
&0 =  \frac{1}{2\pi \i}  \oint_{\Gamma} dz \sum_{A \subseteq \llbracket 1, m \rrbracket }  \prod_{a \in A^c} \left( \frac{n^{-1}}{(v_a - z)(v_a - z + n^{-1} )} \right) \frac{\Phi^{-}_{n,K}(zn) e^{-\theta G(z, n/K)}}{2n(z-v_0)} \\
&\times  M\left(1 - \frac{\theta \hat{G}_{n}^n(z)}{n} + \frac{\theta^2 \partial_z G(z,n/K)}{2n} ; A\right) +  \\
&\frac{\Phi^{+}_{n, K}(zn)e^{\theta G(z, n/K)}}{2 n (z- v_0)} M\left(1 + \frac{\theta \hat{G}_{n}^n(z)}{n} + \frac{(\theta^2-2\theta) \partial_z G(z,n/K)}{2n}; \llbracket 1, m \rrbracket \right) + \frac{\mathbb{E}[\xi_n(z)]}{n^2}\\
&+ n^{-2} \cdot \sum_{A \subseteq \llbracket 1, m \rrbracket } M \left(\xi_n(z) [ \hat{G}_{n}^n(z)]^2 +  \xi_n(z) \partial_z \hat{G}_{n}^n(z) +\xi_n(z) \hat{G}_{n}^n(z) +\xi_n(z); A\right).
\end{split}  
\end{equation}
In (\ref{S3Loop2}) and the sequel $\xi_n(z)$ stands for a generic random function of $v_0, \dots, v_m \in \mathcal{V}_1$ and $z \in \Gamma$ that is a.s. uniformly $O(1)$ over these sets. This constant depends on $\Gamma, \mathcal{V}_1, \theta, \delta, \eta$. We next analyze (\ref{S3Loop2}) when $m = 0$ and $m > 0$ separately and derive some simpler equations that will be used later.\\

When $m = 0$ we get from (\ref{S3Loop2}) that 
\begin{equation}\label{S33E5}
\begin{split}
&0 = \frac{1}{2\pi \i}  \oint_{\Gamma} dz \frac{\Phi^{-}_{n,K}(zn) e^{-\theta G(z, n/K)} + \Phi^{+}_{n, K}(zn)e^{\theta G(z, n/K)} }{2n (z - v_0) } \\
& +  \frac{\theta \mathbb{E}\left[ \hat{G}_{n}^n(z) \right]}{2n^2(z -v_0)} \cdot \left[ \Phi^{+}_{n, K}(zn)e^{\theta G(z, n/K)} - \Phi^{-}_{n,K}(zn) e^{-\theta G(z, n/K)}\right] +  \frac{\mathbb{E}[\xi_n(z)]}{n},
\end{split}
\end{equation}
where we mention that the $\partial_zG(z,n/K)$ terms got absorbed into $\xi_n(z)$ and we used that 
\begin{equation}\label{S33E6}
n^{-1} \Phi^{-}_{n,K}(zn) = z + \theta = \Phi^-(z, n/K) \mbox{ and } n^{-1} \Phi^{+}_{n,K}(zn) = \Phi^+(z, n/K) + n^{-1}(1-\theta),
\end{equation}
which can be deduced from (\ref{S31E3}) and (\ref{S31E10}). If we focus on the first line of (\ref{S33E5}), we see from (\ref{S31E7}), (\ref{S31E9}) and (\ref{S33E6}) that it equals
\begin{equation}\label{S33E7}
\frac{1}{2\pi \i}  \oint_{\Gamma} dz \frac{(K/n - \theta) + n^{-1}(1-\theta) e^{\theta G(z, n/K)} }{2 (z - v_0) }  =  \frac{1}{2\pi \i}  \oint_{\Gamma} dz \frac{ n^{-1}(1-\theta) e^{\theta G(z, n/K)} }{2 (z - v_0) },
\end{equation}
where we used Cauchy's theorem and the fact that $v_0$ is outside of $\Gamma$ to evaluate the integral of the first term to $0$. On the other hand, we have from (\ref{S31E16}) and (\ref{S33E6}) that 
\begin{equation}\label{S33E8}
\begin{split}
&n^{-1}\left[ \Phi^{+}_{n, K}(zn)e^{\theta G(z, n/K)} - \Phi^{-}_{n,K}(zn) e^{-\theta G(z, n/K)}\right] \\
&= 2 \sqrt{(z - z_+(n/K))(z - z_-(n/K))} + n^{-1}(1- \theta)e^{\theta G(z, n/K)} .
\end{split}
\end{equation}
Substituting (\ref{S33E7}) and (\ref{S33E8})  into (\ref{S33E5}) and utilizing that $n^{-1} \hat{G}_{n}^n(z) = \xi_n(z)$ we get
\begin{equation*}
\begin{split}
&0 = \frac{1}{2\pi \i}  \oint_{\Gamma} dz \frac{\mathbb{E}[\xi_n(z)]}{n} +  \frac{\theta \mathbb{E}\left[ \hat{G}_{n}^n(z) \right] \cdot \sqrt{(z - z_+(n/K))(z - z_-(n/K))} }{n(z -v_0)}.
\end{split}
\end{equation*}
We can evaluate the integral of the second term as (minus) the residue at $z= v_0$ (note there is no residue at infinity), multiply both sides by $n$ and divide by $ \theta\sqrt{(v_0 - z_+(n/K))(v_0 - z_-(n/K))}$. The result of these operations is
\begin{equation}\label{S33E9}
\begin{split}
&\mathbb{E}\left[ \hat{G}_{n}^n(v_0) \right] = \frac{\theta^{-1} }{2\pi \i \cdot \sqrt{(v_0 - z_+(n/K))(v_0 - z_-(n/K))}}  \oint_{\Gamma}dz\mathbb{E}[\xi_n(z)] = \frac{1 }{2\pi \i }  \oint_{\Gamma}dz\mathbb{E}[\xi_n(z)],
\end{split}
\end{equation}
where in the last equality we used (\ref{S31E13}). 

The last thing we do in this step is simplify (\ref{S3Loop2}) when $m > 0$. We use
$$\frac{1}{(v_a - z)(v_a - z + n^{-1} )} = \frac{1}{(v_a - z)^2} + \frac{\xi_n(z)}{n}$$
and (\ref{S33E6}), linearity of cumulants and the fact that the joint cumulant of any nonempty collection of bounded random variables and a constant is zero, see (\ref{S3Linearity}), to simplify (\ref{S3Loop2}) to
\begin{equation}\label{S33E10}
\begin{split}
&0 =  \frac{1}{2\pi \i}  \oint_{\Gamma} dz  \frac{\theta \left[ \Phi^{+}_{n, K}(zn)e^{\theta G(z, n/K)} - \Phi^{-}_{n,K}(zn) e^{-\theta G(z, n/K)}\right]}{2n^2(z -v_0)} \cdot  M\left( \hat{G}_{n}^n(z); \llbracket 1 ,m \rrbracket \right) \\
&+  {\bf 1}\{m = 1\} \cdot \frac{\Phi^-(z, n/K) e^{-\theta G(z, n/K)}}{2n(z- v_0)(v_1 - z)^2 } \\
&+ n^{-2} \cdot \sum_{A \subseteq \llbracket 1, m \rrbracket } M \left(\xi_n(z) [ \hat{G}_{n}^n(z)]^2 +  \xi_n(z) \partial_z \hat{G}_{n}^n(z) +\xi_n(z) \hat{G}_{n}^n(z) +\xi_n(z); A\right).
\end{split}  
\end{equation}
We can now use (\ref{S33E8}) to simplify the first line in (\ref{S33E10}) to
$$ \frac{1}{2\pi \i}  \oint_{\Gamma} dz  \frac{\theta \sqrt{(z - z_+(n/K))(z - z_-(n/K))} }{n(z -v_0)} \cdot  M\left( \hat{G}_{n}^n(z); \llbracket 1 ,m \rrbracket \right)  + M(n^{-2} \xi_n(z) \hat{G}_{n}^n(z); \llbracket 1 , m \rrbracket).$$
As before, we can evaluate the first term as (minus) the residue at $z = v_0$ put the result back into (\ref{S33E10}), multiply both sides by $n$ and divide by $ \theta\sqrt{(v_0 - z_+(n/K))(v_0 - z_-(n/K))}$. The result of these operations is
\begin{equation}\label{S33E11}
\begin{split}
&M\left( \hat{G}_{n}^n(v_0); \llbracket 1 ,m \rrbracket \right)  =  \frac{1}{2\pi \i}  \oint_{\Gamma} dz     \frac{{\bf 1}\{m = 1\} \cdot \theta^{-1} \cdot \Phi^-(z, n/K) e^{-\theta G(z, n/K)}}{2(z- v_0)(v_1 - z)^2  \sqrt{(v_0 - z_+(n/K))(v_0 - z_-(n/K))} } \\
&+ n^{-1} \cdot \sum_{A \subseteq \llbracket 1, m \rrbracket } M \left(\xi_n(z) [ \hat{G}_{n}^n(z)]^2 +  \xi_n(z) \partial_z \hat{G}_{n}^n(z) +\xi_n(z) \hat{G}_{n}^n(z) +\xi_n(z); A\right).
\end{split}  
\end{equation}

{\bf \raggedleft Step 3.}  In this step we reduce the proof of the proposition to the establishment of the following self-improvement estimate claim.\\

{\bf \raggedleft Claim:} Suppose that for some $H, M \in \mathbb{N}$, and all $n, K \in\mathbb{N}$ with $n/K \in [\delta, \delta^{-1}]$ we have that 
\begin{equation}\label{S33E12}
\mathbb{E} \left[ \prod_{a = 1}^m \left|\hat{G}_{n}^n(v_a) \right| \right]= O(1) + O\left(n^{m/2 + 1 - M/2}\right) \mbox{ for $m = 1, \dots, 4H + 4$,}
\end{equation}
then
\begin{equation}\label{S33E13}
\mathbb{E} \left[ \prod_{a = 1}^m \left|\hat{G}_{n}^n(v_a)\right| \right]= O(1) + O\left(n^{m/2 + 1 - (M+1)/2}\right) \mbox{ for $m = 1, \dots, 4H $,}
\end{equation}
where the constants in the big $O$ notations are uniform as $v_a$ vary over compacts in $\mathbb{C} \setminus [-\theta, \delta^{-1}]$.
We prove the above claim in the following steps. Here we assume its validity and establish (\ref{S33E1}). \\

Fix $k \in \mathbb{N}$. From Proposition \ref{S32P1} with $\varepsilon = 1/2$ we have
\begin{equation*}
\mathbb{E} \left[ \prod_{a = 1}^m \left|\hat{G}_{n}^n(v)\right|^m \right] =O \left( n^{m/2 + 1/2} \right) \mbox{ for $m = 1, \dots, 8k + 4$}.
\end{equation*}
Using H{\"o}lder's inequality, the above implies that (\ref{S33E12}) holds for the for the pair $H = 2k$ and $M = 1$. The conclusion is that (\ref{S33E12}) holds for the pair $H  = 2k- 1$ and $M = 2$. Iterating the argument an additional $k$ times we conclude that (\ref{S33E12}) holds with $H = k - 1$ and $M = k+2$, which implies (\ref{S33E1}).\\

{\bf \raggedleft Step 4.}  In this step we will prove (\ref{S33E13}) except for a single case, which will be handled separately in the next step. 

The first thing we show is that 
\begin{equation}\label{S33E14}
M\left( \hat{G}_{n}^n(v_0),\hat{G}_{n}^n(v_1), \dots, \hat{G}_{n}^n(v_m) \right) = O(1) + O\left( n^{m/2 + 1 - M/2}\right) \mbox{ for $m = 1, \dots, 4H+2$},
\end{equation}
where constant in the big $O$ notation are uniform over $v_0, \dots, v_m$ in compact subsets of $\mathbb{C} \setminus [-\theta, \delta^{-1}]$.

We start by fixing $\mathcal{V}_1$ to be a compact subset of $\mathbb{C} \setminus [-\theta, \delta^{-1}]$, which is invariant under conjugation. We also fix $\eta, \Gamma$ as in Step 2. From (\ref{S33E11}) we have
\begin{equation}\label{S33E15}
\begin{split}
&M\left( \hat{G}_{n}^n(v_0),\hat{G}_{n}^n(v_1), \dots, \hat{G}_{n}^n(v_m) \right)    \\
& = \frac{1}{2\pi \i}  \oint_{\Gamma} dz     \frac{{\bf 1}\{m = 1\} \cdot \theta^{-1} \cdot \Phi^-(z, n/K) e^{-\theta G(z, n/K)}}{2(z- v_0)(v_1 - z)^2  \sqrt{(v_0 - z_+(n/K))(v_0 - z_-(n/K))} } \\
&+ n^{-1} \cdot \sum_{A \subseteq \llbracket 1, m \rrbracket } M \left(\xi_n(z) [ \hat{G}_{n}^n(z)]^2 +  \xi_n(z) \partial_z \hat{G}_{n}^n(z) +\xi_n(z) \hat{G}_{n}^n(z) +\xi_n(z); A\right).
\end{split}  
\end{equation}
In addition, from (\ref{S33E9}) we have
\begin{equation}\label{S33E16}
\begin{split}
&M\left( \hat{G}_{n}^n(v_0) \right)  = \mathbb{E} \left[\hat{G}_{n}^n(v_0) \right]= O(1).
\end{split}  
\end{equation}
We next use the fact that cumulants can be expressed as linear combinations of products of moments, see (\ref{Mal2}). This means that $M(\xi_1, \dots, \xi_r)$ can be controlled by the quantities $1$ and $\mathbb{E} \left[ |\xi_i|^r \right]$ for $ i = 1, \dots, r$. This and (\ref{S33E12}) together imply for $A \subseteq \llbracket 1, m \rrbracket$
$$n^{-1} \cdot M \left(\xi_n(z) + \xi_n(z) \hat{G}_{n}^n(z) ; A\right) = O(1) + O\left(n^{m/2 + 1/2 - M/2} \right), \hspace{2mm} $$
$$n^{-1} \cdot M \left(\xi_n(z) \partial_z \hat{G}_{n}^n(z) ; A\right) = O(1) + O\left(n^{m/2 + 1/2 - M/2} \right), \hspace{2mm} $$
$$n^{-1} \cdot M \left(\xi_n(z) [ \hat{G}_{n}^n(z)]^2 ; A\right) = O(1) + O\left(n^{m/2 + 1 - M/2} \right).$$
We mention that the second equality above utilizes Cauchy's inequalities, see e.g. \cite[Corollary 4.3]{SS}, which shows that the moment bounds we have for $\mathbb{E} \left[ |\hat{G}_{n}^n(z)|^m \right]$ in (\ref{S33E12}) imply analogous ones for $\mathbb{E} \left[ |\partial_z \hat{G}_{n}^n(z)|^m \right]$. In addition, the third equality above utilizes a special case of Malyshev's formula, see (\ref{Mal3}). Combining the last three estimates with (\ref{S33E15}) (where we mention that the second line is $O(1)$) gives (\ref{S33E14}).\\

Notice that by H{\"o}lder's inequality we have
\begin{equation*}
\sup_{v_1, \dots, v_m \in \mathcal{V}_1} \mathbb{E} \left[ \prod_{a = 1}^m \left|\hat{G}_{n}^n(v_a) \right| \right]\leq \sup_{v \in \mathcal{V}_1} \mathbb{E} \left[ \left|\hat{G}_{n}^n(v) \right|^m \right],
\end{equation*}
and so to finish the proof it suffices to show that for $m = 1, \dots, 4H$ we have
\begin{equation}\label{S33E17}
 \mathbb{E} \left[ \left|\hat{G}_{n}^n(v)\right|^m \right] = O(1)  + O(n^{m/2 + 1 - (M+1)/2}).
\end{equation}
Using the fact that one can express joint moments as linear combinations of products of joint cumulants, see (\ref{Mal1}), we deduce from (\ref{S33E14}) and (\ref{S33E16}) that
\begin{equation}\label{S33E18}
\sup_{ v_0, v_1, \dots, v_{m-1} \in \mathcal{V}_1} \mathbb{E} \left[ \prod_{a = 0}^{m-1}\hat{G}_{n}^n(v_a) \right] = O(1) + O( n^{(m-1)/2 + 1 - M/2}) \mbox{ for $m = 1, \dots, 4H+2$}.
\end{equation}

If $m = 2m_1$, we set $v_0 = v_1 = \cdots = v_{m_1 - 1} = v$ and $v_{m_1} = \cdots = v_{2m_1 - 1} = \overline{v}$ in (\ref{S33E18}) and get
\begin{equation}\label{S33E19}
\sup_{ v \in \mathcal{V}_1} \mathbb{E} \left[ \left|\hat{G}_{n}^n(v)\right|^m \right] = O(1) + O( n^{m/2 + 1/2 - M/2}) \mbox{ for $m = 1, \dots, 4H+2$, $m$ even}.
\end{equation}
In deriving the above we used that $\hat{G}_{n}^n(\overline{v}) = \overline{\hat{G}_{n}^n(v)}$ and so $\hat{G}_{n}^n(v) \cdot \hat{G}_{n}^n(\overline{v}) = |\hat{G}_{n}^n(v)|^2$. 

We next let $m = 2m_1 + 1$ be odd and notice that by the Cauchy-Schwarz inequality and (\ref{S33E19}) 
\begin{equation}\label{S33E20}
\begin{split}
&\sup_{ v \in \mathcal{V}_1} \mathbb{E} \left[ \left|\hat{G}_{n}^n(v)\right|^{2m_1 + 1} \right] \leq \sup_{ v \in \mathcal{V}_1} \mathbb{E} \left[ \left|\hat{G}_{n}^n(v)\right|^{2m_1 + 2} \right]^{1/2} \cdot \mathbb{E} \left[ \left|\hat{G}_{n}^n(v) \right|^{2m_1} \right]^{1/2} =  \\  
&O(1) + O( n^{m_1 + 1 - M/2}) + O(n^{m_1/2 + 3/4 - M/4}) \mbox{ for $m = 1 , \dots, 4H + 1$, $m$ odd}.
\end{split}
\end{equation}
We note that the bottom line of (\ref{S33E20}) is $O(1) + O(n^{m_1 + 1 - M/2})$ except when $M = 2m_1 + 2$, since
$$m_1/2 + 3/4 - M/4 \leq \begin{cases} m_1 + 1 - M/2 &\mbox{ when $M \leq 2m_1 + 1$,} \\ 0 &\mbox{ when $M \geq 2m_1 + 3$.} \end{cases}$$
Consequently, (\ref{S33E19}) and (\ref{S33E20}) together imply (\ref{S33E13}), except when $M = 2m_1 +2$ and $m = 2m_1 + 1$. We will handle this case in the next and final step.\\

{\bf \raggedleft Step 5.}  In this step we will show that (\ref{S33E13}) holds even when $M = 2m_1 + 2$ and $4H > m = 2m_1 + 1$. In the previous step we showed in (\ref{S33E19}) that
$\sup_{v \in \mathcal{V}_1} \mathbb{E} \left[ \left|\hat{G}_{n}^n(v)\right|^{2m_1 + 2} \right] = O(n^{1/2})$, and below we will improve this estimate to
\begin{equation}\label{S33E21}
\sup_{v \in \mathcal{V}_1} \mathbb{E} \left[ \left|\hat{G}_{n}^n(v)\right|^{2m_1 + 2} \right] = O(1).
\end{equation}
The trivial inequality $|x|^{2m_1 + 2} + 1 \geq |x|^{2m_1 + 1}$ together with (\ref{S33E21}) imply
$$\sup_{v \in \mathcal{V}_1} \mathbb{E} \left[ \left|\hat{G}_{n}^n(v)\right|^{2m_1 + 1} \right]  = O(1).$$
Consequently, we have reduced the proof of the claim to establishing (\ref{S33E21}). \\

Let us list the relevant estimates we will need
\begin{equation}\label{S33E22}
\begin{split}
& \mathbb{E} \left[ \prod_{a = 1}^{2m_1 + 2} \left|\hat{G}_{n}^n(v_a)\right| \right] = O(n^{1/2})\mbox{, } \hspace{2mm} \mathbb{E} \left[ \prod_{a = 1}^j \left|\hat{G}_{n}^n(v_a)\right| \right] = O(1) \mbox{ for $0 \leq j \leq 2m_1$,} \\\
&\mathbb{E} \left[ \prod_{a = 1}^{2m_1 + 3} \left|\hat{G}_{n}^n(v_a)\right| \right] = O(n) , \hspace{2mm} \mathbb{E} \left[ \prod_{a = 1}^{2m_1 + 1} \left|\hat{G}_{n}^n(v_a)\right| \right] = O(n^{1/4}).
\end{split}
\end{equation}
All of these identities follow from (\ref{S33E19}) and (\ref{S33E20}), which we showed to hold in the previous step. Below we feed the improved estimates of (\ref{S33E22}) into Step 4, which will ultimately yield (\ref{S33E21}).\\

Using (\ref{S33E22}) in place of (\ref{S33E12}) in Step 4 we get for $A \subseteq \llbracket 1, m \rrbracket$
$$n^{-1} \cdot M \left(\xi_n(z) + \xi_n(z) \hat{G}_{n}^n(z) + \xi_n(z) \partial_z \hat{G}_{n}^n(z) + \xi_n(z) [ \hat{G}_{n}^n(z)]^2 ; A\right) = O(1),$$
which together with (\ref{S33E15}) gives the following improvement over (\ref{S33E14})
\begin{equation}\label{S33E23}
M\left( \hat{G}_{n}^n(v_0),\hat{G}_{n}^n(v_1), \dots, \hat{G}_{n}^n(v_m)\right) = O(1). 
\end{equation}
We now use (\ref{S33E23}) in place of (\ref{S33E14}) to get the following improvement over (\ref{S33E18})
\begin{equation}\label{S33E24}
\sup_{ v_0, v_1, \dots, v_{m} \in \mathcal{V}_1} \mathbb{E} \left[ \prod_{a = 0}^{m}\hat{G}_{n}^n(v_a) \right] = O(1).
\end{equation}
Setting $v_0 = v_1 = \cdots = v_{m_1} = v$ and $v_{m_1+1} = \cdots = v_{2m_1 + 1} = \overline{v}$ in (\ref{S33E24}) gives (\ref{S33E21}).
\end{proof}

%-------------------------------------------------------------------------------------------------------------------------------------------------------------------------------------------------
% Section 4
%
%-------------------------------------------------------------------------------------------------------------------------------------------------------------------------------------------------
\section{Two-level analysis}\label{Section4} In Proposition \ref{S33P1} we showed that the $k$-th moments of $\hat{G}_{n}^n(z)$ as in (\ref{S32E1}) are all bounded. In this section we show that the same is true for $n^{1/2} \cdot [ \hat{G}_{n}^n(z) - \hat{G}_{n}^{n-1}(z)]$ -- this is Proposition \ref{S43P1}. In Section \ref{Section4.1} we derive a certain product expansion formula, Proposition \ref{ProdExp1}, which improves (\ref{S33E2}) by utilizing the moment bounds from Proposition \ref{S33P1}. In Section \ref{Section4.2} we utilize the two-level Nekrasov's equations from \cite{DK2020} to obtain certain cumulant identities involving $n^{1/2} \cdot [ \hat{G}_{n}^n(z) - \hat{G}_{n}^{n-1}(z)]$. Those identities are then utilized in Section \ref{Section4.3} to prove Proposition \ref{S43P1}.

%-------------------------------------------------------------------------------------------------------------------------------------------------------------------------------------------------
% Section 4.1
%
%-------------------------------------------------------------------------------------------------------------------------------------------------------------------------------------------------
\subsection{Product expansion formula}\label{Section4.1} The goal of this section is to establish the following result.

\begin{proposition}\label{ProdExp1} Fix $\theta >0$ and $\delta \in (0,1)$. Suppose that $(\ell^1, \ell^2, \dots)$ is distributed according to $\P$ as in Definition \ref{BKCC}, and let $n, L$ be integers such that $n/K, n/L \in [\delta, \delta^{-1}]$. Let us fix a compact set $\mathcal{V}_2 \subset \mathbb{C}$, an $\eta > 0$ and let $\mathcal{D}_{\eta}'$ be 
\begin{equation}\label{Deta2}
\mathcal{D}_{\eta}' = \{z \in \mathbb{C}: |z - x| > \eta \mbox{ for all $x \in [-\theta \delta^{-1}, \delta^{-2}]$} \}.
\end{equation}
Note that since $\ell_i^n/n \in [-\theta, \delta^{-1}]$ almost surely, we can find $B > 0$ (depending on $\eta, \delta, \mathcal{V}_2$) such that if $K \geq B$ and $n/K, n/L \in [\delta, \delta^{-1}]$ we have $\ell_i^n/L \pm x/L \not \in \mathcal{D}_{\eta}'$ for each $x \in \mathcal{V}_2$ and $i \in \llbracket 1, n \rrbracket$.

With the above data we define for $K \geq B$, $x,y \in \mathcal{V}_2$ the deterministic functions
\begin{equation}\label{S41E1}
\begin{split}
&  A^{\mathsf{s},1}_{n,L}(z;x,y) =  1 , \hspace{2mm} A^{\mathsf{s},2}_{n,L}(z;x,y) =  \frac{L(x^2 - y^2)}{2n} \cdot \partial_z G\left(zL/n, n/K\right) ,  \\
&A^{\mathsf{s},3}_{n,L}(z;x,y) = \frac{L^2(x^3 - y^3) \partial_z^2 G\left( zL/n, n/K\right)}{6n^2}   + \frac{L^2(x^2 - y^2)^2 [\partial_z G\left(zL/n, n/K\right) ]^2}{8n^2}  ,
\end{split}
\end{equation}
and the random functions 
\begin{equation}\label{S41E2}
\begin{split}
 B_{n,L}^{\mathsf{s},1}(z;x,y) = \hspace{2mm}& (x-y) \cdot  \hat{G}_{L}^n(z), \\
 B_{n,L}^{\mathsf{s},2}(z;x,y) =\hspace{2mm} &  \frac{x^2 - y^2}{2} \cdot \partial_z \hat{G}_{L}^n\left(z\right) + \frac{1}{2} \left( (x-y) \cdot \hat{G}_{L}^n\left(z \right) \right)^2 +   \\
&    \frac{L(x-y)(x^2 - y^2)}{2n} \cdot \partial_z G\left(zL/n, n/K\right)  \cdot  \hat{G}_{L}^n\left(z \right)   , \\
B_{n,L}^{\mathsf{s},3}(z;x,y) =\hspace{2mm} & L^3 \hspace{-1mm} \left( \hspace{-1mm} e^{(y-x) G\left(zL/n, n/K\right)} \hspace{-1mm} \prod_{ i = 1}^n \hspace{-1mm} \frac{zL - \ell_i^n  + x }{zL - \ell_i^n +y}  -\sum_{i = 1}^3 \frac{A_{n,L}^{\mathsf{s},i}(z;x,y)}{L^{i-1}} - \sum_{i = 1}^2 \frac{B_{n,L}^{\mathsf{s},i}(z;x,y) }{L^i}\hspace{-1mm} \right) \hspace{-1mm}.
\end{split}
\end{equation}
We claim that the functions in (\ref{S41E1}) and (\ref{S41E2}) are $\P$-a.s. holomorphic in $\mathcal{D}_{\eta}'$ and for each $k \in \mathbb{N}$
\begin{equation}\label{S41E3}
\begin{split}
&\sum_{i = 1}^3\mathbb{E}\left[ \left|A^{\mathsf{s},i}_{n,L}(z;x,y)  \right|^k + \left|B^{\mathsf{s},i}_{n,L}(z;x,y)  \right|^k \right]  = O(1),
\end{split}
\end{equation}
where the constant in the big $O$ notation depends on $\theta, \delta, \eta, k$ and is uniform as $z$ varies over compact subsets of $\mathcal{D}_{\eta}'$ and $x,y \in \mathcal{V}_2$. 
\end{proposition}
\begin{remark}\label{S41R1} In plain words, the functions $A^{\mathsf{s},i}_{n,L}(z;x,y)$ and $B^{\mathsf{s},i}_{n,L}(z;x,y)$ are the deterministic and random functions that arise when we linearize 
$$ e^{(y-x) G\left(zL/n, n/K\right)} \prod_{ i = 1}^n \frac{zL - \ell_i^n  + x }{zL - \ell_i^n +y},$$
and group terms according to their size (as powers of $L$). The introduction of $B$ and $\mathcal{D}_{\eta}'$ is to ensure that $zL +x$ and $zL + y$ are bounded away from $\ell_i^n$ (this is where the above product is singular), which allows the linearization.
\end{remark}
\begin{remark}\label{S41R2} The formulation in Proposition \ref{ProdExp1} is made with an outlook to applying the result in our multi-level analysis in Section \ref{Section5}. In the present section we will only use the result for the variables $\ell_i^n$ and $\ell_i^{n-1}$ when $L = n$. In this case, for any $\eta > 0$ we can find $B > 0$ (depending on $\eta, \delta, \mathcal{V}_2$) such that if $K \geq B$ and $(n-1)/K, n/K \in [\delta, \delta^{-1}]$ then $\ell_i^n/n \pm x/n \not \in \mathcal{D}_{\eta}$ (as in (\ref{Deta})) for each $x \in \mathcal{V}_2$ and $i \in \llbracket 1, n \rrbracket$ and also $\ell_i^{n-1}/n \pm x/n \not \in \mathcal{D}_{\eta}$ for each $x \in \mathcal{V}_2$ and $i \in \llbracket 1, n-1 \rrbracket$. One can then essentially verbatim repeat the the proof below to conclude that the functions in (\ref{S41E1}) and (\ref{S41E2}) are analytic in $\mathcal{D}_{\eta}$ (as opposed to $\mathcal{D}_{\eta}'$) and the bound in (\ref{S41E3}) is uniform over compact subsets of $\mathcal{D}_{\eta}$ (as opposed to $\mathcal{D}_{\eta}'$), provided that $K \geq B$ and $(n-1)/K, n/K \in [\delta, \delta^{-1}]$. 
\end{remark}

\begin{proof} We first establish the analyticity of the functions in (\ref{S41E1}) and (\ref{S41E2}). From (\ref{S31E6}) we know that $G(z,s)$ is analytic in $\mathbb{C}\setminus [-\theta, s^{-1}]$ and since $n/K, n/L \in [\delta, \delta^{-1}]$ we conclude that $G(zL/n,n/K)$ is analytic in $\mathcal{D}_{\eta}'$. In addition, since $\P$-a.s. we know that $\ell_i^n/n \in [-\theta, K/n]$, we conclude from (\ref{S32E1}) that $G_{L}^n(z)$ and $\hat{G}_{L}^n(z)$ are both analytic in $\mathcal{D}_{\eta}'$. The latter observations show that all the functions in (\ref{S41E1}) and also $ B_{n,L}^{\mathsf{s},1}(z;x,y)$, $ B_{n,L}^{\mathsf{s},2}(z;x,y)$ are analytic in $\mathcal{D}_{\eta}'$. For $B_{n,L}^{\mathsf{s},3}(z;x,y) $ we note that analyticity follows from the analyticity of the earlier functions as well as that of the product (here is where we use that $K \geq B$, which ensures that $zL + y $ is bounded away from $\ell_i^n$ for $i = 1, \dots, n$, $z \in \mathcal{D}_{\eta}'$ and $y \in \mathcal{V}_2$). \\

In the remainder of the proof we focus on (\ref{S41E3}). We fix a compact set $\mathcal{V}_1 \subset \mathcal{D}_{\eta}'$ and throughout the proof write $\xi_{n,L}(z;x,y)$ and $\zeta_{n,L}(z;x,y)$ to mean generic random functions of $n, L$ such that $n/K, n/L  \in [\delta , \delta^{-1}]$,  $x,y \in \mathcal{V}_2$ and $z \in \mathcal{D}_{\eta}'$ that are analytic on $\mathcal{D}_{\eta}'$ for fixed $x,y \in \mathcal{V}_2$  and such that $\P$-almost surely
$$   \max_{\delta^{-1}  \geq n/L, n/K \geq \delta }  \sup_{z \in \mathcal{V}_1} \sup_{ x, y \in \mathcal{V}_2}  \left| \xi_{n,L}(z;x,y) \right| = O(1), \mbox{ and  }$$
$$\max_{\delta^{-1}  \geq n/L, n/K \geq \delta } \max_{\delta K \geq n \geq \delta^{-1} K}  \sup_{z \in \mathcal{V}_1}\sup_{ x, y \in \mathcal{V}_2}  \mathbb{E} \left[ \left|\zeta_{n,L}(z;x,y) \right|^k \right] = O(1),$$
for $k \in \mathbb{N}$, where the constants in the big $O$ notations depend on $\mathcal{V}_1$, $\mathcal{V}_2$, $\theta, \delta, \eta$ and also on $k$ for the second identity, but not on $K$ provided it is larger than $B$ as in the statement of the proposition. To ease the notation, we drop $x,y, n,L$ and simply write $\xi(z), \zeta(z)$. The meaning of $\xi(z)$ and $\zeta(z)$ may be different from line to line or even within the same equation. In the sequel we also assume that $K \geq B$, $n/K, n/L \in [\delta, \delta^{-1}]$, $x, y \in \mathcal{V}_2$ and $z \in \mathcal{D}_{\eta}'$ without mention.\\

Recalling the definition of $G_{L}^n(z)$ from (\ref{S32E1}) we have that 
$$G_{L}^n(z) = \sum_{i = 1}^{n} \frac{1}{z - \ell_i^n/L}, \hspace{2mm} \partial_z G_{L}^n(z) = - \sum_{i = 1}^{n} \frac{1}{(z - \ell_i^n/L)^2}, \hspace{2mm} \partial^2_z G_{L}^n(z) =  \sum_{i = 1}^{n} \frac{2}{(z - \ell_i^n/L)^3}.$$
By Taylor expanding the logarithm and using the latter identities we have
\begin{equation}\label{S41E4}
\begin{split}
& \prod_{ i = 1}^n \frac{zL - \ell^n_i + x}{zL - \ell^n_i + y} = \exp \left[ \sum_{i = 1}^n \log \left( 1 + \frac{1}{L} \cdot \frac{x}{z - \ell^n_i/L } \right) -  \sum_{i = 1}^n \log \left( 1 + \frac{1}{L} \cdot \frac{y}{z - \ell^n_i/L } \right) \right] = \\
&  \exp\left[ \frac{(x - y)  G_{L}^n(z) }{L}+ \frac{(x^2 - y^2)  \partial_z G_{L}^n(z)}{2L^2} + \frac{(x^3 - y^3)  \partial^2_zG_{L}^n(z)}{6L^3} +  \frac{ \xi(z)}{L^3}\right].
\end{split}
\end{equation}
By Taylor expanding the exponential function we get
\begin{equation}\label{S41E5}
\begin{split}
\exp \left(  \frac{(x^3 - y^3)  \partial^2_zG_{L}^n(z)}{6L^3}\right) = 1 +  \frac{(x^3 - y^3)  \partial^2_zG_{L}^n(z)}{6L^3} +\frac{\xi(z)}{L^4},
\end{split}
\end{equation}
\begin{equation}\label{S41E6}
\begin{split}
\exp \left(  \frac{(x^2 - y^2)  \partial_z G_{L}^n(z)}{2L^2}  \right) = 1 +   \frac{(x^2 - y^2)  \partial_z G_{L}^n(z)}{2L^2}  +  \frac{(x^2 - y^2)^2 [\partial_z G_{L}^n(z)]^2}{8L^4} + \frac{\xi(z)}{L^3}.
\end{split}
\end{equation}

From Proposition \ref{S33P1}, see (\ref{MBRem}) in Remark \ref{MomentBoundRem}), we have
\begin{equation}\label{S41E7}
\begin{split}
&\hat{G}_{L}^n(z) =  \zeta(z), \hspace{2mm} \partial_z  \hat{G}_{L}^n(z) =  \zeta(z), \hspace{2mm} \partial^2_z  \hat{G}_{L}^n(z)  = \zeta(z).
\end{split}
\end{equation}
Combining (\ref{S41E5}), (\ref{S41E6}) and (\ref{S41E7}) with the definition of $\hat{G}_{L}^n(z)$ from (\ref{S32E1})  we get
\begin{equation}\label{S41E8}
\begin{split}
&\exp \left(  \frac{(x^3 - y^3) \partial^2_zG_{L}^n(z)}{6L^3}\right) = 1 +  \frac{(x^3 - y^3)  \partial^2_zG(zL/n, n/K)}{6n^2} +\frac{\zeta(z)}{L^3}, \\
\end{split}
\end{equation}
\begin{equation}\label{S41E9}
\begin{split}
\exp \left(  \frac{(x^2 - y^2) \partial_z G_{L}^n(z)}{2L^2}  \right) = \hspace{2mm} 1 &+  \frac{(x^2 - y^2)  \partial_z G(zL/n, n/K)}{2n} +  \frac{(x^2 - y^2)  \partial_z \hat{G}^n_{L}(z)}{2L^2} \\
& +\frac{(x^2 - y^2)^2 [\partial_z G(zL/n,n/K)]^2}{8n^2}+\frac{\zeta(z)}{L^3}.
\end{split}
\end{equation}
Setting $F_2(u) = u^{-3}(e^u - 1  -u - u^2/2)$ and noting that $L^{-1} G_{L}^n(z) = \xi(z)$, we get
\begin{equation}\label{S41E10}
\begin{split}
&e^{(y- x) G(zL/n, n/K)}  \cdot \exp\left(\frac{(x - y)  G_{L}^n(z) }{L} \right) =  \exp\left(\frac{(x - y)  \hat{G}_{L}^n(z) }{L} \right) \\
&= 1 +  \frac{(x - y)  \hat{G}_{L}^n(z) }{L} + \frac{(x-y)^2 [ \hat{G}_{L}^n(z)]^2 }{2L^2} +  \frac{(x-y)^3 \hat{G}_{L}^n(z)^3}{L^3}   \cdot F_2 \left(\frac{(x-y)  \hat{G}^n_{L}(z) }{L}  \right)  \\
& = 1 +  \frac{(x-y)  \hat{G}_{L}^n(z) }{L} + \frac{(x-y)^2  [\hat{G}^n_{L}(z)]^2 }{2L^2}  + \frac{\zeta(z)}{L^3}.
\end{split}
\end{equation}
Substituting (\ref{S41E8}), (\ref{S41E9}) and (\ref{S41E10}) into (\ref{S41E4}) we get
\begin{equation}\label{S41E11}
\begin{split}
& e^{(y- x) G(zL/n, n/K)}  \prod_{ i = 1}^n \frac{zL - \ell^n_i + x}{zL - \ell^n_i + y} = \sum_{i = 1}^3 \frac{A_{n,L}^{\mathsf{s},i}(z;x,y)}{L^{i-1}} + \sum_{i = 1}^2 \frac{B_{n,L}^{\mathsf{s},i}(z;x,y) }{L^i} + \frac{\zeta(z)}{L^3}.
\end{split}
\end{equation}
The latter shows that $B_{n,L}^{\mathsf{s},3}(z;x,y) = \zeta(z)$ and so its $k$-th moments are all $O(1)$. The fact that the other functions in (\ref{S41E1}) and (\ref{S41E2}) have $O(1)$ moments is a consequence of (\ref{S41E7}). This completes the proof of (\ref{S41E3}).
\end{proof}

%-------------------------------------------------------------------------------------------------------------------------------------------------------------------------------------------------
% Section 4.2
%
%-------------------------------------------------------------------------------------------------------------------------------------------------------------------------------------------------
\subsection{Two-level loop equations}\label{Section4.2} In this section we state the consequence of applying the two-level loop equations from \cite{DK2020} to our model. The main outputs of this section are equations (\ref{S42E21}) and (\ref{S42E22}) below. Throughout we assume that $(\ell^1, \ell^2, \dots)$ are distributed according to $\P$ from Definition \ref{BKCC}. We start by introducing some useful notation.

We fix $\delta \in (0,1)$ and a compact set $\mathcal{V}_1 \subset \mathbb{C} \setminus [-\theta, \delta^{-1}]$. We can find an $\eta > 0$ and a positively oriented contour $\Gamma$ that encloses $[-\theta , \delta^{-1}]$ such that $\Gamma \subset \mathcal{D}_\eta$, $\mathcal{V}_1 \subset \mathcal{D}_{\eta}$ with $\mathcal{D}_{\eta}$ as in (\ref{Deta}) and with $\Gamma$ and $\mathcal{V}_1$ being at least distance $\eta/2$ away from each other. We then fix $m \in \mathbb{Z}_{\geq 0}$ and $m+1$ points $v_0, v_1, \dots, v_m \in \mathcal{V}_1$. For any set $A= \{a_1, \dots, a_r\}  \subseteq \llbracket 1, m \rrbracket $ and bounded complex-valued random variable $\xi$ we write
\begin{equation}\label{S4Cum}
M(\xi; A): = M\left(\xi, n^{1/2} \cdot [\hat{G}_{n}^n(v_{a_1}) - \hat{G}_{n}^{n-1}(v_{a_1})], \dots, n^{1/2} \cdot [\hat{G}_{n}^n(v_{a_r}) - \hat{G}_{n}^{n-1}(v_{a_r})] \right),
\end{equation}
where $M(\xi_1, \dots, \xi_r)$ denotes the joint cumulant of the variables $\xi_1, \dots, \xi_r$, see Appendix \ref{AppendixA}. We mention that the notation $M(\xi; A)$ coincides with (\ref{S3Cum}), where it means something else, but that notation is only used in Section \ref{Section3}, while the present one is the only one used in Section \ref{Section4}.

We have the following consequence of \cite[(4.20)]{DK2020} 
\begin{equation}\label{S42E1}
\begin{split}
0 =  & \frac{1}{2\pi \i} \oint_{\Gamma} dz  \frac{\Phi_{n,K}^-(zn) }{2  \theta (z - v_0) n^{m/2 + 1}} \cdot M \left( \prod_{ i = 1}^n \frac{zn - \ell_i^n - \theta}{zn - \ell_i^n} ; \llbracket 1, m \rrbracket  \right) \\
&+ \frac{\Phi_{n,K}^+(zn)}{2 \theta (z - v_0) n^{m/2 + 1}}  \cdot M \left( \prod_{i = 1}^{n-1} \frac{zn - \ell_i^{n-1} + \theta - 1}{zn - \ell_i^{n-1} - 1}; \llbracket 1, m \rrbracket  \right) \\
& + \sum_{A \subseteq \llbracket 1, m \rrbracket} \frac{\Phi^+_{n,K}(zn)}{2 \theta (z- v_0) n^{|A|/2 + 1}} \cdot \prod_{a \in A^c} \left( \frac{n^{-1}}{(v_a - z)(v_a - z + n^{-1}) } \right) \cdot M \left( \Pi_1^{\theta}(zn); A \right),
\end{split}
\end{equation}
where $A^ c = \llbracket 1, m \rrbracket \setminus A$, and 
\begin{equation}\label{S42E2}
\Pi_1^{\theta}(z) = \begin{cases} \mathlarger{\frac{\theta}{1- \theta} \cdot \prod_{i = 1}^n \frac{z - \ell_i^n - \theta}{z - \ell_i^n - 1} \cdot \prod_{i = 1}^{n-1} \frac{z - \ell_i^{n-1} + \theta - 1}{z - \ell_i^{n-1} }} &\mbox{ if } \theta \neq 1, \\ \mathlarger{\sum_{i = 1}^n \frac{1}{z - \ell_i^n - 1} - \sum_{i = 1}^{n-1} \frac{1}{z - \ell_i^{n-1}}} &\mbox{ if } \theta = 1. \end{cases}
\end{equation}
We also have the following consequence of \cite[(4.22)]{DK2020}
\begin{equation}\label{S42E3}
\begin{split}
&0 =  \frac{1}{2\pi \i} \oint_{\Gamma} dz \sum_{A \in \llbracket 1, m \rrbracket}  \frac{\Phi_{n,K}^+(zn) }{2  \theta (z - v_0) n^{|A|/2 + 1}} \cdot \prod_{a \in A^c} \left( \frac{-\theta n^{-2}(2 v_a - z_{\theta} - z^-)}{(v_a - z)(v_a - z^-)(v_a - z_{\theta})(v_a - z_{\theta}^-) } \right) \\
& \times M \left( \prod_{ i = 1}^n \frac{zn - \ell_i^n + \theta - 1}{zn - \ell_i^n - 1} ; A  \right) + \frac{\Phi_{n,K}^-(zn)}{2 \theta (z - v_0) n^{m/2 + 1}}  \cdot M \left( \prod_{i = 1}^{n-1} \frac{zn - \ell_i^{n-1} }{zn - \ell_i^{n-1} + \theta }; \llbracket 1, m \rrbracket  \right)  \\
& + \sum_{A \subseteq \llbracket 1, m \rrbracket} \frac{\Phi^-_{n,K}(zn)}{2 \theta (z- v_0) n^{|A|/2 + 1}} \cdot \prod_{a \in A^c} \left( \frac{n^{-1}}{(v_a - z_{\theta})(v_a - z_{\theta}^-) } \right) \cdot M \left( \Pi_2^{\theta}(zn); A \right),
\end{split}
\end{equation}
where $z^{\pm} = z \pm n^{-1}$, $z_{\theta} = z + \theta n^{-1}$, $z_{\theta}^{\pm} = z^{\pm} + \theta n^{-1}$ and 
\begin{equation}\label{S42E4}
\Pi_2^{\theta}(z) = \begin{cases} \mathlarger{\frac{\theta}{1- \theta} \cdot \prod_{i = 1}^n \frac{z - \ell_i^n + \theta - 1}{z - \ell_i^n} \cdot \prod_{i = 1}^{n-1} \frac{z - \ell_i^{n-1}}{z - \ell_i^{n-1} + \theta - 1}} &\mbox{ if } \theta \neq 1, \\  \mathlarger{\sum_{i = 1}^{n-1} \frac{1}{z - \ell_i^{n-1}} - \sum_{i = 1}^n \frac{1}{z - \ell_i^n}}  &\mbox{ if } \theta = 1. \end{cases}
\end{equation}
To identify (\ref{S42E1}) with \cite[(4.20)]{DK2020} and (\ref{S42E3}) with \cite[(4.22)]{DK2020} one needs to set $r = m$, $s = t = 0$ and use the following correspondences
$$n \leftrightarrow N, \hspace{3mm} \Phi^{\pm}_{n,K}(zn) \leftrightarrow N\cdot \Phi_N^{\pm}(zN + N \theta) , \hspace{3mm} \ell_i^{n} + n \theta \leftrightarrow \ell_i,\hspace{3mm} \ell_i^{n-1} + n \theta \leftrightarrow m_i, \hspace{3mm} 2\theta \leftrightarrow S(z), \hspace{3mm}  $$
$$K + 1 + (n-1)\theta \leftrightarrow s_N, \hspace{3mm} \Gamma \leftrightarrow \Gamma - \theta ,  \hspace{3mm}, v_0 \leftrightarrow v, \hspace{3mm},\mathbb{C} \leftrightarrow \mathcal{M},$$
$$ G^n_n(z) \leftrightarrow G^t_N(z +\theta), \hspace{3mm}G^{n-1}_n(z) \leftrightarrow G^b_N(z + \theta), \hspace{3mm} n^{1/2}[G_n^n(z) - G_n^{n-1}(z)] \leftrightarrow \Delta X_N(v_a),$$
where we recall that $\mathcal{M}, \Phi^{\pm}_N$ are as in \cite[Assumption 3, Section 3]{DK2020} and $G^t_N(z), G^{b}_N(z)$ and $\Delta X_{N}(z)$ are as in \cite[(3.18) and (4.1)]{DK2020}. We further note that $\tilde{V}_N^1$ and $\tilde{V}_N^2$ in \cite[(4.20)]{DK2020} and \cite[(4.22)]{DK2020} are equal to zero, since the random variables in \cite[(4.10)]{DK2020} are equal to zero, which in turn follows from (\ref{S31E3}) that implies 
$$\Phi_N^-(0) = n^{-1}\Phi_{n,K}^-(-n\theta) = 0 \mbox{ and } \Phi^+_N(s_N) = n^{-1} \Phi^{+}_{n,K}(K + 1 - \theta) = 0.$$
Finally, we remark that our joint cumulants in (\ref{S42E1}) and (\ref{S42E3})  involve the variables  $n^{1/2} \cdot [\hat{G}_{n}^n(v_{a}) - \hat{G}_{n}^{n-1}(v_{a})]$ as opposed to $n^{1/2} \cdot [{G}_{n}^n(v_{a}) - {G}_{n}^{n-1}(v_{a})]$ like in \cite{DK2020}, but this is not an issue since from (\ref{S32E1}) $\hat{G}_L^n(z)$ differs from $G_L^n(z)$ by a deterministic function, and second and higher order cumulants remain unchanged upon constant shifts, cf. (\ref{S3Linearity}).\\

In the remainder of the section, we use Proposition \ref{ProdExp1} to linearize (\ref{S42E1}) and (\ref{S42E3}), simplify the resulting equations and finally take their difference. The output of these procedures after considerable simplifications can be found in equations (\ref{S42E21}) and (\ref{S42E22}).

Below we write $\xi_n(z)$ for a generic random function of $v_0, \dots, v_m \in \mathcal{V}_1$ and $z \in \Gamma$ that that a.s. $O(1)$ over these sets. We also write $\zeta_n(z)$ for a generic random function of $v_0, \dots, v_m \in \mathcal{V}_1$ and $z \in \Gamma$ that that has uniformly $O(1)$ moments over these sets. 

We now apply Proposition \ref{ProdExp1} (with $L = n$ for the variables $\ell_i^n$ and $\ell_i^{n-1}$ as in Remark \ref{S41R2}) to (\ref{S42E1}).  The conclusion is that there exists $B$ (depending on $\theta$, $\delta$, $\eta$) such that if $K \geq B$ and $n/K, (n-1)/K \in [\delta, \delta^{-1}]$ we have
\begin{equation}\label{S42E5}
\begin{split}
&  \frac{1}{2\pi \i} \oint_{\Gamma} dz \frac{\Phi_{n,K}^-(zn) e^{-\theta G(z, n/K)}}{2  \theta (z - v_0) n^{m/2 + 1}} \cdot M \left(1 + \frac{\theta^2}{2n} \cdot \partial_z G(z,n/K) - \frac{\theta \hat{G}_n^n(z)}{n} ; \llbracket 1, m \rrbracket  \right) \\
&+ \frac{\Phi_{n,K}^+(zn) e^{\theta G(zn/(n-1), (n-1)/K)}}{2 \theta (z - v_0) n^{m/2 + 1}}  \\
&\times  M \left( 1 + \frac{\theta (\theta - 2)}{2(n-1)} \cdot \partial_z G\left(\frac{zn}{n-1}, \frac{n-1}{K}\right) + \frac{\theta \hat{G}_n^{n-1}(z)}{n}; \llbracket 1, m \rrbracket  \right) \\
& + \sum_{A \subseteq \llbracket 1, m \rrbracket} \frac{\Phi^+_{n,K}(zn) }{2 \theta (z- v_0) n^{|A|/2 + 1}}\cdot \prod_{a \in A^c} \left( \frac{n^{-1}}{(v_a - z)(v_a - z + n^{-1}) } \right) \cdot M\left(\Pi^{\theta}_1(zn);  A \right) \\
& + \frac{M(\zeta_n(z); \llbracket 1, m \rrbracket)}{n^{m/2 + 2}} = 0.
\end{split}
\end{equation}
When applying Proposition \ref{ProdExp1} we absorbed the terms $A_{n-1,n}^{\mathsf{s},3}(z;x,y)$, $A_{n,n}^{\mathsf{s},3}(z;x,y)$, $B_{n-1,n}^{\mathsf{s},2}(z;x,y)$, $B_{n-1,n}^{\mathsf{s},3}(z;x,y) $, $B_{n,n}^{\mathsf{s},2}(z;x,y)$ and $B_{n,n}^{\mathsf{s},3}(z;x,y) $ into $\zeta_n(z)$ terms above. We also used (\ref{S33E6}).

We now observe that 
\begin{equation}\label{S42PiExp1}
\begin{split}
&\Pi^{\theta}_1(zn) = \hspace{2mm} \frac{\theta {\bf 1}\{ \theta \neq 1\} }{1- \theta} +  \theta G(z,n/K) -  \theta G\left(\frac{zn}{n-1}, \frac{n-1}{K}\right)  - \frac{\theta (1+ \theta)}{2n} \cdot \partial_z G(z, n/K)  \\
& + \frac{\theta (1-\theta)}{2(n-1)} \cdot \partial_z G(zn/(n-1), (n-1)/K)+ \frac{\theta [\hat{G}_n^n(z) - \hat{G}_n^{n-1}(z)]}{n}  + \frac{\zeta_n(z)}{n^2}.
\end{split}
\end{equation}
When $\theta \neq 1$ equation (\ref{S42PiExp1}) follows by applying Proposition \ref{ProdExp1} to each product in the definition of $\Pi_1^{\theta}$ in (\ref{S42E2}). When $\theta = 1$ we have from (\ref{S32E1}) and (\ref{S42E2}) that 
\begin{equation*}
\begin{split}
&\Pi_1^1(zn) = \sum_{i = 1}^n \frac{1}{zn - \ell_i^n - 1} - \sum_{i = 1}^{n-1} \frac{1}{zn - \ell_i^{n-1}} = \frac{G_n^n(z) - G_n^{n-1}(z)}{n} - \frac{\partial_z G_n^n(z)}{n^2} + \frac{\xi_n(z)}{n^2}\\
& = \frac{\hat{G}_n^n(z) - \hat{G}_n^{n-1}(z)}{n} + G(z,n/K) -  G\left(\frac{zn}{n-1}, \frac{n-1}{K}\right) - \frac{\partial_z \hat{G}_n^n(z)}{n^2} - \frac{\partial_z G(z,n/K)}{n} + \frac{\xi_n(z)}{n^2},
\end{split}
\end{equation*}
which implies (\ref{S42PiExp1}) when $\theta = 1$ in view of (\ref{MBRem}).\\

We now proceed to substitute (\ref{S42PiExp1}) into (\ref{S42E5}) and simplify the resulting expression. We make use of the formulas 
\begin{equation}\label{S42E7}
G\left(\frac{zn}{n-1}, \frac{n-1}{K}\right) = G\left(z, n/K\right) + \frac{z \partial_zG(z,n/K) - (n/K) \partial_s G(z,n/K)}{n} + \frac{\xi_n(z)}{n^2} ,
\end{equation}
\begin{equation}\label{S42E8}
\partial_z G\left(\frac{zn}{n-1}, \frac{n-1}{K}\right) = \partial_z G\left(z, n/K\right) + \frac{\xi_n(z)}{n},
\end{equation}
\begin{equation}\label{S42E9}
e^{\theta G(zn/(n-1), (n-1)/K)} = e^{\theta G(z, n/K)} \cdot \left(1 +  \frac{\theta z \partial_zG(z,n/K) -  \theta (n/K) \partial_s G(z,n/K)}{n} \right) + \frac{\xi_n(z)}{n^2} .
\end{equation}

Suppose first that $m = 0$. Substituting (\ref{S42PiExp1}) (\ref{S42E7}), (\ref{S42E8}) and (\ref{S42E9}) into (\ref{S42E5}) we get
\begin{equation}\label{S42E10}
\begin{split}
&  \frac{1}{2\pi \i} \oint_{\Gamma}dz  \left[ \frac{\Phi_{n,K}^-(zn) e^{-\theta G(z, n/K)} + \Phi_{n,K}^+(zn) e^{\theta G(z, n/K)}}{2\theta (z-v_0)n} +  \frac{ {\bf 1}\{ \theta \neq 1\} \Phi^+_{n,K}(zn) }{2 (1- \theta) (z-v_0) n}  \right] \\
&+  \frac{ \partial_z G(z,n/K) \left( \theta \Phi_{n,K}^-(zn) e^{-\theta G(z, n/K)} + (\theta-2)  \Phi_{n,K}^+(zn) e^{\theta G(z, n/K)} - 2\theta \Phi_{n,K}^+(zn) \right) }{2 (z-v_0) n^2} \\ 
&+\frac{  \Phi_{n,K}^+(zn) [e^{\theta G(z, n/K)} - 1] \hat{G}_n^{n-1}(z) - [\Phi_{n,K}^-(zn) e^{-\theta G(z, n/K)} -  \Phi_{n,K}^+(zn)] \hat{G}_n^n(z) }{2(z-v_0) n^2}  \\
& + \frac{ [z \partial_zG(z,n/K) - (n/K) \partial_s G(z,n/K)] \cdot \Phi_{n,K}^+(zn)  \left(e^{\theta G(z, n/K)} -1 \right)}{ (z-v_0) n^2} + \frac{\mathbb{E}[\zeta_n(z)]}{n^2} = 0. \\
\end{split}
\end{equation}
We suppose now that $m > 0$, substitute (\ref{S42PiExp1}), (\ref{S42E7}), (\ref{S42E8}) and (\ref{S42E9}) into (\ref{S42E5}) and get
\begin{equation}\label{S42E11}
\begin{split}
&  \frac{1}{2\pi \i} \oint_{\Gamma}  dz M \left( \frac{\Phi_{n,K}^+(zn) e^{\theta G(z, n/K)}\hat{G}_n^{n-1}(z) - \Phi_{n,K}^-(zn) e^{-\theta G(z, n/K)} \hat{G}_n^{n}(z) }{2 (z - v_0) n^{m/2 + 2}} ; \llbracket 1, m \rrbracket  \right) \\
& + \sum_{A \subseteq \llbracket 1, m \rrbracket} \frac{\Phi^+_{n,K}(zn) }{2 \theta (z- v_0) n^{|A|/2 + 1}}\cdot \prod_{a \in A^c} \left( \frac{n^{-1}}{(v_a - z)(v_a - z + n^{-1})} \right)   \\
&\times M \left( \frac{\theta {\bf 1}\{ \theta \neq 1\} }{1- \theta} + \frac{\theta [\hat{G}_n^n(z) - \hat{G}_n^{n-1}(z)]}{n} ; A \right) + \frac{\mathbb{E}[\zeta_n(z)]}{n^{m  +1} }+  \sum_{A \subseteq \llbracket 1, m \rrbracket} \frac{M(\zeta_n(z); A)}{n^{m - |A|/2 + 2}} = 0.
\end{split}
\end{equation}
We mention that in deriving (\ref{S42E11}) we used the linearity of cumulants and the fact that the joint cumulant of any nonempty collection of bounded random variables and a constant is zero, see (\ref{S3Linearity}). Equations (\ref{S42E10}) and (\ref{S42E11}) are the main output of linearizing (\ref{S42E1}).\\

We next proceed to linearize (\ref{S42E3}). From Proposition \ref{ProdExp1} we have
\begin{equation}\label{S42E12}
\begin{split}
&0 =  \frac{1}{2\pi \i} \oint_{\Gamma}  dz  \frac{\Phi_{n,K}^+(zn) e^{\theta G(z, n/K)} }{2  \theta (z - v_0) n^{m/2 + 1}}  M \left( 1 + \frac{\theta(\theta-2)}{2n} \cdot \partial_z G(z, n/K) +  \frac{\theta \hat{G}_n^n(z) }{n}; \llbracket 1, m \rrbracket\right)\\
& + \frac{\Phi_{n,K}^-(zn) e^{-\theta G(zn/(n-1), (n-1)/K)} }{2  \theta (z - v_0) n^{m/2 + 1}}  \\
&\times M \left(1 - \frac{\theta^2}{2(n-1)} \cdot  \partial_z G\left(\frac{zn}{n-1}, \frac{n-1}{K} \right)  - \frac{\theta \hat{G}^{n-1}_n(z)}{n} ; \llbracket 1, m \rrbracket \right) \\
& + \sum_{A \subseteq \llbracket 1, m \rrbracket} \frac{\Phi^-_{n,K}(zn)}{2 \theta (z- v_0) n^{|A|/2 + 1}} \cdot \prod_{a \in A^c} \left( \frac{n^{-1}}{(v_a - z_{\theta})(v_a - z_{\theta}^-) } \right) \cdot M \left( \Pi_2^{\theta}(zn); A \right) +\\
& +    \sum_{A \subseteq \llbracket 1, m \rrbracket} \frac{M(\zeta_n(z); A)}{n^{2m - 3|A|/2 + 2}} = 0.
\end{split}
\end{equation}
We mention that in deriving (\ref{S42E12}) we absorbed all terms in the first line of (\ref{S42E3}) corresponding to $A \neq \llbracket 1, m \rrbracket$ in $\zeta_n(z)$ terms.

We next observe that 
\begin{equation}\label{S42PiExp2}
\begin{split}
&\Pi^{\theta}_2(zn) = \hspace{2mm} \frac{\theta {\bf 1}\{ \theta \neq 1\} }{1- \theta} -  \theta G(z,n/K) +  \theta G\left(\frac{zn}{n-1}, \frac{n-1}{K}\right)  - \frac{\theta (1- \theta)}{2n} \cdot \partial_z G(z, n/K)  \\
& + \frac{\theta (1-\theta)}{2(n-1)} \cdot \partial_z G(zn/(n-1), (n-1)/K) - \frac{\theta [\hat{G}_n^n(z) - \hat{G}_n^{n-1}(z)]}{n}  + \frac{\zeta_n(z)}{n^2}.
\end{split}
\end{equation}
When $\theta \neq 1$ equation (\ref{S42PiExp2}) follows by applying Proposition \ref{ProdExp1} to each product in the definition of $\Pi_2^{\theta}$ in (\ref{S42E4}). When $\theta = 1$ we have from (\ref{S32E1}) and (\ref{S42E4}) that 
\begin{equation*}
\begin{split}
&\Pi_2^1(zn) = \sum_{i = 1}^{n-1} \frac{1}{zn - \ell_i^{n-1}} - \sum_{i = 1}^n \frac{1}{zn - \ell_i^n} = \frac{G_n^{n-1}(z) - G_n^n(z)}{n} +  G\left(\frac{zn}{n-1}, \frac{n-1}{K}\right) - G(z, n/K),
\end{split}
\end{equation*}
which implies (\ref{S42PiExp2}) when $\theta = 1$.\\

Suppose that $m = 0$. Substituting (\ref{S42E7}), (\ref{S42E8}), (\ref{S42E9}) and (\ref{S42PiExp2}) into (\ref{S42E12}) we get
\begin{equation}\label{S42E13}
\begin{split}
&  \frac{1}{2\pi \i} \oint_{\Gamma}dz  \left[ \frac{\Phi_{n,K}^-(zn) e^{-\theta G(z, n/K)} + \Phi_{n,K}^+(zn) e^{\theta G(z, n/K)}}{2\theta (z-v_0)n} +  \frac{ {\bf 1}\{ \theta \neq 1\} \Phi^-_{n,K}(zn) }{2 (1- \theta) (z-v_0) n}  \right] \\
&+  \frac{ \partial_z G(z,n/K) \left( (\theta -2) \Phi_{n,K}^+(zn) e^{\theta G(z, n/K)} -  \theta  \Phi_{n,K}^-(zn) e^{-\theta G(z, n/K)} \right) }{2 (z-v_0) n^2} \\ 
&+\frac{[\Phi_{n,K}^+(zn) e^{\theta G(z, n/K)} - \Phi_{n,K}^-(zn)]  \hat{G}_n^n(z) -\Phi_{n,K}^-(zn) [e^{-\theta G(z, n/K)} - 1]  \hat{G}_n^{n-1}(z) }{2(z-v_0) n^2} \\
& + \frac{ [z \partial_zG(z,n/K) - (n/K) \partial_s G(z,n/K)] \cdot \Phi_{n,K}^-(zn)  \left(- e^{-\theta G(z, n/K)} +1 \right)}{ (z-v_0) n^2} + \frac{\mathbb{E}[\zeta_n(z)]}{n^2} = 0. \\
\end{split}
\end{equation}
We suppose now that $m > 0$, substitute (\ref{S42E7}), (\ref{S42E8}), (\ref{S42E9}) and (\ref{S42PiExp2}) into (\ref{S42E12}) and get
\begin{equation}\label{S42E14}
\begin{split}
&  \frac{1}{2\pi \i} \oint_{\Gamma}  dz M \left( \frac{\Phi_{n,K}^+(zn) e^{\theta G(z, n/K)}\hat{G}_n^{n}(z) - \Phi_{n,K}^-(zn) e^{-\theta G(z, n/K)} \hat{G}_n^{n-1}(z) }{2 (z - v_0) n^{m/2 + 2}} ; \llbracket 1, m \rrbracket  \right) \\
& + \sum_{A \subseteq \llbracket 1, m \rrbracket} \frac{\Phi^-_{n,K}(zn) }{2 \theta (z- v_0) n^{|A|/2 + 1}}\cdot \prod_{a \in A^c} \left( \frac{n^{-1}}{(v_a - z_{\theta})(v_a - z_{\theta}^-)} \right)   \\
&\times M \left( \frac{\theta {\bf 1}\{ \theta \neq 1\} }{1- \theta} - \frac{\theta [\hat{G}_n^n(z) - \hat{G}_n^{n-1}(z)]}{n} ; A \right)  + \frac{\mathbb{E}[\zeta_n(z)]}{n^{m  +1} } +   \sum_{A \subseteq \llbracket 1, m \rrbracket} \frac{M(\zeta_n(z); A)}{n^{m - |A|/2 + 2}} = 0.
\end{split}
\end{equation}
As before, we used linearity of cumulants and that the joint cumulant of any nonempty collection of bounded random variables and a constant is zero.\\

The last thing we do is we take the difference of (\ref{S42E10}) and (\ref{S42E13}) and the difference of (\ref{S42E11}) and (\ref{S42E14}) and do some simplifications of the result.

Firstly, when we consider (\ref{S42E10}) and (\ref{S42E13}) we note that the terms involving ${\bf 1}\{\theta \neq 1\}$ integrate to $0$ by Cauchy's theorem (here we used that $\Phi_{n,K}^{\pm}$ are polynomials and $v_0$ is outside of $\Gamma$). Taking this into account and subtracting (\ref{S42E13}) from (\ref{S42E10})  we get
\begin{equation}\label{S42E15}
\begin{split}
&  \frac{1}{2\pi \i} \oint_{\Gamma}dz  \frac{ \partial_z G(z,n/K) \left(  \theta \Phi_{n,K}^-(zn) e^{-\theta G(z, n/K)}  - \theta \Phi_{n,K}^+(zn) \right) }{ (z-v_0) n^2} +\frac{ S_{n,K}(z) [ \hat{G}_n^{n-1}(z) - \hat{G}_n^n(z)] }{2(z-v_0) n}  \\ 
&+\frac{ S_{n,K}(z)  [z \partial_zG(z,n/K) - (n/K) \partial_s G(z,n/K)] }{(z-v_0)n} + \frac{\mathbb{E}[\zeta_n(z)]}{n^2} = 0,
\end{split}
\end{equation}
where 
\begin{equation}\label{S42E16}
S_{n,K}(z) = n^{-1} \cdot \left[\Phi_{n,K}^+(zn) e^{\theta G(z, n/K)}  + \Phi_{n,K}^-(zn) e^{-\theta G(z, n/K)}  - \Phi_{n,K}^+(zn) - \Phi_{n,K}^-(zn) \right].
\end{equation}

Next we consider (\ref{S42E11}) and (\ref{S42E14}). In this case we note that the ${\bf 1}\{\theta \neq 1\}$ term either does not contribute if $A \neq \emptyset$ (as the joint cumulant of any nonempty collection of bounded random variables and a constant is zero) or $A = \emptyset$ in which case it does not contribute as it integrates to zero by Cauchy's theorem. Taking this into account and subtracting (\ref{S42E14}) from (\ref{S42E11})  we get
\begin{equation}\label{S42E17}
\begin{split}
&  \frac{1}{2\pi \i} \oint_{\Gamma}  dz M \left(\frac{ S_{n,K}(z) [ \hat{G}_n^{n-1}(z) - \hat{G}_n^n(z)] }{2(z-v_0) n^{m/2 + 1}} ; \llbracket 1, m \rrbracket \right)   + \frac{\mathbb{E}[\zeta_n(z)]}{n^{m  +1} } +  \sum_{A \subseteq \llbracket 1, m \rrbracket} \frac{M(\zeta_n(z); A)}{n^{m -|A|/2 + 2}} \\
& + \sum_{A \subsetneq \llbracket 1, m \rrbracket} \prod_{a \in A^c} \left( \frac{n^{-1}}{(v_a - z)^2} \right) \cdot \frac{[  \Phi_{n,K}^+(zn) +\Phi_{n,K}^-(zn)  ] M\left([ \hat{G}_n^{n}(z) - \hat{G}_n^{n-1}(z)] ; A\right)}{2 ( z - v_0) n^{|A|/2 + 2}}  = 0.
\end{split}
\end{equation}
where we mention that we used 
\begin{equation}\label{S42E18}
 \frac{1}{(v_a -z)(v_a - z + n^{-1})} = \frac{1}{(v_a - z)^2} + \frac{\xi_n(z)}{n} \mbox{ and }  \frac{1}{(v_a - z_{\theta})(v_a - z_{\theta}^-) } =    \frac{1}{(v_a - z)^2} + \frac{\xi_n(z)}{n}.
\end{equation}

We now use (\ref{S31E7}), (\ref{S31E17}) and (\ref{S33E6}) to get
\begin{equation}\label{S42E19}
S_{n,K}(z) = R(z,n/K) - \Phi^+(z,n/K) - \Phi^-(z,n/K) + n^{-1} \xi_n(z) = - 2\theta +n^{-1} \xi_n(z),
\end{equation}
and also use the latter equations and (\ref{SpecEqn}) to get by a direct computation
\begin{equation}\label{S42E20}
\begin{split}
&n^{-1} \partial_z G(z,n/K) \left( \theta  \Phi_{n,K}^-(zn) e^{-\theta G(z, n/K)}  - \theta \Phi_{n,K}^+(zn) \right)  \\
&+ S_{n,K}(z)  [z \partial_zG(z,n/K) - (n/K) \partial_s G(z,n/K)] = n^{-1} \xi_n(z).
\end{split}
\end{equation}
We substitute (\ref{S42E19}) and (\ref{S42E20}) into (\ref{S42E15}) to get
\begin{equation*}
\begin{split}
&  \frac{1}{2\pi \i} \oint_{\Gamma}dz \frac{ (-2\theta) [ \hat{G}_n^{n-1}(z) - \hat{G}_n^n(z)] }{2(z-v_0) n}  + \frac{\mathbb{E}[\zeta_n(z)]}{n^2} = 0.
\end{split}
\end{equation*}
We can now evaluate the first part of the integral as (minus) the residue at $z = v_0$ (note that there is no residue at $\infty$) and multiply the expression by $\theta^{-1} n^{3/2}$ to get
\begin{equation}\label{S42E21}
\begin{split}
& \mathbb{E} \left[ n^{1/2} [ \hat{G}_n^{n-1}(v_0) - \hat{G}_n^n(v_0)] \right] =   \frac{n^{-1/2}}{2\pi \i} \oint_{\Gamma}dz\mathbb{E}[\zeta_n(z)].
\end{split}
\end{equation}
Similarly, we substitute (\ref{S33E6}) and (\ref{S42E19}) into (\ref{S42E17}) to get
\begin{equation*}
\begin{split}
&  \frac{1}{2\pi \i} \oint_{\Gamma}  dz M \left(\frac{ (-2\theta) [ \hat{G}_n^{n-1}(z) - \hat{G}_n^n(z)] }{2(z-v_0) n^{m/2 + 1}} ; \llbracket 1, m \rrbracket \right)   + \frac{\mathbb{E}[\zeta_n(z)]}{n^{m  +1} } +  \sum_{A \subseteq \llbracket 1, m \rrbracket} \frac{M(\zeta_n(z); A)}{n^{m -|A|/2 + 2}} \\
& + \sum_{A \subsetneq \llbracket 1, m \rrbracket} \prod_{a \in A^c} \left( \frac{n^{-1}}{(v_a - z)^2} \right) \cdot \frac{(\theta + K/n) M\left([ \hat{G}_n^{n}(z) - \hat{G}_n^{n-1}(z)] ; A\right)}{2 ( z - v_0) n^{|A|/2 + 1}}  = 0.
\end{split}
\end{equation*}
As before, we can evaluate the first part of the integral as (minus) the residue at $z = v_0$ (again there is no pole at infinity) and multiply the expression by $\theta^{-1} n^{m/2 + 3/2}$ to get
\begin{equation}\label{S42E22}
\begin{split}
&   M \left( n^{1/2} [ \hat{G}_n^{n-1}(z) - \hat{G}_n^n(z)] ; \llbracket 1, m \rrbracket \right)   =  \frac{1}{2\pi \i} \oint_{\Gamma}  dz \frac{\mathbb{E}[\zeta_n(z)]}{n^{m/2  -1/2} } +  \sum_{A \subseteq \llbracket 1, m \rrbracket} \frac{M(\zeta_n(z); A)}{n^{m/2 -|A|/2 + 1/2}} \\
& + \sum_{A \subsetneq \llbracket 1, m \rrbracket} \prod_{a \in A^c} \left( \frac{1}{(v_a - z)^2} \right) \cdot \frac{(\theta + K/n) \cdot M\left(n^{1/2}[ \hat{G}_n^{n}(z) - \hat{G}_n^{n-1}(z)] ; A\right)}{2 \theta ( z - v_0) n^{m/2 -|A|/2}} .
\end{split}
\end{equation}
Equations (\ref{S42E21}) and (\ref{S42E22}) are the main outputs of this section.

%-------------------------------------------------------------------------------------------------------------------------------------------------------------------------------------------------
% Section 4.3
%
%-------------------------------------------------------------------------------------------------------------------------------------------------------------------------------------------------
\subsection{Moment bounds for differences}\label{Section4.3} The goal of this section is to establish the following result.

\begin{proposition}\label{S43P1} Fix $\theta >0$ and $\delta \in (0,1)$. Suppose that $(\ell^1, \ell^2, \dots)$ is distributed according to $\P$ as in Definition \ref{BKCC}, and that $\hat{G}_L^n$ are as in (\ref{S32E1}). For each $k \in \mathbb{N}$ and compact set $\mathcal{V} \subset \mathbb{C} \setminus [-\theta , \delta^{-1}]$
\begin{equation}\label{S43E1}
 \sup_{z \in \mathcal{V}} \max_{\delta K \leq n-1, n  \leq \delta^{-1} K} \mathbb{E}\left[\left|\hat{G}_{n}^n(z) - \hat{G}_{n}^{n-1}(z) \right|^k \right] = O(n^{-k/2}),
\end{equation}
where the constant in the big $O$ notation depends on $\theta, \delta, \mathcal{V}$ and $k$ alone.
\end{proposition}
\begin{remark}\label{S43MR}
Before we go to the proof let us explain how Proposition \ref{S43P1} will be applied later in the text. Fixing $\theta$ and $\delta$ as in the statement of the proposition, we consider $n, L, K$ such that $ n \geq 2$, $n/L, n/K, (n-1)/L, (n-1)/K \in [\delta, \delta^{-1}]$. For these parameters we have for $\hat{G}^n_L$ as in (\ref{S32E1}) that
\begin{equation}\label{S43MBRem}
\begin{split}
&L^{1/2} \cdot  [ \hat{G}_{L}^n(z) - \hat{G}_{L}^{n-1}(z)]  =  \zeta(z) \mbox{, and } L^{1/2} \cdot \partial_z  [ \hat{G}_{L}^n(z) - \hat{G}_{L}^{n-1}(z)] =  \zeta(z).
\end{split}
\end{equation}
In (\ref{S43MBRem}) we have written $\zeta(z)$ for a generic random analytic function on $\mathbb{C} \setminus [-\theta \delta^{-1}, \delta^{-2}]$ such that 
\begin{equation}\label{S43REQ1}
\mathbb{E}\left[|\zeta(z)|^k \right] = O(1),
\end{equation}
where the constant in the big $O$ notation depends on $\theta, \delta, k$ and a compact subset $\mathcal{V}_1 \subset \mathbb{C} \setminus [-\theta \delta^{-1}, \delta^{-2}]$ and (\ref{S43REQ1}) holds uniformly as $z$ varies in $\mathcal{V}_1$. One deduces (\ref{S43MBRem}) from Proposition \ref{S43P1} using the same argument as in Remark \ref{MomentBoundRem}.
\end{remark}

\begin{proof} The proof we present is similar to Steps 3-5 from the proof of Proposition \ref{S33P1}. Essentially, we reduce the statement to a certain self-improvement estimate claim. The base case of the claim is established using Proposition \ref{S33P1} and by iterating it finitely many times one ultimately obtains the bounds in (\ref{S43E1}). The way one obtains the improvement at each step in the claim is by utilizing equations (\ref{S42E21}) and (\ref{S42E22}) from the previous section. For clarity we split the proof into three steps.\\

{\bf \raggedleft Step 1.} In this step we reduce the proof of the proposition to the establishment of the following self-improvement estimate claim.\\

{\bf \raggedleft Claim:} Suppose that for some $H, M \in \mathbb{N}$, and all $n, K \in\mathbb{N}$ with $n/K, (n-1)/K \in [\delta, \delta^{-1}]$ we have 
\begin{equation}\label{S43E2}
\mathbb{E} \left[ n^{m/2} \prod_{a = 1}^m \left|\hat{G}_{n}^n(v_a) - \hat{G}_n^{n-1}(v_a) \right| \right]= O(1) + O\left(n^{m/2 + 1/2 - M/2}\right) \mbox{ for $m = 1, \dots, 4H + 4$,}
\end{equation}
then
\begin{equation}\label{S43E3}
\mathbb{E} \left[ n^{m/2} \prod_{a = 1}^m \left|\hat{G}_{n}^n(v_a) - \hat{G}_n^{n-1}(v_a) \right| \right]= O(1) + O\left(n^{m/2 - M/2}\right) \mbox{ for $m = 1, \dots, 4H $,}
\end{equation}
where the constants in the big $O$ notations are uniform as $v_a$ vary over compacts in $\mathbb{C} \setminus [-\theta , \delta^{-1}]$.
We prove the above claim in the following steps. Here we assume its validity and establish (\ref{S43E1}). \\

Fix $k \in \mathbb{N}$. From Proposition \ref{S33P1} we have for $n/K, (n-1)/K \in [\delta, \delta^{-1}]$
\begin{equation*}
\mathbb{E} \left[ \left|\hat{G}_{n}^n(v)\right|^m \right] =O (1) \mbox{ and } \mathbb{E} \left[ \left|\hat{G}_{n-1}^{n-1}(v)\right|^m \right] = O(1)  \mbox{ for $m = 1, \dots, 8k + 4$},
\end{equation*}
and the constants in the big $O$ notations depend on $\theta, \delta, k$ and are uniform as $v$ varies over compact subsets of $\mathbb{C} \setminus [-\theta, \delta^{-1}]$. Using that 
$$\hat{G}^{n-1}_{n}(z) = \sum_{i = 1}^n \frac{1}{z -  \ell_i^{n-1}/n } - n G\left( \frac{z n}{n-1}, \frac{n-1}{K} \right) = \frac{n}{n-1} \cdot \hat{G}_{n-1}^{n-1}\left(\frac{z n }{n-1} \right),$$
we conclude that uniformly over compact subsets of $\mathbb{C} \setminus [-\theta, \delta^{-1}]$ we have
\begin{equation*}
\mathbb{E} \left[ n^{m/2}\left|\hat{G}_{n}^n(v) - \hat{G}_{n}^{n-1}(v) \right|^m \right] =O \left(n^{m/2}\right)  \mbox{ for $m = 1, \dots, 8k + 4$},
\end{equation*}
Using H{\"o}lder's inequality, the above implies that (\ref{S43E2}) holds for the for the pair $H = 2k$ and $M = 1$. The conclusion is that (\ref{S43E2}) holds for the pair $H  = 2k- 1$ and $M = 2$. Iterating the argument an additional $k-1$ times we conclude that (\ref{S43E2}) holds with $H = k $ and $M = k+1$, which implies (\ref{S43E1}).\\

{\bf \raggedleft Step 2.}  In this step we will prove (\ref{S43E3}) except for a single case, which will be handled separately in the next step. For brevity we denote 
$$\Delta \hat{G}_n(z):= n^{1/2} \left[ \hat{G}_{n}^n(z) - \hat{G}_{n}^{n-1}(z) \right].$$

The first thing we show is that 
\begin{equation}\label{S43E4}
\begin{split}
&M\left( \Delta \hat{G}_n(v_0),  \dots, \Delta \hat{G}_n(v_m)  \right) = O(1) + O\left( n^{m/2 + 1/2 - M/2}\right) \mbox{ for $m = 1, \dots, 4H+2$},
\end{split}
\end{equation}
where constant in the big $O$ notation are uniform over $v_0, \dots, v_m$ in compact subsets of $\mathbb{C} \setminus [-\theta, \delta^{-1}]$.

We start by fixing $\mathcal{V}_1$ to be a compact subset of $\mathbb{C} \setminus [-\theta , \delta^{-1}]$, which is invariant under conjugation. We also fix $\eta, \Gamma$ as in the beginning of Section \ref{Section4.2} and adopt the same notation from that section. From (\ref{S42E22}) we have for $m \geq 1$
\begin{equation}\label{S43E5}
\begin{split}
M\left( \Delta \hat{G}_n(v_0),  \dots, \Delta \hat{G}_n(v_m)  \right)    = \hspace{2mm}& O(1) + \sum_{A \subseteq \llbracket 1, m \rrbracket } \frac{M(\zeta_n(z); A)}{n^{m/2 - |A|/2 + 1/2}}  \\
&+  \sum_{A \subsetneq \llbracket 1, m \rrbracket } \frac{M \left(\xi_n(z) \cdot \Delta \hat{G}_n(z) ; A\right)}{n^{m/2 -|A|/2}} .
\end{split}  
\end{equation}
In addition, from (\ref{S42E21}) we have
\begin{equation}\label{S43E6}
\begin{split}
&M\left( \Delta\hat{G}_{n}(v_0) \right)  = \mathbb{E} \left[\Delta\hat{G}_{n}(v_0) \right]= O(n^{-1/2}).
\end{split}  
\end{equation}
We next use the fact that cumulants can be expressed as linear combinations of products of moments, see (\ref{Mal2}). This means that $M(\xi_1, \dots, \xi_r)$ can be controlled by the quantities $1$ and $\mathbb{E} \left[ |\xi_i|^r \right]$ for $ i = 1, \dots, r$. This and (\ref{S43E2}) together imply 
$$ \frac{M(\zeta_n(z); A)}{n^{m/2 - |A|/2 + 1/2}} = O(1) + O \left( n^{|A| + 1/2  - M/2 - m/2 } \right) \mbox{ for } A \subseteq \llbracket 1, m \rrbracket,$$
$$\frac{M \left(\xi_n(z) \cdot \Delta \hat{G}_n(z) ; A\right)}{n^{m/2 -|A|/2}} = O(1) + O \left( n^{|A| + 1 - M/2 - m/2 } \right)  \mbox{ for } A \subsetneq \llbracket 1, m \rrbracket .$$
Combining the last two estimates with (\ref{S43E5}) gives (\ref{S43E4}).\\

Notice that by H{\"o}lder's inequality we have
\begin{equation*}
\sup_{v_1, \dots, v_m \in \mathcal{V}_1} \mathbb{E} \left[ \prod_{a = 1}^m \left|\Delta \hat{G}_{n}(v_a) \right| \right]\leq \sup_{v \in \mathcal{V}_1} \mathbb{E} \left[ \left|\Delta \hat{G}_{n}(v) \right|^m \right],
\end{equation*}
and so to finish the proof it suffices to show that for $m = 1, \dots, 4H$ we have
\begin{equation}\label{S43E7}
 \mathbb{E} \left[ \left|\Delta \hat{G}_{n}(v) \right|^m \right] = O(1)  + O(n^{m/2  - M/2}).
\end{equation}
Using the fact that one can express joint moments as linear combinations of products of joint cumulants, see (\ref{Mal1}), we deduce from (\ref{S43E4}) and (\ref{S43E6}) that
\begin{equation}\label{S43E8}
\sup_{ v_0, v_1, \dots, v_{m-1} \in \mathcal{V}_1} \mathbb{E} \left[ \prod_{a = 0}^{m-1}\Delta \hat{G}_{n}(v_a) \right] = O(1) + O( n^{(m-1)/2 + 1/2 - M/2}) \mbox{ for $m = 1, \dots, 4H+2$}.
\end{equation}

If $m = 2m_1$, we set $v_0 = v_1 = \cdots = v_{m_1 - 1} = v$ and $v_{m_1} = \cdots = v_{2m_1 - 1} = \overline{v}$ in (\ref{S43E8}) and get
\begin{equation}\label{S43E9}
\sup_{ v \in \mathcal{V}_1} \mathbb{E} \left[ |\Delta \hat{G}_{n}(v)|^m \right] = O(1) + O( n^{m/2  - M/2}) \mbox{ for $m = 1, \dots, 4H+2$, $m$ even}.
\end{equation}
In deriving the above we used that $\Delta \hat{G}_{n}(\overline{v}) = \overline{\Delta \hat{G}_{n}(v)}$ and so $\Delta \hat{G}_{n}(v) \cdot \Delta \hat{G}_{n}(\overline{v}) = |\Delta \hat{G}_{n}(v)|^2$. 

We next let $m = 2m_1 + 1$ be odd and notice that by the Cauchy-Schwarz inequality and (\ref{S43E9}) 
\begin{equation}\label{S43E10}
\begin{split}
&\sup_{ v \in \mathcal{V}_1} \mathbb{E} \left[ |\Delta \hat{G}_{n}(v)|^{2m_1 + 1} \right] \leq \sup_{ v \in \mathcal{V}_1} \mathbb{E} \left[ |\Delta \hat{G}_{n}(v)|^{2m_1 + 2} \right]^{1/2} \cdot \mathbb{E} \left[ |\Delta \hat{G}_{n}(v)|^{2m_1} \right]^{1/2} =  \\  
&O(1) + O( n^{m_1 + 1/2 - M/2}) + O(n^{m_1/2 + 1/2 - M/4}) \mbox{ for $m = 1 , \dots, 4H + 1$, $m$ odd}.
\end{split}
\end{equation}
We note that the bottom line of (\ref{S43E10}) is $O(1) + O(n^{m_1 + 1/2 - M/2})$ except when $M = 2m_1 + 1$, since
$$m_1/2 + 1/2 - M/4 \leq \begin{cases} m_1 + 1/2 - M/2 &\mbox{ when $M \leq 2m_1$,} \\ 0 &\mbox{ when $M \geq 2m_1 + 2$.} \end{cases}$$
Consequently, (\ref{S43E9}) and (\ref{S43E10}) together imply (\ref{S43E3}), except when $M = m = 2m_1 +1$. We will handle this case in the next and final step.\\

{\bf \raggedleft Step 3.}  In this step we will show that (\ref{S43E3}) holds even when $M = 2m_1 + 1$ and $4H > m = 2m_1 + 1$. In the previous step we showed in (\ref{S43E9}) that
$\sup_{v \in \mathcal{V}_1} \mathbb{E} \left[ |\Delta \hat{G}_{n}(v)|^{2m_1 + 2} \right] = O(n^{1/2})$, and below we will improve this estimate to
\begin{equation}\label{S43E11}
\sup_{v \in \mathcal{V}_1} \mathbb{E} \left[ \left|\Delta \hat{G}_{n}(v) \right|^{2m_1 + 2} \right] = O(1).
\end{equation}
The trivial inequality $|x|^{2m_1 + 2} + 1 \geq |x|^{2m_1 + 1}$ together with (\ref{S43E11}) imply
$$\sup_{v \in \mathcal{V}_1} \mathbb{E} \left[ \left|\Delta \hat{G}_{n}(v) \right|^{2m_1 + 1} \right]  = O(1).$$
Consequently, we have reduced the proof of the claim to establishing (\ref{S43E11}). \\

Let us list the relevant estimates we will need
\begin{equation}\label{S43E12}
\begin{split}
& \mathbb{E} \left[ \prod_{a = 1}^{2m_1 + 2} \left|\Delta \hat{G}_{n}(v_a) \right| \right] = O(n^{1/2})\mbox{, } \hspace{2mm} \mathbb{E} \left[ \prod_{a = 1}^j \left|\Delta \hat{G}_{n}(v_a) \right| \right] = O(1) \mbox{ for $0 \leq j \leq 2m_1$,} \\\
&\mathbb{E} \left[ \prod_{a = 1}^{2m_1 + 1} \left|\Delta \hat{G}_{n}(v_a) \right| \right] = O(n^{1/4}).
\end{split}
\end{equation}
All of these identities follow from (\ref{S43E9}) and (\ref{S43E10}), which we showed to hold in the previous step. Below we feed the improved estimates of (\ref{S43E12}) into Step 2, which will ultimately yield (\ref{S43E11}).\\

Using (\ref{S43E12}) in place of (\ref{S43E2}) in Step 2 we get 
$$ \frac{M(\zeta_n(z); A)}{n^{m/2 - |A|/2 + 1/2}} = O(1)  \mbox{ for } A \subseteq \llbracket 1, m \rrbracket, \hspace{2mm} \frac{M \left(\xi_n(z) \cdot \Delta \hat{G}_n(z) ; A\right)}{n^{m/2 -|A|/2}} = O(1) \mbox{ for } A \subsetneq \llbracket 1, m \rrbracket .$$
which together with (\ref{S43E5}) gives the following improvement over (\ref{S43E4})
\begin{equation}\label{S43E13}
M\left( \Delta \hat{G}_{n}(v_0), \Delta \hat{G}_{n}(v_1), \dots, \Delta \hat{G}_{n}(v_m) \right) = O(1). 
\end{equation}
We now use (\ref{S43E13}) in place of (\ref{S43E4}) to get the following improvement over (\ref{S43E8})
\begin{equation}\label{S43E14}
\sup_{ v_0, v_1, \dots, v_{m} \in \mathcal{V}_1} \mathbb{E} \left[ \prod_{a = 0}^{m} \Delta \hat{G}_{n}(v_a) \right] = O(1).
\end{equation}
Setting $v_0 = v_1 = \cdots = v_{m_1} = v$ and $v_{m_1+1} = \cdots = v_{2m_1 + 1} = \overline{v}$ in (\ref{S43E14}) gives (\ref{S43E11}).
\end{proof}

%-------------------------------------------------------------------------------------------------------------------------------------------------------------------------------------------------
% Section 5
%
%-------------------------------------------------------------------------------------------------------------------------------------------------------------------------------------------------
\section{Multi-level analysis}\label{Section5} The goal of this section is to prove Proposition \ref{S52P1}, which establishes a certain joint cumulant formula involving the functions $\hat{G}^n_K(z)$ from (\ref{S32E1}). The proof of  Proposition \ref{S52P1} is given in Section \ref{Section5.2} and is based on the multi-level loop equations from \cite{DK21}. In Section \ref{Section5.1} we establish a certain double product expansion formula that is used to linearize the multi-level loop equations and it heavily relies on the product expansion formula from Proposition \ref{ProdExp1} as well as the moment bounds from Proposition \ref{S43P1}.

%-------------------------------------------------------------------------------------------------------------------------------------------------------------------------------------------------
% Section 5.1
%
%-------------------------------------------------------------------------------------------------------------------------------------------------------------------------------------------------
\subsection{Double product expansion}\label{Section5.1} The goal of this section is to establish the following result.
\begin{proposition}\label{ProdExp2} Fix $\theta >0$ and $\delta \in (0,1)$. Suppose that $(\ell^1, \ell^2, \dots )$ is distributed according to $\P$ as in Definition \ref{BKCC} and let $n, L$ be integers such that $n/K, n/L, (n-1)/K, (n-1)/L \in [\delta, \delta^{-1}]$. Let us fix a compact set $\mathcal{V}_2 \subset \mathbb{C}$, an $\eta > 0$ and let $\mathcal{D}_{\eta}'$ be as in (\ref{Deta2}). Note that since $\ell_i^n/n \in [-\theta, \delta^{-1}]$ almost surely, we can find $B > 0$ (depending on $\eta, \delta, \mathcal{V}_2$) such that if $K \geq B$ and $n/K, n/L, (n-1)/K, (n-1)/L \in [\delta, \delta^{-1}]$ we have $\ell_i^n/L \pm x/L \not \in \mathcal{D}_{\eta}'$ for each $x \in \mathcal{V}_2$ and $i \in \llbracket 1, n \rrbracket$ and $\ell_i^{n-1}/L \pm x/L \not \in \mathcal{D}_{\eta}'$  for each for each $x \in \mathcal{V}_2$ and $i \in \llbracket 1, n -1 \rrbracket$.

With the above data we define for $K \geq B$, $x,y, u+x, u+y \in \mathcal{V}_2$ the deterministic functions
\begin{equation}\label{S51E0}
D(z) = L(x-y)  \left(G\left( \frac{zL}{n}, \frac{n}{L} \right) - G\left( \frac{zL}{n-1}, \frac{n-1}{L} \right) \right),
\end{equation}
\begin{equation}\label{S51E1}
\begin{split}\\
&A_{n,L}^{\mathsf{p},1}(z; x,y,u) = \hspace{2mm} 1, \hspace{2mm}A_{n,L}^{\mathsf{p},2}(z; x,y,u)  = D(z) + A_{n,L}^{\mathsf{s},2}(z; x,y) + A_{n-1,L}^{\mathsf{s},2}(z; u+ y, u+x), \\
&  A_{n,L}^{\mathsf{p},3}(z; x,y,u)= \frac{ D(z)^2 }{2} + A_{n,L}^{\mathsf{s},3} (z;x,y) + A_{n-1,L}^{\mathsf{s},3}(z;u + y, u+x)  \\
& + D(z) \cdot \left[ A_{n,L}^{\mathsf{s},2} (z;x,y)  + A_{n-1,L}^{\mathsf{s},2}(z;u + y, u+x) \right] + A_{n,L}^{\mathsf{s},2} (z;x,y)  \cdot A_{n,L}^{\mathsf{s},2} (z;x,y)  ,
\end{split}
\end{equation}
where $A_{n,L}^{\mathsf{s},i}$ are as in (\ref{S41E1}). We also define the random functions
\begin{equation}\label{S51E2}
\begin{split}
&B_{n,L}^{\mathsf{p},1}(z; x,y,u) =  (x-y) \left( \hat{G}^n_L(z) - \hat{G}^{n-1}_L(z) \right), \hspace{2mm} B_{n,L}^{\mathsf{p},2}(z; x,y,u)   = u (y-x) \cdot \partial_z \hat{G}^n_L(z), \\
&  B_{n,L}^{\mathsf{p},3}(z; x,y,u) = L^{5/2}   \\
& \times \left( \prod_{p = 1}^{n} \frac{zL- \ell^{n}_p  +x }{zL - \ell_p^{n} + y}  \prod_{p = 1}^{n-1} \frac{zL-  \ell_p^{n-1} + u + y}{zL -  \ell_p^{n-1} + u + x}  - \sum_{i = 1}^3 \frac{A_{n,L}^{\mathsf{p},i}(z; x,y,u) }{L^{i-1}} - \sum_{i = 1}^2 \frac{B_{n,L}^{\mathsf{p},i}(z; x,y,u) }{L^i} \right).
\end{split}
\end{equation}
We claim that the functions in (\ref{S51E1}) and (\ref{S51E2}) are $\P$-a.s. holomorphic in $\mathcal{D}_{\eta}'$ and for each $k \in \mathbb{N}$
\begin{equation}\label{S51E3}
\begin{split}
&\sum_{i = 1}^3\mathbb{E}\left[ \left|A^{\mathsf{p},i}_{n,L}(z;x,y,u)  \right|^k + \left|B^{\mathsf{p},i}_{n,L}(z;x,y,u)  \right|^k \right]  = O(1),
\end{split}
\end{equation}
where the constant in the big $O$ notation depends on $\theta, \delta, \eta, k$ and is uniform as $z$ varies over compact subsets of $\mathcal{D}_{\eta}'$ and $x,y, u+x,u+y \in \mathcal{V}_2$. 
\end{proposition}
\begin{proof} The analyticity of the functions in (\ref{S51E1}) and (\ref{S51E2}) on $\mathcal{D}_{\eta}'$ follows from the same arguments as in the first paragraph of the proof of Proposition \ref{ProdExp1}. We are thus left with establishing (\ref{S51E3}). We note that the smoothness of $G(z,s)$ implies that $A_{n,L}^{\mathsf{p},i}(z; x,y,u) $ are all uniformly bounded and then so are their powers. In addition, the moment bounds for $B_{n,L}^{\mathsf{p},1}(z; x,y,u)$ and $B_{n,L}^{\mathsf{p},2}(z; x,y,u)$ follow from the moment bounds for $\hat{G}^n_L(z)$ in (\ref{MBRem}). We are thus left with showing that
\begin{equation}\label{S51E3.5}
\begin{split}
&\mathbb{E}\left[  \left|B^{\mathsf{p},3}_{n,L}(z;x,y,u)  \right|^k \right]  = O(1),
\end{split}
\end{equation}
with constant dependence as explained below (\ref{S51E3}). In the remainder we focus on (\ref{S51E3.5}). The idea is to apply Proposition \ref{ProdExp1} to each of the products in the last line of (\ref{S51E2}) and utilize the moment bounds from Proposition \ref{S43P1}.\\

We fix a compact set $\mathcal{V}_1 \subset \mathcal{D}_{\eta}'$ and throughout the proof write $\xi_{n,L}(z;x,y,u)$ and $\zeta_{n,L}(z;x,y,u)$ to mean generic random functions of $n, L$ such that $n/K, n/L  \in [\delta, \delta^{-1}]$,  $x,y, x+u, y + u \in \mathcal{V}_2$ and $z \in \mathcal{D}_{\eta}'$ that are analytic on $\mathcal{D}_{\eta}'$ for fixed $x,y, x+u, y + u \in \mathcal{V}_2$  and such that 
$$   \max_{\delta^{-1}  \geq n/L, n/K \geq \delta }  \sup_{z \in \mathcal{V}_1} \sup_{ x, y, x+u, y+ u \in \mathcal{V}_2}  \left| \xi_{n,L}(z;x,y, u) \right| = O(1) \mbox{ $\P$-almost surely and  }$$
$$\max_{\delta^{-1}  \geq n/L, n/K \geq \delta }   \sup_{z \in \mathcal{V}_1}\sup_{ x, y, x+u, y+ u \in \mathcal{V}_2}  \mathbb{E} \left[ \left|\zeta_{n,L}(z;x,y) \right|^k \right] = O(1),$$
for $k \in \mathbb{N}$, where the constants in the big $O$ notations depend on $\mathcal{V}_1$, $\mathcal{V}_2$, $\theta, \delta, \eta$ and also on $k$ for the second identity, but not on $K$ provided $K \geq B$ as in the statement of the proposition. To ease the notation, we drop $x,y,u, n,L$ and simply write $\xi(z), \zeta(z)$. The meaning of $\xi(z)$ and $\zeta(z)$ may be different from line to line or even within the same equation. In the sequel we also assume that $K \geq B$, $n/K, n/L, (n-1)/K, (n-1)/L \in [\delta, \delta^{-1}]$, $x, y, x+u, y + u \in \mathcal{V}_2$ and $z \in \mathcal{D}_{\eta}'$ without mention.\\ 

From Proposition \ref{S43P1} we have
\begin{equation}\label{S51E4}
\begin{split}
&e^{(y-x) G\left(\frac{zL}{n}, \frac{n}{K}\right)} \prod_{ p = 1}^n\frac{zL - \ell_p^n  + x }{zL - \ell_p^n +y} = \sum_{i = 1}^3 \left( \frac{A_{n,L}^{\mathsf{s},i}(z;x,y)}{L^{i-1}} + \frac{B_{n,L}^{\mathsf{s},i}(z;x,y) }{L^i}  \right),\\
&e^{(x-y) G\left(\frac{zL}{n-1}, \frac{n-1}{K} \right)} \prod_{p = 1}^{n-1} \frac{zL-  \ell_p^{n-1} + u + y}{zL -  \ell_p^{n-1} + u + x} \\
&= \sum_{i = 1}^3 \left( \frac{A_{n-1,L}^{\mathsf{s},i}(z;u + y, u+x)}{L^{i-1}} + \frac{B_{n-1,L}^{\mathsf{s},i}(z;u + y, u + x) }{L^i}  \right).
\end{split}
\end{equation}
We further have from the smoothness of $G(z,s)$ that 
\begin{equation}\label{S51E5}
\begin{split}
&\exp\left( (x-y) \left[ G\left(\frac{zL}{n}, \frac{n}{K}\right) - G\left(\frac{zL}{n-1}, \frac{n-1}{K} \right) \right] \right) = 1 + \frac{D(z)}{L} + \frac{D(z)^2}{2L^2}  + \frac{\xi(z)}{L^3},
\end{split}
\end{equation}
where we recall that $D(z)$ is as in (\ref{S51E0}). 

Combining (\ref{S51E4}) and (\ref{S51E5}) with the definitions (\ref{S41E1}) and (\ref{S41E2}) we get 
\begin{equation*}
 \prod_{ p = 1}^n\frac{zL - \ell_p^n  + x }{zL - \ell_p^n +y}  \prod_{p = 1}^{n-1} \frac{zL-  \ell_p^{n-1} + u + y}{zL -  \ell_p^{n-1} + u + x}  = \sum_{i = 1}^3 \frac{A_{n,L}^{\mathsf{p},i}(z; x,y,u) }{L^{i-1}} + \frac{B_{n,L}^{\mathsf{p},1}(z; x,y,u)}{L} + \frac{B_L}{L^2} + \frac{\xi(z)}{L^3},
\end{equation*}
where 
\begin{equation}\label{S51E6}
\begin{split}
&B_L = B_{n,L}^{\mathsf{s},2}(z;x,y) + B_{n-1,L}^{\mathsf{s},2}(z;u+y,u+x) +  B_{n,L}^{\mathsf{s},1}(z;x,y) \cdot A_{n-1,L}^{\mathsf{s},2}(z;u+y,u+x) \\
&+ B_{n-1,L}^{\mathsf{s},1}(z;u+y,u+x) \cdot A_{n,L}^{\mathsf{s},2}(z;x,y) +  D(z)\cdot \left[   B_{n,L}^{\mathsf{s},1}(z;x,y) + B_{n-1,L}^{\mathsf{s},1}(z;u+y,u+x) \right] \\
& +  B_{n,L}^{\mathsf{s},1}(z;x,y) \cdot B_{n-1,L}^{\mathsf{s},1}(z;u+y,u+x).
\end{split}
\end{equation}
From the last equations, we would obtain (\ref{S51E3.5}) if we can show
\begin{equation}\label{S51E7}
\begin{split}
&B_L = B_{n,L}^{\mathsf{p},2}(z; x,y,u) + L^{-1/2} \cdot \zeta (z).
\end{split}
\end{equation}

From (\ref{MBRem}) and  (\ref{S43MBRem}) we have
\begin{equation}\label{S51E8}
\begin{split}
& \hat{G}^n_L(z) = \zeta(z), \hspace{2mm } \partial_z \hat{G}^n_L(z) = \zeta(z), \hspace{2mm} \left[ \hat{G}^{n}_{L}(z)  - \hat{G}^{n-1}_K(z) \right]= L^{-1/2} \zeta(z),  \mbox{ and } \\
& \partial_z \left[ \hat{G}^{n}_{L}(z)  - \hat{G}^{n-1}_L(z) \right]= L^{-1/2} \zeta(z).
\end{split}
\end{equation}
Combining (\ref{S51E8}) with the smoothness of $G(z,s)$ we conclude 
$$B_{n-1,L}^{\mathsf{s},i}(z;u+y,u+x)  = B_{n,L}^{\mathsf{s},2}(z;u+y,u+x) + L^{-1/2} \cdot \zeta(z) \mbox{ for $i = 1,2$}$$
$$A_{n-1,L}^{\mathsf{s},2}(z;u+y,u+x) = A_{n,L}^{\mathsf{s},2}(z;u+y,u+x) + L^{-1} \cdot \xi(z).$$
Substituting the last two identities into (\ref{S51E6}), the definitions of the functions $A_{n,L}^{\mathsf{s},i}, B_{n,L}^{\mathsf{s},i}$ from (\ref{S41E1}) and (\ref{S41E2}), and performing a bit of cancellation we arrive at (\ref{S51E7}).
\end{proof}

%-------------------------------------------------------------------------------------------------------------------------------------------------------------------------------------------------
% Section 5.2
%
%-------------------------------------------------------------------------------------------------------------------------------------------------------------------------------------------------
\subsection{Multi-level cumulant relations}\label{Section5.2} Suppose that $(\ell^1, \ell^2, \dots)$ is distributed according to $\P$ as in Definition \ref{BKCC}. For $ n \geq k \geq 1$ we define 
\begin{equation}\label{S52E1}
\begin{split}
&\U(z) =  \left[\Phi^-\left(\frac{zK}{n}, \frac{n}{K} \right) e^{-\theta G(zK/n, n/K)} - \Phi^+\left(\frac{zK}{n}, \frac{n}{K} \right) \right] \cdot \hat{G}^{n}_K(z) \\
&+ \theta \Phi^+\left(\frac{zK}{n}, \frac{n}{K} \right) \sum_{i = k+1}^n \frac{\partial_z \hat{G}^i_K(z)}{K}  - \Phi^+\left(\frac{zK}{n}, \frac{n}{K} \right) \left[ e^{\theta G(zK/k, k/K)} - 1 \right] \cdot \hat{G}^{k}_K(z),
\end{split}
\end{equation}
where we recall that $\hat{G}^n_K$ are as in (\ref{S32E1}), $G(z,s)$ is as in (\ref{S31E6}) and $\Phi^{\pm}(z,s)$ are as in (\ref{S31E10}). \\

The main result of the section is as follows.
\begin{proposition}\label{S52P1} Fix $\theta > 0$ and $\delta \in (0,1)$. Suppose that $(\ell^1, \ell^2, \dots )$ is distributed according to $\P$ as in Definition \ref{BKCC} and let $n, k$ be integers such that $n \geq k \geq 1$, $n/K, k/K \in [\delta, \delta^{-1}]$. Fix any $m_k, \dots, m_n \in \mathbb{Z}_{\geq 0}$ with $m = \sum_{i = k}^n m_i$ and $m+1$ parameters $v_0$ and $v_i^j$, for $j \in \llbracket k, n \rrbracket$ and $i \in \llbracket 1, m_j \rrbracket$, that are all in $\mathbb{C} \setminus [-\theta \delta^{-1}, \delta^{-2}]$. Then, we have
\begin{equation}\label{S52Cum1}
M\left(\U(v_0) ; \llbracket 1, m_k \rrbracket, \dots, \llbracket 1, m_n \rrbracket \right) = O(K^{-1/2}) \mbox{ if } m \geq 2,
\end{equation}  
where for $(n-k +1)$ subsets $F_r \subseteq \llbracket 1, m_r \rrbracket$ for $r\in \llbracket k, n \rrbracket$ and a bounded complex-valued random variable $\xi$ we write $M(\xi; F_k, \dots, F_n)$ in place of the joint cumulant of $\xi$ and the variables $\{ \hat{G}_K^j(v_i^j): j \in \llbracket k, n \rrbracket, i \in F_j \}$. 

In addition, for each $v_0, v_1 \in \mathbb{C} \setminus [-\theta \delta^{-1}, \delta^{-2}]$ and $a \in \llbracket k, n \rrbracket$ we have
\begin{equation}\label{S52Cov1}
\begin{split}
&M\left(\U(v_0) , \hat{G}^a_K(v_1)   \right) = O(K^{-1/2}) +    \oint_{\gamma} \frac{dz \Phi^+(zK/n, n/K)  }{2 \pi \i  \theta (z- v_0) (v_1-z)^2}  \cdot  \Bigg{[} -  e^{\theta G(zK/k, k/K)}    \\
&  \\
& +   \frac{\theta}{K} \sum_{j = k+1}^a \left[  \frac{z K^2}{j^2} \cdot \partial_z G\left( \frac{zK}{j}, \frac{j}{K} \right) + \frac{\theta K}{j} \cdot  \partial_z G\left( \frac{zK}{j}, \frac{j}{K}\right) -  \partial_s G \left(\frac{zK}{j} , \frac{j}{K} \right)  \right] \Bigg{]} ,
\end{split}
\end{equation}  
where the constants in the big $O$ notations depend on $\theta, \delta$ and in (\ref{S52Cum1}) also on $m$ and are uniform as $v_i^j$ and $v_0, v_1$ vary over compact subsets of $\mathbb{C} \setminus [-\theta \delta^{-1}, \delta^{-2}]$. In (\ref{S52Cov1}) we have that $\gamma$ is a positively oriented contour that encloses $[-\theta \delta^{-1}, \delta^{-2}]$ and excludes the points $v_0, v_1$. 
\end{proposition}
\begin{proof} For clarity we split the proof into six steps. In the first step we state one form of the multi-level Nekrasov's equations from \cite{DK21}, and linearize these equations using Propositions \ref{ProdExp1} and \ref{ProdExp2}. In the second step we establish (\ref{S52Cum1}) by reducing it to three statements (these are equations (\ref{S52A13rd}), (\ref{S52A23rd}) and (\ref{S52A33rd})), which are in turn established in Steps 3 and 4. Equation (\ref{S52Cov1}) is proved in Step 5 modulo a certain estimate (this is equation (\ref{S52A3V3})), which is then established in Step 6.\\

We start by introducing some notation that will be used throughout the proof. We fix a compact set $\mathcal{V}_1 \subset \mathbb{C}\setminus [-\theta \delta^{-1}, \delta^{-2}]$ and set $\mathcal{V}_2 = [-\theta - 1, 1 + \theta]$. We can find an $\eta > 0$ and a positively oriented contour $\Gamma$ that encloses $[-\theta \delta^{-1}, \delta^{-2}]$ such that $\Gamma \subset \mathcal{D}_{\eta}'$, $\mathcal{V}_1 \subset \mathcal{D}_{\eta}'$ with $\mathcal{D}_{\eta}'$ as in (\ref{Deta2}) and with $\Gamma$ and $\mathcal{V}_1$ at least distance $\eta/2$ away from each other. We further assume that $v_0, v_1, v^j_i \in \mathcal{V}_1$ for all $j \in \llbracket k, n \rrbracket$ and $i \in \llbracket 1, m_j\rrbracket$. We write $\xi_{j,K}(z)$ and $\zeta_{j,K}(z)$ to mean generic random functions of $j, K$ such that $n \geq j \geq k$, and $z \in \mathcal{D}_{\eta}'$ that are analytic on $\mathcal{D}_{\eta}'$  and $\P$-almost surely
$$   \max_{n  \geq j  \geq k }  \sup_{z \in \mathcal{V}} \left| \xi_{j, K}(z) \right| = O(1) \mbox{, and  } \max_{n \geq j \geq k }  \sup_{z \in \mathcal{V}}  \mathbb{E} \left[ \left|\zeta_{j,K}(z) \right|^m \right] = O(1),$$
for $m \in \mathbb{N}$, where $\mathcal{V}$ is a compact subset of $\mathcal{D}_{\eta}'$ and the constants in the big $O$ notations depend on $\mathcal{V},\theta, \delta$ and also on $m$ for the second identity, but not on $K$ provided $K \geq B$ as in the statement of Proposition \ref{ProdExp2} for our choice of $\eta, \delta, \mathcal{V}_2$. To ease the notation, we drop $j,K$ and simply write $\xi(z), \zeta(z)$. The meaning of $\xi(z)$ and $\zeta(z)$ may be different from line to line or even within the same equation. In the sequel we assume $K \geq B$, $n \geq k \geq 1$, $n/K, k/K \in [\delta, \delta^{-1}]$. Also, all constants in the big $O$ notations, unless otherwise specified, depend on $\theta, \delta, \eta, \mathcal{V}_1$ and $\Gamma$ and hold provided that the inequalities in the previous sentence hold. We will not mention this further.

We also record for future use the following consequences of (\ref{MBRem}) and (\ref{S43MBRem})
\begin{equation}\label{S5MB}
\begin{split}
&\partial_z^i \hat{G}_{K}^j(z) = \zeta(z) \mbox{ for $i = 0,1,2$ and $j \in \llbracket k,n \rrbracket$; }\\
& \partial_z^i [\hat{G}_{K}^j(z) -\hat{G}_{K}^{j-1}(z) ]  = K^{-1/2} \cdot \zeta(z) \mbox{ for $i = 0,1$ and $j \in \llbracket k+1,n \rrbracket$.} 
\end{split}
\end{equation}

{\raggedleft \bf Step 1.} From \cite[Lemma 3.5]{DK21} we have the following multi-level loop equations
\begin{equation}\label{S52Loop}
\begin{split}
&0=  \sum_{\substack{F_r \subseteq \llbracket 1, m_r \rrbracket \\  r  \in \llbracket k, n \rrbracket }}    \prod_{i = k}^n {\bf 1 } \{F_i = \llbracket 1, m_i \rrbracket\}  \oint_{\Gamma}\frac{dz \Phi^-_{n,K}(zK) }{2 \pi \i K (z- v_0)}\cdot  M\left( \prod_{p = 1}^n\frac{zK- \ell^n_p -\theta}{zK - \ell^n_p}; F_k, \dots, F_n  \right)+  \\
& + \oint_{\Gamma} \frac{dz \Phi_{n,K}^+(zK)}{2 \pi \i  K (z- v_0)}   \prod_{i = k}^n \prod_{f \in F_i^c}\frac{1}{K(v_f^i-z)(v_f^i - z + K^{-1})} \\
&\times   M\left( \prod_{p = 1}^{k}\frac{zK- \ell^{k}_p + \theta - 1}{zK - \ell^{k}_p - 1}; F_k, \dots, F_N \right) +  \sum\limits_{j=k+1}^{n}  \prod_{i = k}^{j-1} {\bf 1 } \{F_i = \llbracket 1, m_i \rrbracket\} \\
&\times \oint_{\Gamma} \frac{dz\Phi_{n,K}^+(zK)}{2\pi \i K (z - v_0)}    \cdot   \prod_{i = j}^n \prod_{f \in F_i^c}\frac{1}{K(v_f^i-z)(v_f^i - z + K^{-1})}  \cdot   M \left(\Pi_1^{\theta,j}(zK) ; F_k, \dots, F_N \right),
\end{split}
\end{equation}
where $\Pi_1^{\theta, j}$ is given by
\begin{equation}\label{S52Pi1}
\Pi_1^{\theta,j}(z) = \begin{cases} \mathlarger{\frac{\theta}{1- \theta} \cdot \prod_{i = 1}^j \frac{z - \ell_i^j - \theta}{z - \ell_i^j - 1} \cdot \prod_{i = 1}^{n-1} \frac{z - \ell_i^{j-1} + \theta - 1}{z - \ell_i^{j-1} }} &\mbox{ if } \theta \neq 1, \\ \mathlarger{\sum_{i = 1}^j \frac{1}{z - \ell_i^j - 1} - \sum_{i = 1}^{j-1} \frac{1}{z - \ell_i^{j-1}}} &\mbox{ if } \theta = 1. \end{cases}
\end{equation}
We mention that to identify (\ref{S52Loop}) with \cite[(3.10)]{DK21} one needs to use the correspondences
$$ n \leftrightarrow N, \hspace{1mm} K \leftrightarrow L = M, \hspace{1mm} \Phi^{\pm}_{n,K}(zK) \leftrightarrow \Phi^{\pm}_N(zL), \hspace{1mm} w_n^{\theta, K}(z) \leftrightarrow  w_N(z), \hspace{1mm} \mathbb{C} \leftrightarrow \mathcal{M}, \hspace{1mm} K(z- v_0) \leftrightarrow S_v(zL). $$
We also mention that to reconcile \cite[Lemma 3.5]{DK21} with (\ref{S52Loop}), one needs to set $S_v(zL) = L(z- v)$, which is allowed in view of \cite[Remark 3.7]{DK21} with $H(z) = 1$ and the fact that 
$$\Phi_N^-(-N\theta ) = \Phi_{n,K}^-(-n\theta) = 0 \mbox{ and } \Phi^+_N(M+1 - \theta) =  \Phi^{+}_{n,K}(K + 1 - \theta) = 0.$$
Finally, we remark that our joint cumulants in (\ref{S52Loop}) involve $\hat{G}^j_K(v_i^j)$ as opposed to ${G}^j_K(v_i^j)$ like in \cite[Lemma 3.5]{DK21}, but this is not an issue since from (\ref{S32E1}) $\hat{G}^j_K(z)$ differs from ${G}^j_K(z)$ by a deterministic function, and second and higher order cumulants remain unchanged upon constant shifts, cf. (\ref{S3Linearity}).\\

In the remainder of this step we use Propositions \ref{ProdExp1} and \ref{ProdExp2} to linearize (\ref{S52Loop}). We first note that for $j\in \llbracket k + 1, n \rrbracket$ we have
\begin{equation}\label{S52PiExpand}
\begin{split}
&\Pi_1^{\theta, j}(zK) = \frac{\theta {\bf 1}\{ \theta \neq 1\}}{1 - \theta} + \theta \left[G\left( \frac{zK}{j}, \frac{j}{K} \right) - G\left( \frac{zK}{j-1}, \frac{j-1}{K} \right)   \right]  - \frac{\theta(1 + \theta)}{2j}  \partial_z G\left( \frac{zK}{j}, \frac{j}{K}\right)\\
& + \frac{\theta (1- \theta) }{2(j-1)} \cdot \partial_z G\left( \frac{zK}{j-1}, \frac{j-1}{K}\right) + \frac{A_j^{\theta}(z)}{K^2} + \frac{\theta[\hat{G}^j_K(z) - \hat{G}^{j-1}_K(z)] }{K} - \frac{\theta^2 \partial_z \hat{G}^j_K(z)}{K^2} + \frac{\zeta(z)}{K^{5/2}},
\end{split}
\end{equation}
where $A_j^{\theta}(z)$ is a deterministic function and $A_j^{\theta}(z) = \xi(z)$. When $\theta \neq 1$ equation (\ref{S52PiExpand}) follows from applying Proposition \ref{ProdExp2} to the product in (\ref{S52Pi1}). We mention that within the proposition one needs to take $x = -\theta$, $y = -1$, $u = \theta$, $L = K$ and $n = j$. We also mention that we expressed $A_{j,K}^{\mathsf{p},2}(z; -\theta, -1, \theta)$ in terms of the functions in (\ref{S41E1}) and the function $A_{j,K}^{\mathsf{p},3}(z; -\theta, - 1, \theta)$ got absorbed into $A_j^{\theta}(z)$ in (\ref{S52PiExpand}). When $\theta = 1$ we have from (\ref{S52Pi1}) and the definition of $G^n_L$ and $\hat{G}^n_L$ in (\ref{S32E1}) 
\begin{equation*}
\begin{split}
&\Pi_1^{1,j}(zK) = \sum_{i = 1}^j \frac{1}{zK - \ell_i^j } - \sum_{i = 1}^{j-1} \frac{1}{zK - \ell_i^{j-1}} + \sum_{i = 1}^j \frac{1}{(zK -\ell_i^j)^2} + \sum_{i = 1}^j \frac{1}{(zK -\ell_i^j)^2(zK -\ell_i^j - 1)} \\
& = \frac{[\hat{G}^j_K(z) - \hat{G}^{j-1}_K(z)]}{K} + \left[G\left( \frac{zK}{j}, \frac{j}{K} \right) - G\left( \frac{zK}{j-1}, \frac{j-1}{K} \right)   \right] - \frac{\partial_z \hat{G}_K^j(z)}{K^2} - \frac{1}{j} \cdot \partial_z G\left( \frac{zK}{j}, \frac{j}{K}\right) \\
& + \sum_{i = 1}^j \frac{1}{(zK -\ell_i^j)^3} + \frac{\xi(z)}{K^3}.
\end{split}
\end{equation*}
The latter equation implies (\ref{S52PiExpand}) when $\theta = 1$, since the last line can be absorbed into the $A_j^{\theta}(z)$ and the $\zeta(z)$ terms. This is clear for $\xi(z)$ and also we note 
$$\sum_{i = 1}^j \frac{1}{(zK -\ell_i^j)^3}  = \frac{\partial^2_z G^j_K(z)}{2K^3}  = \frac{\partial^2_z \hat{G}^j_K(z) }{2K^3} +  \frac{1}{2j^2} \cdot \partial^2_z G\left( \frac{zK}{j}, \frac{j}{K} \right) = \frac{\zeta(z)}{K^3} + \frac{1}{2j^2} \cdot \partial^2_z G\left( \frac{zK}{j}, \frac{j}{K} \right), $$
where we used the moment bound for $\partial^2_z \hat{G}^j_K(z)$ from (\ref{S5MB}). 

From Proposition \ref{ProdExp1} applied to $L = K$, $x = -\theta$, $y = 0$ we have
\begin{equation}\label{S52Prod1Expand}
\begin{split}
&\prod_{p = 1}^n\frac{zK- \ell^n_p -\theta}{zK - \ell^n_p} = e^{-\theta G(zK/n, n/K)}  +e^{-\theta G(zK/n, n/K)} \cdot  \frac{(\theta^2 -1)}{2n} \cdot \partial_z G(zK/n, n/K) \\
&-   e^{-\theta G(zK/n, n/K)} \cdot \frac{\theta}{K} \cdot \hat{G}^n_K(z)+ \frac{\zeta(z)}{K^2}.
\end{split}
\end{equation}
Also, from Proposition \ref{ProdExp1} applied to $L = K$, $x = \theta - 1$, $y = -1$ we have
\begin{equation}\label{S52Prod2Expand}
\begin{split}
&\prod_{p = 1}^{k}\frac{zK- \ell^{k}_p + \theta - 1}{zK - \ell^{k}_p - 1} = e^{\theta G(zK/k, k/K)} + e^{\theta G(zK/k, k/K)} \cdot \frac{\theta (\theta -2)}{2k} \cdot \partial_z G(zK/k, k/K) \\
& +  e^{\theta G(zK/k, k/K)} \cdot \frac{\theta}{K} \cdot \hat{G}^k_K(z)+ \frac{\zeta(z)}{K^2}.
\end{split}
\end{equation}

Substituting (\ref{S52PiExpand}), (\ref{S52Prod1Expand}) and (\ref{S52Prod2Expand}) into (\ref{S52Loop}) we get for $m = \sum_{i = k}^n m_i \geq 1$
\begin{equation}\label{S52Linear}
O(K^{-3/2}) = A_1 + A_2 + A_3, \mbox{ where }
\end{equation}
\begin{equation}\label{S52A1}
\begin{split}
A_1 =   \oint_{\Gamma}\frac{dz (-\theta) \Phi^-_{n,K}(zK) e^{-\theta G(zK/n, n/K)}  }{2 \pi \i K^2 (z- v_0)}\cdot  M\left( \hat{G}^n_K(z) ; \llbracket 1, m_k \rrbracket, \dots, \llbracket 1 ,m_n \rrbracket \right),
\end{split}
\end{equation}
\begin{equation}\label{S52A2}
\begin{split}
&A_2 =   \sum_{\substack{F_r \subseteq \llbracket 1, m_r \rrbracket \\  r \in \llbracket k ,n \rrbracket }}  \oint_{\Gamma} \frac{dz \Phi_{n,K}^+(zK) e^{\theta G(zK/k, k/K)} }{2 \pi \i  K (z- v_0)}   \prod_{i = k}^n \prod_{f \in F_i^c}\frac{1}{K(v_f^i-z)(v_f^i - z + K^{-1})} \\
&\times   M\Bigg{(}1 + \frac{\theta (\theta -2)}{2k} \cdot \partial_z G(zK/k, k/K)  + \frac{\theta}{K}   \cdot  \hat{G}^k_K(z) ; F_k, \dots, F_n \Bigg{)} ,
\end{split}
\end{equation}
\begin{equation}\label{S52A3}
\begin{split}
&A_3 =   \sum_{\substack{F_r \subseteq \llbracket 1, m_r \rrbracket \\  r \in \llbracket k,n \rrbracket }} \sum\limits_{j=k+1}^{n}  \prod_{i = k}^{j-1} {\bf 1 } \{F_i = \llbracket 1, m_i \rrbracket\} \oint_{\Gamma} \frac{dz\Phi_{n,K}^+(zK)}{2\pi \i K (z - v_0)}      \\
&\times \prod_{i = j}^n \prod_{f \in F_i^c}\frac{1}{K(v_f^i-z)(v_f^i - z + K^{-1})}    \cdot   M \Bigg{(}  \frac{\theta {\bf 1}\{ \theta \neq 1\}}{1 - \theta} + \frac{A_j^{\theta}(z)}{K^2}   \\
& + \theta \left[G\left( \frac{zK}{j}, \frac{j}{K} \right) - G\left( \frac{zK}{j-1}, \frac{j-1}{K} \right)   \right]  - \frac{\theta(1 + \theta)}{2j} \cdot  \partial_z G\left( \frac{zK}{j}, \frac{j}{K}\right) \\
& + \frac{\theta (1- \theta) }{2(j-1)} \cdot \partial_z G\left( \frac{zK}{j-1}, \frac{j-1}{K}\right) + \frac{\theta}{K}  \cdot [\hat{G}^j_K(z) - \hat{G}^{j-1}_K(z)] - \frac{\theta^2}{K^2} \cdot \partial_z \hat{G}^j_K(z) ; F_k, \dots, F_n \Bigg{)}.
\end{split}
\end{equation}
We mention that when we used (\ref{S52Prod1Expand}) in the first line of (\ref{S52Loop}) we got rid of the deterministic terms, since second and higher joint cumulants remain unchanged upon constant shifts,  see (\ref{S3Linearity}), and $m \geq 1$ in our setup. We also mention that joint cumulants of the form $M(\zeta(z); F_k, \dots, F_n)$ are all of unit order, since $\hat{G}_K^j(z)$ for $j \in \llbracket k, n \rrbracket$ have bounded moments in view of (\ref{S5MB}) as does $\zeta(z)$. In particular, all $\zeta(z)$ terms coming from (\ref{S52PiExpand}), (\ref{S52Prod1Expand}) and (\ref{S52Prod2Expand}) got absorbed into $O(K^{-3/2})$. Here, we also use that the number of summands $j \in \llbracket k+1, n \rrbracket$ is $O(K)$ and so the $\zeta(z)$ terms in (\ref{S52PiExpand}) each producing an $O(K^{-5/2})$ errors at most accumulate to an $O(K^{-3/2})$ error. Finally, we also mention that all the functions involving $G(z,s)$ are of unit order by the smoothness of this function, and the coefficients $K^{-1}\Phi^{\pm}_{n,K}(zK)$ are also of unit order since by (\ref{S31E3}) and (\ref{S31E10})
\begin{equation}\label{S52Phis}
\begin{split}
&\frac{\Phi^{-}_{n,K}(zK)}{K}  = \frac{zK + n\theta}{K} = \frac{n}{K} \cdot \Phi^-(zK/n, n/K) \mbox{, and }\\
&\frac{ \Phi^{+}_{n,K}(zK)}{K} = \frac{n}{K} \cdot \Phi^+(zK/n, n/K) + \frac{1-\theta}{K}. 
\end{split}
\end{equation}

{\bf \raggedleft Step 2.} In this and the next two steps we assume that $m = \sum_{i = k}^n m_i \geq 2$. We claim that 
\begin{equation}\label{S52A13rd}
\begin{split}
&A_1 = \oint_{\Gamma}\frac{dz (-\theta) (n/K) \Phi^-(zK/n, n/K) e^{-\theta G(zK/n, n/K)}  }{2 \pi \i K (z- v_0)} \cdot  M\left( \hat{G}^n_K(z) ; \llbracket 1, m_k \rrbracket \dots, \llbracket 1 ,m_n \rrbracket \right) ,
\end{split}
\end{equation} 
\begin{equation}\label{S52A23rd}
\begin{split}
&A_2 = \oint_{\Gamma}\frac{dz \theta (n/K) \Phi^+(zK/n, n/K) e^{\theta G(zK/k, k/K)}  }{2 \pi \i K (z- v_0)} \\
& \times  M\left( \hat{G}^k_K(z) ; \llbracket 1, m_k \rrbracket \dots, \llbracket 1 ,m_n \rrbracket \right) + O(K^{-2}),
\end{split}
\end{equation} 
\begin{equation}\label{S52A33rd}
\begin{split}
&A_3 = \oint_{\Gamma}\frac{dz \theta  (n/K) \Phi^+(zK/n, n/K)  }{2 \pi \i K (z- v_0)} \\
& \times  M\left( \hat{G}^n_K(z) - \hat{G}^k_K(z) - \frac{\theta}{K} \sum_{j = k + 1}^n\partial_z \hat{G}^j_K(z) ; \llbracket 1, m_k \rrbracket \dots, \llbracket 1 ,m_n \rrbracket \right) + O(K^{-2}).
\end{split} 
\end{equation} 
We establish (\ref{S52A13rd}), (\ref{S52A23rd}) and (\ref{S52A33rd}) in the next two steps. Here, we assume their validity and conclude the proof of (\ref{S52Cum1}). \\

Substituting (\ref{S52A13rd}), (\ref{S52A23rd}) and (\ref{S52A33rd}) into (\ref{S52Linear}) and using the definition of $\U$ from (\ref{S52E1}) we conclude that 
\begin{equation}\label{S52U3rd}
O(K^{-3/2}) = \oint_{\Gamma}\frac{dz (-\theta) (n/K)   }{2 \pi \i K (z- v_0)}\cdot  M\left(\U(z) ; \llbracket 1, m_k \rrbracket \dots, \llbracket 1 ,m_n \rrbracket \right).
\end{equation}
We may now evaluate the right side of (\ref{S52U3rd}) as (minus) the residue at $z = v_0$. We remark that there is no residue at infinity. The latter can be seen to be the case since from (\ref{S31E8}) and (\ref{S32E1}) we have for $j \in \llbracket k, n \rrbracket$ as $|z| \rightarrow \infty$
$$e^{\pm \theta G(zK/j,j/K )} = 1 + O(|z|^{-1}), \mbox{ and } \hat{G}^j_K(z) = O(|z|^{-2}), \hspace{2mm}\partial_z \hat{G}^j_K(z) = O(|z|^{-3}),$$
where the constants in the latter big $O$ notations depend on the usual parameters but also on $j$. Using the latter, and the fact that $\Phi^{\pm}(z,s)$ are {\em monic polynomials} ensures that as $|z| \rightarrow \infty$
\begin{equation}\label{S52PoleInf}
 \U(z) = O(|z|^{-1}).
\end{equation}
Since in (\ref{S52U3rd}) we are further dividing by $(z-v_0)$, we conclude that the integrand behaves like $O(|z|^{-2})$ as $|z| \rightarrow \infty$, which ensures there is no residue at infinity. Evaluating (\ref{S52U3rd}) at $z = v_0$ and multiplying the resulting expression by $\theta^{-1} \cdot (K/n) \cdot K$ we arrive at (\ref{S52Cum1}).\\

{\bf \raggedleft Step 3.} In this step we establish (\ref{S52A13rd}) and (\ref{S52A23rd}). Equation (\ref{S52A13rd}) follows from substituting the formula for $\Phi^-_{n,K}$ from (\ref{S52Phis}) into (\ref{S52A1}). Thus, we only need to verify (\ref{S52A23rd}).

Let us denote by $C_2(F_k, \dots, F_n)$ the summand in (\ref{S52A2}) corresponding to the sets $F_k, \dots, F_m$. We investigate the contribution of such terms. Let us denote $\mathsf{F} = |F_k| + \cdots + |F_n|$ and note that if $\mathsf{F} > 0$ we can remove the constant terms from the joint cumulant in (\ref{S52A2})  to get
\begin{equation}\label{S52C2}
\begin{split}
&C_2(F_k, \dots, F_n) = \oint_{\Gamma} \frac{dz \Phi_{n,K}^+(zK)}{2 \pi \i  K (z- v_0)}   \prod_{i = k}^n \prod_{f \in F_i^c}\frac{1}{K(v_f^i-z)(v_f^i - z + K^{-1})} \\
&\times   M\left( e^{\theta G(zK/k, k/K)} \cdot \frac{\theta}{K} \cdot  \hat{G}^k_K(z)  ; F_k, \dots, F_n \right).
\end{split}
\end{equation}
If $m > \mathsf{F}$, then the product in the first line of (\ref{S52C2}) produces at least one factor $K^{-1}$, while the cumulant on the second line of (\ref{S52C2}) is $O(K^{-1})$ due to the moment bounds of $\hat{G}^j_K(z)$ from (\ref{S5MB}). In particular, such terms contribute at most $O(K^{-2})$. If $\mathsf{F} = 0$ (i.e. all $F_i = \emptyset$) then the product on the first line of (\ref{S52A2}) is $O(K^{-2})$ (here we used that $m \geq 2$) so $C_2(\emptyset, \dots, \emptyset) = O(K^{-2})$ as well. Combining the last few observations we see that the only $C_2(F_k, \dots, F_n)$ term that contributes non-trivially is the one corresponding to $\mathsf{F} = m$, i.e. $F_i = \llbracket 1, m_i \rrbracket$ for $i = k, \dots, n$. In particular,
\begin{equation*}
\begin{split}
&A_2 = \oint_{\Gamma} \frac{dz \Phi_{n,K}^+(zK)}{2 \pi \i  K^2 (z- v_0)}   \cdot M\left( \hat{G}^k_K(z) ; \llbracket 1, m_k \rrbracket \dots, \llbracket 1 ,m_n \rrbracket \right) + O(K^{-2}).
\end{split}
\end{equation*}
Substituting the formula for $\Phi^+_{n,K}$ from (\ref{S52Phis}) into the last equation gives (\ref{S52A23rd}). \\

{\bf \raggedleft Step 4.} In this step we establish (\ref{S52A33rd}). Similarly to the previous step, we let $C_3(F_k, \dots, F_n)$ denote the summand in (\ref{S52A3}) corresponding to the sets $F_k, \dots, F_m$ and we investigate the contribution of such terms. We also let $\mathsf{F} = |F_k| + \cdots + |F_n|$ as before. 

Our first observation is that if $\mathsf{F} > 0$ we can remove the constant terms from the cumulant in (\ref{S52A3}) and get
\begin{equation}\label{S52C31}
\begin{split}
&C_3(F_k, \dots, F_n) = \sum\limits_{j=k+1}^{n}  \prod_{i = k}^{j-1} {\bf 1 } \{F_i = \llbracket 1, m_i \rrbracket\} \oint_{\Gamma} \frac{dz \theta \Phi_{n,K}^+(zK)}{2\pi \i K (z - v_0)}      \\
&  \times \prod_{i = j}^n \prod_{f \in F_i^c}\frac{1}{K(v_f^i-z)(v_f^i - z + K^{-1})}     \cdot   M \Bigg{(}\frac{[\hat{G}^j_K(z) - \hat{G}^{j-1}_K(z)] }{K} - \frac{\theta \partial_z \hat{G}^j_K(z)}{K^2} ; F_k, \dots, F_n \Bigg{)}.
\end{split}
\end{equation}
If $\mathsf{F} < m$, then $|F_i| < m_i$ for some $i$ and we let $R$ be the smallest such index. In particular, the summands in (\ref{S52C31}) for $j \geq R+1$ do not contribute because of the presence of the indicator function ${\bf 1}\{ F_R = \llbracket 1, m_R \rrbracket\}$. We conclude that if $m -1 \geq  \mathsf{F} \geq 1$
\begin{equation}\label{S52C32}
\begin{split}
&C_3(F_k, \dots, F_n) = \sum\limits_{j=k+1}^{R}  \oint_{\Gamma} \frac{dz \theta \Phi_{n,K}^+(zK)}{2\pi \i K (z - v_0)}    \\
&  \times \prod_{i = j}^n \prod_{f \in F_i^c}\frac{1}{K(v_f^i-z)(v_f^i - z + K^{-1})}     \cdot   M \Bigg{(}\frac{[\hat{G}^j_K(z) - \hat{G}^{j-1}_K(z)] }{K} - \frac{\theta \partial_z \hat{G}^j_K(z)}{K^2} ; F_k, \dots, F_n \Bigg{)}.
\end{split}
\end{equation}
The product on the second line of (\ref{S52C32}) is $O(K^{ |\mathsf{F}| - m})$ while the cumulant is $O(K^{-1})$ due to the moment bounds for $\hat{G}_K^j$ and $\partial_z \hat{G}_K^j$ from (\ref{S5MB}). Also, the order of the sum over $j \in \llbracket k+1, R \rrbracket$ is $O(K)$. The latter estimates show that for $m-2 \geq \mathsf{F} \geq 1$ we have 
 \begin{equation}\label{S52C33}
\begin{split}
&C_3(F_k, \dots, F_n) = O(K^{-2}).
\end{split}
\end{equation}
If $\mathsf{F} = m-1$, then we still have (\ref{S52C32}) but now we also know that $F_i = \llbracket 1, m_i\rrbracket$ for all $i \neq R$. In particular, the product on the second line of (\ref{S52C32}) becomes independent of $j$ and can be replaced with one from $i = R$ to $i = n$. Once we do this we can push the sum over $j$ all the way inside of the cumulant by linearity and obtain
\begin{equation}\label{S52C34}
\begin{split}
&C_3(F_k, \dots, F_n) = \oint_{\Gamma} \frac{dz \theta \Phi_{n,K}^+(zK)}{2\pi \i K (z - v_0)}   \prod_{i = R}^n \prod_{f \in F_i^c}\frac{1}{K(v_f^i-z)(v_f^i - z + K^{-1})}  \\
&  \times    M \Bigg{(}\frac{[\hat{G}^{R}_K(z) - \hat{G}^{k}_K(z)] }{K} - \sum_{j=k+1}^{R}\frac{\theta \partial_z \hat{G}^j_K(z)}{K^2} ; F_k, \dots, F_n \Bigg{)}.
\end{split}
\end{equation}
We mention that once we pushed the $j$-sum inside the terms $[\hat{G}^j_K(z) - \hat{G}^{j-1}_K(z)]$ telescoped to $[\hat{G}^{R}_K(z) - \hat{G}^{k}_K(z)]$. Looking at (\ref{S52C34}) we have that the product on the first line is $O(K^{-1})$ while the cumulant on the second line is also $O(K^{-1})$ due to the moment bounds for $\hat{G}_K^j$ and $\partial_z \hat{G}_K^j$ from (\ref{S5MB}). In particular we now conclude that (\ref{S52C33}) holds when $\mathsf{F} = m-1$. We point out that the crucial advantage that happened when $\mathsf{F} = m-1$ versus $\mathsf{F} \leq m-2$ is that we could cancel many terms by telescoping.\\

We next consider $\mathsf{F} = 0$, i.e. the case $F_i = \emptyset $ for $i = k, \dots, n$. In this case, the cumulants in (\ref{S52A3}) all become expectations, and the term $ \frac{\theta {\bf 1}\{ \theta \neq 1\}}{1 - \theta} $ on the second line of  (\ref{S52A3}) integrates to $0$ by Cauchy's theorem. Here we used that $v_j^i$ are outside of $\Gamma$, as is $v_0$ and that $\Phi^+_{n,K}$ is a polynomial. The rest of the cumulants (or expectations) in (\ref{S52A3})  now become $O(K^{-1})$, while
$$\prod_{i = k}^{j-1} {\bf 1 } \{\emptyset = \llbracket 1, m_i \rrbracket\}\prod_{i = j}^n \prod_{f = 1}^{m_i}\frac{1}{K(v_f^i-z)(v_f^i - z + K^{-1})}  = O(K^{-m}).$$
The latter is clear when $m_i > 0$ for some $i \leq j-1$ as the indicators make the whole expression zero, while if $m_i = 0$ for all $i = k, \dots, j-1$ the second product has precisely $m$ terms. Combining the last few observations we see that each $j$-summand in $C_3(\emptyset, \dots, \emptyset)$ is $O(K^{-m-1})$. As there are $O(K)$ such summands and $m \geq 2$ by assumption we conclude (\ref{S52C33}) holds when $\mathsf{F} = 0$ as well. 

Overall, we see that (\ref{S52C33}) holds for all $m - 1 \geq \mathsf{F} \geq 0$, and so the only $C_3(F_k, \dots, F_n)$ term that contributes non-trivially is the one corresponding to $\mathsf{F} = m$, i.e. $F_i = \llbracket 1, m_i \rrbracket$ for $i = k, \dots, n$. In particular, from (\ref{S52C31}) we conclude 
\begin{equation*}
\begin{split}
&A_3 = C_3(\llbracket 1, m_k \rrbracket, \dots, \llbracket 1, m_n \rrbracket ) + O(K^{-2}) =O(K^{-2}) +  \sum\limits_{j=k+1}^{n}  \oint_{\Gamma} \frac{dz \theta \Phi_{n,K}^+(zK)}{2\pi \i K (z - v_0)}    \\
&  \times   M \Bigg{(}\frac{[\hat{G}^j_K(z) - \hat{G}^{j-1}_K(z)] }{K} - \frac{\theta \partial_z \hat{G}^j_K(z)}{K^2} ;\llbracket 1, m_k \rrbracket, \dots, \llbracket 1, m_n \rrbracket  \Bigg{)} \\
&= O(K^{-2}) +  \oint_{\Gamma} \frac{dz \theta \Phi_{n,K}^+(Kz)}{2\pi \i K (z - v_0)}   M \Bigg{(}\frac{[\hat{G}^n_K(z) - \hat{G}^{k}_K(z)] }{K} - \sum\limits_{j=k+1}^{n}   \frac{\theta \partial_z \hat{G}^j_K(z)}{K^2} ;\llbracket 1, m_k \rrbracket, \dots, \llbracket 1, m_n \rrbracket  \Bigg{)}.
\end{split}
\end{equation*}
Substituting the formula for $\Phi^+_{n,K}$ from (\ref{S52Phis}) into the last equation gives (\ref{S52A33rd}). \\

{\bf \raggedleft Step 5.} In this and the next step we assume that $m_a = 1$ for some $a \in \llbracket k, n \rrbracket$ and $m_i = 0$ for $i \in \llbracket k, n\rrbracket \setminus \{a\}$ and that $v^a_1 = v_1$. Our goal is to prove (\ref{S52Cov1}).

We still have equation (\ref{S52Linear}), where for the specified choice of parameters the expressions $A_1, A_2, A_3$  from (\ref{S52A1}), (\ref{S52A2}) and (\ref{S52A3}) take the following simpler forms:
\begin{equation}\label{S52A1V2}
\begin{split}
A_1 =   \oint_{\Gamma}\frac{dz (-\theta) \Phi^-_{n,K}(zK) e^{-\theta G(zK/n, n/K)}  }{2 \pi \i K^2 (z- v_0)}\cdot  M\left( \hat{G}^n_K(z) , \hat{G}^a_K(v_1) \right),
\end{split}
\end{equation}
\begin{equation}\label{S52A2V2}
\begin{split}
&A_2 =   \oint_{\Gamma} \frac{dz \theta \Phi_{n,K}^+(zK) e^{\theta G(zK/k, k/K)}}{2 \pi \i  K^2 (z- v_0)} \cdot M \left(\hat{G}^k_K(z), \hat{G}^a_K(v_1) \right) + \oint_{\Gamma} \frac{dz \Phi_{n,K}^+(zK)}{2 \pi \i  K (z- v_0)} \\
&    \times  \frac{e^{\theta G(zK/k, k/K)}}{K(v_1-z)(v_1- z + K^{-1})} \cdot   \Bigg{[}1  +   \frac{\theta (\theta -2)}{2k} \cdot  \partial_z G(zK/k, k/K) +\frac{\theta }{K}  \cdot \mathbb{E}\left[\hat{G}^k_K(z) \right] \Bigg{]} ,
\end{split}
\end{equation}
\begin{equation}\label{S52A3V2}
\begin{split}
&A_3 =    \sum_{j=k+1}^{n}   \oint_{\Gamma} \frac{dz\Phi_{n,K}^+(zK)}{2\pi \i K (z - v_0)}     \cdot   M \left(  \frac{\theta[\hat{G}^j_K(z) - \hat{G}^{j-1}_K(z)] }{K} - \frac{\theta^2 \partial_z \hat{G}^j_K(z)}{K^2} , \hat{G}_K^a(v_1) \right) \\
&+   \sum_{j=k+1}^{a}  \oint_{\Gamma} \frac{dz\Phi_{n,K}^+(zK)}{2\pi \i K (z - v_0)}  \cdot \frac{1}{K(v_1 - z) (v_1 - z + K^{-1})} \cdot   \Bigg{[}    \frac{\theta {\bf 1}\{ \theta \neq 1\}}{1 - \theta} + \frac{A_j^{\theta}(z)}{K^2}  \\
&   + \theta \left[G\left( \frac{zK}{j}, \frac{j}{K} \right) - G\left( \frac{zK}{j-1}, \frac{j-1}{K} \right)   \right]    - \frac{\theta(1 + \theta)}{2j} \cdot  \partial_z G\left( \frac{zK}{j}, \frac{j}{K}\right)  \\
&+ \frac{\theta (1- \theta) }{2(j-1)} \cdot \partial_z G\left( \frac{zK}{j-1}, \frac{j-1}{K}\right)  + \frac{\theta }{K} \cdot \mathbb{E}\left[\hat{G}^j_K(z) - \hat{G}^{j-1}_K(z) \right] - \frac{\theta^2 }{K^2}\cdot \mathbb{E}\left[\partial_z  \hat{G}^j_K(z) \right] \Bigg{]}.
\end{split}
\end{equation}
Let us elaborate a bit on the derivation of the above formulas. Equation (\ref{S52A1V2}) is a direct consequence of (\ref{S52A1}).  The sum over $F_r$'s in (\ref{S52A2}) becomes the sum of two terms -- when $F_a = \emptyset$ and $F_a = \{1\}$. In the case when $F_a = \{1\}$ we can remove the constant terms from the joint cumulants in (\ref{S52A2}) and the result becomes the first integral of (\ref{S52A2V2}). When $F_a = \emptyset$ the cumulant in  (\ref{S52A2}) becomes an expectation and this produces the second integral in (\ref{S52A2V2}). Similarly, the sum over $F_r$'s in (\ref{S52A3}) becomes the sum of two terms -- when $F_a = \emptyset$ and $F_a = \{1\}$. The summand $F_a = \{1\}$ becomes the first line in (\ref{S52A3V2}) once we remove constant terms from the joint cumulant. When $F_a = \emptyset$, we note that only the summands over $j \in \llbracket k+1, a\rrbracket$ contribute because terms for $j \geq a+1$ include the factor ${\bf 1}\{F_a = \llbracket 1,1 \rrbracket\}$ that is zero. In addition, when $F_a = \emptyset$ the cumulant in (\ref{S52A3}) becomes an expectation and that term becomes lines 2-4 in (\ref{S52A3V2}). \\

Using the formula for $\Phi^-_{n,K}(zK)$ from (\ref{S52Phis})  in (\ref{S52A1V2}) we conclude 
\begin{equation}\label{S52A1V3}
\begin{split}
A_1 =   \oint_{\Gamma}\frac{dz (-\theta) (n/K) \Phi^-(zK/n, n/K) e^{-\theta G(zK/n, n/K)}  }{2 \pi \i K (z- v_0)} \cdot  M\left( \hat{G}^n_K(z) , \hat{G}^a_K(v_1) \right).
\end{split}
\end{equation}
Using the formula for $\Phi^+_{n,K}(Kz)$ from (\ref{S52Phis}) and the identity 
\begin{equation}\label{S52Frac}
\frac{1}{(v_1 - z) (v_1 - z + K^{-1})} = \frac{1}{(v_1 - z)^2} + O(K^{-1})
\end{equation}
we conclude that 
\begin{equation}\label{S52A2V3}
\begin{split}
& A_2 =   \oint_{\Gamma}\frac{dz \theta (n/K) \Phi^+(zK/n, n/K) e^{\theta G(zK/k, k/K)}  }{2 \pi \i K (z- v_0)} \cdot M \left(\hat{G}^k_K(z), \hat{G}^a_K(v_1) \right) \\
& +   \oint_{\Gamma} \frac{dz  (n/K) \Phi^+(zK/n, n/K) e^{\theta G(zK/k, k/K)}  }{2 \pi \i K (z- v_0) (v_1-z)^2} + O(K^{-2}).
\end{split}
\end{equation}
We mention that in deriving (\ref{S52A2V3}) we used the moment bounds for $\hat{G}_K^j$ from (\ref{S5MB}).

We finally claim that 
\begin{equation}\label{S52A3V3}
\begin{split}
&A_3 =     \oint_{\Gamma} \frac{dz \theta (n/K) \Phi^+(zK/n, n/K)}{2\pi \i  K (z - v_0)}   \cdot   M \left( \hat{G}^n_K(z) - \hat{G}^{k}_K(z)  - \frac{\theta}{K} \cdot \sum_{j=k+1}^{n}  \partial_z \hat{G}^j_K(z) , \hat{G}_K^a(v_1) \right)  \\
& +  \frac{1}{K} \sum_{j = k+1}^a  \oint_{\Gamma}   \frac{dz  \theta (n/K) \Phi^+(zK/n, n/K)}{2\pi \i  (z - v_0) (v_1 - z)^2}  \\
& \times \left[ \partial_s G \left(\frac{zK}{j} , \frac{j}{K} \right) - \frac{z K^2}{j^2} \cdot \partial_z G\left( \frac{zK}{j}, \frac{j}{K} \right) - \frac{\theta K}{j} \cdot  \partial_z G\left( \frac{zK}{j}, \frac{j}{K}\right)  \right] +O(K^{-2}).
\end{split}
\end{equation}
We prove (\ref{S52A3V3}) in the next step. Here we assume its validity and conclude the proof of (\ref{S52Cov1}).\\

Substituting (\ref{S52A1V3}), (\ref{S52A2V3}) and (\ref{S52A3V3}) into (\ref{S52Linear}) and using the definition of $\U$ from (\ref{S52E1}) we conclude that 
\begin{equation}\label{S52U2nd}
\begin{split}
&\oint_{\Gamma}\frac{dz (-\theta) (n/K)   }{2 \pi \i K (z- v_0)}\cdot  M\left(\U(z) , \hat{G}^a_K(v_1) \right) = - \oint_{\Gamma} \frac{dz  (n/K) \Phi^+(zK/n, n/K)  e^{\theta G(zK/k, k/K)}  }{2 \pi \i K (z- v_0) (v_1-z)^2} \\
& +   \frac{1}{K} \sum_{j = k+1}^a  \oint_{\Gamma}   \frac{dz  \theta (n/K) \Phi^+(zK/n, n/K)}{2\pi \i  (z - v_0) (v_1 - z)^2}  \\
& \times \left[  \frac{z K^2}{j^2} \cdot \partial_z G\left( \frac{zK}{j}, \frac{j}{K} \right) + \frac{\theta K}{j} \cdot  \partial_z G\left( \frac{zK}{j}, \frac{j}{K}\right) -  \partial_s G \left(\frac{zK}{j} , \frac{j}{K} \right)  \right] +O(K^{-3/2}).
\end{split}
\end{equation}
Similarly to Step 2 above, we can evaluate the first integral as (minus) the residue at $z = v_0$ and as before there is no residue at infinity. Performing this evaluation in (\ref{S52U2nd}) and multiplying both sides by $\theta^{-1} \cdot (K/n) \cdot K$ brings us to (\ref{S52Cov1}) with $\gamma$ replaced with $\Gamma$. However, by Cauchy's theorem we can deform $\Gamma$ to $\gamma$ without crossing any poles of the integrand and hence without affecting the value of the integral. \\

{\bf \raggedleft Step 6.} In this final step we prove (\ref{S52A3V3}). Starting from (\ref{S52A3V2}) we see that by linearity we can put the sum over $j \in \llbracket k +1, n\rrbracket$ on the first line all the way inside of the cumulant, which makes the terms $[\hat{G}^j_K(z) - \hat{G}^{j-1}_K(z)]$ telescope to $\hat{G}^n_K(z) - \hat{G}^{k}_K(z)$. We can also do the same with the sum over $j \in \llbracket k+1, a \rrbracket$, which causes the $\mathbb{E}\left[\hat{G}^j_K(z) - \hat{G}^{j-1}_K(z) \right]$ terms to telescope to $\mathbb{E}\left[\hat{G}^a_K(z) - \hat{G}^{k}_K(z) \right]$. Finally, we note that the term $ \frac{\theta {\bf 1}\{ \theta \neq 1\}}{1 - \theta} $ on the second line of  (\ref{S52A3V2})  integrates to $0$ by Cauchy's theorem. Performing these simplifications we get
\begin{equation}\label{S52A3V4}
\begin{split}
&A_3 =       \oint_{\Gamma} \frac{dz\Phi_{n,K}^+(zK)}{2\pi \i K (z - v_0)}     \cdot   M \left(  \frac{\theta[\hat{G}^n_K(z) - \hat{G}^{k}_K(z)] }{K} - \sum_{j=k+1}^{n} \frac{\theta^2 \partial_z \hat{G}^j_K(z)}{K^2} , \hat{G}_K^a(v_1) \right) \\
&  \oint_{\Gamma} \frac{dz \theta \Phi_{n,K}^+(zK)}{2\pi \i K (z - v_0)}  \cdot \frac{1}{K(v_1 - z) (v_1 - z + K^{-1})} \cdot \sum_{j = k+1}^a  \Bigg{[} G\left( \frac{zK}{j}, \frac{j}{K} \right) - G\left( \frac{zK}{j-1}, \frac{j-1}{K} \right)  \\
&   -    \frac{(1 + \theta)}{2j}  \cdot \partial_z G\left( \frac{zK}{j}, \frac{j}{K}\right) + \frac{ (1- \theta) }{2(j-1)} \cdot \partial_z G\left( \frac{zK}{j-1}, \frac{j-1}{K}\right) \Bigg{]} \\
&+      \oint_{\Gamma} \frac{dz\Phi_{n,K}^+(zK)}{2\pi \i K (z - v_0)}  \cdot \frac{1}{K(v_1 - z) (v_1 - z + K^{-1})} \cdot   \Bigg{[}  \sum_{j=k+1}^{a}  \frac{A_j^{\theta}(z)}{K^2}  \\
&  + \frac{\theta }{K} \cdot \mathbb{E}\left[\hat{G}^a_K(z) - \hat{G}^{k}_K(z) \right] -\sum_{j=k+1}^{a}   \frac{\theta^2 }{K^2}\cdot \mathbb{E}\left[\partial_z  \hat{G}^j_K(z) \right] \Bigg{]}.
\end{split}
\end{equation}
Note that the last two lines of (\ref{S52A3V4}) are $O(K^{-2})$. We also have by the smoothness of $G$ that 
\begin{equation*}
\begin{split}
&G\left( \frac{zK}{j}, \frac{j}{K} \right) - G\left( \frac{zK}{j-1}, \frac{j-1}{K} \right)   -    \frac{(1 + \theta)}{2j}  \cdot \partial_z G\left( \frac{zK}{j}, \frac{j}{K}\right) + \frac{ (1- \theta) }{2(j-1)} \cdot \partial_z G\left( \frac{zK}{j-1}, \frac{j-1}{K}\right) \\
& =  \frac{1}{K} \cdot \partial_s G \left(\frac{zK}{j} , \frac{j}{K} \right) - \frac{z K}{j^2} \cdot \partial_z G\left( \frac{zK}{j}, \frac{j}{K} \right) - \frac{\theta}{j} \cdot  \partial_z G\left( \frac{zK}{j}, \frac{j}{K}\right) + O(K^{-2}).
\end{split}
\end{equation*}
Using the last few observations to simplify (\ref{S52A3V4}) and using (\ref{S52Frac}) and the formula for $\Phi^+_{n,K}$ from (\ref{S52Phis}) we get (\ref{S52A3V3}).
\end{proof}

%-------------------------------------------------------------------------------------------------------------------------------------------------------------------------------------------------
% Section 6
%
%-------------------------------------------------------------------------------------------------------------------------------------------------------------------------------------------------
\section{Third and higher order cumulants}\label{Section6} In (\ref{S52Cum1}) from Proposition \ref{S52P1} we showed that certain third and higher order joint cumulants vanish as $K \rightarrow \infty$. In this section we leverage this result to show that the third and higher order joint cumulants of $\hat{G}^n_K(z)$ from (\ref{S32E1}) also vanish as $K \rightarrow \infty$ -- the precise statement is in Proposition \ref{S6Prop} below. In Section \ref{Section6.1} we introduce some new notation that essentially replaces the $n$ in $\hat{G}^n_K(z)$ with a continuous variable, state the main result of the section, Proposition \ref{S6Prop}, and state two technical lemmas that will be required for its proof, which is given in Section \ref{Section6.2}. The lemmas from Section \ref{Section6.1} are proved in Section \ref{Section6.3}.

%-------------------------------------------------------------------------------------------------------------------------------------------------------------------------------------------------
% Section 6.1
%
%-------------------------------------------------------------------------------------------------------------------------------------------------------------------------------------------------
\subsection{Notation and technical preliminaries}\label{Section6.1} We begin this section by introducing a bit of notation. Fix $\theta > 0$ and assume that $(\ell^1, \ell^2, \dots)$ is distributed according to $\mathbb{P}^{\theta, K}_{\infty}$ as in Definition \ref{BKCC}. Fix $\delta \in (0,1/2)$, $\delta_0 = 2\delta$ and define
\begin{equation}\label{S6DefU}
U(\delta, \theta) = \mathbb{C} \setminus [-\theta \delta^{-1}, \delta^{-2}].
\end{equation}
For $s \in [\delta_0, \delta_0^{-1}]$, $K \geq \delta_0^{-1}$ and $z \in U(\delta, \theta)$ we define 
\begin{equation}\label{S6DefG}
\G_K(z,s) := G_K^n(z) - \mathbb{E} \left[G_K^n(z) \right] = \sum_{i = 1}^n \frac{1}{z - \ell_i^n/K} - \mathbb{E}\left[\sum_{i = 1}^n \frac{1}{z - \ell_i^n/K}  \right], \mbox{ with } n = \lceil s K \rceil,
\end{equation}
where we mention that the second equality follows from the definition of $G_K^n(z)$ in (\ref{S32E1}). Note that since $\ell_i^n \in [-\theta n, K-\theta]$ for $i \in \llbracket 1, n\rrbracket$, the functions $\G_K(z,s)$ are random analytic functions in $z \in U(\delta, \theta)$ as in (\ref{S6DefU}), provided that $K \geq \delta_0^{-1}$. \\

With the above notation in place we can formulate the main result of the section.
\begin{proposition}\label{S6Prop} Fix $\theta > 0$, and assume that $(\ell^1, \ell^2, \dots)$ is distributed according to $\mathbb{P}^{\theta, K}_{\infty}$ as in Definition \ref{BKCC}. Fix $\delta \in (0,1/2)$, $\delta_0 = 2\delta$ and $U(\delta, \theta)$ as in (\ref{S6DefU}). For any $m \geq 2$, $s_0, s_1, \dots, s_m \in [\delta_0, \delta_0^{-1}]$ and $K \geq \delta_0^{-1}$ we have
\begin{equation}\label{S6VanCum3}
M \left( \G_{K}(z_0,s_0), \G_K(z_1, s_1), \dots,  \G_K(z_m, s_m) \right) = o(1).
\end{equation}
In more details, there is a sequence $a_K$, converging to $0$, that depends on $\theta, \delta, m$ and a compact set $\mathcal{V} \subset U(\delta, \theta)$, such that for  $z_0, \dots, z_m \in \mathcal{V}$ and $K \geq \delta_0^{-1}$
$$| M \left( \G_{K}(z_0,s_0), \G_K(z_1, s_1), \dots,  \G_K(z_m, s_m) \right)| \leq a_K.$$
\end{proposition}
In the remainder of this section we state two technical results, Lemmas \ref{DomLip} and \ref{LUnique}, which we need for the proof of Proposition \ref{S6Prop}, which is given in Section \ref{Section6.2}. The two lemmas are proved in Section \ref{Section6.3}. 

\begin{lemma}\label{DomLip} Fix $\theta > 0, \delta \in (0,1/2)$ and $U(\delta, \theta)$ as in (\ref{S6DefU}). For each $n \in \mathbb{N}$ define 
$$F_n = \left\{z \in U(\delta, \theta): n \geq \inf_{x \in [-\theta \delta^{-1}, \delta^{-2} ]} |x-z| \geq n^{-1} \right\}.$$
Then, $\{F_n\}_{n = 1}^\infty$ forms a compact exhaustion of $U(\delta, \theta)$. In addition, for each $n \in \mathbb{N}$ there exists a constant $\lambda_n > 0$, depending on $\delta,\theta$ and $n$ such that for each $x,y \in F_n$ there exists a piecewise smooth curve $\gamma_{x,y}$, contained in $F_n$, that connects $x$ and $y$ and such that
\begin{equation}\label{LipDomE}
\|\gamma_{x,y} \| \leq \lambda_n |x-y|,
\end{equation}
where $\|\gamma_{x,y} \| $ is the length of $\gamma_{x,y}$.
\end{lemma}
\begin{remark}\label{DomLipRem} In the proof of Proposition \ref{S6Prop} we will deal with certain sequences of analytic functions on $U(\delta, \theta)$. We will seek to show that these functions have subsequential limits, for which some type of equicontinuity is required. An analytic function on $U(\delta, \theta)$ can have jump discontinuities across $[-\theta \delta^{-1}, \delta^{-2}]$ and so it is not feasible to prove equicontinuity over the whole domain. Lemma \ref{DomLip} is useful in that it allows us to establish equicontinuity over each $F_n$. 
\end{remark}

\begin{lemma}\label{LUnique} Fix $\theta > 0$, $\delta \in (0,1/2)$, $\delta_0 = 2 \delta$ and let $U(\delta, \theta)$ be as in (\ref{S6DefU}). Let $\delta_0 \leq p < q \leq \delta_0^{-1}$. Suppose that $f: U(\delta,\theta ) \times [p,q] \rightarrow \mathbb{C}$ is a continuous function, such that for each $s \in [p,q]$ we have that $f(z,s)$ is analytic in $z$ on $U(\delta,\theta )  $. Then, the function $\partial_z f(z,s)$ is continuous on $U(\delta,\theta )   \times [p,q]$. If we further suppose that for $z \in U(\delta,\theta ) $
\begin{equation}\label{S5LID0}
f(z,q) = 0,
\end{equation}
and for every $s \in [p,q]$ and $z \in U(\delta,\theta)$ we have
\begin{equation}\label{S5LID}
  \left[e^{\theta G\left(z/q, q\right)} - 1 \right] f(z,q) + \theta \cdot \int_{s}^{q}  \partial_{z} f(z,u)du -   \left[e^{\theta G\left(z/s, s\right)} - 1 \right] f(z,s)  = 0,
\end{equation}
where $G(z,s)$ is as in (\ref{S31E6}), then $f(z,s) = 0$ for all $(z,s) \in U(\delta,\theta)  \times [p,q]$.
\end{lemma}
\begin{remark}\label{LUniqueRem} In the proof of Proposition \ref{S6Prop} we will deal with certain sequences of functions on $U(\delta, \theta) \times [p,q]$ for some $\delta_0 \leq p < q \leq \delta_0^{-1}$, which we will seek to show are converging to zero. The way this will be established is by showing that any subsequential limit of these functions satisfies the conditions of Lemma \ref{LUnique}. We mention that the statement of Lemma \ref{LUnique} is probably not optimal in terms of parameters. We chose the present formulation to make the result directly applicable within the proof of Proposition \ref{S6Prop}.
\end{remark}

%-------------------------------------------------------------------------------------------------------------------------------------------------------------------------------------------------
% Section 6.2
%
%-------------------------------------------------------------------------------------------------------------------------------------------------------------------------------------------------
\subsection{Proof of Proposition \ref{S6Prop}}\label{Section6.2} In this section we prove Proposition \ref{S6Prop}.  We continue with the same notation as in Section \ref{Section6.1} and introduce some additional notation that will be used.\\

When computing joint cumulants of certain variables, we will sometimes replace them with other variables that differ from them by a deterministic quantity. Since second and higher order cumulants remain unchanged upon deterministic shifts, see (\ref{S3Linearity}), this will not affect the value of the cumulant. To ease this replacement we will denote by $\Df$ a generic deterministic quantity, whose meaning may be different from line to line or even within the same equation. For example, from (\ref{S32E1}) and (\ref{S6DefG}) 
\begin{equation}\label{S6DefG2}
\G_K(z,s) = G_K^{\lceil s K \rceil} (z) + \Df = \hat{G}_K^{\lceil s K \rceil} (z) + \Df.
\end{equation}
Notice that if $s \in [\delta_0, \delta_0^{-1}]$ and $K \geq \delta_0^{-1}$, we have that $n = \lceil s K \rceil \in [\delta K, \delta^{-1} K ]$. The latter observation, equation (\ref{S6DefG}), the fact that $\G_K(z,s)$ has zero mean and (\ref{MBRem}) together imply
\begin{equation}\label{S6MB}
\begin{split}
&\mathbb{E} \left[ \left| \partial_z^i \G_K(z,s)  \right|^k \right] = O(1) \mbox{ for $i = 0,1$ and $k \in \mathbb{N}$, }
\end{split}
\end{equation}
where the constant in the big $O$ notation depends on $\theta, \delta, k$ and a compact set $\mathcal{V} \subset U(\delta, \theta) $ and (\ref{S6MB}) holds whenever $z \in \mathcal{V}$, $s \in [\delta_0, \delta_0^{-1}]$ and $K \geq \delta_0^{-1}$. \\

Fix $s,t \in \mathbb{R}$ such that $\delta_0^{-1} \geq s \geq t \geq \delta_0$,  $z \in U(\delta, \theta)$ and define for $K \geq \delta_0^{-1}$
\begin{equation}\label{S6DefInt}
\begin{split}
&\sfU_{K}(z;s,t) = \left[\Phi^-\left(z/s, s \right)e^{-\theta G(z/s,s)} - \Phi^+(z/s, s)\right] \cdot \G_K(z,s) \\
& + \theta \cdot \Phi^+(z/s,s) \int_{t}^s \partial_z \G_K(z,u) du - \Phi^+ \left( z/s,s \right) \cdot \left[ e^{\theta G(z/t,t)} - 1 \right] \cdot \G_K(z,t),
\end{split}
\end{equation}
where we recall from (\ref{S31E10}) and (\ref{S31E16}) the functions 
\begin{equation}\label{QRecalled}
\begin{split}
&\Phi^-(z/s,s) = z/s + \theta \mbox{, } \Phi^+(z/s,s) = (1/s) \cdot (1 - z) , \mbox{ and } Q(z/s,s)   \\
&= \Phi^-\left(z/s, s \right)e^{-\theta G(z/s,s)} -   \Phi^+\left(z/s, s \right)e^{\theta G(z/s,s)}=   (2/s) \sqrt{(z - sz_-(s))(z- sz_+(s))}.
\end{split}
\end{equation}
We also recall that $G(z,s)$ is as in (\ref{S31E6}) and $z_{\pm}(s)$ as in (\ref{S31E12}). 
\begin{remark}\label{Divide} In the arguments below we will frequently seek to divide by $Q(z/s,s)$, or $\Phi^{\pm}(z/s,s)$, or $e^{\theta G(z/s,s)} - 1$, where $z$ varies in some compact subset of $U(\delta, \theta)$. In view of (\ref{S31E17}) and the sentence that follows it, and (\ref{QRecalled}) these functions are uniformly bounded away from zero and infinity over compact sets of $U(\delta, \theta)$, provided that $s \in [\delta_0, \delta_0^{-1}]$, which will always be assumed. In particular, there is no issue with this division.
\end{remark}

Fix $m$ as in the statement of the proposition and define for $r \in \llbracket 1, m+1\rrbracket$ the sets
\begin{equation}\label{S6DefV}
\begin{split}
 V_r = &\{(s_0, \dots, s_m, t_0) \in [\delta_0, \delta_0^{-1}]^{m+2}: s_0 \geq t_0, s_1, \dots, s_{m-r+1} \in [t_0, s_0] \}.
\end{split}
\end{equation}
We claim that for any $r \in \llbracket 1, m +1 \rrbracket$, $(s_0, \dots, s_m, t_0) \in V_r$ and $K \geq \delta_0^{-1}$ we have 
\begin{equation}\label{TR1}
M \left( \sfU_{K}(z_0;s_0,t_0), \G_K(z_1, s_1), \dots,  \G_K(z_m, s_m) \right) = o(1),
\end{equation}
where as before the $K$-indexed sequence in the $o(1)$ notation depends on $\theta, \delta, m$ and a compact set $\mathcal{V} \subset U(\delta, \theta)$ and (\ref{TR1}) holds whenever $z_0, \dots, z_m \in \mathcal{V}$ and $(s_0, \dots, s_m, t_0) \in V_r$.

Notice that for any $s_0, \dots, s_m \in [\delta_0, \delta_0^{-1}]$ with $s_0 \geq s_i$ for $i \in \llbracket 1, m \rrbracket$ we have that $(s_0, \dots, s_m, s_0) \in V_{m+1}$. Applying (\ref{TR1}) for this choice of parameters and using that 
\begin{equation}\label{UTop}
\sfU_{K}(z_0;s_0,s_0) = Q(z_0/s_0,s_0) \cdot \G_K(z_0, s_0), 
\end{equation} 
which follows from (\ref{S6DefInt}) and (\ref{QRecalled}), we conclude that 
\begin{equation}\label{TR0}
Q(z_0/s_0,s_0) \cdot M \left( \G_K(z_0, s_0), \G_K(z_1, s_1), \dots,  \G_K(z_m, s_m) \right) = o(1).
\end{equation}
In view of Remark \ref{Divide} we can divide both sides of (\ref{TR0}) by $Q(z_0/s_0,s_0)$, which gives (\ref{S6VanCum3}), provided that $s_0 \geq s_i$ for $i \in \llbracket 1, m \rrbracket$. However, since joint cumulants are symmetric in their arguments we conclude (\ref{S6VanCum3}) for all $s_0, \dots, s_m \in [\delta_0, \delta_0^{-1}]$.\\

The argument in the previous paragraph reduces the proof of the proposition to establishing (\ref{TR1}), which we do by induction on $r$. We mention that the form of (\ref{TR1}) is somewhat important for carrying out the induction argument. Specifically, the insertion of $\sfU_{K}$ into the joint cumulant is what will allow us to apply Proposition \ref{S52P1}, while the definition of the sets $V_r$ will make sure we have enough information from the induction hypothesis to complete the induction step. The difficult part of the proof is completing the induction step, while the base case is an almost immediate consequence of Proposition \ref{S52P1} as we explain next. 

Suppose $r= 1$, i.e. $s_1, \dots, s_m \in [t_0, s_0]$. From (\ref{S52Cum1}) in Proposition \ref{S52P1} we know for $K \geq \delta_0^{-1}$  
\begin{equation}\label{TR2}
M \left( \sfU_{K}(z_0;s_0,t_0), \G_K(z_1, s_1), \dots,  \G_K(z_m, s_m) \right) = O(K^{-1/2}),
\end{equation}
where the constant in the big $O$ notation depends on $\theta, \delta, m$ and a compact set $\mathcal{V} \subset U(\delta, \theta)$ and (\ref{TR2}) holds whenever $z_0, \dots, z_m \in \mathcal{V}$. We mention that in deriving (\ref{TR2}) we used that second and higher order cumulants remain unchanged upon deterministic shifts, see (\ref{S3Linearity}), and (\ref{S6DefG2}). In addition, we used that if $S = \lceil s_0 K \rceil$, $T = \lceil t_0 K \rceil$ we have 
\begin{equation}\label{TR3}
\sfU_{K}(z_0;s_0,t_0) = U_{K}^{S,T}(z_0) + \Df + K^{-1} \cdot \zeta(z_0),
\end{equation}
where $S = \lceil s_0 K \rceil$, $T = \lceil t_0 K \rceil$ and $\zeta(z)$ is a random analytic function on $U(\delta, \theta)$ that has $O(1)$ moments as $z$ varies over compact subsets of $U(\delta, \theta)$. To see why (\ref{TR3}) holds note that 
\begin{equation}\label{TR3.5}
\begin{split}
&\Phi^{\pm}(z/s_0,s_0) = \Phi^{\pm}\left( \frac{zK}{S}, \frac{S}{K} \right) + O(K^{-1}), \hspace{2mm} e^{\pm \theta G(zK/S, S/K)} = e^{\pm \theta G(z/s_0, s_0)} + O(K^{-1}),\\
&e^{\pm \theta G(zK/T, T/K)} = e^{\pm \theta G(z/t_0, t_0)} + O(K^{-1}),
\end{split}
\end{equation}
where the latter set of identities follows from (\ref{QRecalled}) and the smoothness of $G(z/s,s)$. In particular, the last set of identities and the moment bounds from (\ref{S6MB}) give (\ref{TR3}) once we compare the definitions of $U_{K}^{S,T}(z)$ from (\ref{S52E1}) and $\sfU_{K}(z; s_0, t_0)$ from (\ref{S6DefInt}). 

Equation (\ref{TR2}) in particular implies (\ref{TR1}) when $(s_0, \dots, s_m, t_0) \in V_1$, which completes the base case of our induction. In the remainder of the proof we assume that $m+1 \geq r \geq 2$, that (\ref{TR1}) holds for $(s_0, \dots, s_m, t_0) \in V_{r-1}$, and proceed to show that it holds for $(s_0, \dots, s_m, t_0) \in V_{r}$.\\

The rest of the proof is split into six steps and before we go into the details we give an overview of the argument. In Step 1 we introduce a certain random analytic in $z$ function, called $\sfV_K(z; s,t)$, which is the same as (\ref{S5LID}) in Lemma \ref{S5LID} with $f(z,s)$ replaced with $\G_K(z,s)$. In the same step we prove (\ref{TR2}) when we put $\sfV_K(z_{m-r+2}; s_0,s_{m-r+2})$ in place of $\G_K(z_{m-r+2}, s_{m-r+2})$, this is (\ref{PP2}), and deduce an important consequence of this in (\ref{PP22}). We will come back to equations (\ref{PP2}) and (\ref{PP22}) shortly.

In Step 2 we pass to some sequence $K_n, s_i^n, z_i^n$, $t_0^n$ along which the absolute value of the left side of (\ref{TR1}) converges to its largest subsequential limit, and reduce the problem to showing that this limit is zero. We also prove that this limit is zero in a simple corner case we want to rule out from later steps. As it turns out, we lose too much information by passing to a subsequence $s_{m-r+2}^n$ and $z_{m-r+2}^n$ and instead need to keep these parameters free. We do this by introducing a sequence of functions, $f_n(z,s)$, which are the same as the left side of (\ref{TR1}) where all parameters still follow the sequences $K_n, s_i^n, z_i^n$, $t_0^n$ but $s_{m-r+2} = s$ and $z_{m-r+2} = z$ are free. The functions $f_n(z,s)$ are defined in Step 3 and the problem reduces to showing that they converge to zero.

Ideally, we would like to argue that $f_n(z,s)$ converge to zero directly; however, this turns out to be difficult with the information from the induction hypothesis. The idea is then to show that these functions have {\em subsequential limits} and prove that any subsequential limit satisfies enough conditions that force it to be zero. The existence of subsequential limits is pointwise easy to show; however, one needs them to be continuous. For this some equicontinuity is required, and in fact we can prove that $f_n(z,s)$ satisfy a type of Lipschitz condition which becomes an honest Lipschitz condition in the limit. This is done in Step 4 by using (\ref{PP22}) from Step 1, which we mentioned earlier, and Lemma \ref{DomLip}. In Step 5 we show that the identity satisfied by $\sfV_K$ in (\ref{PP2}) from Step 1 in the limit ensures that all subsequential limits of $f_n(z,s)$ satisfy the integral identity in Lemma \ref{LUnique}. We also have from (\ref{PP22}) that any subsequential limit has the zero boundary condition in Lemma \ref{LUnique}, which allows us to conclude that the limit is zero. In Step 6 we briefly wrap up the proof by combining a few statements from Steps 3 and 5.\\

{\bf \raggedleft Step 1.} For $\delta_0^{-1} \geq s \geq t \geq \delta_0$ and $z \in U(\delta, \theta)$ we define 
\begin{equation}\label{PP1}
\begin{split} 
& \sfV_K(z; s,t) = \left[ e^{\theta G(z/s,s)} - 1 \right] \cdot \G_K(z,s) + \theta \int_{t}^s \partial_z \G_K(z,u) du - \left[ e^{\theta G(z/t,t)} - 1 \right] \cdot \G_K(z,t).
\end{split}
\end{equation}
The goal of this step is to prove that for $(s_0, \dots, s_m, t_0) \in V_{r}$ 
\begin{equation}\label{PP2}
\begin{split}
&M (\sfU_{K}(z_0;s_0,t_0), \G_K(z_1, s_1), \dots,  \G_K(z_{m-r+1}, s_{m-r+1}), \sfV_K(z_{m-r+2}; s_0, s_{m-r+2}), \\
& \G_K(z_{m-r+3}, s_{m-r+3}) ,\dots,  \G_K(z_{m}, s_{m}) ) = o(1). 
\end{split}
\end{equation}
Note that $(s_0, s_1, \dots, s_{m-r+1}, s_0, s_{m-r+3}, \dots, s_m, t_0) \in V_{r-1}$ and so by the induction hypothesis we know that (\ref{TR1}) holds for this choice of parameters. Combining the latter with (\ref{PP2}) and linearity of the cumulants, we conclude for $(s_0, \dots, s_m, t_0) \in V_{r}$ 
\begin{equation}\label{PP22}
\begin{split}
&\left[ e^{\theta G(z_{m-r+2}/s_{m-r+2},s_{m-r+2})} - 1 \right]  M (\sfU_{K}(z_0;s_0,t_0), \G_K(z_1, s_1), \dots,  \G_K(z_{m}, s_{m}) ) \\
&= \theta \cdot \int_{s_{m-r+2}}^{s_0} du  M (\sfU_{K}(z_0;s_0,t_0), \G_K(z_1, s_1), \dots,  \G_K(z_{m-r+1}, s_{m-r+1}),\\ &\partial_{z} \G_{K}(z_{m-r+2}, u), \G_K(z_{m-r+3}, s_{m-r+3}) ,\dots,  \G_K(z_{m}, s_{m}) )  +o(1). 
\end{split}
\end{equation}
We will use (\ref{PP2}) and (\ref{PP22}) in the steps below. In the remainder of this step we prove (\ref{PP2}). \\

We first consider the case when $s_{m-r+2} \geq t_0$, which implies that $(s_0, \dots, s_m, t_0) \in V_{r-1}$. By the induction hypothesis we know that (\ref{TR1}) holds on $V_{r-1}$. If we have $u_{m-r+2} \in U(\delta, \theta)$, we can find a positively oriented circle $C_{m-r+2}$ around $u_{m-r+2}$ that is contained in $U(\delta, \theta)$. By Cauchy's integral formula, see \cite[Chapter 2, Corollary 4.2]{SS}, the linearity of the cumulants, see (\ref{S3Linearity}), and (\ref{TR1}) we conclude
\begin{equation}\label{PP3}
\begin{split}
&M (\sfU_{K}(z_0;s_0,t_0), \G_K(z_1, s_1), \dots,  \G_K(z_{m-r+1}, s_{m-r+1}), \partial_z \G_K(u_{m-r+2}, s_{m-r+2}), \\
& \G_K(z_{m-r+3}, s_{m-r+3}) ,\dots,  \G_K(z_{m}, s_{m}) ) = \frac{1}{2\pi \i} \oint_{C_{m-r+2}} \frac{dz_{m-r+2}}{(z_{m-r+2} - u_{m - r +2})^2}  \\
&\times M \left( \sfU_{K}(z_0;s_0,t_0), \G_K(z_1, s_1), \dots,  \G_K(z_m, s_m) \right)  = o(1).
\end{split}
\end{equation}
Combining (\ref{TR1}) and (\ref{PP3}), the linearity of the cumulants and the definition of $\sfV_K(z;s,t)$ in (\ref{PP1}) we conclude (\ref{PP2}).\\

We next suppose that $s_{m-r+2} \in [\delta_0, t_0]$, which implies for all $s \in [t_0, s_0]$ that
$$(s_0, s_1, \dots, s_{m-r+1}, s, s_{m-r+3}, \dots, s_m, s_{m-r+2}) \in V_{r-1}.$$
By the induction hypothesis we know that (\ref{TR1}) holds on $V_{r-1}$, which implies for $s \in [t_0, s_0]$
\begin{equation}\label{PP4}
\begin{split}
&M (\sfU_{K}(z_{m-r+2};s_0,s_{m-r+2} ), \G_K(z_1, s_1), \dots,  \G_K(z_{m-r+1}, s_{m-r+1}), \G_K(z_{0}, s), \\
& \G_K(z_{m-r+3}, s_{m-r+3}) ,\dots,  \G_K(z_{m}, s_{m}) ) = o(1). 
\end{split}
\end{equation}
As before, we can apply Cauchy's integral formula and the linearity of the cumulants to conclude from (\ref{PP4}) that 
\begin{equation}\label{PP5}
\begin{split}
&M (\sfU_{K}(z_{m-r+2};s_0,s_{m-r+2} ), \G_K(z_1, s_1), \dots,  \G_K(z_{m-r+1}, s_{m-r+1}), \partial_z \G_K(z_{0}, s), \\
& \G_K(z_{m-r+3}, s_{m-r+3}) ,\dots,  \G_K(z_{m}, s_{m}) ) = o(1). 
\end{split}
\end{equation}
Using (\ref{PP4}), (\ref{PP5}) and the linearity of the cumulants we conclude 
\begin{equation}\label{PP6}
\begin{split}
&M (\sfU_{K}(z_{m-r+2};s_0,s_{m-r+2} ), \G_K(z_1, s_1), \dots,  \G_K(z_{m-r+1}, s_{m-r+1}),  \sfU_K(z_{0}; s_0, t_0), \\
& \G_K(z_{m-r+3}, s_{m-r+3}) ,\dots,  \G_K(z_{m}, s_{m}) ) = o(1). 
\end{split}
\end{equation}
In addition, from (\ref{TR1}) applied to $(s_0, s_1, \dots, s_{m-r+1}, s_0, s_{m-r+3}, \dots, s_m, t_0)$, linearity of cumulants, and the fact that cumulants are symmetric in their arguments we have
\begin{equation}\label{PP7}
\begin{split}
&\left[ \Phi^-(z_{m-r+2}/s_0, s_0) e^{-\theta G(z_{m-r+2}/s_0, s_0)} - \Phi^+(z_{m-r+2}/s_0, s_0) e^{\theta G(z_{m-r+2}/s_0, s_0)}  \right]\\
&\times  M (\G_K(z_{m-r+2}, s_0), \G_K(z_1, s_1), \dots,  \G_K(z_{m-r+1}, s_{m-r+1}),  \sfU_K(z_{0}; s_0, t_0), \\
& \G_K(z_{m-r+3}, s_{m-r+3}) ,\dots,  \G_K(z_{m}, s_{m}) ) = o(1). 
\end{split}
\end{equation}
Subtracting (\ref{PP7}) from (\ref{PP6}) and using the definition of $\sfV_K$ from (\ref{PP1}) we conclude
\begin{equation}\label{PP8}
\begin{split}
&\Phi^+(z_{m-r+2}/s_0, s_0) \cdot M (\sfV_{K}(z_{m-r+2};s_0,s_{m-r+2} ), \G_K(z_1, s_1), \dots,  \G_K(z_{m-r+1}, s_{m-r+1}),   \\
& \sfU_K(z_{0}; s_0, t_0), \G_K(z_{m-r+3}, s_{m-r+3}) ,\dots,  \G_K(z_{m}, s_{m}) ) = o(1). 
\end{split}
\end{equation}
Equation (\ref{PP8}) implies (\ref{PP3}), since cumulants are symmetric in their arguments and we can divide by $\Phi^+(z_{m-r+2}/s_0, s_0)$ in view of Remark \ref{Divide}. \\

{\bf \raggedleft Step 2.} Let $\{F_v\}_{v \geq 1}$ be as in Lemma \ref{DomLip}. For each $v \in \mathbb{N}$ we denote
\begin{equation}\label{PR1}
A_v = \limsup_{K \rightarrow \infty} \sup_{z_0, \dots, z_m \in F_v} \sup_{(s_0, \dots, s_m, t_0) \in V_r} \left| M \left( \sfU_{K}(z_0;s_0,t_0), \G_K(z_1, s_1), \dots,  \G_K(z_m, s_m) \right) \right|.
\end{equation}
Since $\{F_v\}_{v \geq 1}$ form a compact exhaustion of $U(\delta, \theta)$, we see that to conclude (\ref{TR1}) holds on $V_r$ it suffices to show that $A_v = 0$ for each $v \in \mathbb{N}$. In the remainder we fix $v \in \mathbb{N}$ and show that $A_v = 0$.\\

Let $\{K_n \}_{n \geq 1}$ be a strictly increasing sequence with $K_1 \geq \delta_0^{-1}$, and $(s_0^n, \dots, s_m^n, t_0^n) \in V_r$, $z_0^n, \dots, z_m^n \in F_v$ be sequences such that
\begin{equation}\label{PR2}
\lim_{ n\rightarrow \infty} \left| M \left( \sfU_{K_n}(z^n_0;s^n_0,t^n_0), \G_{K_n}(z^n_1, s^n_1), \dots,  \G_{K_n}(z^n_m, s^n_m) \right)  \right| = A_v.
\end{equation}
By possibly passing to a subsequence, which we continue to index by $n$, we may assume that $\lim_n t_0^n$, $\lim_n s_i^n$, and $\lim z_i^n$ all exist for $i \in \llbracket 0, m\rrbracket$. We denote
\begin{equation}\label{PR3}
t_0^{\infty} = \lim_n t_0^n, \hspace{2mm} s_i^{\infty} = \lim_n s_i^n, \hspace{2mm} z_i^{\infty} = \lim_n z_i^n \mbox{ for } i \in \llbracket 0, m \rrbracket.
\end{equation}

In the remainder of this step we show that $A_v = 0$ if $s_{m-r+2}^{\infty} = s_0^{\infty}$. For a bounded complex random variable $\xi$ we define
\begin{equation}\label{PR4}
\begin{split}
M_n(\xi) = M (& \sfU_{K_n}(z^n_0;s^n_0,t^n_0), \G_{K_n}(z^n_1, s^n_1), \dots,  \G_{K_n}(z^n_{m-r+1}, s^n_{m-r+1}), \xi, \\
& \G_{K_n}(z^n_{m-r+3}, s^n_{m-r+3}), \dots, \G_{K_n}(z^n_{m}, s^n_{m})).
\end{split}
\end{equation}
From (\ref{PP22}) and the moment bounds from (\ref{S6MB}) we have
\begin{equation*}
\begin{split}
\left[ e^{\theta G(z^n_{m-r+2}/s^n_{m-r+2},s^n_{m-r+2})} - 1 \right]  \cdot M_n( \G_{K_n}(z^n_{m-r+2}, s^n_{m-r+2}) ) = o(1) + O(|s^n_{m-r+2} - s^n_0|).
\end{split}
\end{equation*}
We may now divide both sides by $\left[ e^{\theta G(z^n_{m-r+2}/s^n_{m-r+2},s^n_{m-r+2})} - 1 \right]$, see Remark \ref{Divide}, and let $n \rightarrow \infty$, which gives in view of (\ref{PR2}) and $s_{m-r+2}^{\infty} = s_0^{\infty}$ that $A_v = 0$.\\

{\bf \raggedleft Step 3.} In the remaining steps we prove that $A_v = 0$ under the additional assumption that $s^{\infty}_{m- r+2} \in [\delta_0, s_0^{\infty})$, which in particular implies $s_0^{\infty} > \delta_0$. For $s \in [\delta_0, \delta_0^{-1}]$ and $z \in U(\delta, \theta)$ we define the functions
\begin{equation}\label{PS1}
f_n(z,s) =  M_n(\G_{K_n}(z, s)).
\end{equation}
Note that since $\G_K(z,s)$ is analytic in $z \in U(\delta, \theta)$ for each $s \in [\delta_0, \delta_0^{-1}]$, the same is true for $f_n(z,s)$.

We claim that 
\begin{equation}\label{PS2}
\lim_{n \rightarrow \infty} \sup_{z \in F_v} \sup_{s \in [\delta_0, s_0^n]} |f_n(z,s)| = 0,
\end{equation}
which in view of (\ref{PR2}), (\ref{PR4}) and the definition of $f_n(z,s)$ in (\ref{PS1}) implies $A_v = 0$. We have thus reduced our proof to establishing (\ref{PS2}). \\

Let $\{n_h\}_{h \geq 1}$ be a strictly increasing sequence and $\zeta_h \in F_v$, $\sigma_h \in [\delta_0, s^{n_h}_0]$ be such that
\begin{equation}\label{PT1}
\lim_{h \rightarrow \infty} |f_{n_h}(\zeta_h, \sigma_h)| = \limsup_{n \rightarrow \infty} \sup_{z \in F_v} \sup_{s \in [\delta_0, s^{n}_0]} |f_n(z,s)|.
\end{equation}
Let $\mathsf{Z}$ be a countable dense subset of $U(\delta, \theta)$ and $\mathsf{S}$ be a countable dense subset of $[\delta_0, \delta_0^{-1}]$. From (\ref{S6MB}) we know that there exists $R(u) > 0$, depending on $\theta, \delta, u$, such that for $s \in [\delta_0, \delta_0^{-1}]$ and $z \in F_u$
\begin{equation}\label{PT2}
\left|f_{n}(z,s) \right| \leq R(u) \mbox{ and } \left|\partial_z f_{n}(z,s) \right| \leq R(u).
\end{equation}
In particular, we see that $f_{n_h}(z,s)$ is a bounded sequence for each $(z,s) \in \mathsf{Z} \times \mathsf{S}$. By a diagonalization argument we may pass to a subsequence of $n_h$, which we continue to call $n_h$, such that $\lim_h f_{n_h}(z,s)$ exists for each $(z,s) \in \mathsf{Z} \times \mathsf{S}$. We denote this limit by $f_{\infty}(z,s)$. \\

{\bf \raggedleft Step 4.} In this step we prove that $f_{\infty}(z,s)$ has a continuous extension to $F_u \times [\delta_0, s_0^{\infty}]$ for each $u \in \mathbb{Z}_{\geq 2}$, which would imply that $f_{\infty}(z,s)$ has a continuous extension to $U(\delta, \theta) \times [\delta_0, s_0^{\infty}]$, which we continue to call $f_{\infty}(z,s)$. We claim that there exists $D(u) > 0$, depending on $\theta, \delta, u$, such that for $(x,s), (y,t) \in (\mathsf{Z} \cap F_u) \times (\mathsf{S} \cap [\delta_0, s^{\infty}_0))$ we have 
\begin{equation}\label{PT3}
|f_{\infty}(x,s) - f_{\infty}(y,t)| \leq D(u) \cdot \left( |s-t| + |x-y| \right).
\end{equation}
If true, (\ref{PT3}) would imply that $f_{\infty}(z,s)$ is uniformly continuous on $(\mathsf{Z} \cap F_u) \times (\mathsf{S} \cap [\delta_0, s^{\infty}_0))$ and hence has a unique continuous extension to $F_u \times [\delta_0, s^{\infty}_0]$, which also satisfies (\ref{PT3}) for $(x,s), (y,t) \in F_u \times [\delta_0, s^{\infty}_0]$. Here, we implicitly used that $\mathsf{Z} \cap F_u$ is dense in $F_u$, since $u \geq 2$, and that $ \mathsf{S} \cap [\delta_0, s^{\infty}_0)$ is dense in $[\delta_0, s^{\infty}_0]$, since $s^{\infty}_0 > \delta_0$ (see the begining of Step 3).\\

From Lemma \ref{DomLip} there is $\lambda_u > 0$ and for each $x,y \in F_u$ a piecewise smooth contour $\gamma_{x,y}$ connecting $x,y$ such that $\gamma_{x,y} \subset F_u$ and $\| \gamma_{x,y}\| \leq \lambda_v |x-y|$. In view of (\ref{PT2}) we can find a constant $R_z(u)$, depending on $\theta, \delta, u$, such that for $t \in [\delta_0, s_0^{n_h}]$ and $x,y \in F_u$
\begin{equation}\label{PT4}
\left| f_{n_h}(x,t) - f_{n_h}(y,t) \right| = \left| \int_{\gamma_{x,y}} \partial_z f_{n_h}(z,t) dz \right| \leq R_z(u) \cdot |x- y|.
\end{equation}
In addition, from (\ref{PP22}) we can find a constant $R_s(u)$, depending on $\theta, \delta, u$, such that for $z\in F_u$ and $s,t \in [\delta_0, s_0^{n_h}]$
\begin{equation}\label{PT5}
\begin{split}
\left| \left[e^{\theta G(z/s, s)} - 1 \right] f_{n_h}(z,s) - \left[e^{\theta G(z/t, t)} - 1 \right] f_{n_h}(z,t) \right|  \leq o(1) + R_s(u) |s - t|.
\end{split}
\end{equation}
Recall from Section \ref{Section3.1} that $G(z/s,s)$ is smooth for $(z,s) \in U(\theta, \delta) \times [\delta_0, \delta_0^{-1}]$. We also recall from Remark \ref{Divide} that $e^{\theta G(z/s,s )} - 1 \neq 0$ on $U(\theta, \delta) \times [\delta_0, \delta_0^{-1}]$. In particular, we can find $c(u), C(u) > 0$, depending on $\theta, \delta, u$, such that for $x \in F_u$ and $s \in [\delta_0, \delta_0^{-1}]$
\begin{equation}\label{PT6}
\left| e^{\theta G(x/s,s)} -1 \right| \geq c(u) \mbox{ and } \left| \partial_s \left( e^{\theta G(z/s,s)} \right) \right| \leq C(u).
\end{equation}
Combining (\ref{PT5}) and (\ref{PT6}) we conclude that for any sequences $a_h, b_h \in [\delta_0, s_0^{n_h}]$
\begin{equation}\label{PT7}
\limsup_{h \rightarrow \infty} \sup_{z \in F_u} \left| f_{n_h}(z,a_h) - f_{n_h}(z,b_h) \right| \leq \frac{C(u) R(u) + R_s(u)}{c(u)} \cdot \limsup_{ h \rightarrow \infty} |a_h - b_h|.
\end{equation}
Taking the $h \rightarrow \infty$ limit in (\ref{PT4}), using (\ref{PT7}) and applying the triangle inequality we obtain (\ref{PT3}). \\

{\bf \raggedleft Step 5.} In this step we prove that $f_{\infty}(z,s) = 0$ for all $(z,s) \in U(\theta, \delta) \times [\delta_0, s_0^{\infty}]$. Using that (\ref{PT3}) is satisfied for all $(x,s), (y,t) \in F_u \times [\delta_0, s^{\infty}_0]$, (\ref{PT4}), (\ref{PT7}) and the pointwise convergence of $f_{n_h}$ to $f_{\infty}$ on $(\mathsf{Z} \cap F_u) \times (\mathsf{S} \cap [\delta_0, s^{\infty}_0))$ (which is dense in $F_u \times [\delta_0, s_0^{\infty}]$) we conclude for any sequence $a_h \in [\delta_0, s^{n_h}_0]$ such that $a_h \rightarrow a \in [\delta_0, s_0^{\infty}]$ that
\begin{equation}\label{PU1}
\limsup_{h \rightarrow \infty} \sup_{z \in F_u} \left| f_{n_h}(z, a_h) - f_{\infty}(z, a) \right| = 0.
\end{equation}
From (\ref{PU1}) we have for each $u \in \mathbb{Z}_{\geq 2}$ and $s \in [\delta_0, s_0^{\infty}]$ that $f_{\infty}(\cdot, s)$ is the uniform over $F_u$ limit of analytic functions in $U(\theta, \delta)$. This implies that for each $s \in [\delta_0, s_0^{\infty}]$ the function $f_{\infty}(z,s)$ is analytic in $U(\theta, \delta)$, see \cite[Chapter 2, Theorem 5.2]{SS}. In addition, (\ref{PU1}) being true for all $u \in \mathbb{Z}_{\geq 2}$ implies
\begin{equation}\label{PU2}
\limsup_{h \rightarrow \infty} \sup_{z \in F_u} \left| \partial_z f_{n_h}(z, a_h) - \partial_z f_{\infty}(z, a) \right| = 0,
\end{equation}
in view of \cite[Chapter 2, Theorem 5.3]{SS}. 

From (\ref{PP2}) we have for $z \in F_u$ and any sequence $a_h \in [\delta_0, s^{n_h}_0]$ such that $a_h \rightarrow a \in [\delta_0, s_0^{\infty}]$
\begin{equation}\label{PU3}
\begin{split}
&\lim_{ h \rightarrow \infty} \left[ e^{\theta G(z/s^{n_h}_0,s^{n_h}_0)} - 1 \right] \cdot f_{n_h}(z,s^{n_h}_0) + \theta \int_{a_h}^{s^{n_h}_0} \partial_z f_{n_h}(z,u) du \\
& - \left[ e^{\theta G(z/a_h,a_h)} - 1 \right] \cdot f_{n_h}(z, a_h) = 0.
\end{split}
\end{equation} 
Using the continuity of $e^{\theta G(z/s,s)}$, the convergence in (\ref{PU1}), (\ref{PU2}) and of $s^n_0$ from (\ref{PR3}), we can apply the bounded convergence theorem to (\ref{PU3}) and obtain for all $z \in U(\theta, \delta)$ and $a \in [\delta_0, s_0^{\infty}]$
\begin{equation}\label{PU4}
\begin{split}
& \left[ e^{\theta G(z/s^{\infty}_0,s^{\infty}_0)} - 1 \right] \cdot f_{\infty}(z,s^{\infty}_0) + \theta \int_{a}^{s^{\infty}_0} \partial_z f_{\infty}(z,u) du  - \left[ e^{\theta G(z/a,a)} - 1 \right] \cdot f_{\infty}(z, a) = 0.
\end{split}
\end{equation} 
On the other hand, we have from (\ref{PP22}) (with $s_{m-r+2} = s_{0}$) and (\ref{PU1}) for $z \in F_u$ that 
$$ 0 =\lim_{ h \rightarrow \infty} \left[ e^{\theta G(z/s^{n_h}_0,s^{n_h}_0)} - 1 \right] \cdot f_{n_h}(z, s^{n_h}_0) = \left[ e^{\theta G(z/s^{\infty}_0,s^{\infty}_0)} - 1 \right] \cdot f_{\infty}(z, s^{\infty}_0).$$
The latter implies that for $z \in U(\theta, \delta)$ that
\begin{equation}\label{PU5}
\begin{split}
&f_{\infty}(z, s^{\infty}_0) = 0.
\end{split}
\end{equation} 
Our work in this step shows that $f_{\infty}(z,s)$ satisfies the conditions of Lemma \ref{LUnique}, from which we conclude that $f_{\infty}(z,s) = 0$ for all  $(z,s) \in U(\theta, \delta) \times [\delta_0, s_0^{\infty}]$.\\

{\bf \raggedleft Step 6.} In this final step we conclude the proof of (\ref{PS2}). Most of the work is already done and we simply need to put together a few statements from Steps 3 and 5.

Firstly, let $\zeta_h, \sigma_h$ be as in Step 3. By possibly passing to a subsequence we may assume that $\sigma_h \rightarrow \sigma \in [\delta_0, s_0^{\infty}]$ as $h \rightarrow \infty$. The latter and (\ref{PU1}) show that 
$$0 = \lim_{h \rightarrow \infty} |f_{n_h}(\zeta_h, \sigma_h) - f_{\infty}(\zeta_h, \sigma) | = \lim_{h \rightarrow \infty} |f_{n_h}(\zeta_h, \sigma_h)|,$$
where in the second equality we used that $f_{\infty}(z,s) = 0$ for all  $(z,s) \in U(\theta, \delta) \times [\delta_0, s_0^{\infty}]$. The latter and (\ref{PT1}) together imply (\ref{PS2}).

%-------------------------------------------------------------------------------------------------------------------------------------------------------------------------------------------------
% Section 6.3
%
%-------------------------------------------------------------------------------------------------------------------------------------------------------------------------------------------------
\subsection{Proof of Lemmas \ref{DomLip} and \ref{LUnique}}\label{Section6.3} In this section we prove the two lemmas from Section \ref{Section6.1}.

\begin{proof}[Proof of Lemma \ref{DomLip}]
It is clear that $F_n$ is closed and bounded, and hence compact. In addition, one readily observes that $F_n \subset F_{n+1}^{\circ}$ (the interior of $F_{n+1}$) for each $n \in \mathbb{N}$ and $\cup_{n \in \mathbb{N}} F_n = U(\delta, \theta)$ and so $\{F_n\}_{n = 1}^\infty$ forms a compact exhaustion of $U(\delta, \theta)$.

In the remainder we prove (\ref{LipDomE}) with 
$$\lambda_n = 1 + n(\delta^{-2} + \theta \delta^{-1} +  \pi n^{-1} ).$$
We introduce the sets
$$C_n = \left\{ z \in \mathbb{C}: \inf_{x \in [-\theta \delta^{-1}, \delta^{-2} ] } |x-z| = n^{-1} \right\}, \hspace{2mm} D_n = \left\{z \in \mathbb{C}: \inf_{x \in [-\theta \delta^{-1}, \delta^{-2} ] } |x-z| < n^{-1} \right\}.$$ 
Observe that $F_n \cup D_n$ and $D_n$ are convex and $C_n$ consists of two straight segments connecting $-\theta \delta^{-1} \pm \i n^{-1}$ with $\delta^{-2} \pm \i n^{-1}$, and two half-circles of radius $n^{-1}$, centered at $-\theta \delta^{-1}$ and $\delta^{-2}$. In particular, the length $\|C_v\|$ of $C_v$ is $2 ( \delta^{-2} + \theta \delta^{-1} + \pi n^{-1} )$. \\

Let $x,y \in F_n$ and let $\ell_{x,y}$ be the straight segment connecting $x$ and $y$. By convexity, we have that $\ell_{x,y} \subset F_n \cup D_n$. If $\ell_{x,y} \cap D_n = \emptyset$ we let $\gamma_{x,y} = \ell_{x,y}$ and note that (\ref{LipDomE}) holds since $\lambda_n \geq 1$. If $\ell_{x,y} \cap D_n \neq \emptyset$, then $\ell_{x,y}$ intersects $C_n$ and we let $x_1, y_1$ be the points on $\ell_{x,y} \cap C_n$ that are closest to $x$ and $y$, respectively (note that $x_1 = y_1$ is possible). For such $x,y$ we let $\gamma_{x,y}$ denote curve consisting of the straight segment connecting $x$ to $x_1$, the shortest of the clockwise and counterclockwise parts of $C_n$ connecting $x_1$ to $y_1$ and the straight segment connecting $y_1$ to $y$. By construction, $\gamma_{x,y}$ is contained in $F_n$ and is piecewise smooth. We proceed to estimate the length of $\gamma_{x,y}$ and check that it satisfies (\ref{LipDomE}).\\

If $|x-y| \geq  n^{-1}$ we have
$$\|\gamma_{x,y}\| \leq |x - x_1| + |y-y_1| + (1/2)\|C_v\| \leq |x-y| + (\delta^{-2} + \theta \delta^{-1} + \pi n^{-1})  \leq \lambda_n |x-y|,$$
which proves (\ref{LipDomE}) provided that $|x - y| \geq n^{-1}$.

We next suppose that $|x-y| < n^{-1}$, and note that the latter implies $|x_1 - y_1| \leq |x-y| < n^{-1}$. In particular, we conclude that either $x_1, y_1$ belong to the same half-circle in $C_n$ or one of the points belongs to one of the half-circles and the other to one of the straight segments connecting $-\theta \delta^{-1} \pm \i n^{-1}$ with $\delta^{-2} \pm \i n^{-1}$. If $x_1, y_1$ belong to the same half-circle, then the part of $\gamma_{x,y}$ connecting $x_1, y_1$ is a circular arc, which implies that this arc has length at most $(\pi/2) |x_1 -y_1|$. This gives 
$$\|\gamma_{x,y}\| \leq |x - x_1| + |y-y_1| + (\pi/2) |x_1 -y_1| \leq (1 + \pi/2)|x-y|.$$
For the remaining case, we assume without loss of generality that $x_1$ belongs to the half-circle, centered at $\delta^{-2}$ and $y_1$ belongs to the straight segment connecting $-\theta \delta^{-1} + \i n^{-1}$ with $\delta^{-2} + \i n^{-1}$. In this case, the part of $\gamma_{x,y}$ connecting $x_1, y_1$ consists of the straight segment $\ell_1$ connecting $y_1$ to $\delta^{-2}+ \i n^{-1}$ and the circular arc $\ell_2$ connecting $x_1$ to $\delta^{-2}+ \i n^{-1}$. We observe that 
$$\|\ell_1\| = \delta^{-2} - \Re [y_1] \leq \Re [x_1] - \Re [y_1] \leq |x_1 - y_1|.$$
In addition, 
$$\|\ell_2 \| \leq (\pi/2) |x_1 - \delta^{-2} - \i n^{-1}| \leq (\pi/2) |x_1 - y_1|.$$
Combining the last few observations we conclude that
$$\|\gamma_{x,y}\| = |x - x_1| + |y-y_1| + \|\ell_1 \| + \| \ell_2\|  \leq |x-y| + |x_1 - y_1| (1 + \pi/2) \leq |x-y| (2 + \pi/2).$$
Our work in the last paragraph proves (\ref{LipDomE}) when $|x - y| < n^{-1}$, since $\lambda_n \geq 1 + \pi > 2 + \pi/2$. This suffices for the proof.
\end{proof}

Before we go to the proof of Lemma \ref{LUnique} we establish the following useful lemma.
\begin{lemma}\label{ODESol} Fix $\theta > 0$, $\delta \in (0,1/2)$ and $\delta_0 = 2 \delta$. Fix $x_0 \geq 1 + \delta^{-2}$ and $\delta_0 \leq p < q \leq \delta_0^{-1}$. There exists a unique real-valued function $X(s) \in C^1([p,q])$ such that $X(p) = x_0$ and for each $s \in [p,q]$ 
\begin{equation}
X'(s) = \theta \left[\exp \left(\theta G \left( X(s)/s, s \right) \right) - 1 \right]^{-1},
\end{equation}
where $G(z,s)$ is as in (\ref{S31E6}). In addition, $X(s)$ is increasing on $[p,q]$ and 
$$x_0 \leq X(s) \leq  x_0\exp \left( (q-p)/p \right) \mbox{ for each $s \in [p,q]$.}$$
If $1 +\delta^{-2} \leq x_1 < x_2 < \infty$ and $X(s;x_1)$, $X(s;x_2)$ denote the function $X(s)$ with $x_0 = x_1$ and $x_0 = x_2$, respectively, as above, then we have $X(s;x_1) < X(s;x_2)$ for all $s \in [p,q]$. 
\end{lemma}
\begin{remark}\label{ODESolR} We mention that the statement of Lemma \ref{ODESol} is probably not optimal in terms of parameters. We chose the present formulation to make the result directly applicable within the proof of Lemma \ref{LUnique} below.
\end{remark}
\begin{proof} Recall from Section \ref{Section3.1} that $G(z,s)$ is smooth in the region $\{ (z, s) \in \mathbb{C} \times (0,\infty): z \in \mathbb{C} \setminus [-\theta, s^{-1}]\}$, which implies that $G(z/s, s)$ is smooth in the region $\{ (z, s) \in \mathbb{C} \times (0,\infty): z \in \mathbb{C} \setminus [-\theta s , 1]\}$. In addition, from (\ref{S31E6}) we have that
$$G(z/s,s)= \int_{\mathbb{R}} \frac{\mu(x,s)dx}{z/s - x},$$
where $\mu(x,s)$ is as in (\ref{S31E11}). Since $\mu(x,s)$ is a probability density function, supported on $[-\theta , s^{-1}]$, we see that $G(z/s,s)$ is real-valued for $z \geq 1 + \delta^{-2}$ and is positive there, provided that $s \in [\delta_0, \delta_0^{-1}]$. In addition, we observe that $G(y/s,s)$ is decreasing in $y$ for any fixed $s \in [\delta_0, \delta_0^{-1}]$ on $[1 +\delta^{-2}, \infty)$ and also for $y \geq 1 + \delta^{-2}$ 
\begin{equation}\label{S5BoundST1}
0 \leq G( y/s,s) \leq  \int_{-\theta}^{1/s} \frac{\mu(x,s)dx}{y/s - x} \leq \int_{-\theta}^{1/s} \frac{\mu(x,s)dx}{y/s - 1/s} = \frac{s}{y-1}.
\end{equation}

We also recall, see the discussion  after (\ref{S31E17}), that $\exp \left( \theta G\left(z/s, s\right) \right) - 1 \neq 0$ on $\{ (z, s) \in \mathbb{C} \times (0,\infty): z \in \mathbb{C} \setminus [-\theta s , 1]\}$. The latter implies that 
$$F(y,s):= \theta \left[\exp \left(\theta G \left( y/s, s \right) \right) -1 \right]^{-1}$$
is smooth in $[1 +\delta^{-2}, \infty) \times [p,q]$, is increasing in $y$ for each $s \in [p,q]$ and also 
\begin{equation}\label{S5BoundEST1}
0 \leq F(y,s) \leq  \theta \left[\exp \left( \frac{\theta s }{y-1} \right) -1 \right]^{-1} \leq \frac{y-1}{s} \leq \frac{y-1}{p},
\end{equation}
where in the next to last inequality we used that $e^x - 1 \geq x$ for $x \geq 0$. \\

We now let $X_0(s) = Y_0(s) = x_0$ for $s \in [p,q]$ and then inductively define for $n \in \mathbb{N}$
\begin{equation}\label{S5Picard}
X_n(s) = x_0 + \int_{p}^s F(X_{n-1}(u),u) du \mbox{ and } Y_n(s) = x_0 + \int_{p}^s H(Y_{n-1}(u),u)du,
\end{equation}
where $H(y,s) = (y-1)/p$. Observe that $X_0(s)$ is well-defined for $s \in [p,q]$, $X_0(s) \in [1 + \delta^{-2} , \infty)$ and $X_0(s)$ is increasing on $[p,q]$. Assuming the same for $X_n(s)$, we observe by the smoothness of $F(y,s)$ in $[1 +\delta^{-2}, \infty) \times [p,q]$ that $X_{n+1}(s)$ is well-defined and from the first inequality in (\ref{S5BoundEST1}) it is increasing and in particular $X_{n}(s) \geq x_0 \geq 1 + \delta^{-2}$. Analogous arguments show that $Y_n$ is well-defined and increasing, and $Y_n(s) \in [ 1 + \delta^{-2}, \infty)$.\\

We next establish several properties about the sequences $X_n$ and $Y_n$. We first observe that 
\begin{equation}\label{XLY}
X_n(s) \leq Y_n(s) \mbox{ for $s \in [p,q]$ and $n \in \mathbb{Z}_{\geq 0}$}.
\end{equation}
Indeed, the latter is clear when $n = 0$ and assuming the inequality for $n \in \mathbb{Z}_{\geq 0}$ we get
\begin{equation*}
\begin{split}
X_{n+1}(s) = x_0 + \int_{p}^s F(X_{n}(u),u) du  \leq x_0 + \int_{p}^s H(X_{n}(u),u) du \leq x_0 + \int_{p}^s H(Y_{n}(u),u) du  = Y_{n+1}(s),
\end{split}
\end{equation*}
where in the first inequality we used (\ref{S5BoundEST1}), in the second we used that $1 + \delta^{-2} \leq X_n(s) \leq Y_n(s)$ by assumption and and that $H(y,s)$ is increasing in $y$ on $[1 + \delta^{-2}, \infty)$. Secondly, we have
\begin{equation}\label{XMon}
X_n(s) \leq X_{n+1}(s)  \mbox{ for $s \in [p,q]$ and $n \in \mathbb{Z}_{\geq 0}$}.
\end{equation}
When $n = 0$, the latter is clear since $X_1$ is increasing, and $X_1(p) = x_0 = X_0(s)$. Assuming (\ref{XMon}) for $n \in \mathbb{Z}_{\geq 0}$ we have
$$X_{n+1}(s) = x_0 + \int_{p}^s F(X_{n}(u),u) du \leq x_0 + \int_{p}^s F(X_{n+1}(u),u) du \leq X_{n+2}(s),$$
where in the middle inequality we used the monotonicity of $F(y,s)$ in $y \in [1 + \delta^{-2}, \infty)$ for each $s \in[p,q]$. Thirdly, we have
\begin{equation}\label{YMon}
Y_n(s) = 1 + (x_0-1) \cdot \sum_{k = 0}^n  \frac{p^{-k}(s- p)^k}{k!},
\end{equation}
which one readily observes by induction on $n$ using (\ref{S5Picard}). \\

From (\ref{YMon}) we know that $Y_n(s)$ increases in $n$ to $1 + (x_0 - 1) e^{(s-p)/p}$ and from (\ref{XLY}) we conclude 
\begin{equation}\label{XUB}
X_n(s) \leq x_0 \exp \left( (s-p)/p \right) \mbox{ for $s \in [p,q]$ and $n \in \mathbb{Z}_{\geq 0}$}.
\end{equation}
Combining the last equation with (\ref{XMon}) we conclude that $X_{\infty}(s) := \lim_{n \rightarrow \infty} X_n(s)$ exists and is finite. Also from our work after (\ref{S5Picard}) we have that $X_{\infty}(s)$ is increasing on $[p,q]$, which together with (\ref{XUB}) gives 
$$x_0 \leq  X_{\infty}(s) \leq x_0 \exp \left( (q-p)/p \right).$$

From the monotone convergence theorem and (\ref{S5Picard}) we conclude that for each $s \in [p,q]$
$$X_{\infty}(s) = x_0 + \int_{p}^s F(X_{\infty}(u),u) du,$$
which implies that $X_{\infty}$ is differentiable and hence continuous. Differentiating the last equation gives $X'_{\infty}(s) = F(X_{\infty}(s),s)$, which from the continuity of $X_{\infty}(s)$, the fact that $X_{\infty}(s) \in [ 1 + \delta^{-2} ,\infty)$ and the smoothness of $F(y,s)$ in $[1 + \delta^{-2}, \infty) \times [p,q]$ shows that $X'_{\infty}$ is continuous. The last two paragraphs show that $X:=X_{\infty}$ satisfies all the conditions in the lemma. We mention that the uniqueness of $X$ follows from classical ODE theory, since $X$ solves the differential equation $Y'(s) = F(Y(s), s)$ on the domain $[x_0, x_0 \exp((q-p)/p)] \times [p,q]$ and $F(y,s)$ is smooth there, see \cite[Lemma 1 and Theorem 5, Chapter 1]{Rota}.\\

We finally turn to the last part of the lemma and fix $x_1, x_2$, and $X(s;x_1)$, $X(s;x_2)$ as in the statement of the lemma. Suppose for the sake of contradiction that $X(s_0;x_1) \geq X(s_0;x_2)$ for some $s_0 \in [p,q]$. By assumption, $X(p;x_1) = x_1 < x_2 = X(p;x_2)$ and so by continuity there exists $s_1 \in [p,s_0]$ such that $X(s_1;x_1) = X(s_1; x_2)$. In addition, from the properties established earlier in this lemma we know that both $ X(s;x_1)$ and $X(s;x_2)$ solve the differential equation 
$Y'(s) = F(Y(s),s)$ on the domain $[x_1, x_2 \exp \left( (q-p)/p \right)] \times [p,q]$ and $F(y,s)$ is smooth there. Again, using \cite[Lemma 1 and Theorem 5, Chapter 1]{Rota} we conclude that $X(s;x_1)  = X(s;x_2)$ for all $s \in [p,q]$, which implies $x_1 = X(p;x_1) = X(p;x_2) = x_2 $. This is our desired contradiction, and so $X(s;x_1) < X(s;x_2)$ for all $s \in [p,q]$ as desired.
\end{proof}
We end this section with the proof of Lemma \ref{LUnique}.

\begin{proof}[Proof of Lemma \ref{LUnique}] Suppose that $s_n \in [p,q]$ and $z_n \in U(\delta,\theta)$ for $n \in \mathbb{N} \cup \{\infty\}$ are such that $\lim_{n \rightarrow \infty} s_n = s_\infty$ and $\lim_{n \rightarrow \infty} z_n = z_\infty$. We first show that 
\begin{equation}\label{S5Cont}
\lim_{n \rightarrow \infty} \partial_z f(z_n,s_n) = \partial_z f(z_{\infty},s_{\infty}).
\end{equation}
Let $\epsilon > 0$ be such that $B_{2\epsilon}(z_{\infty}) \subset U(\delta,\theta)$ (here $B_{r}(x)$ is the disc of radius $r$ around the point $x$). Let $N\in \mathbb{N}$ be sufficiently large so that for $n \geq N$ we have $|z_n - z_{\infty}| < \epsilon/2$. Then, by Cauchy's integral formula, see \cite[Chapter 2, Theorem 4.1]{SS},  for $n \in \mathbb{N} \cup \{\infty\}$ such that $n \geq N$ we have
$$\partial_z f(z_n,s_n) = \frac{1}{2\pi \i } \int_{C_{\epsilon}(z_{\infty})} \frac{f(\zeta, s_n)}{(\zeta - z_n)^2} d\zeta,$$
where $C_{\epsilon}(z_{\infty})$ is the positively oriented circle of radius $\epsilon$, centered at $z_{\infty}$. Since $f$ is continuous, we know that $f(\zeta, s_n) \rightarrow f(\zeta, s_{\infty})$ as $n \rightarrow \infty$ and so the last equation implies (\ref{S5Cont}) by the bounded convergence theorem. We mention that the fact that $ \partial_{z} f(z,u)$ is continuous on $U(\delta,\theta) \times [p,q]$ in particular implies that the integral in (\ref{S5LID}) is well-defined.\\

Recall from Section \ref{Section3.1} that $G(z,s)$ is smooth in the region $\{ (z, s) \in \mathbb{C} \times (0, \infty): z \in \mathbb{C} \setminus [-\theta, s^{-1}]\}$, which implies that $G(z/s, s)$ is smooth on $(z,s) \in U(\delta, \theta) \times [\delta_0, \delta_0^{-1}]$. We also recall, see the discussion after (\ref{S31E17}), that $e^{\theta G\left(z/s, s\right)} - 1 \neq 0$ on $U(\delta, \theta) \times [\delta_0, \delta_0^{-1}]$. The last observations and (\ref{S5LID}) together imply that $\partial_s f(z,s) $ exists for each $(z,s) \in U(\delta, \theta) \times [p,q]$. By differentiating (\ref{S5LID}) with respect to $s$ we see that $\partial_s f(z,s) $ satisfies for $(z,s) \in U(\delta, \theta) \times [p,q]$
\begin{equation}\label{S6II1}
  \partial_s f(z,s) + f(z,s) \cdot \frac{\theta e^{\theta G\left(z/s, s\right)}  \left[\partial_s G\left(z/s, s\right) - (z/s^2) \cdot \partial_z G\left(z/s, s\right) \right]  }{\left[e^{\theta G\left(z/s, s\right)} - 1 \right]} +  \frac{ \theta \cdot \partial_z f(z,s)  }{\left[e^{\theta G\left(z/s, s\right)} - 1 \right]}    = 0.
\end{equation}
Equation (\ref{S6II1}) is a first order linear PDE, which can be solved using the method of characteristics as we explain next. 

For $x_0 \in [1 +\delta^{-2}, \infty)$ we let $X_{x_0}: [p,q] \rightarrow \mathbb{R}$ be as in Lemma \ref{ODESol}. Observe that from Lemma \ref{ODESol} we know that $X_{x_0}(s) \in \left[x_0,  x_0 \exp \left( (q-p)/p \right)\right] \subset U(\delta,\theta)$, which implies that $f(X_{x_0}(s), s)$ and $ G\left(X_{x_0}(s)/s, s\right)$ are well-defined for $s \in [p,q]$. Setting $v_{x_0}(s) =f(X_{x_0}(s),s)$ we obtain from Lemma \ref{ODESol} and (\ref{S6II1}) that $v_{x_0}(s)$ solves
\begin{equation}\label{S6ODEV}
\begin{split}
&v'(s) + v(s) \cdot \frac{\theta e^{\theta G\left(X_{x_0}(s)/s, s\right)} \left[\partial_s G\left(X_{x_0}(s)/s, s\right) - (X_{x_0}(s)/s^2)\partial_z G\left(X_{x_0}(s)/s, s\right)  \right]  }{\left[e^{\theta G\left(X_{x_0}(s)/s, s\right)} - 1 \right]}  = 0,\\
&\mbox{ with initial condition } v(q)= f\left( X_{x_0}(q), q \right) = 0,
\end{split}
\end{equation}
where in the last equality we used (\ref{S5LID0}). Equation (\ref{S6ODEV}) is a linear first order ODE with a $C^1$ coefficient (here we used the smoothness of $G(z/s,s)$ and the fact that $X_{x_0} \in C^1([p,q])$ from Lemma \ref{ODESol}), which from classical ODE theory, see e.g. \cite[Lemma 1 and Theorem 5, Chapter 1]{Rota}, has a unique solution. On the other hand, it is clear that $v(s) = 0$  is a solution and so we conclude that $v_{x_0}(s) = 0$ for $s \in [p,q]$. 

The work in the previous paragraph shows that for each $s\in [p,q]$ and $x_0 \geq 1 + \delta^{-2}$ we have $0 = v_{x_0}(s) = f\left( X_{x_0}(s), s \right)$. From Lemma \ref{ODESol} we have that the map $h : [1 + \delta^{-2}, 2 + \delta^{-2}] \rightarrow \mathbb{R}$, given by $h(x_0) = X_{x_0}(s)$, is injective and the image of $h$ is contained in $[1 + \delta^{-2}, (2 + \delta^{-2})  \exp \left( (q-p)/p \right)]$. The latter implies that $f(z,s) = 0$ for infinitely many $z \in [1 + \delta^{-2}, (2 + \delta^{-2})  \exp \left( (q-p)/p \right)] \subset U(\delta, \theta)$ and from the analyticity of $f(z,s)$ in $z$ on $U(\delta, \theta)$ we conclude that $f(z,s) = 0$ for all $z \in U(\delta, \theta)$, see \cite[Chapter 2, Theorem 4.8]{SS}. Since $s \in [p,q]$ was arbitrary, we conclude that $f(z,s) = 0$ for all $z \in U(\delta, \theta)$ and $s \in [p,q]$ as desired.
\end{proof}

%-------------------------------------------------------------------------------------------------------------------------------------------------------------------------------------------------
% Section 7
%
%-------------------------------------------------------------------------------------------------------------------------------------------------------------------------------------------------
\section{Asymptotic Gaussian fields}\label{Section7} In Proposition \ref{S6Prop} we showed that the third and higher order joint cumulants of $\G_K(z,s)$ from (\ref{S6DefG}) vanish as $K \rightarrow \infty$. In this section we prove that $\G_K(z,s)$ converge to a Gaussian field with an explicit covariance -- the precise statement is in Theorem \ref{MainTechThm} below. In Section \ref{Section7.1} we formulate Theorem \ref{MainTechThm}, and state a few technical lemmas that are used in its proof, which is given in Section \ref{Section7.2}. The lemmas from Section \ref{Section7.1} are proved in Section \ref{Section7.3}.
%-------------------------------------------------------------------------------------------------------------------------------------------------------------------------------------------------
% Section 7.1
%
%-------------------------------------------------------------------------------------------------------------------------------------------------------------------------------------------------
\subsection{Notation and technical preliminaries}\label{Section7.1} We continue with the same notation as in Section \ref{Section6.1}. For $s > 0$ we define 
\begin{equation}\label{S7DefAB}
a(s) = s \cdot z_-(s) = (1/2) (1- \theta s) - \sqrt{\theta s}  \mbox{ and } b(s) = s \cdot z_+(s) = (1/2) (1- \theta s) - \sqrt{\theta s} ,  
\end{equation}
where we recall that $z_{\pm}(s)$ were defined in (\ref{S31E12}). If $s \geq t > 0$ we define the function $F(\cdot; s,t) : \mathbb{C} \setminus (a(t), b(t)) \rightarrow \mathbb{C}$ via
\begin{equation}\label{Transport}
F(z; s,t) = z + \frac{\theta (s-t)}{e^{\theta G(z/t,t)} -1} = z + \frac{(s - t)(2z - 1 - \theta t)}{4t} + \frac{s -t }{2t} \cdot \sqrt{(z- a(t))(z-b(t))},
\end{equation}
where for the second equality we used (\ref{S31E14}). We recall that the square root in (\ref{Transport}) is as in Definition \ref{Branch}. The following lemma summarizes the key properties we require from $F(z;s,t)$. Its proof is given in Section \ref{Section7.3}.
\begin{lemma}\label{FProp} Fix $\theta > 0$, and $s \geq t > 0$. The function $F(z; s,t) $ in (\ref{Transport}) has the following properties.
\begin{enumerate}
\item $F(z;s,t)$ is holomorphic in $z \in \mathbb{C} \setminus [a(t), b(t)]$.
\item If $z \in \mathbb{H}$, then $F(z; s,t) \in \mathbb{H}$, and if $z \in \overline{\mathbb{H}}$, then $F(z; s,t) \in \overline{\mathbb{H}}$. 
\item For $z \in \mathbb{C} \setminus [a(t), b(t)]$ we have $F(\bar{z}; s,t) = \overline{F(z;s,t)}$.
\item The restriction of $F(z;s,t)$ to $[b(t), \infty)$ is strictly increasing and defines a bijection between $[b(t), \infty)$ and $[d(s,t), \infty)$, where $d(s,t) = b(s) + (1/2) \cdot \sqrt{\theta/t} \cdot (\sqrt{s} - \sqrt{t})^2$.
\item The restriction of $F(z;s,t)$ to $(-\infty, a(t)]$ is strictly increasing and defines a bijection between $(-\infty, a(t)]$ and $(-\infty, c(s,t)]$, where $c(s,t) = a(s) -(1/2) \cdot \sqrt{\theta/t} \cdot (\sqrt{s} - \sqrt{t})^2$.
\end{enumerate}
\end{lemma}

With the above notation in place we are ready to state the main result of the section.
\begin{theorem}\label{MainTechThm} Fix $\theta > 0$, and assume that $(\ell^1, \ell^2, \dots)$ is distributed according to $\mathbb{P}^{\theta, K}_{\infty}$ as in Definition \ref{BKCC}. Fix $\delta \in (0,1/2)$, $\delta_0 = 2\delta$ and $U(\delta, \theta)$ as in (\ref{S6DefU}). Then, as $K \rightarrow \infty$ the random field $\{ \G_K(z,s): z \in U(\delta, \theta), s \in [\delta_0, \delta_0^{-1}]\}$ converges in the sense of joint moments, uniformly in $z$ in compact subsets of $U(\delta, \theta)$ and $s \in [\delta_0, \delta_0^{-1}]$, to a complex Gaussian random field $\{ \G(z,s): z \in U(\delta, \theta), s \in [\delta_0, \delta_0^{-1}]\}$. The field $\G(z,s)$ has mean zero and covariance for $z_1, z_2 \in U(\delta, \theta)$ and $\delta_0^{-1} \geq s_1 \geq s_2 \geq \delta_0$ given by
\begin{equation}\label{LimCov}
\begin{split}
&\mathrm{Cov}\left(\G(z_1,s_1), \G(z_2, s_2) \right) = C(z_1,s_1;z_2,s_2) = \\
&= -\frac{\theta^{-1}}{2(z_1- x_2)^2} \left( 1 - \frac{(z_1 - b(s_1))(x_2 - a(s_1)) + (x_2 - b(s_1)) (z_1- a(s_1))}{2\sqrt{(z_1 - a(s_1))(z_1- b(s_1))} \cdot \sqrt{(x_2- a(s_1))(x_2 - b(s_1))} }  \right) \\
&\times \frac{[e^{\theta G(z_2/s_2, s_2)} + 1 ]^2 \cdot \theta s_1 - [e^{\theta G(z_2/s_2, s_2)} - 1 ]^2}{[e^{\theta G(z_2/s_2, s_2)} + 1 ]^2 \cdot \theta s_2 - [e^{\theta G(z_2/s_2, s_2)} - 1 ]^2},
\end{split}
\end{equation}
where $x_2 = F(z_2; s_1, s_2)$ as in (\ref{Transport}), and $a(s_1), b(s_1)$ are as in (\ref{S7DefAB}). Equation (\ref{LimCov}) defines $C(z_1,s_1;z_2,s_2)$ for $s_1 \geq s_2$ and for $s_1 < s_2$ we define $C(z_1,s_1; z_2, s_2) := C(z_2, s_2; z_1 ,s_1)$ with the latter as in (\ref{LimCov}).
\end{theorem}
\begin{remark}\label{S6WellDefCov} We mention that from Lemma \ref{FProp} we know that $F(z_2; s_1, s_2) = x_2 \not \in [a(s_1), b(s_1)]$ for $z_2 \in U(\delta, \theta)$. One also has trivially that $z_1 \not \in [a(s_1), b(s_1)]$ for $z_1 \in U(\delta, \theta)$, which shows that the expression inside the brackets on the second line in (\ref{LimCov}) is well-defined and finite. In addition, we mention that by a simple limiting argument one checks that (\ref{LimCov}) has a finite limit as $z_1 \rightarrow x_2$. In particular, the function on the second line of (\ref{LimCov}) is well-defined and finite. One also checks by a direct computation using (\ref{S31E14}) that the only roots of the denominator in the last line of (\ref{LimCov}) are at $z_2 = a(s_2)$ and $z_2 = b(s_2)$. Since $z_2 \in U(\delta, \theta)$, we know that $z_2 \not \in [a(s_2), b(s_2)]$ and so the fraction in the last line of (\ref{LimCov}) is also well-defined. Overall, we conclude that the right side of (\ref{LimCov}) is well-defined and finite for each $z_1, z_2 \in U(\delta, \theta)$ and $\delta_0^{-1} \geq s_1 \geq s_2 \geq \delta_0$.
\end{remark}
\begin{remark}\label{S7SingleLev} As mentioned in the beginning of Section \ref{Section3}, we have that the distribution of $\ell^n$ under $\P$ is that of a discrete $\beta$-ensemble. Specifically, the asymptotic Gaussianity of $\G(z,s)$ for a fixed $s > 0$ and the formula for the covariance in (\ref{LimCov}) when $s_1 = s_2 = s$ have been previously established for general models -- see \cite[Theorem 7.1 and Remark 7.2]{BGG} and \cite[Proposition 3.10]{DK2020}. The novelty of Theorem \ref{MainTechThm} is that it proves asymptotic Gaussianity and provides an exact formula for the covariance not just for a single level of $\G(z,s)$, but for all levels jointly.
\end{remark}

We end this section with a technical result, which we need for the proof of Theorem \ref{MainTechThm} in the next section. Its proof is given in Section \ref{Section7.3}.
\begin{lemma}\label{LUnique2} Fix $\theta > 0$, $\delta \in (0,1/2)$, $\delta_0 = 2 \delta$ and let $U(\delta, \theta)$ be as in (\ref{S6DefU}). Let $\delta_0 \leq p < q \leq \delta_0^{-1}$ and $z_1 \in U(\delta, \theta)$. Suppose that $f: U(\delta,\theta ) \times [p,q] \rightarrow \mathbb{C}$ is a continuous function, such that for each $s \in [p,q]$ we have that $f(z,s)$ is analytic in $z$ on $U(\delta,\theta ) $. Suppose that for $z \in U(\delta,\theta ) $
\begin{equation}\label{S6LID0}
f(z,q) = C(z_1, q; z,q),
\end{equation}
where $C(z_1, s_1; z_2,s_2)$ is as in (\ref{LimCov}), and for every $s \in [p,q]$ and $z \in U(\delta,\theta)$ we have
\begin{equation}\label{S6LID}
  \left[e^{\theta G\left(z/q, q\right)} - 1 \right] f(z,q) + \theta \cdot \int_{s}^{q}  \partial_{z} f(z,u)du -   \left[e^{\theta G\left(z/s, s\right)} - 1 \right] f(z,s)  = 0,
\end{equation}
where $G(z,s)$ is as in (\ref{S31E6}). Then, $f(z,s) = C(z_1, q; z,s)$ for all $(z,s) \in U(\delta,\theta)  \times [p,q]$.
\end{lemma}
\begin{remark}\label{RemGuess} The proof of Lemma \ref{LUnique2} in Section \ref{Section7.3} is of a verification type -- basically we will show that $C(z_1, q; z,s)$ from (\ref{LimCov}) satisfies the conditions of the lemma. In Appendix \ref{AppendixB} we explain how the formula of $C(z_1, q; z,s)$ in (\ref{LimCov}) was discovered.
\end{remark}
\begin{remark}\label{HowUsed} In the proof of Theorem \ref{MainTechThm} we will show that the covariances $\mathrm{Cov}\left(\G_K(z_1,s_1), \G_K(z, s) \right)$, viewed as functions of $z$ and $s$, essentially form a tight sequence and all subsequential limits satisfy the integral equation (\ref{S6LID}), and the boundary condition at $s = s_1$, given by (\ref{S6LID0}) with $q = s_1$. In view of Lemma \ref{LUnique2} there is a unique function, namely $C(z_1,s_1; z,s)$, that satisfies (\ref{S6LID0}) and (\ref{S6LID}), which means that $\mathrm{Cov}\left(\G_K(z_1,s_1), \G_K(z, s) \right)$ need to converge to it.
\end{remark}

%-------------------------------------------------------------------------------------------------------------------------------------------------------------------------------------------------
% Section 7.2
%
%-------------------------------------------------------------------------------------------------------------------------------------------------------------------------------------------------
\subsection{Proof of Theorem \ref{MainTechThm}}\label{Section7.2} In this section we give the proof of Theorem \ref{MainTechThm}. We continue with the same notation as in Sections \ref{Section7.1} and the beginning of Section \ref{Section6.2}. In particular, we will make use of $\sfU_{K}(z;s,t)$ from (\ref{S6DefInt}) and $\sfV_K(z; s,t)$ from (\ref{PP1}). 

From (\ref{S6DefG}) we know that $\G_K(z,s)$ have mean zero, and from Proposition \ref{S6Prop} we know that all third and higher joint cumulants of $\G_K(z,s)$ vanish as $K \rightarrow \infty$. Consequently, to prove the theorem it suffices to show that if $K_n$ is an increasing sequence, $\delta_0^{-1} \geq s_1^n \geq s_2^n \geq \delta_0$ and $z_1^n, z_2^n \in U(\delta, \theta)$ are sequences such that 
$$\lim_{n \rightarrow \infty} s_i^n = s^{\infty}_i \in [\delta_0, \delta_0^{-1}] \mbox{ and } \lim_{n \rightarrow \infty} z_i^n = z^{\infty}_i \in U(\delta, \theta) \mbox{ for $i = 1,2$},$$
then we have 
\begin{equation}\label{YA}
\lim_{n \rightarrow \infty} M\left(\G_{K_n}(z_1^n, s_1^n), \G_{K_n}(z_2^n, s_2^n) \right) = C(z^{\infty}_1,s^{\infty}_1;z^{\infty}_2,s^{\infty}_2),
\end{equation}
where $C(z_1,s_1;z_2,s_2)$ is as in (\ref{LimCov}).\\

For future use we note that from (\ref{S31E16}) and (\ref{S7DefAB}) we have
\begin{equation}\label{S7QRecalled}
\begin{split}
Q(z/s,s) = \Phi^-\left(z/s, s \right)e^{-\theta G(z/s,s)} -   \Phi^+\left(z/s, s \right)e^{\theta G(z/s,s)} = (2/s) \sqrt{(z -a(s))(z- b(s))}.
\end{split}
\end{equation}

The remainder of the proof is similar to the proof of Proposition \ref{S6Prop} and is also split into six steps. In Step 1 we utilize (\ref{S52Cov1}) from Proposition \ref{S52P1} to find a formula for $M \left( \sfU_{K}(v_2;s_0,t_0), \G_K(v_1,t) \right) $. We further simplify the resulting formula, and the final form can be found in (\ref{YA1}). In Step 2 we prove two identities, equations (\ref{YB1}) and (\ref{YB3}), which play the same role as equations (\ref{PP2}) and (\ref{PP22}) in the proof of Proposition \ref{S6Prop} . In the third step we prove (\ref{YA}) when $s^{\infty}_1 = s^{\infty}_2$. In Steps 4 and 5 we free up the $z^n_2$ and $s^n_2$ variables and consider the functions $ f_n(z,s) = M(\G_{K_n} (z,s), \G_{K_n}(z_1^n, s_1^n) )$. Similarly to Steps 4 and 5 in the proof of Proposition \ref{S6Prop}, we use (\ref{YB1}) and (\ref{YB3}) to prove that $f_n$ has a subsequential limit $f_{\infty}$ that satisfies the conditions of Lemma \ref{LUnique2}, and hence is equal to $C(z_1^{\infty},s_1^{\infty}; z,s)$. In Step 5 we show that the latter implies (\ref{YA}). \\

{\bf \raggedleft Step 1.} In this step we show that 
\begin{equation}\label{YA1}
\begin{split}
&M \left( \sfU_{K}(v_2;s_0,t_0), \G_K(v_1,t) \right) = O(K^{-1/2}) -  \oint_{\gamma} dz \frac{ \Phi^+(z/s_0, s_0) e^{\theta G(z/t, t)}    }{2 \pi \i  \theta (z- v_2) (z-v_1)^2},
\end{split}
\end{equation}
provided that $v_1, v_2 \in U(\delta, \theta)$, $\delta_0^{-1} \geq s_0 \geq t \geq t_0 \geq \delta_0$ and $K \geq \delta_0^{-1}$. The constant in the big $O$ notation depends on $\theta, \delta$ and a compact set $\mathcal{V}$ and (\ref{YA1}) holds if $v_1, v_2 \in \mathcal{V}$, and $\gamma$ is a positively oriented contour that encloses $[-\theta \delta^{-1}, \delta^{-2}]$ and excludes the points $v_1, v_2$.

Using (\ref{TR3}), (\ref{TR3.5}), the smoothness of $G(z,s)$ and (\ref{S52Cov1}) we have 
\begin{equation}\label{YA2}
\begin{split}
&M \left( \sfU_{K}(v_2;s_0,t_0), \G_K(v_1,t) \right) = O(K^{-1/2}) -  \oint_{\gamma} dz \frac{ \Phi^+(z/s_0, s_0) e^{\theta G(z/t_0, t_0)}    }{2 \pi \i  \theta (z- v_2) (z-v_1)^2}   \\
& +    \int_{t_0}^{t} ds \oint_{\gamma}dz \frac{ \Phi^+(z/s_0, s_0)  }{2 \pi \i  (z- v_2) (z-v_1)^2} \cdot \left[  \frac{z}{s^2} \cdot \partial_z G\left( \frac{z}{s}, s \right) + \frac{\theta }{s}  \cdot  \partial_z G\left( \frac{z}{s}, s \right) -  \partial_s G \left(\frac{z}{s} , s \right)  \right]  ,
\end{split}
\end{equation}
We next proceed to simplify the right side of (\ref{YA2}). Recall from (\ref{SpecEqn}) that we have
\begin{equation}\label{S7SpecEqn}
\begin{split}
 \partial_s G(z,s) = \frac{\partial_s  e^{\theta G(z,s)} }{ \theta e^{\theta G(z,s)}} = \frac{  (ze^{\theta G(z,s)}   -  \theta -  z) \partial_z G(z,s) }{ s \left(e^{\theta G(z,s)} - 1 \right)}.
\end{split}
\end{equation}
The latter allows us to rewrite
\begin{equation*}
\begin{split}
&\frac{z}{s^2} \cdot \partial_z G\left( \frac{z}{s}, s \right) + \frac{\theta }{s}  \cdot  \partial_z G\left( \frac{z}{s}, s \right) -  \partial_s G \left(\frac{z}{s} , s \right) = \partial_z G \left( \frac{z}{s}, s\right) \\
&  \times \frac{(z/s + \theta) \cdot\left(e^{\theta G(z/s,s)} - 1 \right) - [(z/s)e^{\theta G(z/s,s)}   -  \theta -  (z/s)] }{ s \left(e^{\theta G(z/s,s)} - 1 \right)}=  \partial_z G \left( \frac{z}{s}, s\right) \cdot \frac{\theta e^{\theta G(z/s,s)} }{s \left(e^{\theta G(z/s,s)} - 1 \right)}.
\end{split}
\end{equation*}
Substituting the latter into (\ref{YA2}) gives
\begin{equation}\label{YA3}
\begin{split}
&M \left( \sfU_{K}(v_2;s_0,t_0), \G_K(v_1,t) \right) = O(K^{-1/2}) -  \oint_{\gamma} dz \frac{ \Phi^+(z/s_0, s_0) e^{\theta G(z/t_0, t_0)}    }{2 \pi \i  \theta (z- v_2) (z-v_1)^2}   \\
& +   \theta \int_{t_0}^{t} ds \oint_{\gamma}dz \frac{ \Phi^+(z/s_0, s_0)  \partial_z G \left( z/s, s\right) e^{\theta G(z/s,s)}   }{2 \pi \i  (z- v_2) (z-v_1)^2 s  \left(e^{\theta G(z/s,s)} - 1 \right)}.
\end{split}
\end{equation}

We next claim that the following is true for all $\delta_0^{-1} \geq s_0 \geq t \geq t_0  \geq \delta_0$
\begin{equation}\label{YA4}
\begin{split}
&  \oint_{\gamma} dz \frac{\Phi^+(z/s_0, s_0) e^{\theta G(z /t, t)} }{2\pi \i (z- v_2) (z -v_1)^2} - \oint_{\gamma} dz \frac{\Phi^+(z/s_0, s_0) e^{\theta G(z /t_0, t_0)} }{2\pi \i (z- v_2) (z -v_1)^2} = \\
&- \theta^2 \int_{t_0}^{t}ds \oint_{\gamma} dz \frac{\Phi^+(z/ s_0, s_0)  \cdot \partial_z G \left( z/s, s\right) e^{\theta G(z/s,s)}  }{2 \pi \i (z-v_2) (z-v_1)^2 s \left(e^{\theta G(z/s,s)} - 1 \right) }.
\end{split}
\end{equation}
Substituting the latter into (\ref{YA3}) we obtain (\ref{YA1}). In the remainder of this step we show (\ref{YA4}). Clearly both sides are equal when $t = t_0$ and so it suffices to show that the $t_0$-derivatives agree, which is equivalent to 
\begin{equation}\label{YA5}
\begin{split}
&  - \oint_{\gamma} dz \frac{\Phi^+(z/s_0, s_0) e^{\theta G(z /t_0, t_0)} }{2\pi \i (z- v_2) (z -v_1)^2} \cdot \left[ \partial_s G(z/t_0,t_0) - (z/t_0^2) \partial_z G(z/t_0,t_0) \right] \\
& = \theta \cdot \oint_{\gamma} dz \frac{\Phi^+(z/ s_0, s_0)  \cdot \partial_z G \left( z/t_0, t_0\right) e^{\theta G(z/t_0,t_0)}  }{2 \pi \i (z-v_2) (z-v_1)^2 s_2 \left(e^{\theta G(z/t_0,t_0)} - 1 \right) }.
\end{split}
\end{equation}
From (\ref{S7SpecEqn}) we have
\begin{equation}\label{S7SpecEqn2}
\begin{split}
 &\partial_s G(z/s,s) - (z/s^2) \partial_z G(z/s,s) =  \partial_z G(z/s,s) \\
&\times \frac{  [(z/s)e^{\theta G(z/s,s)}   -  \theta -  (z/s)] - (z/s) \cdot \left(e^{\theta G(z/s,s)} - 1 \right) }{ s \left(e^{\theta G(z/s,s)} - 1 \right)} = -\frac{\theta \cdot  \partial_z G(z/s,s)}{s\left(e^{\theta G(z/s,s)} - 1 \right) }.
\end{split}
\end{equation}
Using the latter, we see that (\ref{YA5}) trivially holds as the integrands agree pointwise.\\

{\bf \raggedleft Step 2.} In this step we establish two identities, these are equations (\ref{YB1}) and (\ref{YB3}), that will be used in later parts of the proof.

From (\ref{UTop}) and (\ref{YA1}) we have for $\delta_0^{-1} \geq s_0 \geq  \delta_0$, and $v_1, v_2 \in U(\delta, \theta)$
\begin{equation}\label{YB1}
\begin{split}
&Q(v_2/s_0, s_0) \cdot M \left(\G_{K}(v_2,s_0), \G_K(v_1,s_0)   \right) = M \left(\sfU_{K}(v_2;s_0,s_0), \G_K(v_1,t)   \right) \\
&= O(K^{-1/2}) -  \oint_{\gamma} dz \frac{ \Phi^+(z/s_0, s_0) e^{\theta G(z/s_0, s_0)}    }{2 \pi \i  \theta (z- v_2) (z-v_1)^2}.
\end{split}
\end{equation}
Subtracting (\ref{YB1}) from (\ref{YA1}) applied to $t = s_0$ we have for $\delta_0^{-1} \geq s_0 \geq t_0 \geq  \delta_0$, and $v_1, v_2 \in U(\delta, \theta)$
$$\Phi^+(v_2/s_0,s_0)  \cdot M \left(\sfV_{K}(v_2;s_0,t_0), \G_K(v_1,s_0)   \right) = O(K^{-1/2}),$$
where we used the formula for $Q$ from (\ref{S7QRecalled}), the formula for $\sfV_K$ from (\ref{PP1}) and $\sfU_K$ from (\ref{S6DefInt}). As explained in Remark \ref{Divide} we can divide both sides by $\Phi^{+}(v_2/s_0, s_0)$ to conclude
\begin{equation}\label{YB2}
\begin{split}
M \left(\sfV_{K}(v_2;s_0,t_0), \G_K(v_1,s_0)   \right)  = O(K^{-1/2}). 
\end{split}
\end{equation}
From the last equation and the definition of $\sfV_K$ in (\ref{PP1}) we obtain
\begin{equation*}
\begin{split}
&\left[e^{\theta G(v_2/t_0, t_0) } - 1 \right] \hspace{-1mm}  M \left(\G_K(v_2,t_0) , \G_K(v_1,s_0)   \right)  - \left[e^{\theta G(v_2/s_0, s_0) } - 1 \right] \hspace{-1mm}  M \left(\G_K(v_2,s_0) , \G_K(v_1,s_0)   \right)  \\
&= O(K^{-1/2}) + \theta \int_{t_0}^{s_0} M(\partial_z \G_K(v_2, u), \G_K(v_1,s_0) ) du  . 
\end{split}
\end{equation*}
Subtracting from the last equation the same equation with $t_0$ replaced with $t$, where $\delta_0^{-1} \geq s_0 \geq t \geq t_0 \geq  \delta_0$, we obtain 
\begin{equation}\label{YB3}
\begin{split}
& \left[e^{\theta G(v_2/t_0, t_0) } - 1 \right] \hspace{-1mm}  M \left(\G_K(v_2,t_0) , \G_K(v_1,s_0)   \right)  - \left[e^{\theta G(v_2/t, t) } - 1 \right] \hspace{-1mm}  M \left(\G_K(v_2,t) , \G_K(v_1,s_0)   \right)   \\
&= O(K^{-1/2}) + \theta \int_{t_0}^{t} M(\partial_z \G_K(v_2, u), \G_K(v_1,s_0) ) du  . 
\end{split}
\end{equation}

{\bf \raggedleft Step 3.} In this step we prove (\ref{YA}) in the case when $s^{\infty}_1 = s^{\infty}_2$. Setting $s_0 = s_1^n$, $v_1 = z_1^n$, $v_2 = z_2^n$ and $K = K_n$ into (\ref{YB1}) and letting $n \rightarrow \infty$, we conclude that 
\begin{equation}\label{YC1}
\begin{split}
&\lim_{n \rightarrow \infty} M \left(\G_K(z^n_2,s_1^n ) , \G_K(z_1^n,s_1^n)   \right) = \frac{-1}{Q(z_2^{\infty}/s_1^{\infty}, s_1^{\infty})} \cdot \oint_{\gamma} dz \frac{ \Phi^+(z/s_1^{\infty}, s_0) e^{\theta G(z/s_1^{\infty}, s_1^{\infty})}    }{2 \pi \i  \theta (z- z^\infty_2) (z- z^{\infty}_1)^2}.
\end{split}
\end{equation}
We mention that the division by $Q(z_2^{\infty}/s_1^{\infty}, s_1^{\infty})$ is fine in view of Remark \ref{Divide}. Next, setting $t_0 = s_2^n$, $s_0 = s_1^n$, $v_1 = z_2^n$, $v_2 = z_1^n$ and $K = K_n$ into (\ref{YB3}) and letting $n \rightarrow \infty$, we conclude that 
\begin{equation*}
\begin{split}
&\lim_{n \rightarrow \infty} \left[e^{\theta G(z^n_2/s_2^n, s_2^n) } - 1 \right]  M \left(\G_K(z^n_2,s_2^n ) , \G_K(z_1^n,s_1^n)   \right)  \\
& = \lim_{n \rightarrow \infty} \left[e^{\theta G(z^n_2/s_1^n, s_1^n) } - 1 \right]  M \left(\G_K(z^n_2,s_1^n ) , \G_K(z_1^n,s_1^n)   \right),
\end{split}
\end{equation*}
since $s^{\infty}_1 = s^{\infty}_2$. We may now divide both sides by $e^{\theta G(z^n_2/s_2^n, s_2^n) } - 1$, which is allowed by Remark \ref{Divide}, and apply (\ref{YC1}) to compute the limit of the right side. The result is 
\begin{equation}\label{YC2}
\begin{split}
&\lim_{n \rightarrow \infty} M \left(\G_K(z^n_2,s_2^n ) , \G_K(z_1^n,s_1^n)   \right) = \frac{-1}{Q(z_2^{\infty}/s_1^{\infty}, s_1^{\infty})} \cdot \oint_{\gamma} dz \frac{ \Phi^+(z/s_1^{\infty}, s_0) e^{\theta G(z/s_1^{\infty}, s_1^{\infty})}    }{2 \pi \i  \theta (z- z^\infty_2) (z- z^{\infty}_1)^2}.
\end{split}
\end{equation}

What remains is to show that the right side of (\ref{YC2}) is equal to $C(z^{\infty}_1,s^{\infty}_1;z^{\infty}_2,s^{\infty}_2)$. From (\ref{S31E7}), (\ref{S31E9}) and (\ref{S31E10}) we have 
 \begin{equation}\label{YC3}
R(z/s,s) =   \Phi^-(z/s,s) \cdot e^{- \theta G(z/s,s)} + \Phi^+(z/s,s) \cdot e^{\theta G(z/s,s)}  =  1/s - \theta .
\end{equation}
In particular, we can write the integrand on the right side of (\ref{YC2}) as 
$$\frac{1}{2} \cdot \left[ R(z/s_1^{\infty}, s_1^{\infty}) - Q(z/s_1^{\infty}, s_1^{\infty}) \right] \cdot \frac{1}{2\pi \i \theta (z- z_2^{\infty}) (z -z_1^{\infty})^2}.$$
In view of (\ref{YC3}) the part corresponding to $R(z/s_1^{\infty}, s_1^{\infty})$ integrates to zero by Cauchy's theorem and the fact that $z_i^{\infty}$ are outside of $\gamma$. Combining the last few observations with the formula for $Q$ from (\ref{S7QRecalled}) gives
 \begin{equation}\label{YC4}
\begin{split}
&\lim_{n \rightarrow \infty} M \left(\G_{K_n}(z^n_2,s_2^n ) , \G_{K_n}(z_1^n,s_1^n)   \right) =  \frac{\theta^{-1}}{2 \sqrt{(z_2^{\infty}- a(s_1^{\infty}))(z_2^{\infty}- b(s_1^{\infty}))}} \\
& \times  \oint_{\gamma} dz \frac{  \sqrt{(z- a(s_1^{\infty}))(z- b(s_1^{\infty})) } }{2\pi \i (z- z^\infty_2) (z- z^{\infty}_1)^2 }  = C(z^{\infty}_1,s^{\infty}_1;z^{\infty}_2,s^{\infty}_1).
\end{split}
\end{equation}
We mention that the equality on the second line of (\ref{YC4}) can be deduced by evaluating the integral as (minus) the sum of the residues at $z = z_1^{\infty}$ and $z = z_2^{\infty}$ (there is no residue at infinity). Equation (\ref{YC4}) implies (\ref{YA}) when $s^{\infty}_1 = s^{\infty}_2$.\\

{\bf \raggedleft Step 4.} In the remaining steps we assume that $s^{\infty}_1 > s^{\infty}_2$ and so in particular $s^{\infty}_1 > \delta_0$.  For $s \in [\delta_0, \delta_0^{-1}]$ and $z \in U(\delta, \theta)$ we define the functions
\begin{equation}\label{YD1}
f_n(z,s) = M(\G_{K_n} (z,s), \G_{K_n}(z_1^n, s_1^n) ).
\end{equation}
Note that since $\G_K(z,s)$ is analytic in $z \in U(\delta, \theta)$ for each $s \in [\delta_0, \delta_0^{-1}]$, the same is true for $f_n(z,s)$.

Let $\mathsf{Z}$ be a countable dense subset of $U(\delta, \theta)$ and $\mathsf{S}$ be a countable dense subset of $[\delta_0, \delta_0^{-1}]$. From (\ref{S6MB}) we know that there exists $R(u) > 0$, depending on $\theta, \delta, u$, such that for $s \in [\delta_0, \delta_0^{-1}]$ and $z \in F_u$
\begin{equation}\label{YD2}
\left|f_{n}(z,s) \right| \leq R(u) \mbox{ and } \left|\partial_z f_{n}(z,s) \right| \leq R(u),
\end{equation}
where we recall that $\{F_u\}_{u \geq 1}$ are as in Lemma \ref{DomLip}. In particular, we see that $f_{n}(z,s)$ is a bounded sequence for each $(z,s) \in \mathsf{Z} \times \mathsf{S}$. By a diagonalization argument we may pass to a subsequence $n_h$, such that $\lim_h f_{n_h}(z,s)$ exists for each $(z,s) \in \mathsf{Z} \times \mathsf{S}$. We denote this limit by $f_{\infty}(z,s)$. \\

In the remainder of this step we prove that $f_{\infty}(z,s)$ has a continuous extension to $F_u \times [\delta_0, s_1^{\infty}]$ for each $u \in \mathbb{Z}_{\geq 2}$, which would imply that $f_{\infty}(z,s)$ has a continuous extension to $U(\delta, \theta) \times [\delta_0, s_1^{\infty}]$, which we continue to call $f_{\infty}(z,s)$. We claim that there exists $D(u) > 0$, depending on $\theta, \delta, u$, such that for $(x,s), (y,t) \in (\mathsf{Z} \cap F_u) \times (\mathsf{S} \cap [\delta_0, s^{\infty}_1))$ we have 
\begin{equation}\label{YD3}
|f_{\infty}(x,s) - f_{\infty}(y,t)| \leq D(u) \cdot \left( |s-t| + |x-y| \right).
\end{equation}
If true, (\ref{YD3}) would imply that $f_{\infty}(z,s)$ is uniformly continuous on $(\mathsf{Z} \cap F_u) \times (\mathsf{S} \cap [\delta_0, s^{\infty}_1))$ and hence has a unique continuous extension to $F_u \times [\delta_0, s^{\infty}_1]$, which also satisfies (\ref{YD3}) for $(x,s), (y,t) \in F_u \times [\delta_0, s^{\infty}_1]$. Here, we implicitly used that $\mathsf{Z} \cap F_u$ is dense in $F_u$, since $u \geq 2$, and that $ \mathsf{S} \cap [\delta_0, s^{\infty}_1)$ is dense in $[\delta_0, s^{\infty}_1]$, since $s^{\infty}_1 > \delta_0$ (see the begining of this step).\\

From Lemma \ref{DomLip} there is $\lambda_u > 0$ and for each $x,y \in F_u$ a piecewise smooth contour $\gamma_{x,y}$ connecting $x,y$ such that $\gamma_{x,y} \subset F_u$ and $\| \gamma_{x,y}\| \leq \lambda_v |x-y|$. In view of (\ref{YD2}) we can find a constant $R_z(u)$, depending on $\theta, \delta, u$, such that for $t \in [\delta_0, s_1^{n_h}]$ and $x,y \in F_u$
\begin{equation}\label{YD4}
\left| f_{n_h}(x,t) - f_{n_h}(y,t) \right| = \left| \int_{\gamma_{x,y}} \partial_z f_{n_h}(z,t) dz \right| \leq R_z(u) \cdot |x- y|.
\end{equation}
In addition, from (\ref{YB3}) we can find a constant $R_s(u)$, depending on $\theta, \delta, u$, such that for $z\in F_u$ and $s,t \in [\delta_0, s_1^{n_h}]$
\begin{equation}\label{YD5}
\begin{split}
\left| \left[e^{\theta G(z/s, s)} - 1 \right] f_{n_h}(z,s) - \left[e^{\theta G(z/t, t)} - 1 \right] f_{n_h}(z,t) \right|  \leq O(K_{n_h}^{-1/2}) + R_s(u) |s - t|.
\end{split}
\end{equation}
Combining (\ref{YD5}) and (\ref{PT6}) we conclude that for any sequences $a_h, b_h \in [\delta_0, s_1^{n_h}]$
\begin{equation}\label{YD7}
\limsup_{h \rightarrow \infty} \sup_{z \in F_u} \left| f_{n_h}(z,a_h) - f_{n_h}(z,b_h) \right| \leq \frac{C(u) R(u) + R_s(u)}{c(u)} \cdot \limsup_{ h \rightarrow \infty} |a_h - b_h|.
\end{equation}
Taking the $h \rightarrow \infty$ limit in (\ref{YD4}), using (\ref{YD7}) and applying the triangle inequality we obtain (\ref{YD3}). \\

{\bf \raggedleft Step 5.}  Using that (\ref{YD3}) is satisfied for all $(x,s), (y,t) \in F_u \times [\delta_0, s^{\infty}_0]$, (\ref{YD4}), (\ref{YD7}) and the pointwise convergence of $f_{n_h}$ to $f_{\infty}$ on $(\mathsf{Z} \cap F_u) \times (\mathsf{S} \cap [\delta_0, s^{\infty}_0))$ (which is dense in $F_u \times [\delta_0, s_0^{\infty}]$) we conclude for any sequence $a_h \in [\delta_0, s^{n_h}_1]$ such that $a_h \rightarrow a \in [\delta_0, s_1^{\infty}]$ that
\begin{equation}\label{YE1}
\limsup_{h \rightarrow \infty} \sup_{z \in F_u} \left| f_{n_h}(z, a_h) - f_{\infty}(z, a) \right| = 0.
\end{equation}
In the remainder of this step we prove that $f_{\infty}(z,s) = C(z_1^{\infty},s_1^{\infty}; z,s)$ for all $(z,s) \in U(\theta, \delta) \times [\delta_0, s_0^{\infty}]$. If true, then (\ref{YE1}) applied to $a_h = s^{n_h}_2$ and $a = s^{\infty}_2$ implies (\ref{YA}).\\

From (\ref{YE1}) we have for each $u \in \mathbb{Z}_{\geq 2}$ and $s \in [\delta_0, s_0^{\infty}]$ that $f_{\infty}(\cdot, s)$ is the uniform over $F_u$ limit of analytic functions in $U(\theta, \delta)$. This implies that for each $s \in [\delta_0, s_0^{\infty}]$ the function $f_{\infty}(z,s)$ is analytic in $U(\theta, \delta)$, see \cite[Chapter 2, Theorem 5.2]{SS}. In addition, (\ref{YE1}) being true for all $u \in \mathbb{Z}_{\geq 2}$ implies
\begin{equation}\label{YE2}
\limsup_{h \rightarrow \infty} \sup_{z \in F_u} \left| \partial_z f_{n_h}(z, a_h) - \partial_z f_{\infty}(z, a) \right| = 0,
\end{equation}
in view of \cite[Chapter 2, Theorem 5.3]{SS}. 

From (\ref{YB3}) we have for $z \in F_u$ and any sequence $a_h \in [\delta_0, s^{n_h}_0]$ such that $a_h \rightarrow a \in [\delta_0, s_0^{\infty}]$
\begin{equation}\label{YE3}
\begin{split}
&\lim_{ h \rightarrow \infty} \left[ e^{\theta G(z/s^{n_h}_1,s^{n_h}_1)} - 1 \right] \cdot f_{n_h}(z,s^{n_h}_1) + \theta \int_{a_h}^{s^{n_h}_1} \partial_z f_{n_h}(z,u) du \\
& - \left[ e^{\theta G(z/a_h,a_h)} - 1 \right] \cdot f_{n_h}(z, a_h) = 0.
\end{split}
\end{equation} 
Using the continuity of $e^{\theta G(z/s,s)}$, the convergence in (\ref{YE1}), (\ref{YE2}) and of $s^{n_h}_1$ to $s_1^{\infty}$, we can apply the bounded convergence theorem to (\ref{YE3}) and obtain for all $z \in U(\theta, \delta)$ and $a \in [\delta_0, s_0^{\infty}]$
\begin{equation}\label{YE4}
\begin{split}
& \left[ e^{\theta G(z/s^{\infty}_0,s^{\infty}_0)} - 1 \right] \cdot f_{\infty}(z,s^{\infty}_0) + \theta \int_{a}^{s^{\infty}_0} \partial_z f_{\infty}(z,u) du  - \left[ e^{\theta G(z/a,a)} - 1 \right] \cdot f_{\infty}(z, a) = 0.
\end{split}
\end{equation} 
On the other hand, we have from (\ref{YC4}) (with $z_2^n = z$ and $s_2^n = s_1^n$) and (\ref{YE1}) for $z \in F_u$ that 
\begin{equation}\label{YE5}
\begin{split}
&f_{\infty}(z, s^{\infty}_1) = C(z^{\infty}_1,s^{\infty}_1; z ,s^{\infty}_1).
\end{split}
\end{equation} 
Our work in this step shows that $f_{\infty}(z,s)$ satisfies the conditions of Lemma \ref{LUnique2}, from which we conclude that $f_{\infty}(z,s) = C(z_1^{\infty},s_1^{\infty}; z,s)$ for all  $(z,s) \in U(\theta, \delta) \times [\delta_0, s_0^{\infty}]$.

%-------------------------------------------------------------------------------------------------------------------------------------------------------------------------------------------------
% Section 7.3
%
%-------------------------------------------------------------------------------------------------------------------------------------------------------------------------------------------------
\subsection{Proof of Lemmas \ref{FProp} and \ref{LUnique2}}\label{Section7.3} In this section we give the proofs of the two lemmas from Section \ref{Section7.1}.

\begin{proof}[Proof of Lemma \ref{FProp}] The first three properties in the lemma follow immediately from the second formula in (\ref{Transport}) and the definition of the square root in Definition \ref{Branch}. Computing the derivative of $F(x;s,t)$ with respect to $x$ on $\mathbb{R} \setminus [a(t), b(t)]$ we get
\begin{equation}\label{FDer}
F'(x;s,t) = \frac{s+t}{2t} + \frac{s-t}{2t} \cdot \frac{2x - a(t) - b(t)}{\sqrt{(x- a(t))(x-b(t))}}.
\end{equation}
In view of Definition \ref{Branch} each summand is positive on $(-\infty, a(t))$ and $(b(t), \infty)$ and so $F'(x;s,t) > 0$ for $x \in \mathbb{R} \setminus [a(t), b(t)]$. In addition, by a direct computation using (\ref{S7DefAB}) and (\ref{Transport}) we have $F(a(t); s,t) = c(s,t)$ and $F(b(t); s,t) = d(s,t)$. This proves the last two properties in the lemma.
\end{proof}

\begin{proof}[Proof of Lemma \ref{LUnique2}] Let us define $\tilde{f}(z,s) = C(z_1, q; z,s)$. If we can show that $\tilde{f}(z,s)$ satisfies (\ref{S6LID}), then $g(z,s):= \tilde{f}(z,s) - f(z,s)$ also satisfies (\ref{S6LID}) and $g(z,q) = 0$ for all $z \in U(\delta, \theta)$. From Lemma \ref{LUnique} we conclude that $g(z,s) = 0$ and hence $\tilde{f}(z,s) = f(z,s)$ as desired. We have thus reduced the problem to showing that $\tilde{f}(z,s)$ satisfies (\ref{S6LID}). Since both sides of (\ref{S6LID}) clearly agree when $s = q$, it suffices to show that the $s$-derivatives agree, which is equivalent to
\begin{equation}\label{WX1}
\begin{split}
&\partial_s \left( [e^{\theta G(z/s, s)} - 1] \tilde{f}(z,s) \right) + \theta \partial_z \tilde{f}(z,s)   = 0.
\end{split}
\end{equation}

Let $\gamma$ be a positively oriented contour that encloses $[a(s_1), b(s_1)]$, and excludes the points $z_1$ and $F(z;q,s)$ for $s \in [p,q]$. Notice that such a choice is possible in view of Lemma \ref{FProp}. We then have the following formula for $\tilde{f}(z,s)$
\begin{equation}\label{WX2}
\begin{split}
&\tilde{f}(z,s) =  \frac{\theta^{-1}}{2 \sqrt{(F(z;q,s) - a(q))(F(z;q,s)- b(q))}} \cdot \oint_{\gamma} d\zeta \frac{  \sqrt{(\zeta- a(q))(\zeta- b(q)) } }{2\pi \i (\zeta- F(z;q,s)) (\zeta- z_1)^2 } \\
& \times \frac{[e^{\theta G(z/s, s)} + 1 ]^2 \cdot \theta s_1 - [e^{\theta G(z/s, s)} - 1 ]^2}{[e^{\theta G(z/s, s)} + 1 ]^2 \cdot \theta s - [e^{\theta G(z/s, s)} - 1 ]^2},
\end{split}
\end{equation}
which one observes by directly computing the integral as (minus) the residues at $\zeta = F(z;q,s)$ and $\zeta = z_1$ (there is no residue at infinity). In addition, from (\ref{Transport}) we have
\begin{equation}\label{WX3}
F(z;q,s) = z + \frac{\theta (q - s)}{e^{\theta G(z/s,s)} - 1} = \frac{e^{\theta G(z/s,s)}}{e^{\theta G(z/s,s)} + 1} + \frac{\theta q}{e^{\theta G(z/s,s)} - 1},
\end{equation}
where the second equality can be checked directly using (\ref{S31E14}).

Let us denote 
\begin{equation}\label{WX4}
\begin{split}
&H(\zeta,w, s) = \frac{\theta^{-1}}{2 \sqrt{\left( \frac{w}{w+1} + \frac{\theta q}{w-1} - a(q) \right) \cdot \left(  \frac{w}{w+1} + \frac{\theta q}{w-1} - b(q)  \right)}}\\
&\times \frac{1}{\zeta - \frac{w}{w+1} - \frac{\theta q}{w - 1}} \cdot \frac{(w+1)^2 \cdot \theta s_1 - (w-1)^2}{(w+1)^2\cdot \theta s - (w - 1)^2} \mbox{ and } W(z,s) = e^{\theta G(z/s,s)}.
\end{split}
\end{equation}
Combining (\ref{WX2}), (\ref{WX3}) and (\ref{WX4}), we see that (\ref{WX1}) is equivalent to
\begin{equation}\label{WX5}
\begin{split}
\frac{1}{2\pi \i} \oint_{\gamma}d\zeta \frac{\Phi^+(\zeta/q,q)  e^{\theta G(\zeta/q, q)}}{(\zeta-z_1)^2} \cdot \left[A(\zeta, z,s) + B(\zeta, z,s) + C(\zeta, z,s) \right] = 0,
\end{split}
\end{equation}
where we have set
\begin{equation}\label{WX6}
\begin{split}
&A(\zeta, z,s) =  \left[ (W(z,s) - 1) \cdot \partial_s W(z,s) + \theta \partial_z W(z,s)\right] \cdot \partial_w H(\zeta, W(z,s), s) ,
\end{split}
\end{equation}
\begin{equation}\label{WX7}
\begin{split}
&B(\zeta, z,s) =   (W(z,s) - 1) \cdot \partial_s H(\zeta, W(z,s), s) \\
&= \frac{(-\theta)(W(z,s) - 1)  (W(z,s) + 1)^2}{(W(z,s)+1)^2\cdot \theta s - (W(z,s) - 1)^2} \cdot H(\zeta, W(z,s), s),
\end{split}
\end{equation}
\begin{equation}\label{WX8}
\begin{split}
&C(\zeta, z,s) = \partial_s W(z,s) \cdot H(\zeta, W(z,s), s).
\end{split}
\end{equation}

In the remainder of the proof we show that
\begin{equation}\label{WX9}
A(\zeta, z,s) = 0 \mbox{ and } B(\zeta, z,s) + C(\zeta, z,s) = 0,
\end{equation}
which proves (\ref{WX5}) and hence the lemma. Recall from (\ref{SpecEqn}) that
\begin{equation*}
\begin{split}
 &\partial_s G(z/s,s) - (z/s^2) \partial_z G(z/s,s)  = -\frac{\theta \cdot  \partial_z G(z/s,s)}{s\left(e^{\theta G(z/s,s)} - 1 \right) }.
\end{split}
\end{equation*}
The latter gives
$$(W(z,s) - 1) \cdot \partial_s W(z,s) =(W(z,s) - 1) \cdot W(z,s) \cdot \theta \cdot \left( \partial_s G(z/s,s) - (z/s^2) \partial_z G(z/s,s) \right) $$
$$= -(\theta^2/ s) \cdot W(z,s) \cdot \partial_z G(z/s,s) = - \theta \partial_z W(z,s),$$
which in turn proves the first equality in (\ref{WX9}). Using the formulas in (\ref{WX7}) and (\ref{WX8}) we see that to prove the second equality in (\ref{WX9}) it suffices to show that 
\begin{equation*}
\partial_s W(z,s) = \frac{ \theta (W(z,s) - 1) \cdot (W(z,s) +1)^2  }{(W(z,s)+1)^2 \cdot \theta s - (W(z,s)-1)^2},
\end{equation*}
which one verifies by a direct computation using (\ref{S31E14}).
\end{proof}

%-------------------------------------------------------------------------------------------------------------------------------------------------------------------------------------------------
% Section 8
%
%-------------------------------------------------------------------------------------------------------------------------------------------------------------------------------------------------
\section{Convergence to the GFF}\label{Section8} The goal of this section is to prove Theorems \ref{ThmLLN} and \ref{ThmMain}. In Section \ref{Section8.1} we recall the formulation and some basic properties of the Gaussian free field. In Section \ref{Section8.2} we prove an analogue of Theorem \ref{ThmMain}, see Theorem \ref{MainThmNew}, which is in terms of a slightly different height function. In Section \ref{Section8.3} we give the proofs of Theorems \ref{ThmLLN} and \ref{ThmMain}, and the latter will be deduced by a simple change of variables from Theorem \ref{MainThmNew}.

%-------------------------------------------------------------------------------------------------------------------------------------------------------------------------------------------------
% Section 8.1
%
%-------------------------------------------------------------------------------------------------------------------------------------------------------------------------------------------------
\subsection{The Gaussian free field}\label{Section8.1} In this section we briefly recall the formulation and some basic properties of the Gaussian free field (GFF). Our discussion will follow the exposition in \cite[Section 8.1]{DK19}, which in turn is based on \cite[Section 4.5]{BGJ}. For a more thorough background on the subject we refer to \cite{Sheff}, \cite[Section 4]{Dub}, \cite[Section 2]{HMP}, and the references therein.

\begin{definition} The Gaussian free field with Dirichlet boundary conditions in the upper half-plane $\mathbb{H}$ is a (generalized) centered Gaussian field $\mathcal{F}$ on $\mathbb{H}$ with covariance given by
\begin{equation}\label{GFFCov}
\mathbb{E} \left[ \mathcal{F}(z) \mathcal{F}(w)\right] = - \frac{1}{2\pi} \log \left| \frac{z - w}{z - \overline{w}} \right|, \hspace{2mm} z,w \in \mathbb{H}.
\end{equation}
\end{definition}
We remark that $\mathcal{F}$ can be viewed as a probability Gaussian measure on a suitable class of generalized functions on $\mathbb{H}$; however, one cannot define the value of $\mathcal{F}$ at a given point $z \in \mathbb{H}$ (this is related to the singularity of (\ref{GFFCov}) at $z = w$).

Even though $\mathcal{F}$ does not have a pointwise value; one can define the (usual distributional) pairing $\mathcal{F}(\phi)$, whenever $\phi$ is a smooth function of compact support, and the latter is a mean zero normal random variable. In general, one can characterize the distribution of $\mathcal{F}$ through pairings with test functions as follows. If $\{\phi_k\}$ is any sequence of compactly supported smooth functions on $\mathbb{H}$, then the pairings $\{ \mathcal{F}(\phi_k) \}$ form a sequence of centered normal variables with covariance
$$\mathbb{E} \left[ \mathcal{F}(\phi_k) \mathcal{F} (\phi_l)\right] = \int_{\mathbb{H}^2} \phi_k(z) \phi_l(w) \left( - \frac{1}{2\pi} \log \left| \frac{z - w}{z - \overline{w}} \right| \right)|dz|^2|dw|^2.$$

An important property of $\mathcal{F}$ that will be useful for us is that it can be integrated against smooth functions on smooth curves $\gamma \subset \mathbb{H}$. We isolate the statement in the following lemma.
\begin{lemma}\label{GFFCH}\cite[Lemma 4.6]{BGJ} Let $\gamma \subset \mathbb{H}$ be a smooth curve and $\mu$ a measure on $\mathbb{H}$, whose support is $\gamma$ and whose density with respect to the natural (arc length) measure on $\gamma$ is a given by a smooth function $g(z)$ such that
\begin{equation}\label{VarCont}
\iint\limits_{\gamma \times \gamma} g(z) g(w) \left(  - \frac{1}{2\pi} \log \left| \frac{z - w}{z - \overline{w}}\right| \right) dz dw < \infty.
\end{equation}
Then
$$ \int_{\mathbb{H}} \mathcal{F} d\mu = \int_\gamma \mathcal{F}(u) g(u) du$$
 is a well-defined Gaussian centered random variable of variance given by (\ref{VarCont}).
 Moreover, if we have two such measures $\mu_1$ and $\mu_2$ (with two curves $\gamma_1$ and $\gamma_2$ and two densities $g_1$ and $g_2$), then $X_1 = \int_{\gamma_1} \mathcal{F}(u) g_1(u)du$, $X_2 = \int_{\gamma_2} \mathcal{F}(u) g_2(u)du$ are jointly Gaussian with covariance 
$$\mathbb{E}[X_1 X_2] = \iint\limits_{\gamma_1 \times \gamma_2} g_1(z) g_2(w) \left(  - \frac{1}{2\pi} \log \left| \frac{z - w}{z - \overline{w}}\right| \right) dz dw.$$
\end{lemma}

Another property of $\mathcal{F}$ that we require is that it behaves well under bijective maps, which leads to the notion of a pullback.
\begin{definition}\label{Pullback}
Given a domain $D$ and a bijection $\Omega: D \rightarrow \mathbb{H}$, the pullback $\mathcal{F} \circ \Omega$ is a generalized centered Gaussian field on $D$ with covariance
$$\mathbb{E} \left[\mathcal{F}(\Omega(z)) \mathcal{F}(\Omega(w)) \right] =  - \frac{1}{2\pi} \log \left| \frac{\Omega(z) - \Omega(w)}{\Omega(z) - \overline{\Omega}(w)} \right|, \hspace{2mm} z,w \in D.$$
Integrals of $\mathcal{F} \circ \Omega$ with respect to measures can be computed through
$$ \int_D (\mathcal{F} \circ \Omega)d\mu = \int_{\mathbb{H}} \mathcal{F} d\Omega(\mu),$$
where $d\Omega(\mu)$ stands for the pushforward of the measure $\mu$.
\end{definition}
The above definition immediately implies the following analogue of Lemma \ref{GFFCH}.
\begin{lemma}\label{GFFCD}\cite[Lemma 4.8]{BGJ} In the notation of Definition \ref{Pullback}, let $\mu$ be a measure on $D$ whose support is a smooth curve $\gamma$ and whose density with respect to the natural (length) measure on $\gamma$ is given by a smooth function $g(z)$ such that
\begin{equation}\label{VarContD}
\iint\limits_{\gamma \times \gamma} g_1(z) g_2(w) \left(  - \frac{1}{2\pi} \log \left| \frac{\Omega(z) - \Omega(w)}{\Omega(z) - \overline{\Omega}(w)}\right| \right) dz dw < \infty.
\end{equation}
Then
$$ \int_{D} (\mathcal{F} \circ \Omega) d\mu = \int_\gamma \mathcal{F}(\Omega(u)) g(u) du$$
 is a well-defined Gaussian centered random variable of variance given by (\ref{VarContD}).
 Moreover, if we have two such measures $\mu_1$ and $\mu_2$ (with two curves $\gamma_1$ and $\gamma_2$ and two densities $g_1$ and $g_2$), then $X_1 = \int_{\gamma_1} \mathcal{F}(\Omega(u)) g_1(u)du$, $X_2 = \int_{\gamma_2} \mathcal{F}(\Omega(u)) g_2(u)du$ are jointly Gaussian with covariance 
$$\mathbb{E}[X_1 X_2] = \iint\limits_{\gamma_1 \times \gamma_2} g_1(z) g_2(w) \left(  - \frac{1}{2\pi} \log \left| \frac{\Omega(z) - \Omega(w)}{\Omega(z) - \overline{\Omega}(w)}\right| \right) dz dw.$$
\end{lemma}

%-------------------------------------------------------------------------------------------------------------------------------------------------------------------------------------------------
% Section 8.2
%
%-------------------------------------------------------------------------------------------------------------------------------------------------------------------------------------------------
\subsection{From $\P$ to the GFF}\label{Section8.2} In this section we prove an analogue of Theorem \ref{ThmMain} -- this is Theorem \ref{MainThmNew} below. Theorem \ref{MainThmNew} differs from Theorem \ref{ThmMain} in two ways. Firstly, it is formulated more generally, so that unlike Theorem \ref{ThmMain} where $s_1, \dots, s_m$ and the test functions $f_1, \dots, f_m$ are fixed, will now allow the latter to vary together with $K$ as long as they converge in the appropriate sense. Secondly, Theorem \ref{MainThmNew} is formulated in terms of a slightly different height function $\hat{\mathcal{H}}_K$, which is defined in equation (\ref{S8Height}) below. This choice is made to ease the application of Theorem \ref{MainTechThm} and also to simplify various formulas in the proof.\\

We start by introducing some useful notation. For $s > 0$ we recall from (\ref{S7DefAB}) that
\begin{equation}\label{S8DefAB}
a(s) = s \cdot z_-(s) = (1/2) (1- \theta s) - \sqrt{\theta s}  \mbox{ and } b(s) = s \cdot z_+(s) = (1/2) (1- \theta s) - \sqrt{\theta s} ,  
\end{equation}
where $z_{\pm}(s)$ were defined in (\ref{S31E12}). We also introduce the domain
\begin{equation}\label{WA1}
\hat{\mathcal{D}} = \{(x,s) \in \mathbb{R}^2: s > 0 \mbox{ and } x \in (a(s), b(s)) \},
\end{equation}
and also the map $\hat{\Omega}: \hat{\mathcal{D}} \rightarrow \mathbb{H}$ via
\begin{equation}\label{WA2}
\hat{\Omega}(x,s) = \frac{1 - \theta s}{2} - x + \i \cdot \sqrt{x - a(s)} \cdot \sqrt{ b(s) - x}.
\end{equation}
One readily verifies that $\hat{\Omega}(x,s)$ defines a smooth bijection between $\hat{\mathcal{D}}$ and $\mathbb{H}$ with inverse 
\begin{equation}\label{WA3}
\hat{\Omega}^{-1}(z) = \left( \frac{1 - |z|^2}{2} - \Re[z], \frac{|z|^2}{\theta} \right).
\end{equation}
We finally introduce the height function $\hat{\mathcal{H}}_K(x,s)$ for $(x,s) \in \mathbb{R} \times (0,\infty)$ via
\begin{equation}\label{S8Height}
\hat{\mathcal{H}}_K(x,s) = \sum_{i = 1}^n {\bf 1} \{\ell^n_i \geq x K \}, \mbox{ where } n = \lceil s K \rceil
\end{equation}
and as usual $(\ell^1, \ell^2, \dots)$ is distributed according to $\P$ as in Definition \ref{BKCC}. 

With the above notation we can state the main result of the section.
\begin{theorem}\label{MainThmNew}  Fix $\theta > 0$, assume that $(\ell^1, \ell^2, \dots)$ is distributed according to $\P$ as in Definition \ref{BKCC}, and let $\hat{\mathcal{H}}_K$ be as in (\ref{S8Height}). Then, the centered random height functions
$$\sqrt{\theta \pi} \cdot \left( \hat{\mathcal{H}}_K (s,x) - \mathbb{E} \left[ \hat{\mathcal{H}}_K(s,x) \right] \right)$$
converge to the pullback of the Gaussian free field with Dirichlet boundary condition on the upper half-plane $\mathbb{H}$ with respect to the map $\hat{\Omega}$ from (\ref{WA2}) in the following sense. Fix $m \in \mathbb{N}$, numbers $s_1, \dots, s_m \in (0, \infty)$, entire functions $f_1, \dots, f_m$ as well as sequences $s_i^K \in (0, \infty)$, and entire functions $f_i^K$ for $i \in \llbracket 1, m \rrbracket$ such that 
$$\lim_K s_i^K = s_i \mbox{ and } \lim_K f_i^K(z) = f_i(z),$$
where the latter convergence happens uniformly over compact subsets of $\mathbb{C}$. 

With the above data, the random variables
\begin{equation} \label{HeightPair}
\int_{\mathbb{R}} \sqrt{\theta \pi} \left( \hat{\mathcal{H}}_K(x, s_i^K) - \mathbb{E} \left[ \hat{\mathcal{H}}_K(x, s_i^K)  \right] \right) f^K_i(x) dx, \hspace{3mm} i \in \llbracket 1, m \rrbracket,
\end{equation}
converge in the sense of moments to the joint distribution of the similar averages
$$\int_{a(s_i)}^{b(s_i)} \mathcal{F} \left( \hat{\Omega}(x,s) \right) f_i(x) dx, \hspace{3mm} i \in \llbracket 1, m \rrbracket$$
of the pullback of the GFF. In the above formula $a(s), b(s)$ are as in (\ref{S8DefAB}). 

Equivalently, the variables in (\ref{HeightPair}) converge jointly in the sense of moments to a Gaussian vector $(X_1, \dots, X_m)$ with mean zero and covariance
\begin{equation}\label{CovGFF}
\mathrm{Cov}(X_i, X_j)= \int_{a(s_i)}^{b(s_i)} \int_{a(s_j)}^{b(s_j)} f_i(x) f_j(y) \left( - \frac{1}{2\pi} \log \left| \frac{\hat{\Omega}(x,s_i) - \hat{\Omega}(y,s_j)  }{\hat{\Omega}(x,s_i) - \overline{\hat{\Omega}}(y,s_j) } \right|\right) dy dx.
\end{equation}
\end{theorem}
\begin{proof} For $i \in \llbracket 1, m \rrbracket$ and $K \geq 1$ we set
$$\mathcal{L}_i^K = \int_{\mathbb{R}} \left( \hat{\mathcal{H}}_K(x, s_i^K) - \mathbb{E} \left[ \hat{\mathcal{H}}_K(x, s_i^K)  \right] \right) f^K_i(x) dx.$$
Since $\mathcal{L}_i^K$ are centered, to prove the theorem it suffices to show that all third and higher order joint cumulants of $\mathcal{L}_1^K, \dots, \mathcal{L}_m^K$ converge to zero and also
\begin{equation}\label{GA1}
\lim_{K \rightarrow \infty} \mathrm{Cov} \left(\mathcal{L}_i^K, \mathcal{L}_j^K \right) =  \int_{a(s_i)}^{b(s_i)} \int_{a(s_j)}^{b(s_j)} f_i(x) f_j(y) \left( - \frac{\theta^{-1}}{2\pi^2} \cdot \log \left| \frac{\hat{\Omega}(x,s_i) - \hat{\Omega}(y,s_j)  }{\hat{\Omega}(x,s_i) - \overline{\hat{\Omega}}(y,s_j) } \right|\right) dy dx.
\end{equation}
For clarity we split the proof into six steps. In Step 1 we show that all third and higher order joint cumulants of $\mathcal{L}_1^K, \dots, \mathcal{L}_m^K$ converge to zero and find a formula for the left side of (\ref{GA1}), see (\ref{HA0}). In Step 2 we reduce the problem to showing that the formula in (\ref{HA0}) agrees with the left side of (\ref{GA1}), see (\ref{HB1}), and find an alternative expression for the LHS of (\ref{HB1}), see (\ref{HB3}). In Steps 3 and 4 we establish (\ref{HB1}) in the case when the two levels, called $s_1$ and $s_2$ in that equation, are distinct, while in Steps 5 and 6 we prove (\ref{HB1}) when the two levels are the same.\\

{\bf \raggedleft Step 1.} Let $\delta \in (0,1/2)$, $\delta_0 = 2\delta$ be sufficiently small so that $s_i^K \in [\delta_0, \delta_0^{-1}]$ for all $ i \in \llbracket 1, m \rrbracket$ and $K \geq 1$, and let $U(\delta, \theta)$ be as in (\ref{S6DefU}). In this step we prove that all third and higher order joint cumulants of $\mathcal{L}_1^K, \dots, \mathcal{L}_m^K$ converge to zero as $K \rightarrow \infty$ and also
\begin{equation}\label{HA0}
\begin{split}
&  \lim_{K \rightarrow \infty} \mathrm{Cov} \left( \mathcal{L}^K_{i}, \mathcal{L}^K_{j} \right) =  \frac{1}{(2\pi \i)^2} \oint_{\gamma} \oint_{\gamma} C(z_1, s_i; z_2, s_j) \cdot R_i(z_1) R_j(z_2) dz_1 dz_2.
\end{split}
\end{equation}
In equation (\ref{HA0}) we have that $C(z_1, s_1; z_2, s_2)$ is as in (\ref{LimCov}) and $\gamma$ is a positively oriented contour that encloses $[-\theta \delta^{-1}, \delta^{-2}]$ and is contained in $U(\delta, \theta)$. In addition, we have that $R_i$ is the primitive of $f_i$, i.e. $R'_i(z) = f_i(z)$, such that $R_i(0) = 0$ for $i \in \llbracket 1, m \rrbracket$. We similarly denote by $R_i^K(z)$ the primitive of $f_i^K(z)$ such that $R_i^K(0) = 0$ for $i \in \llbracket 1, m \rrbracket$ and $K \geq 1$.

Note that if $K \geq \delta_0^{-1}$, we have that $\hat{\mathcal{H}}_K(x,s)$ is deterministic for $x \not \in [-\theta \delta^{-1}, 1]$ and so we conclude that for $K \geq \delta_0^{-1}$ and $i \in \llbracket 1, m \rrbracket$
\begin{equation}\label{HA1}
\int_{\mathbb{R}}  \hspace{-1mm}\left( \hat{\mathcal{H}}_K(x, s_i^K) - \mathbb{E} \hspace{-1mm}\left[ \hat{\mathcal{H}}_K(x, s_i^K)  \right] \right)  \hspace{-1mm} f^K_i(x) dx = \int_{-\theta \delta^{-1}}^1  \hspace{-1mm} \left( \hat{\mathcal{H}}_K(x, s_i^K) - \mathbb{E}\hspace{-1mm} \left[ \hat{\mathcal{H}}_K(x, s_i^K)  \right] \right)  \hspace{-1mm} f^K_i(x) dx.
\end{equation}
We next observe for $K \geq \delta_0^{-1}$ and $i \in \llbracket 1, m \rrbracket$ that
\begin{equation}\label{HA2}
\int_{-\theta \delta^{-1}}^1\hat{\mathcal{H}}_K(x, s_i^K)  f^K_i(x) dx = \sum_{j = 1}^{n} \int_{\ell_{j+1}^n/K}^{\ell_{j}^n/K} j \cdot f^K_i(x) dx =  \sum_{j = 1}^n R^K_i(\ell_j^n/K) - n \cdot R^K_i(-\theta  \delta^{-1}),
\end{equation}
where $\ell_{n+1}^n = -\theta  \delta^{-1} K $ and $n = \lceil s_i^K \cdot K \rceil$. Using (\ref{HA1}), (\ref{HA2}), the residue theorem and the formula for $\G_K(z,s)$ in (\ref{S6DefG}) we conclude for $K \geq \delta_0^{-1}$ and $i \in \llbracket 1, m \rrbracket$
\begin{equation}\label{HA3}
\int_{\mathbb{R}} \left( \hat{\mathcal{H}}_K(x, s_i^K) - \mathbb{E} \left[ \hat{\mathcal{H}}_K(x, s_i^K)  \right] \right) f^K_i(x) dx = \frac{1}{2\pi \i} \oint_{\gamma} R^K_i(z) \cdot \G_K(z, s_i^K) dz,
\end{equation}
where $\gamma$ is as above. From (\ref{HA3}) and the linearity of cumulants, see (\ref{S3Linearity}), we conclude for $k \geq 1$
\begin{equation}\label{HA4}
M \left( \mathcal{L}^K_{i_1}, \dots , \mathcal{L}^K_{i_k}\right) = \frac{1}{(2\pi \i)^k} \oint_{\gamma} \cdots \oint_{\gamma} M \left(  \G_K(z_1, s_{i_1}^K), \dots,  \G_K(z_k, s_{i_k}^K) \right) \cdot \prod_{r = 1}^k R^K_{i_r}(z_r) dz_r.
\end{equation}
Since the third and higher order joint cumulants of $\G_K$'s vanish as $K \rightarrow \infty$ (in view of Theorem \ref{MainTechThm} and the fact that $\gamma \subset U(\delta, \theta)$) we conclude from (\ref{HA4}) that the third and higher order cumulants of $\mathcal{L}_i^K$ vanish. Finally, taking the limit on both sides of (\ref{HA4}) when $k = 2$ and applying Theorem \ref{MainTechThm} to the right side we obtain (\ref{HA0}). We mention that in the last two limits we used that $R_i^K(z)$ converge uniformly over $\gamma$ to $R_i(z)$ as $K \rightarrow \infty$ due to the assumed convergence of $f_i^K(z)$.\\

{\bf \raggedleft Step 2.} From our work in Step 1 we see that to complete the proof it suffices to show that the right side of (\ref{HA0}) agrees with the right side of (\ref{GA1}). Using the formula for $C(z_1, s_1; z_2, s_2)$ from (\ref{LimCov}), we see that it suffices to show for $\delta_0^{-1} \geq s_1 \geq s_2 \geq \delta_0$ that 
\begin{equation}\label{HB1}
\begin{split}
&\frac{1}{(2\pi \i)^2} \oint_{\gamma} \oint_{\gamma} \frac{1}{2(z_1- x_2)^2} \left( - 1 + \frac{(z_1 - b(s_1))(x_2 - a(s_1)) + (x_2 - b(s_1)) (z_1- a(s_1))}{2\sqrt{(z_1 - a(s_1))(z_1- b(s_1))} \cdot \sqrt{(x_2- a(s_1))(x_2 - b(s_1))} }  \right) \\
&\times \frac{[e^{\theta G(z_2/s_2, s_2)} + 1 ]^2 \cdot \theta s_1 - [e^{\theta G(z_2/s_2, s_2)} - 1 ]^2}{[e^{\theta G(z_2/s_2, s_2)} + 1 ]^2 \cdot \theta s_2 - [e^{\theta G(z_2/s_2, s_2)} - 1 ]^2} \cdot R_1(z_1) R_2(z_2) dz_1 dz_2 \\
& =  \int_{a(s_1)}^{b(s_1)} \int_{a(s_2)}^{b(s_2)} R'_1(x) R'_2(y) \left( - \frac{1}{2\pi^2} \cdot \log \left| \frac{\hat{\Omega}(x,s_1) - \hat{\Omega}(y,s_2)  }{\hat{\Omega}(x,s_1) - \overline{\hat{\Omega}}(y,s_2) } \right|\right) dy dx,
\end{split}
\end{equation}
where $x_2 = F(z_2; s_1, s_2)$ and we recall from (\ref{Transport}) that
\begin{equation}\label{HB2}
\begin{split}
x_2 &= F(z_2;s_1,s_2) = z_2 + \frac{\theta (s_1-s_2)}{e^{\theta G(z/s_2,s_2)} -1} \\
&= z_2 + \frac{(s_1 - s_1)(2z_2 - 1 - s_2 \theta)}{4s_2} + \frac{s_1 -s_2 }{2s_2} \cdot \sqrt{(z_2- a(s_2))(z_2-b(s_2))}.
\end{split}
\end{equation}
The equality in (\ref{HB1}) will be established over the next several steps. By the linearity of both sides in $R_1(z_1)$ and $R_2(z_2)$, we may assume that $R_i(z)$ are real-valued on $\mathbb{R}$, which we do in the sequel. In the remainder of this step we find an alternative formula for the LHS of (\ref{HB1}) -- see (\ref{HB3}) below.\\

We start with the LHS of (\ref{HB1}) and deform the $z_1$ contour so that it traverses $[a(s_1), b(s_1)]$ once in the positive and once in the negative direction. Observe that the square roots are purely imaginary and come with opposite sign when we approach $[a(s_1), b(s_1)]$ from the upper and lower half-planes. On the other hand, the term coming from the $-1$ in the first line of (\ref{HB1}) cancels when we integrate over $[a(s_1), b(s_1)]$ in the positive and negative direction. By Cauchy's theorem we do not change the value of the integral in the process of deformation and so from the bounded convergence theorem we see that LHS of (\ref{HB1}) equals
\begin{equation*}
\begin{split}
&\frac{\i }{(2\pi \i)^2} \oint_{\gamma} \int_{a(s_1)}^{b(s_1)} \frac{R_1(v_1) R_2(z_2) }{(v_1- x_2)^2}\cdot \frac{(v_1 - b(s_1))(x_2 - a(s_1)) + (x_2 - b(s_1)) (v_1- a(s_1))}{2\sqrt{v_1 - a(s_1)} \cdot \sqrt{b(s_1)- v_1} \cdot \sqrt{(x_2- a(s_1))(x_2 - b(s_1))} } \\
&\times \frac{[e^{\theta G(z_2/s_2, s_2)} + 1 ]^2 \cdot \theta s_1 - [e^{\theta G(z_2/s_2, s_2)} - 1 ]^2}{[e^{\theta G(z_2/s_2, s_2)} + 1 ]^2 \cdot \theta s_2 - [e^{\theta G(z_2/s_2, s_2)} - 1 ]^2} dv_1 dz_2,
\end{split}
\end{equation*}
where we mention that we relabelled $z_1$ to $v_1$ to emphasize that it is now real. We integrate by parts in the $v_1$ variable and change the order of the integrals, which leads to 
\begin{equation}\label{HB3}
\begin{split}
&\mbox{LHS of (\ref{HB1})} = \frac{-\i }{(2\pi \i)^2}\int_{a(s_1)}^{b(s_1)} \oint_{\gamma}  \frac{R'_1(v_1) R_2(z_2) }{(v_1- x_2)}\cdot \frac{\sqrt{v_1 - a(s_1)}\sqrt{b(s_1)- v_1}}{\sqrt{(x_2- a(s_1))(x_2 - b(s_1))}} \\
&\times \frac{[e^{\theta G(z_2/s_2, s_2)} + 1 ]^2 \cdot \theta s_1 - [e^{\theta G(z_2/s_2, s_2)} - 1 ]^2}{[e^{\theta G(z_2/s_2, s_2)} + 1 ]^2 \cdot \theta s_2 - [e^{\theta G(z_2/s_2, s_2)} - 1 ]^2} dz_2 dv_1.
\end{split}
\end{equation}

{\bf \raggedleft Step 3.} In this step we assume that $s_1 > s_2$ and proceed to find an alternative formula for the LHS of (\ref{HB1}) -- see (\ref{HC14}) below. \\

\begin{figure}[h]
\centering
\scalebox{0.7}{\includegraphics{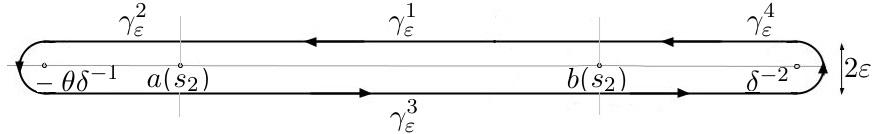}}
\caption{ The contour $\gamma_{\varepsilon} = \cup_{i = 1}^4 \gamma_{\varepsilon}^i$.}
\label{S8_1}
\end{figure}
For $\varepsilon > 0$ we let $\gamma_{\epsilon}$ be the positively oriented contour that goes around $[- \theta \delta^{-1}, \delta^{-2}]$ and is exactly distance $\varepsilon$ away from this segment. In particular, $\gamma_{\varepsilon}$ consists of two segments connecting $-\theta \delta^{-1} +\i \varepsilon$ with $\delta^{-2} + \i \varepsilon$ and $-\theta \delta^{-1} -\i \varepsilon$ with $\delta^{-2} - \i \varepsilon$, as well as two half-circles centered at $-\theta \delta^{-1}$ and $\delta^{-2}$ of radius $\varepsilon$, see Figure \ref{S8_1}. Let $\gamma_{\varepsilon}^1$ be the segment connecting $b(s_2) + \i \varepsilon$ to $a(s_2) + \i \varepsilon$, $\gamma^3_{\varepsilon}$ be the segment connecting $a(s_2) - \i \varepsilon$ to $b(s_2) + \i \varepsilon$, $\gamma_{\varepsilon}^2$ be the portion of $\gamma_{\varepsilon}$ to the left of $a(s_2)$ and $\gamma_{\varepsilon}^4$ the portion of $\gamma_{\varepsilon}$ to the right of $b(s_2)$. Note that $\gamma_{\varepsilon} = \cup_{i = 1}^4 \gamma_{\varepsilon}^i$. By Cauchy's theorem we can deform $\gamma$ in (\ref{HB3}) to $\gamma_{\varepsilon}$ without affecting the value of the integral and so
\begin{equation}\label{HC1}
\begin{split}
&\mbox{LHS of (\ref{HB1})} = \frac{-\i }{(2\pi \i)^2}\int_{a(s_1)}^{b(s_1)}R'_1(v_1) \cdot \left[ A^1_{\varepsilon} + A^{2}_{\varepsilon} + A^{3}_{\varepsilon} + A^4_{\varepsilon} \right] dv_1, \mbox{ where }
\end{split}
\end{equation}
\begin{equation}\label{HC2}
\begin{split}
&A^{i}_{\varepsilon} = \int_{\gamma^i_{\varepsilon}} \frac{ R_2(z_2) }{(v_1- x_2)}\cdot \frac{\sqrt{v_1 - a(s_1)}\sqrt{b(s_1)- v_1}}{\sqrt{(x_2- a(s_1))(x_2 - b(s_1))}} \\
& \times  \frac{[e^{\theta G(z_2/s_2, s_2)} + 1 ]^2 \cdot \theta s_1 - [e^{\theta G(z_2/s_2, s_2)} - 1 ]^2}{[e^{\theta G(z_2/s_2, s_2)} + 1 ]^2 \cdot \theta s_2 - [e^{\theta G(z_2/s_2, s_2)} - 1 ]^2} dz_2.
\end{split}
\end{equation}

From (\ref{HB2}) we have for $v_2 \in (a(s_2), b(s_2))$
\begin{equation}\label{HC3}
\lim_{ \varepsilon \rightarrow 0+} F(v_2 \pm \i \varepsilon; s_1,s_2) =  v_2 + \frac{(s_1 - s_2)(2v_2 - 1 - \theta s_2)}{4s_2} \pm \i \cdot \frac{(s_1 - s_2) \sqrt{v_2 - a(s_2)} \sqrt{b(s_2) - v_2} }{2s_2},
\end{equation}
while for $v_2 \not \in  (a(s_2), b(s_2))$
\begin{equation}\label{HC3.5}
\lim_{ \varepsilon \rightarrow 0+} F(v_2 \pm \i\varepsilon; s_1,s_2) =  v_2 + \frac{(s_1 - s_2)(2v_2 - 1 - \theta s_2)}{4s_2} + \frac{(s_1 - s_2) \sqrt{(v_2 - a(s_2))(v_2 - b(s_2))} }{2s_2}.
\end{equation}
In addition, from (\ref{S31E14}) we have
\begin{equation}\label{S8E14}
W(z,s):= e^{\theta G(z/s,s)} = \frac{(1 - \theta s) - 2 \sqrt{(z - a(s))(z - b(s))}  }{2 (1 - z) },
\end{equation}
and so for $v_2 \in (a(s_2), b(s_2))$ we have
\begin{equation}\label{HC4}
\lim_{ \varepsilon \rightarrow 0+} W(v_2 \pm \i \varepsilon;s_2) =  \frac{1 - \theta s}{2 (1 - v_2)} \mp \frac{1}{1 - v_2} \cdot \sqrt{v_2 -a (s_2)} \cdot \sqrt{b(s_2) - v_2},
\end{equation}
while for $v_2 \not \in (a(s_2), b(s_2))$ we have
\begin{equation}\label{HC5}
\lim_{ \varepsilon \rightarrow 0+} W(v_2 \pm \i \varepsilon;s_2) =  \begin{cases}\frac{(1 - \theta s) - 2 \sqrt{(v_2 - a(s_2))(v_2 - b(s_2))}  }{2 (1 - v_2) }  &\mbox{ if } v_2 \neq 1 \\ 1 &\mbox{ if } v_2 = 1.\end{cases}
\end{equation}
We mention that from (\ref{S8DefAB}) and (\ref{S31E13}) we have that $[a(s_2), b(s_2)] \subset [-\theta s_2, 1]$. \\

If we start letting $\varepsilon \rightarrow 0+$, we see that the integral over $\gamma_{\varepsilon}^{2}$ converges to the integral over $[-\theta \delta^{-1}, a(s_2)]$ once in the positive and once in the negative direction. From (\ref{HC5}) the term on the second line of (\ref{HC2}) converges to the same limit as we approach $[-\theta \delta^{-1}, a(s_2)]$ from the upper and lower half-planes. From (\ref{HC3.5}) the same is true for the term $\frac{ R_2(z_2) }{(v_1- x_2)}$. Finally, from Lemma \ref{FProp} we have for $v_2 \not \in [a(s_2), b(s_2)]$ that $F(v_2; s_1, s_2) \not \in [a(s_1), b(s_1)]$. The latter shows that the term $\sqrt{(x_2- a(s_1))(x_2 - b(s_1))}$ also converges to the same limit as we approach $[-\theta \delta^{-1}, a(s_2)]$ from the upper and lower half-planes. Overall, we conclude that the integrand in (\ref{HC2}) for $i = 2$ converges to the same limit as we approach $[-\theta \delta^{-1}, a(s_2)]$ from the upper and lower half-planes, and since the two integrals come with opposite orientation by the bounded convergence theorem we conclude
\begin{equation}\label{HC6}
\lim_{\varepsilon \rightarrow 0+} A^{i}_{\varepsilon} = 0,
\end{equation}
if $i = 2$. One analogously shows that (\ref{HC6}) holds if $i = 4$. \\

We next note that if $z_2 \in \gamma_{\varepsilon}^{1}$ and the integrand in $A_{\varepsilon}^1$ is $a^1(z_2)$, then $\overline{z_2} \in \gamma_{\varepsilon}^3$ and $a_1(z_2) = \overline{a^3(\overline{z_2})}$, where $a^3$ is the integrand of $A^3_{\varepsilon}$. Here, we also used that $R_2(z)$ is real-valued on $\mathbb{R}$, see the discussion after (\ref{HB2}). The latter shows that the imaginary parts of $A_{\varepsilon}^1$ and $A_{\varepsilon}^3$ are equal, while the the real parts have opposite sign (note that the integrals have opposite orientation). We conclude that
\begin{equation}\label{HC7}
A^{1}_{\varepsilon} + A^3_{\varepsilon} = 2 \i \cdot \Im[A^1_{\varepsilon}].
\end{equation}
We next perform integration by parts in $z_2$ to conclude 
\begin{equation}\label{HC8}
\begin{split}
&A^1_{\varepsilon} =- \i \int_{\gamma_{\varepsilon}^1}   V(z_2,v_1) \cdot R_2'(z_2) dz_2 \\
&- \i \cdot V(b(s_2) + \i \varepsilon, v_1) \cdot R_2( b(s_2) + \i \varepsilon) + \i \cdot V(a(s_2) + \i \varepsilon, v_1) \cdot R_2( a(s_2) + \i \varepsilon), 
\end{split}
\end{equation}
where 
\begin{equation}\label{HC9}
\begin{split}
V(z_2, v_1) = &\log\left(\frac{\sqrt{v_1 - a(s_1)} \sqrt{x_2 - b(s_1)} - \i \sqrt{x_2 - a(s_1)} \sqrt{b(s_1) -v_1}}{\sqrt{v_1 - a(s_1)} \sqrt{x_2 - b(s_1)} + \i \sqrt{x_2 - a(s_1)} \sqrt{b(s_1) -v_1}} \right).
\end{split}
\end{equation}
In (\ref{HC9}) all the logarithms and square roots are with respect to the principal branch as usual. We mention that in deriving (\ref{HC8}) we used that 
\begin{equation}\label{HC10}
\partial_{z_2} x_2 = \partial_{z_2} F(z_2; s_1, s_2) = \frac{[e^{\theta G(z_2/s_2, s_2)} + 1 ]^2 \cdot \theta s_1 - [e^{\theta G(z_2/s_2, s_2)} - 1 ]^2}{[e^{\theta G(z_2/s_2, s_2)} + 1 ]^2 \cdot \theta s_2 - [e^{\theta G(z_2/s_2, s_2)} - 1 ]^2},
\end{equation}
which one verifies by a direct computation using (\ref{HB2}) and (\ref{S8E14}).

Combining (\ref{HC3}), (\ref{HC7}), (\ref{HC8}) and (\ref{HC9}), we conclude by the bounded convergence theorem 
\begin{equation}\label{HC11}
\lim_{\varepsilon \rightarrow 0+} A^{1}_{\varepsilon} + A^3_{\varepsilon} = 2\i \int_{a(s_2)}^{b(s_2)} \tilde{V}(v_2, v_1)dv_2] -\i \tilde{V}(b(s_2), v_1) R_2(b(s_2))  + \i \tilde{V}(a(s_2), v_1)R_2(a(s_2)),
\end{equation}
where 
\begin{equation}\label{HC12}
\begin{split}
\tilde{V}(v_2, v_1) = \log\left|\frac{\sqrt{v_1 - a(s_1)} \sqrt{y_2 - b(s_1)} - \i \sqrt{y_2 - a(s_1)} \sqrt{b(s_1) -v_1}}{\sqrt{v_1 - a(s_1)} \sqrt{y_2 - b(s_1)} + \i \sqrt{y_2 - a(s_1)} \sqrt{b(s_1) -v_1}} \right|
\end{split}
\end{equation}
and we have set 
\begin{equation}\label{HC13}
\begin{split}
y_2 = v_2 + \frac{(s_1 - s_2)(2v_2 - 1 - \theta s_2)}{4s_2} + \i \cdot \frac{(s_1 - s_2) \cdot \sqrt{v_2 - a(s_2)} \cdot \sqrt{b(s_2) - v_2} }{2s_2}.
\end{split}
\end{equation}
We mention that the sign in front of the integral in (\ref{HC11}) got changed from the orientation of $\gamma_{\varepsilon}^1$.

Notice that if $v_2 = b(s_2)$ we have from Lemma \ref{FProp} that $y_2 \in (b(s_1), \infty)$ and so the numerator and denominator in (\ref{HC12}) are conjugates of each other meaning that $\tilde{V}(b(s_2),v_1) = 0$. One shows analogously that $\tilde{V}(a(s_2), v_1) = 0$. Combining the last observation with (\ref{HC1}), (\ref{HC6}), (\ref{HC11}) and the bounded convergence theorem we conclude that 
\begin{equation}\label{HC14}
\begin{split}
&\mbox{LHS of (\ref{HB1})} = \frac{-1}{2 \pi^2}\int_{a(s_1)}^{b(s_1)} \int_{a(s_2)}^{b(s_2)} R'_1(v_1) R_2'(v_2) \tilde{V}(v_2, v_1) dv_2 dv_1.
\end{split}
\end{equation}

{\bf \raggedleft Step 4.} In this step we show that (\ref{HB1}) holds when $s_1 > s_2$. In view of (\ref{HC12}) and (\ref{HC14}), we see that it suffices to show for $v_1 \in (a(s_1, b(s_1))$ and $v_2 \in (a(s_2), b(s_2))$ that
\begin{equation}\label{HD1}
\begin{split}
&\left|\frac{\sqrt{v_1 - a(s_1)} (y_2 - b(s_1)) - \i \sqrt{y_2 - a(s_1)} \sqrt{y_2 - b(s_1)} \sqrt{b(s_1) -v_1}}{\sqrt{v_1 - a(s_1)} (y_2 - b(s_1)) + \i \sqrt{y_2 - a(s_1)} \sqrt{y_2 - b(s_1)} \sqrt{b(s_1) -v_1}} \right|^2   \\
&=   \left| \frac{\hat{\Omega}(v_1,s_1) - \hat{\Omega}(v_2,s_2)  }{\hat{\Omega}(v_1,s_1) - \overline{\hat{\Omega}}(v_2,s_2) } \right|^2.
\end{split}
\end{equation}

We first show that
\begin{equation}\label{HD2}
 \frac{(s_1 - s_2)(2v_2 -1+ \theta s_2)}{4s_2 } + \frac{s_1 + s_2}{2s_2} \cdot \i \sqrt{(v_2 -a(s_2))( b(s_2) - v_2)} =\sqrt{y_2 - a(s_1)}\cdot \sqrt{y_2 - b(s_1)}.
\end{equation}
Since both sides are complex numbers in $\mathbb{H}$ (using Lemma \ref{FProp} for the right side), it suffices to show that their squares are the same. Squaring both sides we seek to show that 
\begin{equation*}
\begin{split}
 &\frac{(s_1 - s_2)^2(2v_2 -1+s_2\theta)^2}{16s_2^2 } + \frac{(s_1 + s_2)^2(v_2-a(s_2))(v_2-b(s_2))}{4s_2^2} \\
&+ \i \cdot  \frac{s_1 + s_2}{s_2}  \cdot \frac{(s_1 - s_2)(2v_2 -1+s_2\theta)}{4s_2 } \cdot \sqrt{(v_2 - a(s_2))(b(s_2) - v_2)}  = (y_2 - a(s_1))(y_2 - b(s_1)),
\end{split}
\end{equation*}
which can now be verified directly by comparing the real and imaginary parts using the formula for $y_2$ from (\ref{HC13}).

Let us denote 
\begin{equation}\label{HD3}
a = \sqrt{v_1 - a(s_1)}, \hspace{2mm} b = \sqrt{b(s_1) - v_1}, \hspace{2mm} x = \sqrt{v_2 - a(s_2)}, \hspace{2mm} y = \sqrt{b(s_2) - v_2},
\end{equation}
from which one readily deduces the equaitons
\begin{equation}\label{HD4}
\begin{split}
&\theta s_1 = (1/4) (a^2+ b^2)^2, \hspace{2mm} \theta s_2 = (1/4) (x^2 + y^2)^2, \hspace{2mm}\\
& v_1 = 1/2 + (1/2)(a^2 - b^2) - (1/8)(a^2+b^2)^2, \hspace{2mm} v_2 = 1/2 + (1/2)(x^2 - y^2) - (1/8)(x^2+y^2)^2\\
&y_2 - b(s_1) = \frac{[(a^2+ b^2)^2 + (x^2 + y^2)^2] \cdot [1/2 + (1/2)(x^2 - y^2) - (1/8)(x^2+y^2)^2 ]}{2 (x^2 + y^2)^2}\\
&- \frac{[1+ (x^2 + y^2)^2/4] \cdot [(a^2+ b^2)^2 - (x^2 + y^2)^2]}{4 (x^2 + y^2)^2}-\frac{1}{2} + \frac{(a^2+ b^2)^2}{8} - \frac{a^2+ b^2}{2} \\
& + \i \cdot \frac{xy[(a^2+ b^2)^2 - (x^2 + y^2)^2]}{2(x^2 + y^2)^2}.
\end{split}
\end{equation}
Using (\ref{HD3}) and (\ref{HD4}) as well as the formula for $\hat{\Omega}(x,s)$ from (\ref{WA2}) we conclude that 
\begin{equation}\label{HD5}
\begin{split}
\mbox{RHS of (\ref{HD1})} = \frac{A^2 + B^2}{A^2 +C^2}, \mbox{ where }
\end{split}
\end{equation}
\begin{equation}\label{HD6}
\begin{split}
&A = (1/2) \cdot (x^2 -y^2 - a^2 + b^2), \hspace{2mm} B = xy-ab, \hspace{2mm} C = ab + xy.
\end{split}
\end{equation}
Using (\ref{HD3}) and (\ref{HD4}) as well as the formula (\ref{HD2}) we conclude that 
\begin{equation}\label{HD7}
\begin{split}
\mbox{LHS of (\ref{HD1})} = \frac{(X_1 +Y_2)^2 + (X_2 - Y_1)^2}{(X_1-Y_2)^2 + (X_2 + Y_1)^2}, \mbox{ where }
\end{split}
\end{equation}
\begin{equation}\label{HD8}
\begin{split}
 X_1 = &\frac{a[(a^2+ b^2)^2 + (x^2 + y^2)^2] \cdot [1/2 + (1/2)(x^2 - y^2) - (1/8)(x^2+y^2)^2 ]}{2 (x^2 + y^2)^2}\\
&- \frac{a [1+ (x^2 + y^2)^2/4] \cdot [(a^2+ b^2)^2 - (x^2 + y^2)^2]}{4 (x^2 + y^2)^2}  -\frac{a}{2} + \frac{a(a^2 + b^2)^2}{8} - \frac{a(a^2 + b^2)}{2},\\
&X_2 = \frac{axy[(a^2+ b^2)^2 - (x^2 + y^2)^2]}{2(x^2 + y^2)^2},
\end{split}
\end{equation}
\begin{equation}\label{HD9}
\begin{split}
 Y_1 = \frac{b[(a^2+b^2)^2 - (x^2 + y^2)^2](x^2 - y^2)}{4 (x^2 + y^2)^2}, \hspace{2mm} Y_2 = \frac{bxy[(a^2+b^2)^2 + (x^2 + y^2)^2]}{2 (x^2 + y^2)^2}.
\end{split}
\end{equation}
One checks directly that 
$$(A^2 +B^2) \cdot \left[ (X_1-Y_2)^2 + (X_2 + Y_1)^2 \right] = (A^2 + C^2) \left[(X_1 +Y_2)^2 + (X_2 - Y_1)^2 \right],$$
which from (\ref{HD5}) and (\ref{HD7}) allows us to conclude (\ref{HD1}).\\

{\bf \raggedleft Step 5.} In the remaining steps we assume that $s_1 = s_2$ and proceed to establish (\ref{HB1}) in this case. To simiplify the notation we write $a = a(s_1) = a(s_2)$ and $b = b(s_1) = b(s_2)$. We claim that for each $v_1 \in (a,b)$ we have
\begin{equation}\label{HE1}
\begin{split}
& \oint_{\gamma}  \frac{R_2(z_2) }{(v_1- z_2)}\cdot \frac{\sqrt{v_1 - a}\sqrt{b- v_1}}{\sqrt{(z_2- a)(z_2 - b)}} dz_2 = -2 \i \int_{a}^b R_2'(v_2) \\
& \times \left[ 2 \log \left(\sqrt{v_1 - a}\sqrt{b-v_2} + \sqrt{v_2-a} \sqrt{b - v_1} \right) - \log|v_1 - v_2| - \log (b-a) \right] dv_2.  
\end{split}
\end{equation}
We will prove (\ref{HE1}) in the next step. Here, we assume its validity and conclude the proof of (\ref{HB1}).

Combining (\ref{HE1}) with (\ref{HB3}), where we remark that the second line disappears and $x_2 = z_2$ as $s_1 = s_2$, we see that to prove (\ref{HB1}) it suffices to show for $v_1, v_2 \in (a,b)$
\begin{equation}\label{HE2}
\begin{split}
- \log \left| \frac{\hat{\Omega}(v_1,s_1) - \hat{\Omega}(v_2,s_1)  }{\hat{\Omega}(v_1,s_1) - \overline{\hat{\Omega}}(v_2,s_1) } \right|  =  & \hspace{1mm} 2 \log \left(\sqrt{v_1 - a}\sqrt{b-v_2} + \sqrt{v_2-a} \sqrt{b - v_1} \right) \\
&- \log|v_1 - v_2| - \log (b-a).
\end{split}
\end{equation}
One readily verifies that $\log|v_1 - v_2| + \log (b-a)$ is equal to 
$$\log \left(\sqrt{v_1 - a}\sqrt{b-v_2} + \sqrt{v_2-a} \sqrt{b - v_1} \right) + \log \left|\sqrt{v_1 - a}\sqrt{b-v_2} - \sqrt{v_2-a} \sqrt{b - v_1} \right|,$$
which shows that to prove (\ref{HE2}) we only need to show that 
\begin{equation}\label{HE3}
\begin{split}
 \left| \frac{\hat{\Omega}(v_1,s_1) - \hat{\Omega}(v_2,s_1)  }{\hat{\Omega}(v_1,s_1) - \overline{\hat{\Omega}}(v_2,s_1) } \right|^2  =  \left(\frac{\sqrt{v_1 - a}\sqrt{b-v_2} - \sqrt{v_2-a} \sqrt{b - v_1}}{\sqrt{v_1 - a}\sqrt{b-v_2} + \sqrt{v_2-a} \sqrt{b - v_1}}\right)^2.
\end{split}
\end{equation}
Using the definition of $\hat{\Omega}(x,s)$ from (\ref{WA2}) we see that 
$$\left|  \frac{\hat{\Omega}(v_1,s_1) - \hat{\Omega}(v_2,s_1)  }{\Omega(v_1,s_1) - \overline{\hat{\Omega}}(v_2,s_1) } \right|^2 = \left|\frac{(v_2 - v_1) + \i \sqrt{v_1 -a}\sqrt{b-v_1} - \i \sqrt{v_2 -a}\sqrt{b-v_2}}{(v_2 - v_1) + \i \sqrt{v_1 -a}\sqrt{b-v_1} + \i \sqrt{v_2 -a}\sqrt{b-v_2}} \right|^2.$$
By a direct computation one verifies that the right side of the last equation and the right side of (\ref{HE3}) are both equal to 
$$\frac{(v_1 - a)(b-v_2) + (v_2-a)(b - v_1) - 2\sqrt{v_1 - a}\sqrt{v_2 - a}\sqrt{b-v_1}\sqrt{b-v_2} }{(v_1 - a)(b-v_2) + (v_2-a)(b - v_1) + 2\sqrt{v_1 - a}\sqrt{v_2 - a}\sqrt{b-v_1}\sqrt{b-v_2}}.$$
This completes the proof of (\ref{HE3}) and hence (\ref{HB1}) when $s_1 = s_2$. \\

{\bf \raggedleft Step 6.} In this final step we prove (\ref{HE1}). In the sequel we assume that $v_1 \in (a,b)$ is fixed. Let $\epsilon > 0$ be such that $(v_1 - \epsilon, v_1 + \epsilon) \subset (a,b)$. For $\varepsilon \in (0, \epsilon)$ we define the contour $\Gamma_{\varepsilon, \epsilon}$ as follows. $\Gamma_{\varepsilon, \epsilon}$ starts from the point $b - \i \varepsilon$ and follows the circle centered at $b$ with radius $\varepsilon$ counterclockwise until the point $b + \i \varepsilon$, afterwards it goes to the left along the segment connecting the points $b + \i \varepsilon$ and $v_1 + \epsilon + \i \varepsilon$; it follows the circle centered at $v_1 + \i \varepsilon$ and radius $\epsilon$ counterclockwise until the point $v_1 - \epsilon + \i \varepsilon$ and goes to the left along the segment connecting $v_1 - \epsilon + \i \varepsilon$ and $a + \i \varepsilon$; it then follows the circle centered at $a$ with radius $\varepsilon$ counterclockwise until the point $a - \i \varepsilon$ and then goes to the right along the segment connecting $a - \i \varepsilon$ and $v_1 - \epsilon - \i \varepsilon$; finally it follows the circle centered at $v_1 - \i \varepsilon$ and radius $\epsilon$ counterclockwise until the point $v_1 - \epsilon - \i \varepsilon$ and goes to the right along the segment connecting $v_1 - \epsilon - \i \varepsilon$ and $b - \i \varepsilon$, see Figure \ref{S8_2}.
\begin{figure}[h]
\centering
\scalebox{0.6}{\includegraphics{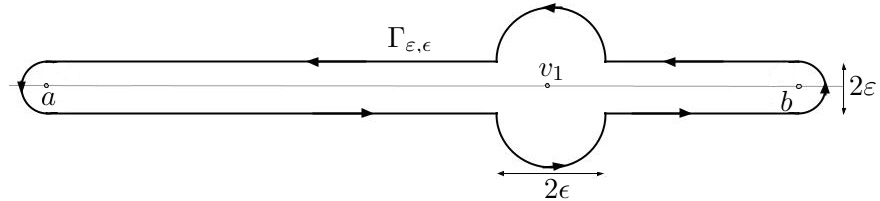}}
\caption{ The contour $\Gamma_{\varepsilon, \epsilon}$.}
\label{S8_2}
\end{figure}

By Cauchy's theorem we see that 
\begin{equation}\label{HF1}
\oint_{\gamma}  \frac{R_2(z_2) }{(v_1- z_2)}\cdot \frac{\sqrt{v_1 - a}\sqrt{b- v_1}}{\sqrt{(z_2- a)(z_2 - b)}}dz_2  = \oint_{\Gamma_{\varepsilon, \epsilon}}  \frac{R_2(z_2) }{(v_1- z_2)}\cdot \frac{\sqrt{v_1 - a}\sqrt{b- v_1}}{\sqrt{(z_2- a)(z_2 - b)}} dz_2.
\end{equation}
We let $\varepsilon \rightarrow 0+$ and conclude
\begin{equation}\label{HF2}
\begin{split}
\lim_{\varepsilon \rightarrow 0+} \oint_{\Gamma_{\varepsilon, \epsilon}}  \frac{R_2(z_2) }{(v_1- z_2)}\cdot \frac{\sqrt{v_1 - a}\sqrt{b- v_1}}{\sqrt{(z_2- a)(z_2 - b)}} dz_2 = T_1(\epsilon) + T_2(\epsilon) + T_3(\epsilon), \mbox{ where }
\end{split}
\end{equation}
\begin{equation}\label{HF3}
\begin{split}
&T_1(\epsilon) = 2 \i \int_{a}^{v_1 - \epsilon} \frac{R_2(v_2)}{(v_1 - v_2)} \cdot \frac{\sqrt{v_1 -a } \sqrt{b- v_1} }{ \sqrt{v_2 -a} \sqrt{b-v_2}}dv_2, \\
& T_2(\epsilon) = 2 \i \int_{v_1 + \epsilon}^b \frac{R_2(v_2)}{(v_1 - v_2)} \cdot \frac{\sqrt{v_1 -a } \sqrt{b- v_1} }{ \sqrt{v_2 -a} \sqrt{b-v_2}}dv_2, \\
&T_3(\epsilon) = \int_{C_{\epsilon}^+(v_1) \cup C_{\epsilon}^-(v_1)} \frac{R_2(z_2) }{(v_1- z_2)}\cdot \frac{\sqrt{v_1 - a}\sqrt{b- v_1}}{\sqrt{(z_2- a)(z_2 - b)}} dz_2,
\end{split}
\end{equation}
with $C^{+}_{\epsilon}(v_1)$, $C^-_{\epsilon}(v_1)$ being positively oriented half-circles of radius $\epsilon$ around $v_1$ in the upper and lower half-planes, respectively. We mention that in deriving (\ref{HF2}) we used the bounded convergence theorem and fact that the square roots are purely imaginary and come with opposite sign when we approach $[a,b]$ from the upper and lower half-planes. \\

We next integrate by parts the integrals in $T_1(\epsilon)$ and $T_2(\epsilon)$ to get
 \begin{equation}\label{HF4}
\begin{split}
&T_1(\epsilon) = -2 \i \int_{a}^{v_1 - \epsilon} R_2'(v_2) H(v_2) dv_2 + 2\i \left[ R_2(v_1 - \epsilon) H(v_1 - \epsilon) - R_2(a) H(a) \right], \\
& T_2(\epsilon) = -2 \i \int_{v_1 + \epsilon}^{b} R_2'(v_2) H(v_2) dv_2 + 2\i \left[ R_2(b) H(b) - R_2(v_1 + \epsilon) H(v_1 + \epsilon) \right], \mbox{ where } \\
&H(v_2) = 2 \log \left(\sqrt{v_1 - a}\sqrt{b-v_2} + \sqrt{v_2-a} \sqrt{b - v_1} \right) - \log|v_1 - v_2| .
\end{split}
\end{equation}
We observe that $H(a) = \log (b-a) = H(b)$ and $R_2(v_1-\epsilon) G(v_1 - \epsilon) - R_2(v_1 + \epsilon) G(v_1 + \epsilon) = O(\epsilon \log \epsilon^{-1})$. The latter statements and the dominated convergence theorem imply
\begin{equation}\label{HF5}
\lim_{\epsilon \rightarrow 0+} T_1(\epsilon) + T_2(\epsilon) = -2 \i \int_a^b R_2'(v_2) H(v_2) dv_2 + 2\i \cdot [R_2(b) - R(a)] \log (b-a).
\end{equation}

We next turn to $T_3(\epsilon)$ and parametrize $C_{\epsilon}^+(v_1)$ through $z_2 = v_1 + \epsilon e^{\i \phi}$ with $\phi \in (0, \pi)$ and $C_{\epsilon}^-(v_1)$ through $z_2 = v_1 + \epsilon e^{\i  \phi}$ with $\phi \in (-\pi, 0)$. This leads to 
$$T_3(\epsilon) = \i \sqrt{v_1 - a}\sqrt{b-v_1} \left[ \int_{0}^{\pi}  \hspace{-2mm} \frac{R_2(v_1 + \epsilon e^{\i \phi}) d\phi}{\sqrt{\epsilon e^{\i \phi} + v_1 - a} \sqrt{\epsilon e^{\i \phi} + v_1 - b}} +\int_{-\pi}^{0} \frac{R_2(v_1 + \epsilon e^{\i \phi}) d\phi}{\sqrt{\epsilon e^{\i \phi} + v_1 - a} \sqrt{\epsilon e^{\i \phi} + v_1 - b}}  \right] \hspace{-1mm}.$$
We can let $\epsilon$ converge to $0+$ above, which by the bounded convergence theorem gives 
\begin{equation}\label{HF6}
\lim_{\epsilon \rightarrow 0+} T_3(\epsilon) = \sqrt{v_1 - a}\sqrt{b-v_1} \cdot  \left[ \int_0^{\pi} \frac{R_2(v_1) d\phi}{\sqrt{v_1 - a} \sqrt{b-v_1}} -   \int_{-\pi}^{0} \frac{R_2(v_1) d\phi}{\sqrt{v_1 - a} \sqrt{b-v_1}} \right] = 0,
\end{equation}
where the sign change comes from the fact that we are approaching the real line from the upper and lower half-planes in the two integrals. Combining (\ref{HF1}), (\ref{HF2}), (\ref{HF5}) and (\ref{HF6}) we conclude 
$$\oint_{\gamma}  \frac{R_2(z_2) }{(v_1- z_2)}\cdot \frac{\sqrt{v_1 - a}\sqrt{b- v_1}}{\sqrt{(z_2- a)(z_2 - b)}}dz_2   = -2 \i \int_a^b R_2'(v_2) H(v_2) dv_2 + 2\i \cdot[R_2(b) - R(a)] \log (b-a).$$
The latter is equivalent to (\ref{HE1}) once we use that $R_2(b) - R_2(a) = \int_a^b R'_2(v_2) dv_2$. 
\end{proof}

%-------------------------------------------------------------------------------------------------------------------------------------------------------------------------------------------------
% Section 8.3
%
%-------------------------------------------------------------------------------------------------------------------------------------------------------------------------------------------------
\subsection{Proof of Theorems \ref{ThmLLN} and \ref{ThmMain}}\label{Section8.3} In this section we give the proofs of Theorems \ref{ThmLLN} and \ref{ThmMain}.

\begin{proof}[Proof of Theorem \ref{ThmLLN}] For clarity we split the proof into three steps. In the first step we make use of Proposition \ref{S32P2} to prove that the CDF of $\mu_n$ from (\ref{S31E1}) is close to that of $\mu(\cdot,n/K)$ from (\ref{S31E11}) -- see equation (\ref{YU3}). In Step 2 we prove an analogue of (\ref{S1LLNEq}) for a different height function, see (\ref{YV2}), and then in Step 3 we deduce (\ref{S1LLNEq}) by applying a suitable change of variables to (\ref{YV2}).\\

{\bf \raggedleft Step 1.} Recall from (\ref{S31E1}) that 
\begin{equation}\label{YU1}
\mu_{n} = \frac{1}{n} \cdot \sum_{i = 1}^n \delta(\ell_i^n/n),
\end{equation}
and let $\mu(x,s)$ be as in (\ref{S31E11}). The goal of this step is to prove that for each $\delta > 0$, there exist positive constants $c_1, C_1 > 0$, depending on $\delta$ and $\theta$, such that if $n \in [\delta K, \delta^{-1} K]$ and $x \in \mathbb{R}$ 
\begin{equation}\label{YU3}
\P\left( \left| \int_x^{\infty} \mu_n(dy) - \int_{x}^{\infty} \mu(y,n/K)dy \right| \geq K^{-1/16} \right) \leq C_1 \cdot e^{-c_1 K^{5/4}}.
\end{equation}
The idea behind proving (\ref{YU3}) is to replace the integrals in (\ref{YU3}) with ones involving a Lipschitz function $g_x(y)$, which is approximating ${\bf 1}\{x \geq y\}$, and use Proposition \ref{S32P2}.\\

We first note that almost surely $\ell_i^n \in [-n \theta, K]$ for $i \in \llbracket 1, n \rrbracket$. The latter and (\ref{S31E11}) imply that $\mu_n$ and $\mu(\cdot, n/K)$ are both supported on $[-\theta, \delta^{-1}]$, and so (\ref{YU3}) is satisfied for any $x \not \in [-\theta, \delta^{-1}]$ for any choice of $c_1, C_1 > 0$ as the difference of the two integrals is exactly zero. In the remainder we assume that $x \in  [-\theta, \delta^{-1}]$, and let $g_x(y)$ be defined as
\begin{equation}\label{YU4}
g_x(y) = \begin{cases} 0 &\mbox{ if } y < x, \\ K^{1/4} \cdot (y - x) &\mbox{ if } y \in [x, x+K^{-1/4}] \\ 1 &\mbox{ if }  y \in [x + K^{-1/4}, \delta^{-1}] \\ 1 - K^{1/4} (y - \delta^{-1}) &\mbox{ if } y \in [\delta^{-1}, \delta^{-1} + K^{-1/4}] \\ 0 &\mbox{ if } y \geq \delta^{-1} + K^{-1/4}. \end{cases}
\end{equation}
Note that since $\ell_i^n - \ell_{i+1}^n \geq \theta$ for all $i \in \llbracket 1, n-1\rrbracket$ and $\mu(y, n/K)$ is at most $\theta^{-1}$, we can find a constant $a_1 > 0$, depending on $\theta$, such that 
\begin{equation}\label{YU5}
\begin{split}
&a_1 \cdot K^{-1/4} \geq \int_x^{\infty} \mu_n(dy) - \int_{\mathbb{R}} g_x(y) \mu_n(dy) \geq 0, \mbox{ and }\\
&a_1 \cdot K^{-1/4} \geq \int_x^{\infty} \mu(y, n/K)dy - \int_{\mathbb{R}} g_x(y) \mu(y, n/K)dy \geq 0.
\end{split}
\end{equation}

Recalling the notation in (\ref{S32E3}) we observe that 
\begin{equation}\label{YU6}
\|{g}_x\|_{\operatorname{Lip}} = K^{1/4}.
\end{equation}
In addition, observe that 
$$[g_x]_{H^{1/2}(\mathbb{R})} := \int_{\mathbb{R}} \int_{\mathbb{R}} \frac{| g_x(y_1) -  g_x(y_2)|^2}{|y_1 - y_2|^2} \leq K^{1/2} \cdot (\delta^{-1} + \theta + 2)^2.$$
On the other hand, as can be deduced from the proof of \cite[Proposition 3.4]{NPV} we have
$$[g_x]_{H^{1/2}(\mathbb{R})}  = 2 C(1,1/2)^{-1} \cdot \|g_x \|_{1/2}^2, \mbox{ where }C(1,1/2) = \int_{\mathbb{R}} \frac{1 - \cos (x)}{x^2}dx \in (0, \infty),$$ 
which implies that we can find a positive constant $a_2 >0 $, depending on $\theta$ and $\delta$, such that
\begin{equation}\label{YU7}
\|g_x \|_{1/2} \leq a_2 \cdot K^{1/4}.
\end{equation}
Combining (\ref{YU6}), (\ref{YU7}) and Proposition \ref{S32P2}, applied to $p = 2$, $\gamma = K^{-3/8}$, we conclude that there is a constant $C_2 > 0$, depending on $\theta, \delta$, such that for $n \geq 2$ and $n/K \in [\delta, \delta^{-1}]$
\begin{equation}\label{YU8}
\begin{split}
& \P  \left(\left|  \int_{\mathbb{R}}  g_x(y)\mu_{n}(dy)- \int_{\mathbb{R}} g_x(y)\mu(y, n/K) dy \right|\ge a_2 K^{-1/8}+ \theta K^{1/4} n^{-2} \right)  \\
& \le  \exp \left(-2\pi^2 \theta K^{-3/4} n^2+ C_2 \cdot  n\log n  \right),
\end{split}
\end{equation}
Equations (\ref{YU5}) and (\ref{YU8}) together imply (\ref{YU3}).\\

{\bf \raggedleft Step 2.} For $x \in \mathbb{R}$ and $s > 0$ we define 
\begin{equation}\label{YV1}
\hat{\mathsf{h}}(x,s) = \int_x^{\infty} \mu(y,s) dy \mbox{ and } \hat{\mathsf{h}}_K(x,s) = \int_x^{\infty} \mu_n(dy),\mbox{ where } n = \lceil s \cdot K \rceil.
\end{equation}
Fix a compact set $\mathcal{V}_1 \subset \mathbb{R} \times (0, \infty)$. In this step we prove that for any $\epsilon > 0$ we have 
\begin{equation}\label{YV2}
\lim_{K \rightarrow \infty} \P \left( \sup_{(x,s) \in \mathcal{V}_1} \left| \hat{\mathsf{h}}_K(x,s)  - \hat{\mathsf{h}}(x,s) \right| > \epsilon \right) = 0.
\end{equation}

Let $\delta \in (0,1/2), \delta_0 = 2\delta$ be small enough so that $\mathcal{V}_1 \subset [-\delta_0^{-1}, \delta_0^{-1}] \times [\delta_0, \delta_0^{-1}]$. Let $\mathsf{X} = \{ x \in K^{-1} \cdot \mathbb{Z}: x \in [-\delta_0^{-1}-1/K, \delta_0^{-1} + 1/K] \}$, $\mathsf{S} = \{s \in K^{-1} \cdot \mathbb{Z}: \delta_0 \leq s \leq \delta_0^{-1} + K^{-1} \}$, and note that we can find a constant $A_1 > 0$, depending on $\delta$, such that 
\begin{equation}\label{YV3}
|\mathsf{X}| \leq A_1 \cdot K \mbox{ and } |\mathsf{S}| \leq A_1 \cdot K
\end{equation}
By taking a union bound of (\ref{YU3}) and using (\ref{YV3}) we conclude that we can find a constant $B_1 > 0$, depending on $\delta, \theta$, such that for $K \geq \delta_0^{-1}$
\begin{equation}\label{YV3.5}
\P\left( \max_{x \in \mathsf{X}, s \in \mathsf{S}} \left| \int_x^{\infty} \mu_{sK} (dy) - \int_{x}^{\infty} \mu(y,s)dy \right| \geq K^{-1/16} \right) \leq B_1 K^2 \cdot e^{-c_1 K^{5/4}}.
\end{equation}

We can also find a constant $A_2 > 0$, depending on $\theta, \delta$, such that if $(x,s) \in  [-\delta_0^{-1}, \delta_0^{-1}] \times [\delta_0, \delta_0^{-1}]$
\begin{equation}\label{YV4}
\left| \hat{\mathsf{h}}_K(x,s) - \hat{\mathsf{h}}_K(\bar{x},s) \right| \leq A_2 \cdot K^{-1},
\end{equation}
where $\bar{x} = \max\{ y \in \mathsf{X}: y \leq x\}$. In deriving the last expression we used that $\ell_i^n - \ell_{i+1}^n \geq \theta$ for all $n \geq 1$ and $i \in \llbracket 1, n-1 \rrbracket$. As a consequence of the bounded convergence theorem, we have that $\hat{\mathsf{h}}(x,s) $ is continuous on $\mathbb{R} \times (0,\infty)$ and hence uniformly continuous on $[-\delta^{-1}, \delta^{-1}] \times [\delta, \delta^{-1}]$. In particular, we can find $\eta > 0$ such that for $(x,s), (y,t) \in [-\delta^{-1}, \delta^{-1}] \times [\delta, \delta^{-1}]$
\begin{equation}\label{YV5}
\left| \hat{\mathsf{h}}(x,s) - \hat{\mathsf{h}}(y,t) \right| \leq \epsilon/2 \mbox{ if } |x-y| + |s - t| < \eta.
\end{equation}

Let $K_0$ be sufficiently large so that $K_0 \geq \delta_0^{-1}, 2 K_0^{-1} < \eta$ and $ K_0^{-1/16} + A_2 K_0^{-1} < \epsilon/2$. Combining (\ref{YV1}), (\ref{YV3.5}), (\ref{YV4}) and (\ref{YV5}) we conclude for $K \geq K_0$ that
\begin{equation*}
\begin{split}
&\P \left( \sup_{(x,s) \in \mathcal{V}_1} \left| \hat{\mathsf{h}}_K(x,s)  - \hat{\mathsf{h}}(x,s) \right| > \epsilon \right)  \leq \P \left( \sup_{(x,s) \in \mathcal{V}_1} \left| \hat{\mathsf{h}}_K(\bar{x},\bar{s} )  - \hat{\mathsf{h}}(\bar{x}, \bar{s} ) \right| > \epsilon/2 - A_2 \cdot K^{-1}  \right)\\
& \leq  \P \left( \max_{x \in \mathsf{X}, s \in \mathsf{S}} \left| \hat{\mathsf{h}}_K(x,s )  - \hat{\mathsf{h}}(x, s ) \right| > \epsilon/2 - A_2 \cdot K^{-1}  \right) \\
&\leq  \P \left( \max_{x \in \mathsf{X}, s \in \mathsf{S}} \left| \hat{\mathsf{h}}_K(x,s )  - \hat{\mathsf{h}}(x, s ) \right| >  K^{-1/16}  \right) \leq B_1 K^2 \cdot e^{-c_1 K^{5/4}},
\end{split}
\end{equation*}
where as before $\bar{x} = \max\{ y \in \mathsf{X}: y \leq x\}$ and also $\bar{s} = K^{-1} \cdot \lceil s K \rceil$. The last equation implies (\ref{YV2}).\\

{ \bf \raggedleft Step 3.} In this step we fix a compact set $\mathcal{V} \subset \mathbb{R} \times (0, \infty)$ and prove (\ref{S1LLNEq}). From the definition of $\mu(x,s)$ in (\ref{S31E11}) and $\nu(x,s)$ in (\ref{S1Density}) we observe that 
\begin{equation*}
\nu\left(x,s \right) = \mu\left( (\theta/s) \cdot  \left( x - (s-1)/2 \right), s \theta^{-1} \right),
\end{equation*}
which implies that 
\begin{equation}\label{YW1}
\begin{split}
&h(x,s) = \int_{x}^{\infty} \nu(y,s) dy = (s/\theta) \cdot \int_{(\theta/s)(x - (s-1)/2) }^{\infty} \mu(y, s \theta^{-1})dy \\
&= (s/\theta) \cdot \hat{\mathsf{h}}\left((\theta/s)(x - (s-1)/2), s \theta^{-1}\right) .
\end{split}
\end{equation}
We also note from (\ref{S1DHF}) that for $n = \lceil s \cdot \theta^{-1} \cdot K \rceil$  
\begin{equation}\label{YW2}
\begin{split}
&\mathcal{H}_K(x,s) = \sum_{i = 1}^n {\bf 1}\{\tilde{\ell}_i^n/K \geq x\} = \sum_{i = 1}^n {\bf 1}\{\ell_i^n/n  \geq Kx/n + K/(2n) - \theta(n+1)/(2n) \} \\
& = n \cdot \int_{(\theta/s)(x - (s-1)/2) }^{\infty} \mu_n(dx) + O(1) = K \cdot (s/\theta) \cdot \hat{\mathsf{h}}_K\left( (\theta/s)(x - (s-1)/2), s\theta^{-1} \right) + O(1),
\end{split}
\end{equation} 
where the constants in the big $O$ notations depend on $\theta, \mathcal{V}$ and (\ref{YW2}) holds if $(x,s) \in \mathcal{V}$. 

We define the set 
$$\mathcal{V}_1 = \{(y,t): y = (\theta/s) \cdot (x-(s-1)/2) \mbox{ and } t = s \theta^{-1} \mbox{ for some } (x,s) \in \mathcal{V} \}.$$
Since $\mathcal{V}$ is a compact subset of $\mathbb{R} \times (0, \infty)$, we conclude the same of $\mathcal{V}_1$. From (\ref{YW1}) and (\ref{YW2}) we conclude that there is a constant $C > 0$, depending on $\theta, \mathcal{V}$, such that
\begin{equation*}
\sup_{(x,s) \in \mathcal{V}} \left| K^{-1} \cdot  \mathcal{H}_K(x,s) - h(x,s) \right|  \leq C K^{-1} + \sup_{(x,s) \in \mathcal{V}_1} (s/\theta) \cdot \left| \hat{\mathsf{h}}_K(x,s)  - \hat{\mathsf{h}}(x,s) \right|.
\end{equation*}
Combining the last equation with (\ref{YV2}) we conclude (\ref{S1LLNEq}). This suffices for the proof.
\end{proof}

\begin{proof}[Proof of Theorem \ref{ThmMain}] The idea is to recast the problem into the setup of Theorem \ref{MainThmNew}, apply that result, and then perform a simple change of variables to conclude (\ref{S1CovGFF}). We first note from (\ref{S1DHF}) and (\ref{S8Height}) that for $s > 0$
$$\mathcal{H}_K(x,s) = \hat{\mathcal{H}}_K\left( x + \frac{1}{2} - \frac{\theta (n+1)}{2K} , \theta^{-1} \cdot s \right), \mbox{ where } n = \lceil s \cdot \theta^{-1} \cdot K \rceil.$$
The latter shows that if $X^K_i$ are as in the statement of the theorem, we have
$$X_i^K = \int_{\mathbb{R}} \sqrt{\theta \pi} \left( \hat{\mathcal{H}}_K(x, \theta^{-1}  s_i) - \mathbb{E} \left[ \hat{\mathcal{H}}_K(x, \theta^{-1} s_i )  \right] \right) \tilde{f}^K_i(x) dx, \mbox{ where }$$
$$\tilde{f}^K_i(x) =  f_i \left(x - \frac{1}{2} + \frac{\theta ( \lceil s_i \cdot \theta^{-1} \cdot K \rceil+1)}{2K} \right).$$
Note that the conditions of Theorem \ref{MainThmNew} are satisfied with $s_i^K = \theta^{-1} s_i$ and $\tilde{f}_i^K(x)$, which converge to $\tilde{f}_i(x) = f_i(x-1/2 + s_i/2)$ uniformly over compacts in $\mathbb{C}$. We conclude that $X^K = (X^K_1, \dots, X^K_m)$ converge in the sense of moments to a Gaussian vector $(X_1, \dots, X_m)$ with mean zero and covariance
\begin{equation}\label{S8IP1}
\begin{split}
&\mathrm{Cov}(X_i, X_j) = \int_{a(s_i/\theta)}^{b(s_i/\theta)} \int_{a(s_j/\theta)}^{b(s_j/\theta)} \tilde{f}_i(y_1) \tilde{f}_j(y_2) \left( - \frac{1}{2\pi} \log \left| \frac{\hat{\Omega}(y_1,s_i/\theta) - \hat{\Omega}(y_2,s_j/ \theta )  }{\hat{\Omega}(y_1,s_i/\theta) - \overline{\hat{\Omega}}(y_2,s_j/\theta ) } \right|\right) dy_2 dy_1.
\end{split}
\end{equation}
We may now apply the change of variables $y_1 = x_1 + (1/2)(1-s_i)$ and $y_2 = x_2 + (1/2) (1-s_j)$ to (\ref{S8IP1}) and conclude
\begin{equation}\label{S8IP2}
\begin{split}
&\mathrm{Cov}(X_i, X_j) = \int_{-\sqrt{s_i}}^{\sqrt{s_i}} \int_{-\sqrt{s_j}}^{\sqrt{s_j}} f_i(x_1) f_j(x_2) \\
&\times  \left( - \frac{1}{2\pi} \log \left| \frac{-x_1 + \i \sqrt{s_i - x_1^2} +x_2 - \i  \sqrt{s_j - x_2^2}  }{-x_1 + \i \sqrt{s_i - x_1^2} +x_2 + \i  \sqrt{s_j - x_2^2}  } \right|\right) dx_2 dx_1,
\end{split}
\end{equation}
where we also used the formula for $a(s), b(s)$ from (\ref{S8DefAB}) and for $\hat{\Omega}$ from (\ref{WA2}). The right side of (\ref{S8IP2}) is readily seen to equal the right side of (\ref{S1CovGFF}) as the two absolute values agree in view of the definition of $\Omega(x,s)$ in (\ref{S1Map}). This suffices for the proof.
\end{proof}

%-------------------------------------------------------------------------------------------------------------------------------------------------------------------------------------------------
% Appendix A
%
%-------------------------------------------------------------------------------------------------------------------------------------------------------------------------------------------------
\begin{appendix}

%-------------------------------------------------------------------------------------------------------------------------------------------------------------------------------------------------
% Appendix A.1
%
%-------------------------------------------------------------------------------------------------------------------------------------------------------------------------------------------------
\section{Joint cumulants} \label{AppendixA} In this section we summarize some notation and results about joint cumulants of random variables. For more background we refer the reader to \cite[Chapter 3]{Taqqu}. For $n$ bounded complex-valued random variables $X_1, \dots, X_n$ we let $M(X_1, \dots, X_n)$ denote their joint cumulant. Explicitly,
\begin{equation}\label{CumDef}
M(X_1, \dots, X_n) = \frac{\partial^n}{\partial z_1 \cdots \partial z_k} \log \mathbb{E} \left[ \exp \left( \sum_{i = 1}^n z_i \cdot X_i \right) \right] \Bigg{\vert}_{z_1 = \cdots = z_k = 0}.
\end{equation}
If $n = 1$, then $M(X_1) = \mathbb{E}[X_1]$. For every subset $J = \{j_1, \dots, j_k \} \subseteq \{ 1, \dots, n \}$ we write 
$$\bX_J = \{ X_{j_1}, \dots, X_{j_k} \} \mbox{ and } \bX^J = X_{j_1} \times\cdots \times X_{j_k},$$
where $\times$ denotes the usual product. If $Y$ is a bounded random variable we also write 
$$M(\bX_J) = M (X_{j_1}, \dots, X_{j_k}) \mbox{, } \mathbb{E}\left[ \bX^J \right] = \mbox{$\mathbb{E} \left[ \prod_{j \in J} X_j \right]$}, \mbox{ and } M(Y;\bX_J) = M(Y, X_{j_1}, \dots, X_{j_k}) $$
where we remark that the definitions make sense as the joint cumulant of a set of random variables $\bX_J$ is invariant with respect to permutations of $J$.

We record the following basic properties of $M(X_1, \dots, X_n)$, see \cite[Section 3.1]{Taqqu}. If $Y$ is a bounded complex-valued random variable and $c_1, \dots, c_n \in \mathbb{C}$ we have 
\begin{equation}\label{S3Linearity}
\begin{split}
&M(X_1 + Y,  X_2, \dots,  X_n )  = M(X_1, \dots, X_n)  +   M(Y, X_2, \dots, X_n) , \\
&M(c_1 X_1, \dots, c_n X_n) = \prod_{i = 1}^n c_i \cdot M(X_1, \dots, X_n) \\
&M( X_1 + c_1, \dots,  X_n + c_n) = M(X_1, \dots, X_n) \mbox{ if $n \geq 2$, while }M(X_1 + c_1) = M(X_1) + c_1.\\
\end{split}
\end{equation}

From \cite[Proposition 3.2.1]{Taqqu} we have the following identities that relate joint moments and joint cumulants of the bounded complex-valued random variables $X_1, \dots, X_n$. 
\begin{lemma}\label{CumToMom} For any $J \subseteq \llbracket 1, n \rrbracket$ we have
\begin{equation}\label{Mal1}
\begin{split}
&\mathbb{E} \left[ \bX^J \right] = \sum_{ \pi = \{ b_1, \dots, b_k \} \in \mathcal{P}(J)} M(\bX_{b_1}) \cdots M(\bX_{b_k})\\
\end{split}
\end{equation}
\begin{equation}\label{Mal2}
\begin{split}
&M(\bX_J) = \sum_{\sigma = \{a_1, \dots, a_r\} \in \mathcal{P}(J)} (-1)^{r-1} (r-1)! \cdot \mathbb{E}\left[ \bX^{a_1} \right] \cdots \mathbb{E} \left[ \bX^{a_r} \right],
\end{split}
\end{equation}
where $\mathcal{P}(J)$ is the set of partitions of $b$, see \cite[Section 2.2]{Taqqu}. If $X,Y$ are bounded complex-valued random variables we also have
\begin{equation}\label{Mal3}
\begin{split}
&M(XY,X_1, \dots, X_n) = M(X,Y,X_1, \dots, X_n) + \sum_{I \subseteq \llbracket 1, n \rrbracket} M(X;\bX_I ) \cdot M(Y;\bX_{I^c} ),
\end{split}
\end{equation}
where $I^c = \llbracket 1, n \rrbracket \setminus I$.
\end{lemma}

%-------------------------------------------------------------------------------------------------------------------------------------------------------------------------------------------------
% Appendix A.2
%
%-------------------------------------------------------------------------------------------------------------------------------------------------------------------------------------------------
\section{Guessing the covariance} \label{AppendixB} In this section we explain how the formula for the limiting covariance $C(z_1, s_1; z_2, s_2)$ in (\ref{LimCov}) was guessed. The goal is to provide a heuristic argument rather than a formal derivation, and so we will not dwell on technical points such as well-posedness of certain functions, existence of enough derivatives, divisions by zero etc. Throughout we assume that all $s$ variables lie in $[\delta, \delta^{-1}]$ and $z$ variables lie in $\mathcal{U}= \mathbb{C}\setminus [-\theta \delta^{-1}, \delta^{-1}]$ for some $\delta \in (0,1)$.  \\

As mentioned in Remark \ref{S7SingleLev}, the formula for $C(z_1, s_1;z_2, s_1)$ was known from previous works. On the other hand, the argument in the proof of Theorem \ref{MainTechThm} suggests that if $f(z,s) = C(z_1, s_1; z, s)$, then $f(z,s)$ satisfies the integral equation from (\ref{S5LID}), namely
\begin{equation*}
  \left[e^{\theta G\left(z/s_1, s_1\right)} - 1 \right] f(z,s_1) + \theta \cdot \int_{s}^{s_1}  \partial_{z} f(z,u)du -   \left[e^{\theta G\left(z/s, s\right)} - 1 \right] f(z,s)  = 0.
\end{equation*}
Differentiating the latter with respect to $s$ we conclude that $f(z,s)$ satisfies the differential equation
\begin{equation}\label{PO1}
\begin{split}
&\partial_s f(z,s) + f(z,s) \cdot \frac{\theta e^{\theta G\left(z/s, s\right)}  \left[\partial_s G\left(z/s, s\right) - (z/s^2) \cdot \partial_z G\left(z/s, s\right) \right]  }{\left[e^{\theta G\left(z/s, s\right)} - 1 \right]} +  \frac{ \theta \cdot \partial_z f(z,s)  }{\left[e^{\theta G\left(z/s, s\right)} - 1 \right]}    = 0,
\end{split}
\end{equation}
with initial condition $f(z,s_1) = C(z_1, s_1;z, s_1)$. Equation (\ref{PO1}) is a first order linear PDE, which can be solved using the method of characteristics as we explain next. 

For $x_0 \in \mathcal{U}$ we let $X_{x_0}: [\delta , s_1] \rightarrow \mathbb{C}$ be the solution to the ODE
\begin{equation}\label{PO2}
X'(s) = \theta \left[\exp \left(\theta G \left( X(s)/s, s \right) \right) - 1 \right]^{-1},
\end{equation}
with initial condition $X(s_1) = x_0$. Setting $v_{x_0}(s) =f(X_{x_0}(s),s)$ from (\ref{PO1}) we see $v_{x_0}(s)$ solves
\begin{equation}\label{PO3}
\begin{split}
&v'(s) + v(s) \cdot \frac{\theta e^{\theta G\left(X_{x_0}(s)/s, s\right)} \left[\partial_s G\left(X_{x_0}(s)/s, s\right) - (X_{x_0}(s)/s^2)\partial_z G\left(X_{x_0}(s)/s, s\right)  \right]  }{\left[e^{\theta G\left(X_{x_0}(s)/s, s\right)} - 1 \right]}  = 0,\\
&\mbox{ with initial condition } v(s_1)=  C(z_1, s_1; x_0,s_1).
\end{split}
\end{equation}
Equation (\ref{PO3}) is a linear ODE with solution given by 
\begin{equation}\label{PO4}
\begin{split}
&v(s) = C(z_1, s_1; x_0,s_1)  \\
& \times \exp \left( \int_{s}^{s_1} \frac{\theta e^{\theta G\left(X_{x_0}(u)/u, u\right)} \left[\partial_s G\left(X_{x_0}(u)/u,u\right) - (X_{x_0}(u)/u^2)\partial_z G\left( X_{x_0}(u)/u,u\right)  \right]  }{\left[e^{\theta G\left(X_{x_0}(u)/u, s\right)} - 1 \right]} du \right).
\end{split}
\end{equation}

We also note that (\ref{PO2}) has integral curves
\begin{equation}\label{PO5}
\exp \left(\theta G \left( X(s)/s, s \right) \right) = C.
\end{equation}
Indeed, differentiating the left side of (\ref{PO5}) with respect to $s$, we get
\begin{equation*}
\begin{split}
&\theta \exp \left(\theta G \left( X(s)/s, s \right) \right)  \cdot \left( \partial_s G \left( X(s)/s, s \right) - (X(s)/s^2) \partial_z G \left( X(s)/s, s \right) + (X'(s)/s)  \partial_z G \left( X(s)/s, s \right) \right)  \\
&\theta \exp \left(\theta G \left( X(s)/s, s \right) \right) \cdot \left[ -\frac{\theta \cdot  \partial_z G(X(s)/s,s)}{s\left(e^{\theta G(X(s)/s,s)} - 1 \right) }  + (X'(s)/s)  \partial_z G \left( X(s)/s, s \right)\right] = 0,
\end{split}
\end{equation*}
where in going from the first to the second line we used (\ref{S7SpecEqn2}) which states that
\begin{equation}\label{SASPecEqn2}
\begin{split}
 &\partial_s G(z/s,s) - (z/s^2) \partial_z G(z/s,s) =  -\frac{\theta \cdot  \partial_z G(z/s,s)}{s\left(e^{\theta G(z/s,s)} - 1 \right) },
\end{split}
\end{equation}
and in the last equality we used that $X(s)$ solves (\ref{PO2}). Using (\ref{PO5}) and (\ref{SASPecEqn2}) we rewrite (\ref{PO4}) as 
\begin{equation}\label{PO6}
\begin{split}
&v(s) = C(z_1, s_1; x_0,s_1) \cdot \exp \left( -\int_{s}^{s_1} \frac{\theta^2 e^{\theta G\left(X_{x_0}(u)/u, u\right)} \partial_z G(X_{x_0}(u)/u,u) }{u\left[e^{\theta G\left(X_{x_0}(u)/u, u\right)} - 1 \right]^2} du \right)  \\
&  = C(z_1, s_1; x_0,s_1) \cdot \exp \left( - \frac{\theta^2 e^{\theta G\left(x_0/s_1, s_1\right)}}{\left[e^{\theta G\left(x_0/s_1, s_1\right)} - 1 \right]^2} \int_{s}^{s_1} \frac{\partial_z G(X_{x_0}(u)/u,u) }{u} du \right).
\end{split}
\end{equation}
We now let $C_{x_0} = e^{\theta G(x_0/s_1, s_1)}$ and note from (\ref{S31E14}) that 
\begin{equation}\label{PO7}
\begin{split}
e^{\theta G(z/s,s)} =\frac{(1 - s\theta) - 2 \sqrt{(z - sz_+(s))(z - sz_-(s))}  }{2 (1 - z) } = \frac{(1 - s\theta) - 2 \sqrt{(z - a(s))(z - b(s))}  }{2 (1 - z) },
\end{split}
\end{equation}
where from (\ref{S31E12}) we have 
\begin{equation}\label{PO8}
a(s) = sz_-(s) = (1/2)(1 - \theta s) - \sqrt{\theta s} \mbox{ and }b(s) = sz_+(s) = (1/2)(1 - \theta s) + \sqrt{\theta s} .
\end{equation}

We now have
$$e^{\theta G(z/s, s)} = C_{x_0} \iff ( 1- s\theta) - 2\sqrt{(z - a(s))(z - b(s))}  = 2 C_{x_0} (1 - z) $$
$$\implies 4 (z- a(s)) (z-b(s)) = [2C_{x_0} (1-z) - ( 1- s\theta) ]^2.$$
The above sets up a quadratic equation for $z$ with two roots that are given by
$$z= 1 \mbox{ and } z = \frac{C_{x_0} \cdot \theta s + C_{x_0}^2 + \theta s - C_{x_0}}{C_{x_0}^2 - 1}.$$
Using the latter observation and (\ref{PO5}) we conclude that the solution to (\ref{PO2}) is precisely
\begin{equation}\label{PO9}
X_{x_0}(s) = \frac{C_{x_0} \cdot \theta s + C_{x_0}^2 + \theta s - C_{x_0}}{C_{x_0}^2 - 1} = \frac{C_{x_0}}{C_{x_0} + 1} + \frac{\theta s}{C_{x_0} - 1} .
\end{equation}

Our next task is to compute the integral in the second line of (\ref{PO6}). By direct computation using (\ref{PO7}), (\ref{PO9}) and the identities
$$\frac{\theta \partial_z G(z/u,u) }{u} = \frac{\partial_z \left( e^{\theta G(z/u, u)} \right)}{e^{\theta G(z/u, u)}} \mbox{ and } C_{x_0} = e^{\theta G(X_{x_0}(s_1)/s_1, s_1)} =e^{\theta G(X_{x_0}(u)/u, u)}  ,$$
we obtain
\begin{equation*}
\frac{\theta \partial_z G(X_{x_0}(u)/u,u) }{u} = - \frac{(C_{x_0}^2 - 1) (\theta s - 1) \cdot \sqrt{P} + s^2 (C_{x_0} + 1)^2 \theta^2 +(6-2C_{x_0}^2) + (C_{x_0} - 1)^2}{\sqrt{P} \cdot (s (C_{x_0}+1) \theta - C_{x_0} + 1) \cdot (\theta s + \sqrt{P}) },
\end{equation*}
where 
\begin{equation*}
P = \left( \frac{s \theta (C_{x_0} +1)^2 - (C_{x_0} - 1)^2 }{C_{x_0}^2 - 1 } \right)^2.
\end{equation*}
Opening $\sqrt{P}$ with positive sign and performing some simplifications we arrive at
\begin{equation}\label{PO10}
\frac{\theta \partial_z G(X_{x_0}(u)/u,u) }{u}  = - \frac{(C_{x_0} - 1)^2 (C_{x_0} + 1)^2}{C_{x_0} \cdot (C_{x_0}^2 \theta u + 2 C_{x_0} \theta u - C_{x_0}^2 + \theta u + 2 C_{x_0} - 1)}.
\end{equation}
Using (\ref{PO10}) we compute directly
\begin{equation}\label{PO11}
\begin{split}
\int_{s}^{s_1} \frac{\theta^2 \partial_z G(X_{x_0}(u)/u,u) }{u} du = - \frac{(C_{x_0} - 1)^2 }{C_{x_0}} \cdot \int_{s}^{s_1} \frac{\theta (C_{x_0} + 1)^2 du}{(C_{x_0} + 1)^2 \theta u - (C_{x_0} - 1)^2 }  \\
= - \frac{(C_{x_0} - 1)^2 }{C_{x_0}} \cdot \left[ \log \left( (C_{x_0} + 1)^2 \theta s_1 - (C_{x_0} - 1)^2  \right) -  \log \left( (C_{x_0} + 1)^2 \theta s - (C_{x_0} - 1)^2  \right) \right].
\end{split}
\end{equation}

Substituting (\ref{PO11}) into (\ref{PO6}) and using that $C_{x_0} = e^{\theta G(x_0/s_1, s_1)}$ we conclude
\begin{equation}\label{PO12}
f(X_{x_0}(s), s) = v(s)  =  C(z_1, s_1; x_0,s_1)  \cdot \frac{(C_{x_0} + 1)^2 \cdot \theta s_1 - (C_{x_0} - 1)^2}{(C_{x_0} + 1)^2 \cdot \theta s - (C_{x_0} - 1)^2}.
\end{equation}
Setting $X_{x_0}(s) = z$ we can backwards engineer from the above equations
$$C_{x_0} = e^{\theta G(z/s,s)} \mbox{ and } x_0 = X_{x_0}(s_1)  = z + \frac{\theta (s_1 - s)}{C_{x_0} - 1} = F(z; s_1,s),$$
where we recall that $F(z;s,t)$ was introduced in (\ref{Transport}). With the latter identities (\ref{PO12}) suggests the formula
\begin{equation}\label{PO13}
\begin{split}
&f(z, s) = C(z_1, s_1; F(z; s_1,s),s_1)  \cdot  \frac{\left(e^{\theta G(z/s,s)} + 1 \right)^2 \cdot \theta s_1 - \left(e^{\theta G(z/s,s)} - 1 \right)^2}{\left(e^{\theta G(z/s,s)} + 1\right)^2 \cdot \theta s - \left(e^{\theta G(z/s,s)} - 1 \right)^2},
\end{split}
\end{equation}
which is precisely the formula in (\ref{LimCov}).
\end{appendix}

\bibliographystyle{alpha}
\bibliography{PD}

\end{document}